\documentclass{article}
\usepackage{amsmath}
\usepackage{amssymb}
\usepackage{titlesec}
\usepackage{graphicx}
\usepackage{fancyhdr}
\usepackage{latexsym}
\usepackage{mathrsfs}%
\usepackage{booktabs}
\usepackage{mathrsfs}
\usepackage{mathtools}
\usepackage{amsthm}
\usepackage{wasysym}
\usepackage{cite}
\usepackage{color}
\usepackage{amsfonts}
\usepackage{amsmath,amssymb}
\usepackage{graphics}
\usepackage{graphics}
\usepackage{cite}
\usepackage{latexsym}
\usepackage{amsmath}
\usepackage{amssymb}
\usepackage{epsfig}
\usepackage{amscd}
\usepackage{amsmath, amsthm}
\usepackage[linkcolor=blue,colorlinks=true,urlcolor=red,bookmarksopen=true]{hyperref}
\numberwithin{equation}{section}
\newcommand{\R}{\mathbb{R}}
\newcommand{\I}{\mathbb{I}}
\newcommand{\A}{\mathcal{A}}
\newcommand{\LL}{\mathcal{L}}
\newcommand{\Kop}{\mathcal{K}}
\newcommand{\diver}{\operatorname{div}}
\newcommand{\ubar}{\underline{\rho}_0}
\newcommand{\obar}{\overline{\rho}_0}
\newcommand{\rstar}{\rho_*}
\newcommand{\norm}[1]{\left\|#1\right\|}
\newcommand{\divg}{{\rm  div}}

\newtheorem{theorem}{Theorem}[section]
\newtheorem{lemma}{Lemma}[section]
\newtheorem{remark}{Remark}[section]
\newtheorem{proposition}{Proposition}[section] 

\newtheorem{corollary}{Corollary}[section]

\title{Formation of Implosion Singularities in 3D Compressible Navier--Stokes--Korteweg Equations}
\author{Yongteng G{\small U}$^{a}$\thanks{Email addresses: guyongteng@amss.ac.cn (Y. T. Gu); xdhuang@amss.ac.cn (X. D. Huang)}, Xiangdi H{\small UANG}$^{a}$   \\  
	{\normalsize a. State Key Laboratory of Mathematical Sciences, Academy of Mathematics and Systems Sciences,}\\
	{\normalsize Chinese Academy of Sciences, Beijing 100190, China;}\\
}
\date{}

\begin{document}
	\maketitle
	\begin{abstract}
		We investigate finite-time singularity formation for the three-dimensional compressible Navier--Stokes--Korteweg system with density-dependent viscosity and capillarity coefficients
		\[
		\mu(\rho)=\nu\rho^\alpha,\qquad
		\lambda(\rho)=2\nu(\alpha-1)\rho^\alpha,\qquad
		\kappa(\rho)=\varepsilon^2\alpha^2\rho^{2\alpha-3}.
		\]
		Previous works of Gu--Huang--Meng--Zhou~\cite{Gu-Huang-Meng-Zhou} and Huang--Lei--Zhou~\cite{Huang-Lei-Zhou} established global strong solutions away from vacuum for arbitrarily large initial data when $\alpha$ lies in a suitable range. In contrast, we show that, for a class of small positive exponents $\alpha$ $(\alpha<\frac{1}{2}$), there exist smooth initial data with density uniformly separated from vacuum whose corresponding solutions develop finite-time implosion singularities.
		
		Our construction is based on smooth self-similar imploding profiles of the compressible Euler equations. After reformulating the system in self-similar coordinates, the viscous and capillary effects appear as exponentially decaying perturbations. We control the resulting non-autonomous system through weighted high-order energy estimates, a stable--unstable decomposition of the linearized operator, and a finite-dimensional selection of the unstable components. The constructed solutions remain smooth before the singular time and converge, after rescaling, to the prescribed imploding profile. In particular, at the blowup time $T$, the density becomes infinite at the origin, while the effective velocity $u + d \alpha \rho^{\alpha-2} \nabla \rho$ is unbounded in every neighborhood of the origin. These results complement the aforementioned global existence theory and exhibit a distinct finite-time blowup mechanism for the small-$\alpha$ regime, where the effective bulk-viscosity structure may no longer be positive.
	\end{abstract}
	\tableofcontents
	\section{Introduction}
	In this paper, we study the compressible Navier--Stokes--Korteweg system
	\begin{equation}\label{eq:1.1}
		\begin{cases}
			\rho_t+\operatorname{div}(\rho u)=0,\\[2mm]
			(\rho u)_t+\operatorname{div}(\rho u\otimes u)+\nabla P
			=\operatorname{div}(2\mu(\rho)\mathbf D u)
			+\nabla(\lambda(\rho)\operatorname{div}u)
			+\operatorname{div}\mathbf K ,
		\end{cases}
	\end{equation}
	and the initial and far-field conditions:
	\begin{equation}
		(\rho,u)(0,x)=(\rho_0,u_0)(x) \qquad \text{for}\,\, x \in \mathbb R^3,
	\end{equation}
	\begin{equation}
		(\rho,u)(t,x) \rightarrow(\rstar,0)  \qquad \text{as} \,\,|x|\rightarrow\infty \, .
	\end{equation}
	Here $t\geq0$ denotes the time variable, $x\in\mathbb R^3$  is the spatial variable, $\rho$ and $u=(u_1,u_2,u_3)^\top$ stand for the density and velocity of the fluid, respectively, $\rstar>0$ is a positive constant, and the pressure is given by
	\[
	P=\frac{1}{\gamma}\rho^\gamma,\qquad \gamma\geq1.
	\]
	The deformation tensor is defined by
	\[
	\mathbf D(u)=\frac{\nabla u+(\nabla u)^\top}{2}.
	\]
	
	The capillary force is described through the Korteweg tensor $\mathbf K$, namely
	\[
	\mathbf{K}
	=
	\left(
	\rho \operatorname{div}(\kappa(\rho)\nabla\rho)
	-\frac{\rho\kappa'(\rho)-\kappa(\rho)}{2}|\nabla\rho|^2
	\right)\mathbb{I}
	-\kappa(\rho)\nabla\rho\otimes\nabla\rho,
	\]
	where $\mathbb{I}$ is the identity matrix. Equivalently, its divergence can be written as
	\[
	\operatorname{div}\mathbf K
	=
	\nabla\left(
	\rho\kappa(\rho)\Delta\rho
	+\frac{\kappa(\rho)+\rho\kappa'(\rho)}{2}|\nabla\rho|^2
	\right)
	-\operatorname{div}\big(\kappa(\rho)\nabla\rho\otimes\nabla\rho\big).
	\]
	
	The above model goes back to the pioneering work of Korteweg \cite{Korteweg}, where a stress tensor depending on density gradients was introduced in order to account for capillary effects. This theory was later developed further by Dunn and Serrin \cite{Dunn-Serrin}. When the capillarity coefficient vanishes, namely $\kappa(\rho)=0$, system \eqref{eq:1.1} reduces to the compressible Navier--Stokes equations with density-dependent viscosities. A particularly important class of such systems is associated with the Bresch--Desjardins algebraic relation
	\begin{equation}
		\lambda(\rho)=2\rho\mu'(\rho)-2\mu(\rho),
	\end{equation}
	which yields an additional entropy estimate and has played a central role in the analysis of global weak solutions
	\cite{Bresch-Desjardins-Lin, Bresch-Desjardins}. There is now a substantial literature on multidimensional weak solutions for compressible flows satisfying the BD structure. Notable examples include the analysis of the viscous Saint-Venant system by Vasseur--Yu \cite{Vasseur-Yu} and the treatment of power-law viscosities
	\begin{equation}\label{1.2powerlaw}
		\mu(\rho)=\rho^\alpha,\qquad
		\lambda(\rho)=2(\alpha-1)\rho^\alpha,
	\end{equation}
	by Li--Xin \cite{Li-Xin}. Later, under the power-law condition \eqref{1.2powerlaw} on the viscosity coefficients, a series of works \cite{Gu-Huang, Huang-Meng-Zhang-111, Zhang} established global existence of smooth solutions for arbitrary spherically symmetric initial data far away from vacuum for a partial range of $\alpha \le 1$.\par
	When the capillarity coefficient does not vanish, namely $\kappa(\rho)> 0$, the system becomes the Navier--Stokes--Korteweg system. There are numerous well-posedness results for this system with various capillarity coefficients. Notably, significant progress has recently been made when the viscosity coefficients and the capillarity coefficient satisfy
	\begin{align}\label{vis coff}
		\mu(\rho)=\nu\rho^\alpha,\quad\lambda(\rho)=2\nu(\alpha-1)\rho^\alpha,\quad\kappa(\rho)=\varepsilon^2\alpha^2\rho^{2\alpha-3},\quad \nu\ge\varepsilon>0.
	\end{align}
	
	When $\alpha < 1$, by establishing upper and lower bounds on the density and enabling the fast diffusion, convection, and Lam\'e structure to interact in a compatible way to yield higher-order estimates, Gu--Huang--Meng--Zhou \cite{Gu-Huang-Meng-Zhou} and Huang--Lei--Zhou \cite{Huang-Lei-Zhou} proved the global existence of strong solutions to the periodic and Cauchy problems in two and three dimensions, respectively, for arbitrarily large initial data far away from vacuum, provided $\alpha$ lies in a suitable range. In the critical scenario where $\alpha=1$ and $\nu=\varepsilon$, the global existence of multi-dimensional strong solutions far away from vacuum with arbitrarily large initial data was established in periodic domains and the whole space by Huang--Meng--Zhang \cite{Huang-Meng-Zhang-31} and Huang--Gu--Lei \cite{Huang-Gu-Lei}, respectively. Their approach hinges on a modified Nash--Moser iteration technique, which reveals a pivotal estimate coupling the $L^\infty$-norm of the effective velocity to the lower density bound.\par
	To summarize, for viscosity and capillarity coefficients satisfying \eqref{vis coff}, Huang et al. established that when the spatial dimension $N$ and parameters $\alpha, \gamma$ satisfy
	\begin{align}
		& N=2: \quad \alpha\in \Big(\frac{\sqrt{5}-1}{2},1\Big], \quad \gamma\ge 1,  \\
		& N=3: \quad \alpha\in \left(\frac{9\sqrt3-4\sqrt2}{9\sqrt3-2\sqrt2}, 1\right], \quad \gamma\in \Big[1, \frac{15\alpha-7}{3}\Big),
	\end{align}
	the multi-dimensional system admits global strong solutions for arbitrary initial data far away from vacuum, under a suitable relation between the viscosity constant $\nu$ and the capillarity constant $\varepsilon$.
	
	A natural question is whether solutions to the equation with given initial data exist globally or blow up in finite time when the viscosity coefficients satisfy \eqref{vis coff} and $\alpha$ falls outside the above range. This problem has attracted widespread attention, particularly for the compressible Navier--Stokes (CNS) equations. Recently, Merle--Rapha\"el--Rodnianski--Szeftel \cite{Merle-Raphael-Rodnianski-Szeftel-2022-I} pioneered the construction of smooth radially imploding profiles for compressible Euler equations. Beyond shock formation (For related results, we refer the reader, for example, to
	\cite{Lax-1964,Sideris-1985,Buckmaster-Shkoller-Vicol-2022,
		Buckmaster-Shkoller-Vicol-2023,
		Buckmaster-Shkoller-Vicol-Point-Shocks-2023}.), they established an alternative singular mechanism—known as implosion—where the density and velocity fields themselves, rather than merely their spatial derivatives, blow up in finite time. Specifically, such $C^\infty$-smooth, spherically symmetric imploding solutions were constructed for the three-dimensional barotropic compressible Euler equations whenever $\gamma > 1$, excluding at most a countable set of exceptional values. By exploiting the $C^\infty$-smooth self-similar profiles of compressible Euler equations, Merle--Rapha\"el--Rodnianski--Szeftel \cite{Merle-Raphael-Rodnianski-Szeftel-2022-II} exhibited a family of smooth, finite-energy initial data with a far-field vacuum background that develop finite-time imploding singularities with infinite density blowup. This phenomenon holds for $\gamma \in \big(1, 1 + \frac{2}{\sqrt{3}}\big)$, excluding at most a countable set of exceptional values $\mathcal J$. It is worth emphasizing that the physical cases of monatomic and diatomic gases have been successfully addressed in the context of the 3D barotropic compressible Navier--Stokes equations (CNS). Buckmaster--Cao-Labora--G\'omez-Serrano
	\cite{Buckmaster-Cao-Labora-Gomez-Serrano} and
	Shao--Wang--Wei--Zhang \cite{Shao-Wang-Wei-Zhang} treated the physically
	relevant cases of a diatomic gas, $\gamma=7/5$, and a monatomic gas,
	$\gamma=5/3$, respectively. In particular, their results show that
	$\left\{\frac{5}{3},\frac{7}{5}\right\}\cap\mathcal J=\varnothing$,
	so that the sequence of admissible scaling parameters described in
	\eqref{Eq_condition_P2_Lambda_sequence} exists for both values of
	$\gamma$. More generally, Buckmaster--Cao-Labora--G\'omez-Serrano
	\cite{Buckmaster-Cao-Labora-Gomez-Serrano} constructed smooth
	self-similar imploding profiles for the compressible Euler equations
	for every $\gamma>1$.
	
	Regarding implosion with non-spherically symmetric initial data, Cao-Labora--G\'omez-Serrano--Shi--Staffilani \cite{Cao-Labora-Gomez-Serrano-Shi-Staffilani-2023} constructed smooth solutions away from vacuum that nevertheless form finite-time imploding singularities in $\mathbb{T}^3$ or $\mathbb{R}^3$. Then Chen--Liu--Zhu \cite{Chen-Zhang-Zhu} established finite-time implosion for small values of $\alpha$ under the density-dependent viscosity laws
	\begin{equation}
		\mu(\rho) = a_1 \rho^\alpha, \quad \lambda(\rho) = a_2 \rho^\alpha, \quad (\alpha \ge 0),
	\end{equation}
	where $\mu(\rho)$ and $\lambda(\rho) + \frac{2}{3}\mu(\rho)$ represent the shear and bulk viscosity coefficients, respectively, with constants $a_1 > 0$ and $2a_1 + 3a_2 \ge 0$.
	
	Motivated by these previous studies, addressing the singularity formation for the NSK system is highly desirable. In particular, it remains a key open question whether solutions exist globally or blow up in finite time for the remaining range of $\alpha$ under relation \eqref{vis coff}. To this end, this paper investigates the finite-time blowup of solutions in 3D for small $\alpha$ subject to \eqref{vis coff}. Remarkably, in the regime of small $\alpha$, unlike traditional physical configurations, the system behaves as a Navier--Stokes--Korteweg-type effective model featuring a negative bulk-viscosity effect, i.e.,
	\begin{equation}
		\mu(\rho)>0, \qquad \lambda(\rho)+\frac{2}{3}\mu(\rho)
		=2\nu\rho^\alpha\left(\alpha-\frac{2}{3}\right)<0.
	\end{equation}
	This is due to the fact that, for the three-dimensional NSK system, the bulk viscosity becomes negative whenever $\alpha < 2/3$. Similar negative-viscosity or negative-bulk-viscosity effects have been
	discussed in non-equilibrium acoustic media, generalized hydrodynamics, cosmological viscous fluids,
	and active matter (see, e.g., \cite{Brevik, Tankeshwar}). The Korteweg part is physically motivated by diffuse-interface models of liquid--
	vapor phase transition and capillarity.
	Thus the viscosity structure considered here may include degenerate or sign-changing regimes, provided that the assumptions specified below are satisfied.
	\subsection{Equations for the Effective Velocity}
	
	\quad  To facilitate the analysis in self-similar variables, we first reformulate the original system in terms of an effective velocity. This procedure is motivated by the analytical frameworks developed by Merle et al.~\cite{Merle-Raphael-Rodnianski-Szeftel-2022-II} and Cao-Labora et al.~\cite{Cao-Labora-Gomez-Serrano-Shi-Staffilani-2023}. Treating the dispersive terms directly as a perturbation would yield intractable nonlinearities that fall outside the scope of standard analysis. Consequently, following the renormalization technique introduced by Gu--Huang--Meng--Zhou \cite[Lemma~8.2]{Gu-Huang-Meng-Zhou}, we circumvent the analytical complexities associated with the third-order Korteweg tensor by introducing the following \textit{effective velocities} $v_{\pm}$:
	\begin{equation}\label{eff v}
		v_{\pm} = u + d_\pm \alpha \rho^{\alpha-2} \nabla \rho,
	\end{equation}
	where the coefficients $d_\pm$ are given by $d_\pm = \nu \pm \sqrt{\nu^2 - \varepsilon^2}$. This transformation allows the original Navier--Stokes--Korteweg system to be recast into an equivalent parabolic--parabolic system for the variables $(\rho, v_{\pm})$. 
	
	In the subsequent analysis, we focus on the "negative" branch by defining $v \coloneqq v_-$ and $d \coloneqq d_-$ for convenience. Consequently, the evolution of the fluid is governed by the following reformulated system:
	\begin{equation}
		\label{Equ2}
		\begin{cases}
			\rho_t + \divg(\rho v) - d \Delta(\rho^\alpha) = 0, \\[6pt]
			\rho v_t + \rho u \cdot \nabla v + \nabla P = \mathcal{L}_{\rho, \alpha}(v),
		\end{cases}
	\end{equation}
	where the density-dependent diffusion operator $\mathcal{L}_{\rho, \alpha}(v)$ is defined as:
	\begin{equation*}
		\begin{aligned}
			\mathcal{L}_{\rho, \alpha}(v) \coloneqq & \ \nu \divg(\rho^\alpha \nabla v) + (\sqrt{\nu^2-\varepsilon^2}) \divg(\rho^\alpha (\nabla v)^t) \\
			& + (\alpha - 1)(\nu +\sqrt{\nu^2-\varepsilon^2}) \nabla(\rho^\alpha \divg v).
		\end{aligned}
	\end{equation*}
	
	The system \eqref{Equ2} is supplemented with the initial data $(\rho_0, v_0)$ satisfying the relation $v_0 = u_0 + d \alpha \rho_0^{\alpha-2} \nabla \rho_0$. This reformulation uncovers a hidden dissipative structure: the mass equation now exhibits a degenerate diffusion term $d \Delta(\rho^\alpha)$, which is instrumental for establishing the high-order energy estimates required for the study of implosion singularities. The equivalence follows by direct substitution of the effective velocity into \eqref{eq:1.1}.
	
	To facilitate the analysis of singularity formation, we first introduce a change of variables to renormalize the system. Let $\delta = \frac{\gamma-1}{2}$ and define the sound speed-related variable $c = \frac{\rho^\delta}{\delta}$. In terms of the variables $(c, v)$, the system is transformed into the following reformulated form:
	\begin{equation}
		\label{Eq_reformulated_system_c}
		\begin{cases}
			c_t + v \cdot \nabla c + \delta c \operatorname{div} v - d\alpha (\delta c)^{\frac{\alpha-1}{\delta}} \Delta c - d\alpha(\alpha-\delta) (\delta c)^{\frac{\alpha-\delta-1}{\delta}} |\nabla c|^2 = 0, \\[8pt]
			\begin{aligned}
				v_t &+ \left( v - \alpha ( \nu - \sqrt{\nu^2 - \varepsilon^2})(\delta c)^{\frac{\alpha-\delta-1}{\delta}} \nabla c \right) \cdot \nabla v + \delta c \nabla c \\
				&= (\delta c)^{\frac{\alpha-1}{\delta}} \left[ \nu \Delta v + \left(\sqrt{\nu^2 - \varepsilon^2} + (\alpha - 1) ( \nu + \sqrt{\nu^2 - \varepsilon^2} )\right) \nabla \operatorname{div} v \right] \\
				&\quad + \alpha (\delta c)^{\frac{\alpha-\delta-1}{\delta}} \nabla c \cdot \mathcal{D}v,
			\end{aligned}
		\end{cases}
	\end{equation}
	where the dissipation-related tensor is defined as 
	$$\mathcal{D}v = \nu \nabla v + \sqrt{\nu^2 - \varepsilon^2} (\nabla v)^t + (\alpha - 1) \left( \nu + \sqrt{\nu^2 - \varepsilon^2} \right) \divg v\mathbb{I}.$$
	The initial and far-field conditions:
	\begin{equation}
		(c,v)(0,x)=(\frac{{\rho^{\delta}_0}}{\delta},v_0)(x) \qquad \text{for}\,\, x \in \mathbb R^3,
	\end{equation}
	\begin{equation}
		(c,v)(t,x) \rightarrow(\frac{\rstar^{\delta}}{\delta},0)  \qquad \text{as} \,\,|x|\rightarrow\infty \, .
	\end{equation} 
	\subsection{Self-similar Analysis}
	\quad Next, we investigate the potential implosion by introducing the self-similar scaling. For fixed constants $T > 0$ and $\Lambda > 1$, we define the rescaled time $\tau$ and the similarity variable $y$ as follows:
	\begin{equation}
		\tau = - \frac{\log(T - t)}{\Lambda}, \quad y = \frac{x}{(T - t)^{1/\Lambda}} = e^{\tau} x, \quad \tau_0 = - \frac{\log T}{\Lambda}.
	\end{equation}
	Accordingly, we assume the solution $(\rho, v)$ follows the self-similar ansatz:
	\begin{align}
		c(x,t) &= \frac{(T - t)^{\frac{1}{\Lambda}-1}}{\Lambda} S(\tau, y), \label{cST-t}\\ 
		v(x,t) &= \frac{(T - t)^{\frac{1}{\Lambda}-1}}{\Lambda} V(\tau, y). \label{vVT-t}
	\end{align}
	
	\textbf{This substitution produces the common constant factor
		\(
		C_\Lambda:=\Lambda^{1-\frac{\alpha-1}{\delta}}
		\)
		in all terms involving \(\nu,\varepsilon,\) and \(d\). We absorb it by replacing
		\((\nu,\varepsilon,d)\) with
		\((C_\Lambda\nu,C_\Lambda\varepsilon,C_\Lambda d)\)}, which preserves
	\(d=\nu-\sqrt{\nu^2-\varepsilon^2}\), and retain the same notation for the rescaled coefficients. Substituting the ansatz into \eqref{Eq_reformulated_system_c}, we then obtain the evolution system for $(S, V)$ in the $(y, \tau)$ coordinates:
	\begin{equation}
		\label{Eq_self_similar_S}
		\begin{cases}
			\begin{aligned}
				S_\tau & + y \cdot \nabla_y S + (\Lambda - 1) S + V \cdot \nabla_y S + \delta S \operatorname{div}_y V \\ 
				&= e^{-f_1 \tau} d \alpha (\delta S)^{\frac{\alpha-1}{\delta}} \Delta_y S + e^{-f_1 \tau} d \alpha (\alpha-\delta) (\delta S)^{\frac{\alpha-\delta-1}{\delta}} |\nabla_y S|^2, \\
				V_\tau & + y \cdot \nabla_y V + (\Lambda - 1) V + V \cdot \nabla_y V + \delta S \nabla_y S \\
				&= e^{-f_1 \tau} (\delta S)^{\frac{\alpha-1}{\delta}}\mathbb{L}(V) + e^{-f_1 \tau} \alpha (\delta S)^{\frac{\alpha-\delta-1}{\delta}} \nabla_y S \cdot \mathbb{D}V,
			\end{aligned}
		\end{cases}
	\end{equation}
	where the time-dependent factor $e^{-{f}_1\tau}$ is given by
	\begin{equation}
		e^{-{f}_1\tau}:=e^{-(\frac{(1-\alpha)(\Lambda-1)}{\delta}+(\Lambda-2))\tau}=(T-t)^{\frac{\alpha-1}{\delta}(\frac{1}{\Lambda}-1)+1-\frac{2}{\Lambda}},
	\end{equation}
	the dissipative operator $\mathbb{L}$ is defined as
	\begin{equation}
		\mathbb{L}(V) = \nu \Delta_y V + \left( \sqrt{\nu^2 - \varepsilon^2} + (\alpha - 1) \left( \nu + \sqrt{\nu^2 - \varepsilon^2} \right) \right) \nabla_y \operatorname{div}_y V,
	\end{equation}
	and the dissipation-related tensor $\mathbb{D}V$ is defined as 
	\begin{equation}
		\mathbb{D}V = (2\nu - \sqrt{\nu^2 - \varepsilon^2}) \nabla_y V + \sqrt{\nu^2 - \varepsilon^2} (\nabla_y V)^t + (\alpha - 1) \left( \nu + \sqrt{\nu^2 - \varepsilon^2} \right) \divg_y V \mathbb{I}.
	\end{equation}
	This change in the coefficient results from moving the correction term in the
	transport velocity to the right-hand side of the equation. The initial condition in the self-similar variables is given by
	\begin{equation}\label{SSinitial}
		\begin{aligned}
			(S,V)(\tau_0,y)
			&=(S_0,V_0)(y)=
			\left(
			\frac{\Lambda}{\delta}
			e^{-(\Lambda-1)\tau_0}\rho_0^\delta,\,
			\Lambda e^{-(\Lambda-1)\tau_0}v_0
			\right)(e^{-\tau_0}y),
			\qquad y\in\mathbb R^3.
		\end{aligned}
	\end{equation}
	Moreover, the far-field condition takes the form
	\begin{equation}\label{SSfar}
		(S,V)(\tau,y)
		\rightarrow
		\left(S^*(\tau),0\right)
		:=
		\left(
		\frac{\Lambda}{\delta}
		e^{-(\Lambda-1)\tau}\rstar^{\delta},\,
		0
		\right)
		\quad\text{as }|y|\to\infty,
		\qquad \tau\ge\tau_0.
	\end{equation}

	The dissipative terms decay as \(\tau\to\infty\) precisely when \(f_1>0\), or equivalently
	\begin{equation}
		\Lambda>\frac{1-\alpha+2\delta}{1-\alpha+\delta}.
	\end{equation}
	\subsection{Nonlinear perturbation equations}
	\quad In Section 4, we show that the finite-time blowup problem for system \eqref{Equ2} can be reduced to the stability analysis of system~\eqref{Eq_self_similar_S} near a self-similar profile. We therefore proceed to derive the perturbation equations for \eqref{Eq_self_similar_S} around this profile. Consider a $C^\infty$ self-similar profile $(\bar{S},\bar{V})$ solving
	\eqref{Eq_self_similar_profile_system} for an admissible scaling parameter
	$\Lambda \in (1, \Lambda^*(\gamma))$. The spherical symmetry of
	$(\bar{S}, \bar{V})$ guarantees the existence of a scalar function
	$\bar{\mathcal V}$ such that
	$$\bar{S}(y) = \bar{S}(r), \quad \bar{V}(y) = \bar{\mathcal V}(r)\frac{y}{r} \quad \text{with } r = |y|.$$
	We use the standard profile bounds
	\begin{equation}\label{eq:profile-decay}
		|\nabla^j\bar S(y)|+|\nabla^j\bar V(y)|
		\le C_j\langle y\rangle^{-(\Lambda-1)-j},
		\qquad j\ge0,
	\end{equation}
	and choose the normalization radius $R_0$ so that
	\begin{equation}\label{eq:profile-R0-normalization}
		2\sigma_0\le\bar S(y)\le C\sigma_0,
		\qquad |y|=R_0.
	\end{equation}
	Let $\chi\in C_c^\infty([0,\infty))$, with
	$0\le\chi\le1$, $\chi(r)=1$ for $r\le\frac12$, and $\chi(r)=0$ for
	$r\ge1$, and define
	\begin{equation}\label{eq:time-dependent-cutoff}
		\widehat\chi(\tau,y):=\chi(e^{-\tau}|y|).
	\end{equation}
	The stability analysis is performed around the cutoff profile
	$(\widehat{\chi}\bar{S}, \widehat{\chi}\bar{V})$. Under this motivation, our primary objective is to investigate the algebraic and analytic structures governed by the perturbation variables:
	$$\widetilde{S} = S - \widehat{\chi}\bar{S}, \quad \widetilde{V} = V - \widehat{\chi}\bar{V}.$$
	By subtracting the profile system from the original equations, we obtain the following perturbation equations
	\begin{equation}\label{perturSEq}
		\begin{aligned}
			\partial_\tau \widetilde{S} = &\underbrace{-(\Lambda - 1)\widetilde{S} - (y + \widehat{\chi}\bar{V}) \cdot \nabla \widetilde{S} - \delta(\widehat{\chi}\bar{S}) \operatorname{div} \widetilde{V} - \widetilde{V} \cdot \nabla(\widehat{\chi}\bar{S}) - \delta \widetilde{S} \operatorname{div}(\widehat{\chi}\bar{V})}_{\mathcal{L}_s^1(\widetilde{S},\widetilde{V})} \\
			&\underbrace{- \widetilde{V} \cdot \nabla \widetilde{S} - \delta \widetilde{S} \operatorname{div} \widetilde{V}}_{\mathcal{N}_s(\widetilde{S},\widetilde{V})} \\
			&\underbrace{- (\widehat{\chi}^2 - \widehat{\chi})\bar{V} \cdot \nabla \bar{S} - (1 + \delta)\widehat{\chi}\bar{V} \cdot \nabla \widehat{\chi} \, \bar{S} - \delta(\widehat{\chi}^2 - \widehat{\chi})\bar{S} \operatorname{div} \bar{V}}_{\mathcal{E}_s(\bar{S},\bar{V})} \\
			&+e^{-f_1 \tau} d \alpha (\delta S)^{\frac{\alpha-1}{\delta}} \Delta_y S + e^{-f_1 \tau} d \alpha (\alpha-\delta) (\delta S)^{\frac{\alpha-\delta-1}{\delta}} |\nabla_y S|^2,
		\end{aligned}
	\end{equation}
	and 
	\begin{equation}\label{perturVEq}
		\begin{aligned}
			\partial_\tau \widetilde{V} = &\underbrace{-(\Lambda - 1)\widetilde{V} - (y + \widehat{\chi}\bar{V}) \cdot \nabla \widetilde{V} - \delta(\widehat{\chi}\bar{S})\nabla \widetilde{S} - \widetilde{V} \cdot \nabla(\widehat{\chi}\bar{V}) - \delta \widetilde{S} \nabla(\widehat{\chi}\bar{S})}_{\mathcal{L}_v^1(\widetilde{S},\widetilde{V})} \\
			&\underbrace{- \widetilde{V} \cdot \nabla \widetilde{V} - \delta \widetilde{S} \nabla \widetilde{S}}_{\mathcal{N}_v(\widetilde{S},\widetilde{V})} \\
			&\underbrace{- (\widehat{\chi}^2 - \widehat{\chi})\bar{V} \cdot \nabla \bar{V} - \widehat{\chi}\bar{V} \cdot \nabla \widehat{\chi} \, \bar{V} - \delta(\widehat{\chi}^2 - \widehat{\chi})\bar{S} \nabla \bar{S} - \delta \widehat{\chi}\bar{S}\nabla \widehat{\chi} \, \bar{S}}_{\mathcal{E}_v(\bar{S},\bar{V})} \\
			&+e^{-f_1 \tau} (\delta S)^{\frac{\alpha-1}{\delta}}\mathbb{L}(V) + e^{-f_1 \tau} \alpha (\delta S)^{\frac{\alpha-\delta-1}{\delta}} \nabla_y S \cdot \mathbb{D}V.
		\end{aligned}
	\end{equation}
	
	Our subsequent analysis is mainly devoted to establishing the stability of this perturbation system. We are now in a position to state the main theorem of this paper.
	\begin{theorem}\label{Thm_Main_Blowup_Profiles}
		Let
		\[
		(\overline S,\overline V)\in
		C^\infty(\R^3)\times C^\infty(\R^3;\R^3)
		\]
		be a spherically symmetric self-similar Euler profile satisfying
		\begin{equation}\label{Eq_self_similar_profile_system}
			\left\{
			\begin{aligned}
				(\Lambda-1)\overline S
				+(y+\overline V)\cdot\nabla_y\overline S
				+\delta\overline S\,\operatorname{div}_y\overline V
				&=0,\\
				(\Lambda-1)\overline V
				+(y+\overline V)\cdot\nabla_y\overline V
				+\delta\overline S\,\nabla_y\overline S
				&=0,
			\end{aligned}
			\right.
		\end{equation}
		together with the positivity, decay, and repulsivity properties stated in
		Lemma~\ref{prop:profiles-R}. Define
		\begin{equation}\label{Eq_alpha_star_critical_def}
			\alpha^*(\gamma)
			:=
			\begin{cases}
				\displaystyle
				\frac{\gamma+1}{4}
				-\frac{\sqrt{2(\gamma-1)}}{2},
				&\displaystyle
				1<\gamma<\frac53,\\[3mm]
				\displaystyle
				\frac{1-(2\sqrt3-3)\gamma}{2(3-\sqrt3)},
				&\displaystyle
				\frac53\le\gamma<1+\frac{2}{\sqrt3},
			\end{cases}
		\end{equation}
		and 
		\begin{equation}\label{eq:Lambda-star}
			\Lambda^*(\gamma)
			=
			\begin{cases}
				\displaystyle
				1+\frac{2}
				{\left(1+\sqrt{\frac{2}{\gamma-1}}\right)^2},
				&
				1<\gamma<\frac53,\\[8pt]
				\displaystyle
				\frac{3\gamma-1}
				{2+\sqrt3(\gamma-1)},
				&
				\frac53\le\gamma<1+\frac{2}{\sqrt3}.
			\end{cases}
		\end{equation}
		
		Assume that \((\gamma,\alpha)\) belongs to one of the following two parameter
		regimes:
		
		\smallskip
		\noindent\textnormal{\((\mathrm P_1)\)}
		Suppose that
		\begin{equation}\label{Eq_condition_P1}
			1<\gamma<1+\frac{2}{\sqrt3},
			\qquad
			0<\alpha<\frac12, \quad f_1
			:=
			\frac{(1-\alpha)(\Lambda-1)}{\delta}
			+\Lambda-2
			>0,
		\end{equation}
		where $\Lambda$ is an admissible scaling parameter for the profile defined in Lemma~\ref{thm:existence_profiles} \textnormal {(ii)}.
		
		\smallskip
		\noindent\textnormal{\((\mathrm P_2)\)}
		Suppose that
		\begin{equation}\label{Eq_condition_P2}
			\gamma\in
			\left(1,1+\frac{2}{\sqrt3}\right)\setminus\mathcal J,
			\qquad
			0<\alpha<\alpha^*(\gamma),
		\end{equation}
		where \(\mathcal J\) is the at most countable exceptional set in
		Lemma~\ref{thm:existence_profiles} \textnormal{(i)} and satisfies
		\[
		\left\{\frac75,\frac53\right\}\cap\mathcal J=\varnothing.
		\]
		By Lemma~\ref{thm:existence_profiles} \textnormal{(i)}, there exists a
		sequence of admissible scaling parameters \(\{\Lambda_n\}_{n\ge1}\) such that
		\begin{equation}\label{Eq_condition_P2_Lambda_sequence}
			1<\Lambda_n<\Lambda^*(\gamma),
			\qquad
			\Lambda_n\rightarrow\Lambda^*(\gamma)
			\quad\text{as }n\to\infty,
		\end{equation}
		and each \(\Lambda_n\) generates a smooth, spherically symmetric solution of
		\eqref{Eq_self_similar_profile_system}.

		Then, for sufficiently small \(T>0\) and \(\underline{\rho}_0>0\), one can find
		\(C^\infty\) initial data \((\rho_0,v_0)\) such that
		\[
		\inf_{x\in\R^3}\rho_0(x)>\underline{\rho}_0.
		\]
		The corresponding smooth solution \((\rho,v)\) of
		\eqref{Equ2}, with the constitutive coefficients specified in
		\eqref{vis coff}, exists on \([0,T)\) and develops an implosion singularity at
		\(t=T\). More precisely, for every \(r>0\),
		\begin{equation}\label{Eq_blowup_limit_infinity}
			\lim_{t\to T^-}\rho(t,0)=+\infty,
			\qquad
			\lim_{t\to T^-}\sup_{|x|\le r}|v(t,x)|=+\infty.
		\end{equation}
		
		Furthermore, the singularity is asymptotically described by
		\((\overline S,\overline V)\). For every fixed \(y\in\R^3\),
		\begin{equation}\label{Eq_profile_convergence_limits}
			\begin{aligned}
				\lim_{t\to T^-}
				\left[
				\frac{\Lambda}{\delta}
				(T-t)^{1-\frac1\Lambda}
				\right]^{\frac1\delta}
				\rho\left(t,(T-t)^{\frac1\Lambda}y\right)
				&=
				\overline S(y)^{\frac1\delta},\\
				\lim_{t\to T^-}
				\Lambda(T-t)^{1-\frac1\Lambda}
				v\left(t,(T-t)^{\frac1\Lambda}y\right)
				&=
				\overline V(y).
			\end{aligned}
		\end{equation}
	\end{theorem}
	\begin{remark}\label{rem:Lambda-star-upper-bound}
		It follows directly from \eqref{eq:Lambda-star} that
		\begin{equation}\label{lambdafanwei2}
			1<\Lambda^*(\gamma)<\frac{1+\sqrt3}{2},
			\qquad
			1<\gamma<1+\frac{2}{\sqrt3}.
		\end{equation}
		Hence every admissible scaling parameter
		\(\Lambda\in(1,\Lambda^*(\gamma))\) satisfies \(1<\Lambda<\frac{1+\sqrt3}{2}\).
	\end{remark}
	\begin{remark}
		The same blow-up result remains valid on the periodic domain
		$\mathbb{T}^3$. More precisely, there exists a finite-codimensional
		manifold of initial data with strictly positive initial density for
		which the corresponding solutions develop a finite-time singularity.
		The proof follows arguments analogous to those in
		\cite{Cao-Labora-Gomez-Serrano-Shi-Staffilani-2023} and
		\cite{Chen-Zhang-Zhu}. We therefore focus primarily on the
		whole-space setting $\mathbb{R}^3$.
	\end{remark}
	\begin{remark}
		\label{rmk:finite-codimensional-initial-data}
		We now explain why such initial data exist. Let $\mathcal B$ be a Banach space whose norm controls all the
		quantities in \eqref{eq:conditions_for_tilde}. Since the conditions in
		\eqref{eq:conditions_for_tilde} are open, they hold in a sufficiently
		small open subset of $\mathcal B$.
		
		The additional condition
		\[
		P_{\mathrm{uns}}
		(\widetilde S_0^*,\widetilde V_0^*)=0
		\]
		is a closed finite-dimensional constraint, since
		$P_{\mathrm{uns}}$ projects onto the finite-dimensional space
		$X_{\mathrm{u}}$. Moreover, the initial data are chosen as
		\[
		(\widetilde S_0,\widetilde V_0)
		=
		(\widetilde S_0^*,\widetilde V_0^*)
		+\sum_{j=1}^{m}\hat{k}_j\Phi_j,
		\]
		where $\{\Phi_j\}_{j=1}^m$ is a basis of $X_\mathrm{u}$. Therefore,
		$(\widetilde S_0,\widetilde V_0)$ and
		$(\widetilde S_0^*,\widetilde V_0^*)$ represent the same element of
		$\mathcal B/X_{\mathrm{u}}$. Hence the initial data leading to
		finite-time implosion form a finite-codimensional family.
	\end{remark}
	The following corollary shows that the original system develops a
	finite-time singularity, and it follows directly from the preceding
	theorem.
	\begin{corollary}[Density blow-up for the original system]
		\label{Thm_density_blowup_original_system}
		Assume the hypotheses of Theorem~\ref{Thm_Main_Blowup_Profiles}.
		Then, for sufficiently small $T>0$ and $\underline{\rho}_0>0$, there exist
		smooth initial data $(\rho_0,u_0)$ for the original system
		\eqref{eq:1.1}, with the constitutive coefficients prescribed in
		\eqref{vis coff}, such that
		\begin{equation}\label{Eq_original_initial_density_positive}
			\inf_{x\in\mathbb R^3}\rho_0(x)>\underline{\rho}_0.
		\end{equation}
		The initial physical velocity $u_0$ is recovered from the effective
		velocity $v_0$ through
		\begin{equation}\label{Eq_recover_original_initial_velocity}
			u_0
			=
			v_0-d\alpha\rho_0^{\alpha-2}\nabla\rho_0.
		\end{equation}
		
		The corresponding smooth solution $(\rho,u)$ to the original system
		exists on $[0,T)$, and its density develops an implosion singularity at
		$t=T$. More precisely,
		\begin{equation}\label{Eq_original_density_blowup}
			\lim_{t\to T^-}\rho(t,0)=+\infty.
		\end{equation}
	\end{corollary}
	\begin{remark}
		We can establish the blow-up of the effective velocity $v$. However,
		the blow-up of the physical velocity $u$ cannot be concluded directly,
		since $\nabla\rho$ may also become singular near the origin, and a
		possible cancellation in the relation $
		v=u+d\alpha\rho^{\alpha-2}\nabla\rho$
		cannot be ruled out.
	\end{remark}
	\section{Notation and Sketch of the Proof}
	\subsection{Notation and analytic conventions}
	\label{subsec:notation}
	
	We collect here the notation used in the perturbative, spectral, and weighted
	energy arguments. Unless otherwise stated, all spatial derivatives in the
	self-similar formulation are taken with respect to
	\(y=(y_1,y_2,y_3)\in\R^3\), and all norms are over \(\R^3\).
	
	\paragraph{Derivatives and multi-indices.}
	For a multi-index
	\(\beta=(\beta_1,\beta_2,\beta_3)\in\mathbb N_0^3\), we set
	\[
	|\beta|:=\beta_1+\beta_2+\beta_3,
	\qquad
	\partial_\beta
	:=
	\partial_{y_1}^{\beta_1}
	\partial_{y_2}^{\beta_2}
	\partial_{y_3}^{\beta_3}.
	\]
	
	When it is useful to distinguish the individual derivatives in
	\(\partial_\beta\), we regard \(\beta\) as an ordered list
	\[
	\beta=(\beta_1,\ldots,\beta_k),
	\qquad
	\beta_j\in\{1,2,3\},
	\]
	and write
	\[
	\beta^{(j)}
	:=
	(\beta_1,\ldots,\beta_{j-1},\beta_{j+1},\ldots,\beta_k).
	\]
	Thus \(\partial_{\beta^{(j)}}\) denotes the derivative obtained from
	\(\partial_\beta\) by deleting its \(j\)-th entry. For two multi-indices \(\eta,\beta\in\mathbb N_0^3\), we write
	\[
	\eta\le\beta
	\quad\Longleftrightarrow\quad
	\eta_i\le\beta_i
	\quad\text{for each }i=1,2,3,
	\]
	and
	\[
	\eta<\beta
	\quad\Longleftrightarrow\quad
	\eta_i\le\beta_i
	\quad\text{for each }i=1,2,3,
	\qquad
	|\eta|<|\beta|.
	\]
	Whenever \(\eta\le\beta\), the difference is defined componentwise by
	\[
	\beta-\eta
	:=
	(\beta_1-\eta_1,\beta_2-\eta_2,\beta_3-\eta_3).
	\]

	\paragraph{Commutators.}
	If \(\mathcal A\) is a linear differential operator and \(f\) is sufficiently
	regular, we define
	\[
	[\partial_\beta,\mathcal A]f
	:=
	\partial_\beta(\mathcal Af)
	-
	\mathcal A(\partial_\beta f).
	\]
	For a scalar coefficient \(a\), we use
	\[
	[\partial_\beta,a]f
	:=
	\partial_\beta(af)-a\partial_\beta f
	=
	\sum_{0<\eta\le\beta}
	C_{\beta,\eta}
	(\partial_\eta a)\partial_{\beta-\eta}f.
	\]
	More generally, if \(\mathcal A_a\) is a differential operator whose
	coefficients depend on \(a\), then
	\[
	[\partial_\beta,\mathcal A_a]f
	:=
	\partial_\beta(\mathcal A_af)
	-
	\mathcal A_a(\partial_\beta f).
	\]
	\paragraph{Function spaces and domains.}
	For \(R>0\), we write
	\[
	B(0,R):=\{y\in\R^3:|y|<R\},
	\qquad
	B^c(0,R):=\R^3\setminus B(0,R).
	\]
	Unless a domain is displayed explicitly,
	\[
	\|f\|_{L^p}:=\|f\|_{L^p(\R^3)},
	\qquad
	\|f\|_{H^s}:=\|f\|_{H^s(\R^3)}.
	\]
	For a pair of scalar and vector functions, we use
	\[
	\|(f,F)\|_{H^s}
	:=
	\|f\|_{H^s}+\|F\|_{H^s}.
	\]
	The Japanese bracket is denoted by
	\[
	\langle y\rangle:=(1+|y|^2)^{1/2}.
	\]
	
	Let \(m\) be the sufficiently large integer used in the construction of the
	truncated linearized operator. Its phase space is denoted by
	\begin{equation}\label{X space Hm0}
		X
		:=
		H_0^m(B(0,3R_1))
		\times
		H_0^m(B(0,3R_1);\R^3),
	\end{equation}
	
	\paragraph{Cutoff functions.}
	The radial cutoff \(\chi\) and its time-dependent rescaling
	\(\widehat\chi\) are defined in \eqref{eq:time-dependent-cutoff}.
	
	\paragraph{Choice of the weighted energy function.}
	Fix \(R_0>0\) and \(0<\eta\ll1\). Choose
	\(\chi_\phi\in C^\infty([0,\infty))\) satisfying
	\[
	0\le\chi_\phi\le1,\qquad
	\chi_\phi(r)=0\ \text{for }r\le1,\qquad
	\chi_\phi(r)=1\ \text{for }r\ge4.
	\]
	We define the radial weight
	\begin{equation}\label{eq:weight-definition-notation}
		\phi(y)
		:=
		1-\chi_\phi\left(\frac{|y|}{R_0}\right)
		+
		\chi_\phi\left(\frac{|y|}{R_0}\right)
		\frac{|y|^{2(1-\eta)}}{2R_0^{2(1-\eta)}}.
	\end{equation}
	Then \(\phi\) is smooth and strictly positive, and
	\[
	\phi(y)=1
	\quad\text{for }|y|\le R_0,
	\qquad
	\phi(y)=
	\frac{|y|^{2(1-\eta)}}{2R_0^{2(1-\eta)}}
	\quad\text{for }|y|\ge4R_0.
	\]
	The interpolation on \(R_0<|y|<4R_0\) is fixed once and for all. In
	particular,
	\begin{equation}\label{eq:weight-derivative-properties}
		|\nabla\phi(y)|\le C\phi(y),
		\qquad
		|y\cdot\nabla\phi(y)|\le C\phi(y),
	\end{equation}
	and hence, for every fixed integer \(K\ge 1\),
	\begin{equation}\label{eq:weight-power-properties}
		|\nabla(\phi^K)|
		\le C_K\phi^K,
		\qquad
		|\operatorname{div}(y\phi^K)|
		\le C_K\phi^K.
	\end{equation}
	Indeed, in the far field,
	\[
	y\cdot\nabla\phi=2(1-\eta)\phi,
	\]
	while the estimates in the transition annulus follow from compactness and the
	strict positivity of \(\phi\).
	\begin{remark}\label{remark1-etar}
		The cutoff function $\chi_\phi$ in \eqref{eq:weight-definition-notation}
		may be chosen to be nondecreasing and to complete its transition from $0$
		to $1$ before
		\[
		\frac{r}{R_0}
		=
		2^{\frac{1}{2(1-\eta)}}.
		\]
		More precisely, we may require
		\[
		\chi_\phi'\geq0,
		\qquad
		\operatorname{supp}\chi_\phi'
		\subset
		\left(1,2^{\frac{1}{2(1-\eta)}}\right).
		\]
		For this choice, the inequality
		\begin{equation}\label{eq:weight-radial-ratio-bound}
			\frac{r\partial_r\phi}{2\phi}\leq1-\eta
		\end{equation}
		holds for every $r\geq0$. Indeed, a direct calculation gives
		\[
		\begin{aligned}
			2(1-\eta)\phi-r\partial_r\phi
			&=
			2(1-\eta)
			\left[
			1-\chi_\phi\left(\frac r{R_0}\right)
			\right]\\
			&\quad-
			\frac r{R_0}\chi_\phi'\left(\frac r{R_0}\right)
			\left[
			\frac12\left(\frac r{R_0}\right)^{2(1-\eta)}-1
			\right].
		\end{aligned}
		\]
		On the support of $\chi_\phi'$, one has
		\[
		\frac12\left(\frac r{R_0}\right)^{2(1-\eta)}-1<0.
		\]
		Since $\chi_\phi'\geq0$ and $0\leq\chi_\phi\leq1$, both terms on the
		right-hand side are nonnegative. Hence
		\[
		r\partial_r\phi\leq2(1-\eta)\phi,
		\]
		which proves \eqref{eq:weight-radial-ratio-bound}.
	\end{remark}
	\paragraph{Constants and comparison notation.}
	The symbols \(C\) and \(c\) denote generic positive constants whose values may
	change from line to line. Subscripts indicate the permitted dependence; for
	example, \(C_K\) may depend on \(K\), while \(C_{\tau_1}\) may also depend on
	the local Sobolev bounds on \([\tau_0,\tau_1]\). We write
	\[
	A\lesssim B
	\]
	if \(A\le CB\) for a harmless constant \(C\), and
	\[
	A\sim B
	\]
	if both \(A\lesssim B\) and \(B\lesssim A\) hold. The notation
	\(A\ll B\) means that \(A/B\) is chosen sufficiently small, with the required
	smallness determined by the argument.
	
	\subsection{Sketch of the Proof}
	\quad Our analysis is primarily based on reducing the finite-time blowup problem to the long-time stability of the rescaled system around a self-similar profile. Indeed, compared with the self-similar profile system \eqref{Eq_self_similar_profile_system}, system \eqref{Eq_self_similar_S} includes additional perturbation terms on the right-hand sides of both the density equation and the velocity field equation. Our strategy is to treat the terms on the right-hand side of \eqref{Eq_self_similar_S} as perturbative contributions. This idea can be traced back to the general perturbative framework developed in
	\cite{Cortazar-delPino-Musso-2020,Duyckaerts-Kenig-Merle-2012,Daskalopoulos-delPino-Sesum-2018,Kenig-Merle-2008,Krieger-Schlag-Tataru-2009}. Compared with the radially symmetric setting, the analysis of non-radial
	perturbations is considerably more delicate, since angular modes interact
	with the radial dynamics and may generate additional unstable directions.
	Similar difficulties in extending radial stability to the non-radial setting
	have arisen for the energy-critical defocusing quintic nonlinear
	Schr\"odinger equation \cite{Bourgain-1999, Colliander-Keel-Staffilani-Takaoka-Tao-2008,Grillakis-2000} and the energy-critical wave equation
	\cite{Duyckaerts-Kenig-Merle-2011,Duyckaerts-Kenig-Merle-2012}. We consider the stability around the self-similar profile \eqref{Eq_self_similar_profile_system}. Note that the density and velocity equations are structurally very similar; essentially, the right-hand side of each equation (see system \eqref{Eq_self_similar_S}) consists of a dissipation term and a nonlinear structure. In principle, the dissipation term can absorb the nonlinear terms, which holds true for both equations. In a sense, $V$ and $S$ are equivalent here, except for differences caused by the product terms involving $S$ on the right-hand side. As long as we can establish a lower bound for the density, we expect their high-order derivative decay estimates to be identical. Compared with the work of Cao-Labora et al.~\cite{Cao-Labora-Gomez-Serrano-Shi-Staffilani-2023}, the main difficulty in analyzing the present system lies in the fact that the initial density equation possesses a dissipation structure in addition to its transport structure. Consequently, selecting an appropriate bootstrap argument is crucial when deriving lower-order derivative estimates and lower bounds for the density. As for the higher-order derivative estimates, we believe they can be handled by standard arguments because the smallness originating from the time slab allows the dissipation in the density equation to absorb the nonlinear terms. In summary, the key difficulties and innovations of this work are threefold:
	\begin{itemize}\item First, we need to reformulate the system into an effective velocity formulation. Combining the dissipation estimates with scaling transformations to study the stability around the profile forms the foundational framework of our proof.
		\item Second, we address how to obtain an a priori lower bound for the density. Unlike the method of Chen et al. \cite{Chen-Zhang-Zhu}, we cannot rely solely on the transport structure of the first equation to establish this bound. Instead, we must exploit both the transport and dissipation structures of the density equation via a bootstrap argument. Specifically, we introduce an additional bootstrap assumption:
		\[\frac{\sigma_1}{10} \left\langle \frac{y}{R_0} \right\rangle^{-(\Lambda - 1)} \le S,\]
		Under this assumption, together with the smallness of the time-slab length, we close the bootstrap argument and improve the lower bound back to:
		\[\frac{\sigma_1}{5} \left\langle \frac{y}{R_0} \right\rangle^{-(\Lambda - 1)} \le S.\]
		The core idea is that the density perturbation caused by the right-hand side terms will not exceed the primary effect induced by the transport terms; that is, the transport structure remains the dominant mechanism governing the evolution of the density. 
		
		\item Third, we derive the decay estimates for the gradients. Since the right-hand side of the density equation contains extra nonlinear terms and the decay rate cannot be absorbed by the time-slab smallness parameter, a straightforward analysis without an additional framework would lead to an insufficient decay rate. Nevertheless, this obstacle can be overcome by fully utilizing the intrinsic decay mechanism of the gradient terms, rather than relying on interpolations between zero-order terms and higher-order derivatives. Specifically, we design an additional bootstrap scheme:
		\[\vert{}\nabla \widetilde{S}\vert{} + \vert{}\nabla S\vert{} \le 2C_1 \left\langle \frac{y}{R_0} \right\rangle^{-\Lambda},\]
		\[\vert{}\nabla \widetilde{V}\vert{} + \vert{}\nabla V\vert{} \le 2C_1 \left\langle \frac{y}{R_0} \right\rangle^{-\Lambda},\]
		which allows us to close the estimates back to:
		\[\vert{}\nabla \widetilde{S}\vert{} + \vert{}\nabla S\vert{} \le C_1 \left\langle \frac{y}{R_0} \right\rangle^{-\Lambda},\]
		\[\vert{}\nabla \widetilde{V}\vert{} + \vert{}\nabla V\vert{} \le C_1 \left\langle \frac{y}{R_0} \right\rangle^{-\Lambda},\]
		where $C_1 > 0$ is a constant to be determined later in the paper. Meanwhile, we must verify that the local solution satisfies the weighted assumptions required to initiate the bootstrap argument. In particular, the weighted energy associated with the spatial gradients must remain locally bounded. The proof of this property is given in section 5.4.3.
	\end{itemize}
	
	We now outline the organization of the paper. Section 3 is devoted to the stability analysis of the perturbation system \eqref{perturSEq}--\eqref{perturVEq}. In particular, we establish its local well-posedness and prove the local boundedness of the weighted energy required to initiate the bootstrap argument. In Section 4, we return to the original variables and derive the finite-time blowup result for the original system. Finally, Section~5 contains the appendices, where we collect several auxiliary lemmas, recall the construction and properties of the self-similar Euler profiles, analyze the spectrum of the truncated linearized operator, and prove the local well-posedness of the original system.
	\section{Nonlinear Stability Analysis via Energy Methods}
	\quad In this section, we employ a bootstrap argument to analyze the nonlinear stability of the perturbation system \eqref{perturSEq}--\eqref{perturVEq}.
	\subsection{Parameter Hierarchy}
	\quad We first determine the parameter $R_1$ in the following lemma, which is essential for characterizing the Hilbert space $X$ introduced in \eqref{X space Hm0}.
	\begin{lemma}\label{Lemma_R1_determination}
		We choose a sufficiently large parameter $R_1$, depending exclusively on $\Lambda, \delta$, $M_1$, $C_4$, the self-similar profile $(\bar{S}, \bar{V})$, and the smooth cutoff modifier $\widehat{\chi}$ ( see \eqref{eq:time-dependent-cutoff}). Then, for any temporal variable $\tau \ge \tau_0$, the profile $(\bar{S}, \bar{V})$ satisfies the following estimates:
		\begin{align}
			& | \widehat{\chi} \bar{S}| + |\widehat{\chi} \bar{V}| + | \nabla(\widehat{\chi} \bar{S})| + |\nabla(\widehat{\chi} \bar{V})| \le \frac{1}{50 M_1} \quad \text{for } y \in B^c(0, R_1), \label{est_decay_1} \\
			& \|\nabla(\widehat{\chi} \bar{S})\|_{L^\infty(B^c(0, R_1))} + \|\nabla(\widehat{\chi} \bar{V})\|_{L^\infty(B^c(0, R_1))} \le \frac{1}{10000}, \label{est_decay_2} \\
			& \|\nabla^2(\widehat{\chi} \bar{S})\|_{L^8(B^c(0, R_1))} + \|\nabla^2(\widehat{\chi} \bar{V})\|_{L^8(B^c(0, R_1))} \le \frac{1}{10000C_4}, \label{est_decay_3} \\
			& \sum_{j=3}^{5} \left( \|\nabla^j(\widehat{\chi} \bar{S})\|_{L^2(B^c(0, R_1))} + \|\nabla^j(\widehat{\chi} \bar{V})\|_{L^2(B^c(0, R_1))} \right) \le \frac{1}{10000C_4}. \label{est_decay_4}
		\end{align}
		where \(M_1>0\) is a fixed constant defined in
		\eqref{Eq_constant_choice} and $C_4$ denotes a constant used in \eqref{Eq_energy_inequality_substitution_final}.
	\end{lemma}
	
	\begin{proof}
		In view of the bounds $|\widehat{\chi}| \le 1$ and $|\nabla \widehat{\chi}| \le 2e^{-\tau}$, the localized bounds \eqref{est_decay_1} and \eqref{est_decay_2} are immediate consequences of the asymptotic decay properties of $(\bar{S}, \bar{V})$.
		
		Next, we target the validation of \eqref{est_decay_3} and \eqref{est_decay_4}. For those components where the derivatives do not fall onto the cutoff function $\widehat{\chi}$, the claimed smallness follows from the same profile decay. On the other hand, considering the geometric definition of $\widehat{\chi}$, its support and derivatives satisfy
		\begin{equation}\label{xdecayestim}
			\mathrm{supp} \, (1 - \widehat{\chi}) \subset B^c\left(0, \frac{e^\tau}{2}\right), \quad \operatorname{supp}(\nabla \widehat{\chi}) \subset B(0, e^\tau) \cap B^c\left(0, \frac{e^\tau}{2}\right), \quad |\nabla^j \widehat{\chi}| \le C(j)e^{-j\tau}.
		\end{equation}
		Consequently, evaluating the localized cross-terms $|\nabla^j \widehat{\chi}| |\nabla^k \bar{S}|$, we deduce that
		$$
		\begin{aligned}
			\| |\nabla^j \widehat{\chi}| |\nabla^k \bar{S}| \|_{L^q(\mathbb{R}^3)}^q &\le C(j, k, q) e^{-j q \tau} \int_{\frac{1}{2}e^\tau}^{e^\tau} r^{-q(\Lambda - 1) - qk + 2} \, \mathrm{d}r \\
			&\le C(j, k, q) e^{-j q \tau} e^{(-q(\Lambda - 1) - qk + 3)\tau} \\
			&\le C(j, k, q) e^{-(2q - 3)\tau} \le C(j, k, q) e^{-\tau}.
		\end{aligned}
		$$
		Since $j \ge 1$, $k \ge 0$, $2 \le j+k \le 5$, $q \in \{2, 8\}$, and $\tau \ge \tau_0 \gg 1$, the terms containing at least one derivative acting on $\widehat{\chi}$ can be made arbitrarily small, which completes the proof.
	\end{proof} 
	
	We first specify the following hierarchy among the parameters used in the analysis:
	\begin{equation}
		\label{Eq_parameter_hierarchy}
		\frac{1}{\tau_0} \ll \sigma_0^{\frac{3}{2}} \ll \sigma_1 \ll \sigma_0 \ll \frac{1}{E} \ll \frac{1}{K} \ll \frac{1}{m} \ll \eta \ll \sigma_g = \frac{25}{12} \varpi \ll f_1 = O(1),
	\end{equation}
	In addition, the top-order index \(K\) is chosen sufficiently large relative
	to the profile repulsivity constant (The constant \(\bar{\eta}\) is defined in Lemma~\ref{prop:profiles-R}):
	\begin{equation}\label{eq:K-repulsivity-relation}
		\frac1K\ll\bar\eta.
	\end{equation}
	This condition guarantees that the positive top-order contribution generated
	by the self-similar transport dominates the commutator errors. It also allows
	the weighted Gagliardo--Nirenberg estimates to be applied with the weight
	exponent \(\eta\). More precisely, after \(m\) and the truncation parameter
	\(J\) have been fixed, we take
	\begin{equation}\label{eq:K-large-condition}
		K>
		\max\left\{
		3^m C(J,m),\,
		500,\,
		C\bar\eta^{-1}
		\right\}.
	\end{equation}
	In particular, \(K\ge6\), as required by the local existence and weighted
	interpolation arguments. We further strengthen the parameter hierarchy by requiring $C(K)\sigma_0\ll1$, $C(E)\sigma_0 \ll 1$, $C(K)\ll E$, $C(R_1) \sigma_0\ll 1$, $C(R_1)\ll m$ and $(\frac{3^K}{2})^{2K} \ll E$.
	We next choose the radius \(R_0\) appearing in the weight \(\phi\). By the
	positivity and decay properties of the profile, \(R_0\) can be taken
	sufficiently large so that
	\begin{equation}\label{eq:R0-choice}
		\begin{aligned}
			\bar S(y)
			&\ge 2\sigma_0,
			&& |y|\le R_0,\\
			|\nabla\bar S(y)|
			+
			|\nabla\bar V(y)|
			&\le \frac{\sigma_1}{2},
			&& |y|\ge R_0,\\
			\bar S(y)+|\bar V(y)|
			&\le C\sigma_0,
			&& |y|\ge R_0.
		\end{aligned}
	\end{equation}
	Since
	\[
	\bar S(y)\sim\langle y\rangle^{1-\Lambda}
	\qquad\text{as }|y|\to\infty,
	\]
	the choice of \(R_0\) may be characterized by
	\begin{equation}\label{eq:R0-sigma0-relation}
		\sigma_0 \sim R_0^{1-\Lambda},
	\end{equation}
	and the parameter hierarchy $\sigma_0^{\frac32}\ll \sigma_1\ll \sigma_0\ll1$
	ensures the existence of a radius \(R_0\) satisfying
	\eqref{eq:R0-choice}. The radial weight used in the
	high-order energy estimate is then chosen to satisfy
	\[
	\phi(y)=1
	\quad\text{for }|y|\le R_0,
	\qquad
	\phi(y)
	=
	\frac{|y|^{2(1-\eta)}}{2R_0^{2(1-\eta)}}
	\quad\text{for }|y|\ge4R_0,
	\]
	with a smooth positive interpolation on
	\(R_0<|y|<4R_0\) (for further details, see \eqref{eq:weight-definition-notation}--\eqref{eq:weight-power-properties}) .
	\begin{remark}
		To ensure that $\sigma_0^{\frac{3}{2}} \ll \sigma_1 \ll \sigma_0$, we choose $\sigma_1 = \sigma_0^{\frac{5}{4}}$. For the estimates required later, we further impose
		\begin{equation}\label{eq:sigma-g-hierarchy}
			\sigma_1 \ll \sigma_0 \sigma_g,
			\qquad \text{or equivalently,} \qquad
			\sigma_0^{\frac{1}{4}} \ll \sigma_g.
		\end{equation}
	\end{remark}
	\begin{remark}
		For the estimates below, we choose $\sigma_1$ and $\sigma_g$ such that
		\begin{equation}\label{R1sigma1sigmag1}
			R_1\left(\frac{\sigma_1}{\sigma_g}\right)^{\frac{3}{10}}\ll 1.
		\end{equation}
	\end{remark}
	Finally, after all the above parameters have been fixed, the initial
	self-similar time \(\tau_0\) is again chosen sufficiently large that
	\begin{equation}\label{eq:tau0-large-condition}
		C(E,\sigma_1,\sigma_0,R_0,R_1,K,m,J)
		e^{-f_1\tau_0}
		\ll 1.
	\end{equation}
	In particular, all terms carrying the factor \(e^{-f_1\tau}\) are perturbative
	on \([\tau_0,\infty)\).
	\subsection{Local Existence and Weighted Estimates}
	\quad In this section, based on the local existence result established in the Appendix, we present the local existence for system \eqref{Eq_self_similar_S} along with the weighted local energy estimates required for the bootstrap argument. For simplicity of exposition in establishing the higher-order estimates, we introduce the notation $\mathcal{F}_s$ and $\mathcal{F}_v$ as follows:
	\begin{align}
		\mathcal{F}_s &:= e^{-f_1 \tau} d \alpha (\delta S)^{\frac{\alpha-1}{\delta}} \Delta S + e^{-f_1 \tau} d \alpha (\alpha-\delta) (\delta S)^{\frac{\alpha-\delta-1}{\delta}} |\nabla S|^2, \label{Eq_def_mathcal_F_s} \\ 
		\mathcal{F}_v &:= e^{-f_1 \tau} (\delta S)^{\frac{\alpha-1}{\delta}}\mathbb{L}(V) + e^{-f_1 \tau} \alpha (\delta S)^{\frac{\alpha-\delta-1}{\delta}} \nabla S \cdot \mathbb{D}V. \label{Eq_def_mathcal_F_v}
	\end{align}
	
	First, we state the local existence theorem.
	\begin{theorem}\label{thm:local-self-similar}
		Let the parameters $(\Lambda,\alpha,\delta)$ satisfy
		\begin{equation}\label{eq:physical-assumptions-local}
			1<\Lambda<2,\qquad
			\delta>0,\qquad
			0<\alpha<1.
		\end{equation}
		Let \(K\ge6\) be an integer. Assume that the initial data
		\((S_0,V_0)\) satisfy
		\begin{equation}\label{eq:local-initial-data}
			\inf_{y\in\R^3}S_0(y)>0,\qquad
			S_0-S^*(\tau_0)\in H^K(\R^3),\qquad
			V_0\in H^K(\R^3).
		\end{equation}
		Then there exists a time \(\tau_*>\tau_0\) such that the Cauchy problem
		\eqref{Eq_self_similar_S} admits a unique solution
		\((S,V)\) on \([\tau_0,\tau_*]\times\R^3\). Moreover, this solution satisfies
		\begin{equation}\label{eq:local-regularity-self-similar}
			\begin{split}
				&\inf_{(\tau,y)\in[\tau_0,\tau_*]\times\R^3}S(\tau,y)>0,
				\qquad
				S-S^*(\tau)\in C([\tau_0,\tau_*];H^{K}(\R^3))
				\cap L^2(\tau_0,\tau_*;H^{K+1}(\R^3)),\\
				&V\in C([\tau_0,\tau_*];H^{K}(\R^3))
				\cap L^2(\tau_0,\tau_*;H^{K+1}(\R^3)).
			\end{split}
		\end{equation}
		and
		\begin{equation}\label{eq:local-time-regularity-weighted}
			\begin{split}
				&\partial_\tau\bigl(S-S^*(\tau)\bigr)
				\in C([\tau_0,\tau_*];H^{K-2}(\R^3))
				\cap L^2(\tau_0,\tau_*;H^{K-1}(\R^3)),\\
				&\partial_\tau V\in L^\infty(\tau_0,\tau_*;H^{K-2}(\R^3))
				\cap L^2(\tau_0,\tau_*;H^{K-1}(\R^3)).
			\end{split}
		\end{equation}
		In addition, if the weighted high-order energy is finite initially, namely
		\[
		E_K(\tau_0)=\int (|\nabla^K S(\tau_0)|^2 + |\nabla^K V(\tau_0)|^2) \phi^K \, \mathrm{d}y<\infty,
		\]
		and the initial data are bounded in the weighted Sobolev space:
		\begin{equation}\label{S,VyR0decay}
			\left\langle \frac{y}{R_0}\right\rangle^{\Lambda}\nabla S(\tau_0,y)\in H^2(\R^3),
			\qquad
			\left\langle \frac{y}{R_0}\right\rangle^{\Lambda}\nabla V(\tau_0,y)\in H^2(\R^3).
		\end{equation}
		Then there exists a  $\tau_1$ ($\tau_0<\tau_1\le\tau_*$) such that
		\begin{equation}\label{eq:weighted-energy-local-bound}
			\sup_{\tau_0\le \tau\le \tau_1}E_K(\tau)<\infty ,
		\end{equation}
		and there exists a constant
		\(C>0\) depending only on the initial data and the parameters
		of the system, such that for all \((\tau,y)\in[\tau_0,\tau_1]\times\R^3\),
		\begin{align}
			|\nabla \widetilde S(\tau,y)|+|\nabla S(\tau,y)|
			&\le 
			C\left\langle \frac{y}{R_0}\right\rangle^{-\Lambda},
			\label{nableS1aaa}\\
			|\nabla \widetilde V(\tau,y)|+|\nabla V(\tau,y)|
			&\le
			C\left\langle \frac{y}{R_0}\right\rangle^{-\Lambda}.
			\label{nableU1aaa}
		\end{align}
	\end{theorem}
	\begin{proof}
		From transformations \eqref{cST-t} and \eqref{vVT-t}, we obtain
		$$\rho_0(x) = \left( \frac{\delta e^{(\Lambda - 1)\tau_0}}{\Lambda} \right)^{\frac{1}{\delta}} S_0^{\frac{1}{\delta}}\left( e^{\tau_0} x \right),\quad
		v_0(x) = \frac{e^{(\Lambda - 1)\tau_0}}{\Lambda} V_0\left( e^{\tau_0} x \right).$$
		Therefore,
		\[
		S_0-S^*(\tau_0)\in H^K(\R^3),\qquad
		\inf_{y\in\R^3}S_0(y)>0,\qquad
		V_0\in H^K(\R^3),
		\]
		imply
		\[
		\rho_0-\rstar\in H^K(\R^3),
		\qquad
		v_0\in H^K(\R^3), \qquad \inf_{x\in\R^3}\rho_0(x)>0,
		\]
		By Proposition \ref{thm:main}, we obtain the existence and uniqueness of a local solution $(\rho,v)$ satisfying \eqref{localboundsrho}--\eqref{vtspaceHk-2} and \eqref{decayestx-2}.
		We now verify the regularity of the self-similar variables. Recall that
		\[
		t=T-e^{-\Lambda\tau},
		\qquad
		x=e^{-\tau}y,
		\]
		and
		\[
		S(\tau,y)
		=
		\frac{\Lambda}{\delta}
		e^{-(\Lambda-1)\tau}
		\rho^\delta(t,x),
		\qquad
		V(\tau,y)
		=
		\Lambda e^{-(\Lambda-1)\tau}
		v(t,x).
		\]
		Thus \(S^*(\tau)\) is spatially constant but time dependent; in particular,
		\(\partial_\tau S^*=-(\Lambda-1)S^*\).
		
		We first consider the spatial regularity. For each fixed \(\tau\), define $\mathcal P_\tau f(y):=f(e^{-\tau}y).$
		Then, for any integer \(m\ge0\),
		\[
		\|\mathcal P_\tau f\|_{H^m_y}
		\le
		C_{m,\tau}\|f\|_{H^m_x}.
		\]
		Moreover, on every finite interval \([\tau_0,\tau_*]\) ($\tau^*=-\frac{\log(T-T^*)}{\Lambda}$), the constants
		\(C_{m,\tau}\) are uniformly bounded. Since
		\[
		\rho-\rstar\in C([0,T^*];H^K_x)
		\cap L^2(0,T^*;H^{K+1}_x),
		\]
		and \(\rho\) remains uniformly away from vacuum, the Sobolev product estimate
		gives
		\[
		\rho^\delta-(\rstar)^\delta
		\in
		C([0,T^*];H^K_x)
		\cap L^2(0,T^*;H^{K+1}_x).
		\]
		Therefore,
		\[
		S-S^*(\tau)
		\in
		C([\tau_0,\tau_*];H^K_y)
		\cap L^2(\tau_0,\tau_*;H^{K+1}_y).
		\]
		Similarly, from
		\[
		v\in C([0,T^*];H^K_x)
		\cap L^2(0,T^*;H^{K+1}_x),
		\]
		we obtain
		\[
		V\in C([\tau_0,\tau_*];H^K_y)
		\cap L^2(\tau_0,\tau_*;H^{K+1}_y).
		\]
		The positive lower bound for $S$ follows similarly.
		
		We next verify that the weighted assumptions in the self-similar variables
		imply the weighted assumptions in the original variables, and consequently
		yield the pointwise decay required for the bootstrap argument. We obtain
		\[
		\langle x\rangle^{\Lambda}\nabla\rho_0(x)\in H^2(\R^3),
		\]
		and 
		\[
		\langle x\rangle^{\Lambda}\nabla v_0(x)\in H^2(\R^3),
		\]
		from \eqref{S,VyR0decay}. Therefore the local weighted estimate in the original variables gives
		\begin{equation}\label{decayestx-2-local}
			|\nabla\rho(t,x)|+|\nabla v(t,x)|
			\le
			C\langle x\rangle^{-\Lambda},
			\qquad 0\le t\le T^*.
		\end{equation}
		
		We now transform this decay back to the self-similar variables. Since
		\[
		S(\tau,y)
		=
		\frac{\Lambda}{\delta}e^{-(\Lambda-1)\tau}
		\rho^\delta(t,x),
		\qquad
		V(\tau,y)
		=
		\Lambda e^{-(\Lambda-1)\tau}v(t,x),
		\]
		with
		\[
		t=T-e^{-\Lambda\tau},
		\qquad
		x=e^{-\tau}y,
		\]
		we have
		\[
		\nabla_y S(\tau,y)
		=
		\Lambda e^{-(\Lambda-1)\tau}e^{-\tau}
		\rho^{\delta-1}(t,x)\nabla_x\rho(t,x),
		\]
		and
		\[
		\nabla_y V(\tau,y)
		=
		\Lambda e^{-(\Lambda-1)\tau}e^{-\tau}
		\nabla_x v(t,x).
		\]
		Since \(\rho\) stays in a fixed positive strip, \(\rho^{\delta-1}\) is uniformly
		bounded. Hence \eqref{decayestx-2-local} implies
		\[
		|\nabla_y S(\tau,y)|+|\nabla_y V(\tau,y)|
		\le
		C\langle e^{-\tau}y\rangle^{-\Lambda}.
		\]
		On the finite time interval \([\tau_0,\tau_*]\), the factor \(e^{-\tau}\) is
		bounded above and below by positive constants. There exists a constant $C > 0$ such that
		\[
		\langle e^{-\tau}y\rangle^{-\Lambda}
		\le
		C\left\langle \frac{y}{R_0}\right\rangle^{-\Lambda}.
		\]
		Thus
		\[
		|\nabla S(\tau,y)|+|\nabla V(\tau,y)|
		\le
		C\left\langle \frac{y}{R_0}\right\rangle^{-\Lambda}.
		\]
		
		Finally, the perturbations \(\widetilde S\) and
		\(\widetilde V\) satisfy the same local weighted bounds by the construction of
		the perturbation variables and the compactly supported cutoffs. Therefore,
		after selecting \(C\) if necessary, we obtain \eqref{nableS1aaa} and \eqref{nableU1aaa}.
		The propagated weighted bounds also control the dilation terms
		\(y\cdot\nabla S\) and \(y\cdot\nabla V\). Substitution into
		\eqref{Eq_self_similar_S}, together with the spatial regularity already
		proved, gives \eqref{eq:local-time-regularity-weighted}.
		
		We now prove the local boundedness of the weighted high-order energy. Define
		\[
		E_K(\tau)
		:=
		\sum_{|\beta|=K}
		\int_{\R^3}
		\bigl(
		|\partial_\beta S|^2+|\partial_\beta V|^2
		\bigr)\phi^K\,dy .
		\]
		Let \([\tau_0,\tau_1]\subset[\tau_0,\tau_*)\) be fixed. By the local existence
		theory and the smoothness of the self-similar transformation on finite time
		intervals, \((S,V)\) is smooth on
		\([\tau_0,\tau_1]\times\R^3\). In particular, all the unweighted norms used
		below are bounded. Since \(K\ge6\), Sobolev embedding
		implies
		\[
		\|S\|_{W^{K-2,\infty}}+\|V\|_{W^{K-2,\infty}}
		+\|S-S^*(\tau)\|_{H^K}
		+\|V\|_{H^K}
		\le C_{\tau_1}.
		\]
		In what follows, all constants \(C_{\tau_1}\) may depend on these local
		unweighted norms, on the parameters of the system, and on the fixed weight
		\(\phi\), but not on \(E_K(\tau)\).
		
		Applying \(\partial_\beta\), \(|\beta|=K\), to
		\eqref{Eq_self_similar_S}, and using
		\[
		\partial_\beta(y\cdot\nabla f)
		=
		y\cdot\nabla\partial_\beta f
		+
		K\partial_\beta f,
		\]
		we obtain
		\[
		\left\{
		\begin{aligned}
			&(\partial_\tau+\Lambda-1+K)\partial_\beta S
			+y\cdot\nabla\partial_\beta S
			+\partial_\beta(V\cdot\nabla S)
			+\delta\partial_\beta(S\diver V)
			=
			\partial_\beta\mathcal F_s,\\
			&(\partial_\tau+\Lambda-1+K)\partial_\beta V
			+y\cdot\nabla\partial_\beta V
			+\partial_\beta(V\cdot\nabla V)
			+\delta\partial_\beta(S\nabla S)
			=
			\partial_\beta\mathcal F_v .
		\end{aligned}
		\right.
		\]
		Multiplying the first equation by \(\phi^K\partial_\beta S\), the second one by
		\(\phi^K\partial_\beta V\), summing over \(|\beta|=K\), and integrating over
		\(\R^3\), we get
		\begin{equation}\label{Eq_rewrite_full_energy_integral_system_local1}
			\begin{aligned}
				&\left(\frac12\frac{d}{d\tau}+\Lambda-1+K\right)E_K(\tau) \\
				&=-
				\sum_{|\beta|=K}
				\int_{\R^3}
				\phi^K y\cdot
				\left(
				\nabla\partial_\beta S\,\partial_\beta S
				+
				\sum_{i=1}^3
				\nabla\partial_\beta V_i\,\partial_\beta V_i
				\right)\,dy \\
				&\quad-
				\sum_{|\beta|=K}
				\int_{\R^3}
				\phi^K
				\left(
				\partial_\beta S\,\partial_\beta(V\cdot\nabla S)
				+
				\sum_{i=1}^3
				\partial_\beta V_i\,\partial_\beta(V\cdot\nabla V_i)
				\right)\,dy \\
				&\quad-
				\delta
				\sum_{|\beta|=K}
				\int_{\R^3}
				\phi^K
				\left(
				\partial_\beta S\,\partial_\beta(S\diver V)
				+
				\sum_{i=1}^3
				\partial_\beta V_i\,\partial_\beta(S\partial_{y_i}S)
				\right)\,dy \\
				&\quad+
				\sum_{|\beta|=K}
				\int_{\R^3}
				\phi^K
				\left(
				\partial_\beta S\,\partial_\beta\mathcal F_s
				+
				\partial_\beta V\cdot\partial_\beta\mathcal F_v
				\right)\,dy \\
				&=: -\mathfrak I_1-\mathfrak I_2-\mathfrak I_3+\mathfrak I_4 .
			\end{aligned}
		\end{equation}
		We estimate these four terms separately. For the linear transport contribution,
		using $|\diver(y\phi^K)|\le C\phi^K,$ we have
		\[
		\begin{aligned}
			|\mathfrak I_1|
			&\le
			\frac12
			\sum_{|\beta|=K}
			\int_{\R^3}
			|\diver(y\phi^K)|
			\bigl(
			|\partial_\beta S|^2+|\partial_\beta V|^2
			\bigr)\,dy  \\
			&\le
			C E_K(\tau).
		\end{aligned}
		\]
		
		For the nonlinear transport contribution $\mathfrak I_2$, we use
		\[
		\partial_\beta(V\cdot\nabla S)
		=
		V\cdot\nabla\partial_\beta S
		+
		(\partial_\beta V)\cdot\nabla S
		+
		\sum_{0<\eta<\beta}
		C_{\beta,\eta}
		(\partial_\eta V)\cdot\nabla\partial_{\beta-\eta}S,
		\]
		and
		\[
		\partial_\beta(V\cdot\nabla V_i)
		=
		V\cdot\nabla\partial_\beta V_i
		+
		(\partial_\beta V)\cdot\nabla V_i
		+
		\sum_{0<\eta<\beta}
		C_{\beta,\eta}
		(\partial_\eta V)\cdot\nabla\partial_{\beta-\eta}V_i .
		\]
		The terms $V\cdot\nabla\partial_\beta S$ and $V\cdot\nabla\partial_\beta V_i$ are treated by integration by parts. Indeed,
		\[
		\diver(\phi^K V)
		=
		\phi^K\diver V
		+
		K\phi^{K-1}V\cdot\nabla\phi .
		\]
		Since
		\[
		\|V\|_{W^{1,\infty}}\le C_{\tau_1},
		\qquad
		|\nabla\phi|\le C\phi,
		\]
		we have
		\[
		|\diver(\phi^K V)|\le C_{\tau_1}\phi^K.
		\]
		Thus integrating by parts, we have
		\[
		\left|
		\int_{\R^3}
		\phi^K V\cdot\nabla\partial_\beta S\,\partial_\beta S\,dy
		\right|
		+
		\sum_{i=1}^3
		\left|
		\int_{\R^3}
		\phi^K V\cdot\nabla\partial_\beta V_i\,\partial_\beta V_i\,dy
		\right|
		\le
		C_{\tau_1}E_K(\tau).
		\]
		The terms containing \(\partial_\beta V\) are bounded directly by the
		unweighted \(L^\infty\) bounds:
		\[
		\left|
		\int_{\R^3}
		\phi^K
		\bigl((\partial_\beta V)\cdot\nabla S\bigr)\partial_\beta S\,dy
		\right|
		+
		\sum_{i=1}^3
		\left|
		\int_{\R^3}
		\phi^K
		\bigl((\partial_\beta V)\cdot\nabla V_i\bigr)\partial_\beta V_i\,dy
		\right| \le
		C_{\tau_1}E_K(\tau).
		\]
		It remains to estimate the lower-order terms.
		We now estimate
		\[
		\left\|
		\phi^{K/2}
		(\partial_\eta V)\cdot\nabla\partial_{\beta-\eta}S
		\right\|_{L^2},
		\qquad
		|\beta|=K,\qquad 1\le |\eta|\le K-1 .
		\]
		If \(|\eta|=1\), in this case we simply put the first factor in
		\(L^\infty\). Since \(K\ge6\), Sobolev embedding gives
		\[
		\|\partial_\eta V\|_{L^\infty}\le C_{\tau_1}.
		\]
		Hence
		\[
		\left\|
		\phi^{K/2}
		(\partial_\eta V)\cdot\nabla\partial_{\beta-\eta}S
		\right\|_{L^2}
		\le
		C_{\tau_1}
		\|\phi^{K/2}\nabla\partial_{\beta-\eta}S\|_{L^2}
		\le
		C_{\tau_1}E_K(\tau)^{1/2}.
		\]
		It remains to consider \(2\le |\eta|\le K-1\). Choose \(r_0,r_1\in(1,\infty)\)
		with $\frac1{r_0}+\frac1{r_1}=\frac{1}{2}$ . By Hölder's
		inequality,
		\[
		\begin{aligned}
			&\left\|
			\phi^{K/2}
			(\partial_\eta V)\cdot\nabla\partial_{\beta-\eta}S
			\right\|_{L^2} \\
			&\le
			\left\|
			\langle y\rangle^{-\varepsilon_0}
			\phi^{K\theta_0/2}
			\partial_\eta V
			\right\|_{L^{r_0}}
			\left\|
			\langle y\rangle^{-\varepsilon_1}
			\phi^{K\theta_1/2}
			\nabla\partial_{\beta-\eta}S
			\right\|_{L^{r_1}}  \\
			&\quad\times
			\left\|
			\langle y\rangle^{\varepsilon_0+\varepsilon_1}
			\phi^{\frac K2(1-\theta_0-\theta_1)}
			\right\|_{L^\infty}.
		\end{aligned}
		\]
		Taking 
		\[
		\psi=1,\qquad
		\varphi=\phi^{\frac{1}{2}},
		\qquad
		p=\infty,\qquad
		q=2,\qquad
		l=K.
		\]
		Let \(\theta_0,\theta_1\in(0,1)\) be chosen so that
		\[
		\frac1{r_0}
		=
		\frac{|\eta|}{3}
		+
		\theta_0\left(\frac12-\frac {K}{3}\right), \qquad \frac1{r_1}
		=
		\frac{K-|\eta|+1}{3}
		+
		\theta_1\left(\frac12-\frac {K}{3}\right).
		\]
		The condition $\frac1{r_0}+\frac1{r_1}=\frac12$ implies $\theta_0+\theta_1=\frac{2K-1}{2K-3}$.
		The hypotheses \eqref{eq:weighted-GN-extra-new} of Lemma~\ref{lem:weighted-GN-new} are satisfied provided that $K\eta \gg 1$. Let \(\varepsilon_0,\varepsilon_1>0\) be sufficiently small. 
		Since $\theta_0+\theta_1-1=\frac{2}{2K-3}>0$, and \(\phi(y)\sim |y|^{2(1-\eta)}\) in the exterior region, where
		\(\eta\) denotes the exponent in the definition of the weight, we can choose
		\(\varepsilon_0,\varepsilon_1>0\) so small that $\left\|
		\langle y\rangle^{\varepsilon_0+\varepsilon_1}
		\phi^{\frac K2(1-\theta_0-\theta_1)}
		\right\|_{L^\infty}
		<\infty .$ Applying Lemma~\ref{lem:weighted-GN-new} to \(\partial_\eta V\) gives
		\[
		\left\|
		\langle y\rangle^{-\varepsilon_0}
		\phi^{\frac{K\theta_0}{2}}
		\partial_\eta V
		\right\|_{L^{r_0}}
		\le
		C_{\tau_1}
		\left(
		1+
		E_K(\tau)^{\theta_0/2}
		\right),
		\]
		and applying the same lemma to
		\(\nabla\partial_{\beta-\eta}S\), which has order \(K-|\eta|+1\), gives
		\[
		\left\|
		\langle y\rangle^{-\varepsilon_1}
		\phi^{\frac{K\theta_1}{2}}
		\nabla\partial_{\beta-\eta}S
		\right\|_{L^{r_1}}
		\le
		C_{\tau_1}
		\left(
		1+
		E_K(\tau)^{\theta_1/2}
		\right).
		\]
		Therefore,
		\[
		\begin{aligned}
			\left\|
			\phi^{K/2}
			(\partial_\eta V)\cdot\nabla\partial_{\beta-\eta}S
			\right\|_{L^2}
			&\le
			C_{\tau_1}
			\left(
			1+
			E_K(\tau)^{\theta_0/2}
			\right)
			\left(
			1+
			E_K(\tau)^{\theta_1/2}
			\right)  \\
			&\le
			C_{\tau_1}
			\left(
			1+
			E_K(\tau)^{\frac{\theta_0+\theta_1}{2}}
			\right).
		\end{aligned}
		\]
		Therefore, we finally obtain
		\begin{equation}
			\begin{split}
				\left|\sum_{|\beta|=K}
				\int_{\R^3}
				\phi^K
				\left(
				\partial_\beta S\,\partial_\beta(V\cdot\nabla S)
				\right)\,dy \right| &\le C_{\tau_1}(1+E_K(\tau))\\
				&\quad+\sum_{|\beta|=K}\sum_{0<\eta<\beta}
				C_{\beta,\eta}\left\|
				\phi^{K/2}
				(\partial_\eta V)\cdot\nabla\partial_{\beta-\eta}S\right\|_{L^2}\left\|\phi^{K/2}\partial_\beta S\right\|_{L^2}\\
				&\le C_{\tau_1}(1+E_K(\tau)^{\frac{2K-2}{2K-3}}).
			\end{split}
		\end{equation}
		The term $|\sum_{|\beta|=K}
		\int_{\R^3}\left(
		\sum_{i=1}^3\phi^K
		\partial_\beta V_i\,\partial_\beta(V\cdot\nabla V_i)
		\right)\,dy|$ is estimated in the same way.
		Combining the estimates above, we obtain
		\[
		|\mathfrak I_2|
		\le
		C_{\tau_1}(1+E_K(\tau)^{\frac{2K-2}{2K-3}}).
		\]
		
		Proceeding as in the estimate of $\mathfrak{I}_2$, we likewise have
		\[
		|\mathfrak I_3|
		\le C_{\tau_1}(1+E_K(\tau)^{\frac{2K-2}{2K-3}}).
		\]
		The only difference arises in the treatment of the highest-order terms. Namely,
		we integrate by parts once in
		\[
		\begin{aligned}
			&\delta\sum_{|\beta|=K}
			\int_{\R^3}
			\phi^K
			\left(
			S\partial_\beta S\,\partial_\beta\operatorname{div}V
			+
			\sum_{i=1}^3
			S\partial_\beta V_i\,
			\partial_{y_i}\partial_\beta S
			\right)\,\mathrm dy\\
			&=
			\delta\sum_{|\beta|=K}\sum_{i=1}^3
			\int_{\R^3}
			\phi^K S\partial_\beta S\,
			\partial_{y_i}\partial_\beta V_i\,\mathrm dy
			+
			\delta\sum_{|\beta|=K}\sum_{i=1}^3
			\int_{\R^3}
			\phi^K S\partial_\beta V_i\,
			\partial_{y_i}\partial_\beta S\,\mathrm dy\\
			&=
			-\delta\sum_{|\beta|=K}\sum_{i=1}^3
			\int_{\R^3}
			\partial_\beta V_i\,
			\partial_{y_i}\bigl(\phi^K S\bigr)
			\partial_\beta S\,\mathrm dy.
		\end{aligned}
		\]
		
		It remains to estimate the term \(\mathfrak I_4\). Substituting
		\eqref{Eq_def_mathcal_F_s}--\eqref{Eq_def_mathcal_F_v} into its definition, we
		write
		\[
		\begin{aligned}
			\mathfrak I_4
			&=
			e^{-f_1\tau}
			\sum_{|\beta|=K}
			\int_{\R^3}
			\phi^K\partial_\beta S\,
			\partial_\beta
			\left[
			d\alpha(\delta S)^{\frac{\alpha-1}{\delta}}\Delta S
			+
			d\alpha(\alpha-\delta)
			(\delta S)^{\frac{\alpha-\delta-1}{\delta}}|\nabla S|^2
			\right]\,dy\\
			&\quad+
			e^{-f_1\tau}
			\sum_{|\beta|=K}
			\int_{\R^3}
			\phi^K\partial_\beta V\cdot
			\partial_\beta
			\left[
			(\delta S)^{\frac{\alpha-1}{\delta}}\mathbb L(V)
			+
			\alpha(\delta S)^{\frac{\alpha-\delta-1}{\delta}}
			\nabla S\cdot\mathbb D V
			\right]\,dy\\
			&=:\mathfrak I_{4,1}+\mathfrak I_{4,2}.
		\end{aligned}
		\]
		Set
		\[
		D_K(\tau):=
		\sum_{|\beta|=K}
		\int_{\R^3}
		\phi^K
		\left(
		|\nabla\partial_\beta S|^2
		+
		|\nabla\partial_\beta V|^2
		\right)\,dy .
		\]
		We first consider the term $\mathfrak I_{4,1}$ and write
		\[
		\partial_\beta(d\alpha(\delta S)^{\frac{\alpha-1}{\delta}}\Delta S)
		=
		d\alpha(\delta S)^{\frac{\alpha-1}{\delta}}\Delta\partial_\beta S
		+
		[\partial_\beta,d\alpha(\delta S)^{\frac{\alpha-1}{\delta}}]\Delta S .
		\]
		The principal part gives, after integration by parts,
		\[
		\begin{aligned}
			e^{-f_1\tau}
			\int_{\R^3}
			\phi^K\partial_\beta S\,d\alpha(\delta S)^{\frac{\alpha-1}{\delta}}\Delta\partial_\beta S\,dy
			&=
			-e^{-f_1\tau}
			\int_{\R^3}
			\phi^K d\alpha(\delta S)^{\frac{\alpha-1}{\delta}}|\nabla\partial_\beta S|^2\,dy  \\
			&\quad
			-e^{-f_1\tau}
			\int_{\R^3}
			\partial_\beta S\,\nabla(\phi^K d\alpha(\delta S)^{\frac{\alpha-1}{\delta}})
			\cdot\nabla\partial_\beta S\,dy .
		\end{aligned}
		\]
		Since
		\[
		|\nabla\phi|\le C\phi,
		\qquad
		\|d\alpha(\delta S)^{\frac{\alpha-1}{\delta}}\|_{W^{1,\infty}}\le C_{\tau_1},
		\]
		the second term is bounded by
		\[
		\vartheta e^{-f_1\tau}
		\|\phi^{K/2}\nabla\partial_\beta S\|_{L^2}^2
		+
		C_{\vartheta,\tau_1}e^{-f_1\tau}
		\|\phi^{K/2}\partial_\beta S\|_{L^2}^2 .
		\]
		By Leibniz' rule and the composition formula, the remaining term containing $[\partial_\beta,
		d\alpha(\delta S)^{\frac{\alpha-1}{\delta}}]\Delta S
		+
		d\alpha(\alpha-\delta)
		\partial_\beta\left[
		(\delta S)^{\frac{\alpha-\delta-1}{\delta}}
		|\nabla S|^2
		\right]$ is a
		linear combination of terms of the form
		\[
		C(S)
		\prod_{j=0}^{\ell}
		\partial_{\bar\beta^j}S,
		\qquad
		\sum_{j=0}^{\ell}|\bar\beta^j|=K+2,
		\qquad
		1\le|\bar\beta^j|\le K+1,
		\]
		where \(C(S)\) is a smooth function of \(S\), uniformly bounded on
		\([\tau_0,\tau_1]\). More precisely,
		\[
		\begin{aligned}
			&[\partial_\beta,
			d\alpha(\delta S)^{\frac{\alpha-1}{\delta}}]\Delta S
			+
			d\alpha(\alpha-\delta)
			\partial_\beta\left[
			(\delta S)^{\frac{\alpha-\delta-1}{\delta}}
			|\nabla S|^2
			\right] \\
			&=
			\sum_{\substack{
					\sum_{j=0}^{\ell}|\bar\beta^j|=K+2\\
					1\le|\bar\beta^j|\le K+1}}
			C_{\bar\beta^0,\ldots,\bar\beta^\ell}(S)
			\prod_{j=0}^{\ell}\partial_{\bar\beta^j}S .
		\end{aligned}
		\]
		The terms containing a factor of order \(K+1\) satisfy
		\[
		\begin{aligned}
			&e^{-f_1\tau}
			\sum_{|\beta|=K}
			\left|
			\int_{\R^3}
			\phi^K\partial_\beta S\,
			C(S)\partial_{\bar\beta^0}S
			\prod_{j=1}^{\ell}\partial_{\bar\beta^j}S\,dy
			\right| \\
			&\qquad\le
			C_{\tau_1}e^{-f_1\tau}
			E_K(\tau)^{1/2}
			\|\phi^{K/2}\nabla^{K+1}S\|_{L^2} \\
			&\qquad\le
			\vartheta e^{-f_1\tau}
			\|\phi^{K/2}\nabla^{K+1}S\|_{L^2}^2
			+
			C_{\vartheta,\tau_1}e^{-f_1\tau}E_K(\tau),
			\qquad |\bar\beta^0|=K+1.
		\end{aligned}
		\]
		The terms containing a factor of order \(K\) are estimated by placing that
		factor in weighted \(L^2\) and all the other factors in \(L^\infty\), which
		gives
		\[
		C_{\tau_1}e^{-f_1\tau}\bigl(1+E_K(\tau)\bigr).
		\]
		It remains to consider
		\[
		\sum_{j=0}^{\ell}|\bar\beta^j|=K+2,
		\qquad
		1\le|\bar\beta^j|\le K-1.
		\]
		Set
		\[
		r_j:=\frac{2(K+2)}{|\bar\beta^j|},
		\qquad
		\theta_j:=
		\frac{|\bar\beta^j|-3/r_j}{K-\frac32},
		\qquad
		j=0,\ldots,\ell.
		\]
		Then
		\[
		\sum_{j=0}^{\ell}\frac1{r_j}=\frac12,
		\qquad
		\sum_{j=0}^{\ell}\theta_j
		=
		\frac{K+\frac12}{K-\frac32}
		=
		\frac{2K+1}{2K-3}.
		\]
		By Hölder's inequality and Lemma~\ref{lem:weighted-GN-new}, taking 
		\[
		\psi=1,\qquad
		\varphi=\phi^{\frac{1}{2}},
		\qquad
		p=\infty,\qquad
		q=2,\qquad
		l=K,
		\]
		and for sufficiently
		small \(\varepsilon_j>0\),
		\[
		\begin{aligned}
			&\left\|
			\phi^{K/2}
			C(S)\prod_{j=0}^{\ell}\partial_{\bar\beta^j}S
			\right\|_{L^2} \\
			&\quad\le
			C_{\tau_1}
			\prod_{j=0}^{\ell}
			\left\|
			\langle y\rangle^{-\varepsilon_j}
			\phi^{K\theta_j/2}
			\nabla^{|\bar\beta^j|}S
			\right\|_{L^{r_j}}
			\left\|
			\langle y\rangle^{\sum_{j=0}^{\ell}\varepsilon_j}
			\phi^{\frac K2(1-\sum_{j=0}^{\ell}\theta_j)}
			\right\|_{L^\infty} \\
			&\quad\le
			C_{\tau_1}
			\prod_{j=0}^{\ell}
			\left(
			1+E_K(\tau)^{\theta_j/2}
			\right)
			\le
			C_{\tau_1}
			\left(
			1+E_K(\tau)^{\frac{2K+1}{2(2K-3)}}
			\right).
		\end{aligned}
		\]
		Here the last weighted \(L^\infty\)-norm is finite because
		\[
		\sum_{j=0}^{\ell}\theta_j-1
		=
		\frac{4}{2K-3}>0,
		\]
		provided that the parameters \(\varepsilon_j\) are sufficiently small and
		\(K\eta\) is sufficiently large. Consequently,
		\[
		\begin{aligned}
			& e^{-f_1\tau}
			\sum_{|\beta|=K}
			\left|
			\int_{\R^3}
			\phi^K\partial_\beta S
			\sum_{\substack{
					\sum_{j=0}^{\ell}|\bar\beta^j|=K+2\\
					1\le|\bar\beta^j|\le K-1}}
			C_{\bar\beta^0,\ldots,\bar\beta^\ell}(S)
			\prod_{j=0}^{\ell}\partial_{\bar\beta^j}S\,dy
			\right| \\
			&\quad\le
			C_{\tau_1}e^{-f_1\tau}
			E_K(\tau)^{1/2}
			\left(
			1+E_K(\tau)^{\frac{2K+1}{2(2K-3)}}
			\right) \\
			&\quad\le
			C_{\tau_1}e^{-f_1\tau}
			\left(
			1+E_K(\tau)^{\frac{2K-1}{2K-3}}
			\right).
		\end{aligned}
		\]
		Combining the above estimates and taking \(\vartheta>0\) sufficiently small,
		We conclude that there exists a positive constant \(\kappa_0>0\) such that
		\[
		\mathfrak I_{4,1}
		\le
		-\frac{\kappa_0}{2}e^{-f_1\tau}
		\sum_{|\beta|=K}
		\|\phi^{K/2}\nabla\partial_\beta S\|_{L^2}^2 +
		C_{\tau_1}e^{-f_1\tau}
		\left(
		1+E_K(\tau)^{\frac{2K-1}{2K-3}}
		\right).
		\]
		Similarly, $\mathfrak{I}_{4,2}$ can be estimated as (using the ellipticity of the operator \(\mathbb L\), as stated in \eqref{Lvellpiticoo})
		\[
		\begin{split}
			\mathfrak I_{4,2}
			&\le
			-\frac{\kappa_0}{2}e^{-f_1\tau}
			\sum_{|\beta|=K}
			\|\phi^{K/2}\nabla\partial_\beta V\|_{L^2}^2 +\frac{\kappa_0}{4}e^{-f_1\tau}\sum_{|\beta|=K}
			\|\phi^{K/2}\nabla\partial_\beta S\|_{L^2}^2\\
			&\quad+
			\sum_{|\beta|=K}
			C_{\tau_1}e^{-f_1\tau}
			\left(
			1+E_K(\tau)^{\frac{2K-1}{2K-3}}
			\right).
		\end{split}
		\]
		In particular, since \(e^{-f_1\tau}\) is bounded on
		\([\tau_0,\tau_1]\), this implies there exists a positive constant $c_{\tau_1}$ depending on $\tau_1$ such that
		\[
		\mathfrak I_4
		\le -c_{\tau_1}D_K(\tau)+C_{\tau_1}
		\left(
		1+E_K(\tau)^{\frac{2K-1}{2K-3}}
		\right)
		.
		\]
		
		Substituting the estimates for
		\(\mathfrak I_1,\mathfrak I_2,\mathfrak I_3,\mathfrak I_4\) into
		\eqref{Eq_rewrite_full_energy_integral_system_local1}, we obtain
		\[
		\frac{d}{d\tau}E_K(\tau)
		\le
		C_{\tau_1}
		\left(
		1+E_K(\tau)^{\frac{2K-1}{2K-3}}
		\right),
		\qquad
		\tau\in[\tau_0,\tau_1].
		\]
		Set
		\[
		p:=\frac{2K-1}{2K-3}>1,
		\qquad
		Z(\tau):=1+E_K(\tau).
		\]
		Since
		\[
		1+E_K(\tau)^p\le 2Z(\tau)^p,
		\]
		after enlarging \(C_{\tau_1}\) if necessary, we have
		\[
		Z'(\tau)\le C_{\tau_1}Z(\tau)^p.
		\]
		By the comparison principle for ordinary differential inequalities,
		\[
		Z(\tau)
		\le
		\left[
		Z(\tau_0)^{1-p}
		-
		(p-1)C_{\tau_1}(\tau-\tau_0)
		\right]^{-\frac1{p-1}},
		\]
		as long as
		\[
		Z(\tau_0)^{1-p}
		-
		(p-1)C_{\tau_1}(\tau-\tau_0)>0.
		\]
		Consequently, after choosing \(\tau_1-\tau_0>0\) sufficiently small so that
		\[
		(p-1)C_{\tau_1}(\tau_1-\tau_0)
		\le
		\frac12 Z(\tau_0)^{1-p},
		\]
		we obtain
		\[
		\sup_{\tau_0\le\tau\le\tau_1}E_K(\tau)
		\le
		2^{\frac1{p-1}}
		\bigl(1+E_K(\tau_0)\bigr)-1.
		\]
		Thus the weighted high-order energy remains finite on a sufficiently short
		time interval.
		
	\end{proof}
	\subsection{Truncated Equations and Choice of Initial Data}
	In this section, we truncate the system (or equation) and subsequently analyze its linear and nonlinear parts. Meanwhile, we make a preliminary choice of the initial data.
	Choose fixed radial cutoffs \(\widehat\chi_1,\widehat\chi_2\in
	C_c^\infty(\R^3)\) such that
	\[
	\widehat\chi_1=1\ \text{on }B(0,R_1),\quad
	\operatorname{supp}\widehat\chi_1\subset B(0,2R_1),\qquad
	\widehat\chi_2=1\ \text{on }B(0,2R_1),\quad
	\operatorname{supp}\widehat\chi_2\subset B(0,3R_1).
	\]
	We take \(\tau_0\) large enough that \(3R_1<e^{\tau_0}/2\); hence the
	time-dependent cutoff \(\widehat\chi\) equals one on the support of
	\(\widehat\chi_2\). For a sufficiently large damping parameter \(J>0\), we define the truncated
	linearized operator according to the perturbation equations \eqref{perturSEq}--\eqref{perturVEq}
	\begin{equation}\label{Ls,LvSSVV}
		\mathcal{L}=(\mathcal{L}_s,\mathcal{L}_v):= \widehat{\chi}_2(\mathcal{L}_s^1, \mathcal{L}_v^1) - J(1-\widehat{\chi}_1),
	\end{equation}
	We complement our system with the following auxiliary truncated equation:
	\begin{equation}\label{truncatedsystem}
		\left\{
		\begin{aligned}
			& \partial_\tau \widetilde{\widetilde{S}} = \mathcal{L}_s (\widetilde{\widetilde{S}}, \widetilde{\widetilde{V}}) + \widehat{\chi}_2({\mathcal{N}_s(\widetilde{S},\widetilde{V})}+{\mathcal{E}_s(\bar{S},\bar{V})} \\
			&\quad\quad\quad+e^{-f_1 \tau} d \alpha (\delta S)^{\frac{\alpha-1}{\delta}} \Delta S + e^{-f_1 \tau} d \alpha (\alpha-\delta) (\delta S)^{\frac{\alpha-\delta-1}{\delta}} |\nabla S|^2), \\
			& \partial_\tau \widetilde{\widetilde{V}} = \mathcal{L}_v (\widetilde{\widetilde{S}}, \widetilde{\widetilde{V}}) + \widehat{\chi}_2 ({\mathcal{N}_v(\widetilde{S},\widetilde{V})}+{\mathcal{E}_v(\bar{S},\bar{V})} \\
			&\quad\quad\quad+e^{-f_1 \tau} (\delta S)^{\frac{\alpha-1}{\delta}}\mathbb{L}(V) + e^{-f_1 \tau} \alpha (\delta S)^{\frac{\alpha-\delta-1}{\delta}} \nabla S \cdot \mathbb{D}V), \\
			& (\widetilde{\widetilde{S}}, \widetilde{\widetilde{V}})(\tau_0, y) = (\widetilde{\widetilde{S}}_0, \widetilde{\widetilde{V}}_0)(y) \quad \quad \text{for } y \in \mathbb{R}^3, \\
			& (\widetilde{\widetilde{S}}, \widetilde{\widetilde{V}})(\tau, y) \rightarrow (0, 0) \quad \quad \quad \quad \text{as } |y| \rightarrow \infty \quad \text{for } \tau \ge \tau_0,
		\end{aligned}
		\right.
	\end{equation}
	the initial data for the truncated system are chosen as
	\begin{equation}\label{initialdatatruncated}
		\left\{
		\begin{aligned}
			\widetilde{\widetilde S}_0
			&=
			\widehat{\chi}_2\widetilde S_0^*
			+
			\sum_{j=1}^{N}
			\widehat k_j\Phi_{j,S},\\
			\widetilde{\widetilde V}_0
			&=
			\widehat{\chi}_2\widetilde V_0^*
			+
			\sum_{j=1}^{N}
			\widehat k_j\Phi_{j,V},
		\end{aligned}
		\right.
	\end{equation}
	where \((\widetilde S_0^*,\widetilde V_0^*)\) denotes the prescribed stable
	part of the initial perturbation, while the $\Phi_{j,S}$,$\Phi_{j,V}$ $(j=1,2,...,N)$ are an orthonormal basis of \(X_{\rm u}\) constructed in Lemma~\ref{lem:new-stable-unstable-summary} and the coefficients
	\(\{\widehat{k}_j\}_{j=1}^{N}\) will be selected later to control the
	unstable modes.
	Let $\widehat\chi_0(y):=\widehat\chi(\tau_0,y),$ the stable part $(\widetilde S_0^*,\widetilde V_0^*)$
	is selected so that the complete initial perturbation
	\((\widetilde S_0,\widetilde V_0)\) obeys
	\begin{equation}\label{initialdatacondition1}
		\begin{aligned}
			&(\chi_2\widetilde S_0^*,\chi_2\widetilde V_0^*)
			\in X_{\mathrm{s}},\\
			&\max\left\{
			\|\widetilde S_0\|_{L^\infty},
			\|\widetilde V_0\|_{L^\infty},
			\|(\widehat\chi_2\widetilde S_0^*,
			\widehat\chi_2\widetilde V_0^*)\|_X
			\right\}
			\leq \sigma_1,\\
			&E_K(\tau_0)\le\frac{E}{2},\\
			&S_0(y)
			=
			\widetilde S_0(y)+\widehat\chi_0(y)\overline S(y)
			\ge
			\frac{\sigma_1}{2}
			\left\langle\frac{y}{R_0}\right\rangle^{1-\Lambda},
			\qquad y\in\R^3.
		\end{aligned}
	\end{equation}
	We further require the initial data to have the spatial decay needed in the
	bootstrap argument:
	\begin{equation}\label{initialdatacondition2}
		\sup_{y\in\R^3}
		\left\langle\frac{y}{R_0}\right\rangle^{\Lambda}
		\left[
		\left|
		\nabla\bigl(
		\widetilde S_0+\widehat\chi_0\overline S
		\bigr)(y)
		\right|
		+
		\left|
		\nabla\bigl(
		\widetilde V_0+\widehat\chi_0\overline V
		\bigr)(y)
		\right|
		\right]
		<\infty,
	\end{equation}
	The equations for
	\((\widetilde{\widetilde S},\widetilde{\widetilde V})\) form a first-order
	symmetric hyperbolic system with smooth coefficients determined by the local
	solution \((S,V)\). Since Theorem~\ref{thm:local-self-similar} ensures that
	\[
	(\widetilde S,\widetilde V)
	=
	(S-\widehat\chi\,\bar S,\,
	V-\widehat\chi\,\bar V)
	\]
	is smooth on \([\tau_0,\tau_1]\times\R^3\), the standard local theory for
	symmetric hyperbolic systems, yields a unique smooth solution
	$(\widetilde{\widetilde S},\widetilde{\widetilde V})$
	of the truncated Cauchy problem on
	\([\tau_0,\tau_1]\times\R^3\), with initial data
	\eqref{initialdatatruncated}.
	
	We now establish the relation between the truncated perturbation
	\((\widetilde{\widetilde S},\widetilde{\widetilde V})\) and the original
	perturbation \((\widetilde S,\widetilde V)\).
	\begin{lemma}\label{localidentity}
		Assume that $(\widetilde{\widetilde{S}},\widetilde{\widetilde{V}})$ represents a smooth solution to \eqref{truncatedsystem} over the temporal interval $[\tau_0, \tau_1] \times \mathbb{R}^3$. If, at the initial state $\tau = \tau_0$, the variables coincide on the ball:
		\begin{equation*}
			(\widetilde{\widetilde{S}}_0, \widetilde{\widetilde{V}}_0) = (\widetilde{S}_0, \widetilde{V}_0) \quad \text{for all } y \in B(0, R_1),
		\end{equation*}
		then the unique identification holds throughout the evolution:
		\begin{equation*}
			(\widetilde{\widetilde{S}}, \widetilde{\widetilde{V}}) = (\widetilde{S}, \widetilde{V}) \quad \text{for all } (\tau, y) \in [\tau_0, \tau_1] \times B(0, R_1).
		\end{equation*}
	\end{lemma}
	
	\begin{proof}
		Let us introduce the mismatch variables defined by
		\begin{equation*}
			(\mathcal{R}_S, \mathcal{R}_V) := (\widetilde{S} - \widetilde{\widetilde{S}}, \, \widetilde{V} - \widetilde{\widetilde{V}}).
		\end{equation*}
		Then, the dynamics of $(\mathcal{R}_S, \mathcal{R}_V)$ are governed by the following localized perturbation system:
		\begin{equation}\label{4.16}
			\begin{aligned}
				\partial_\tau \mathcal{R}_S + (\Lambda - 1)\mathcal{R}_S &= -(y + \widehat{\chi}\bar{V}) \cdot \nabla \mathcal{R}_S - \delta(\widehat{\chi}\bar{S}) \operatorname{div}\mathcal{R}_V - \mathcal{R}_V \cdot \nabla(\widehat{\chi}\bar{S}) - \delta \mathcal{R}_S \operatorname{div}(\widehat{\chi}\bar{V}), \\
				\partial_\tau \mathcal{R}_V + (\Lambda - 1)\mathcal{R}_V &= -(y + \widehat{\chi}\bar{V}) \cdot \nabla \mathcal{R}_V - \delta(\widehat{\chi}\bar{S}) \nabla \mathcal{R}_S - \mathcal{R}_V \cdot \nabla(\widehat{\chi}\bar{V}) - \delta \mathcal{R}_S \nabla(\widehat{\chi}\bar{S})
			\end{aligned}
		\end{equation}
		within the localized domain $B(0, R_1)$, equipped with the trivial initial data $(\mathcal{R}_S, \mathcal{R}_V)|_{\tau=\tau_0} = (0,0)$.
		
		Recalling that $(\bar{S}(y), \bar{V}(y))$ are spherically symmetric, there exists an underlying scalar field $\bar{\mathcal{V}}(|y|)$ such that its vector representation is given by
		\begin{equation}\label{sphericallysysm}
			(\bar{S}, \bar{V})(y) = \left( \bar{S}(|y|), \, \bar{\mathcal{V}}(|y|)\frac{y}{|y|} \right).
		\end{equation}
		Multiplying \eqref{4.16}$_1$ and \eqref{4.16}$_2$ by $\mathcal{R}_S$ and $\mathcal{R}_V$, respectively, summing the resulting equations, and integrating over the ball $B(0, R_1)$, we deduce from \eqref{sphericallysysm} and the profile bounds \eqref{eq:profile-outgoing-property} that
		$$
		\begin{aligned}
			&\frac{1}{2} \frac{\mathrm{d}}{\mathrm{d}\tau} \int_{B(0,R_1)} \left( \mathcal{R}_S^2 + |\mathcal{R}_V|^2 \right) \mathrm{d}y \\
			&= \int_{B(0,R_1)} ( -\frac{1}{2}(y + \widehat{\chi}\bar{V}) \cdot \nabla(\mathcal{R}_S^2) - \delta (\widehat{\chi}\bar{S}) \mathcal{R}_S \operatorname{div}\mathcal{R}_V \\
			&\quad- (\Lambda - 1)\mathcal{R}_S^2 - \mathcal{R}_S \mathcal{R}_V \cdot \nabla (\widehat{\chi}\bar{S}) - \delta \mathcal{R}_S^2 \operatorname{div}(\widehat{\chi}\bar{V}) ) \mathrm{d}y.\\
			&\quad + \int_{B(0,R_1)} ( -\frac{1}{2}(y + \widehat{\chi}\bar{V}) \cdot \nabla(|\mathcal{R}_V|^2) - \delta (\widehat{\chi}\bar{S}) \nabla \mathcal{R}_S \cdot \mathcal{R}_V \\
			&\quad- (\Lambda - 1)|\mathcal{R}_V|^2 - \mathcal{R}_V \cdot \nabla(\widehat{\chi}\bar{V}) \cdot \mathcal{R}_V - \delta \mathcal{R}_S \nabla (\widehat{\chi}\bar{S}) \cdot \mathcal{R}_V ) \mathrm{d}y \\
			&= \int_{B(0,R_1)} ( \frac{3 + \operatorname{div}\bar{V}}{2}(\mathcal{R}_S^2 + |\mathcal{R}_V|^2) + \delta \nabla \bar{S} \cdot (\mathcal{R}_S \mathcal{R}_V) - (\Lambda - 1)(\mathcal{R}_S^2 + |\mathcal{R}_V|^2) \\
			&\quad- \mathcal{R}_V \cdot \nabla \bar{V} \cdot \mathcal{R}_V - (\delta + 1)\mathcal{R}_S \mathcal{R}_V \cdot \nabla \bar{S} - \delta \mathcal{R}_S^2 \operatorname{div}\bar{V} ) \mathrm{d}y \\
			&\quad +\int_{\partial B(0,R_1)} ( -\frac{|y| + \bar{\mathcal{V}}}{2}(\mathcal{R}_S^2 + |\mathcal{R}_V|^2) - \delta \bar{S} \mathcal{R}_S (\mathcal{R}_V \cdot \boldsymbol{n}) ) \mathrm{d}S \\
			&\le C \int_{B(0,R_1)} (\mathcal{R}_S^2 + |\mathcal{R}_V|^2) \, \mathrm{d}y - \int_{\partial B(0,R_1)} \frac{|y| + \bar{\mathcal{V}} - \delta \bar{S}}{2}(\mathcal{R}_S^2 + |\mathcal{R}_V|^2) \, \mathrm{d}S \\
			&\le C \int_{B(0,R_1)} (\mathcal{R}_S^2 + |\mathcal{R}_V|^2) \, \mathrm{d}y,
		\end{aligned}$$
		with $\boldsymbol{n}$ denoting the outward unit normal vector on the boundary $\partial B(0,R_1)$. Applying the Grönwall inequality to the integral functional and invoking the trivial initial condition $$(\mathcal{R}_S, \mathcal{R}_V)|_{\tau=\tau_0} = (0, 0),$$ within $B(0, R_1)$, we conclude that the mismatch variables vanish identically:$$    (\mathcal{R}_S, \mathcal{R}_V)(\tau, y) = (0, 0) \quad \text{for all } (\tau, y) \in [\tau_0, \tau_1] \times B(0, R_1).$$
	\end{proof}
	\subsection{The Bootstrap Argument}
	Now we prescribe the required bootstrap framework which will be proved later.
	\begin{proposition}
		\label{prop_English}
		Assume that the initial data are chosen by \eqref{initialdatacondition1}--\eqref{initialdatacondition2}.
		We prescribe the following bootstrap hypotheses for all $\tau \in [\tau_0, \tau_1]$:
		\begin{align}
			\|P_{\mathrm{uns}}(\widetilde{\widetilde{S}}, \widetilde{\widetilde{V}})\|_{X} &\le \sigma_1 e^{-\varpi(\tau-\tau_0)}, \label{Puns}\\
			\max \{|\widetilde{S}|, |\widetilde{V}|\} &\le \frac{\sigma_0}{50} e^{-\varpi(\tau-\tau_0)}, \label{S,U,1} \\
			\frac{\sigma_1}{10} \left\langle \frac{y}{R_0} \right\rangle^{-(\Lambda - 1)} &\le S, \label{S,lower,1}\\
			|\nabla \widetilde{S}| + |\nabla S| &\le 2C_1 \left\langle \frac{y}{R_0} \right\rangle^{-\Lambda},\label{nableS1} \\
			|\nabla \widetilde{V}| + |\nabla V| &\le 2C_1 \left\langle \frac{y}{R_0} \right\rangle^{-\Lambda},\label{nableU1}\\
			\max \left\{\|\nabla^4 \widetilde{S}\|_{L^2(B^c(0,R_1))}, \|\nabla^4 \widetilde{V}\|_{L^2(B^c(0,R_1))}\right\} &\le \sigma_0 e^{-\varpi(\tau-\tau_0)},\label{nable41} \\
			E_K = \int (|\nabla^K S|^2 + |\nabla^K V|^2) \phi^K \, \mathrm{d}y &\le E \label{E1}.
		\end{align}
		Then, for $\tau \in [\tau_0, \tau_1]$, one can derive the following sharpened estimates:
		\begin{align}
			\|P_{\mathrm{uns}}(\widetilde{\widetilde{S}}, \widetilde{\widetilde{V}})\|_{X} &\le \sigma_1^{\frac{21}{20}} e^{-\frac{4\varpi}{3}(\tau-\tau_0)},  \label{Puns2}\\
			\max \{|\widetilde{S}|, |\widetilde{V}|\} &\le \frac{\sigma_0}{100} e^{-\varpi(\tau-\tau_0)},  \label{S,U,2}\\
			\frac{\sigma_1}{5} \left\langle \frac{y}{R_0} \right\rangle^{-(\Lambda - 1)}&\le S , \label{S,lower,2}\\
			|\nabla \widetilde{S}| + |\nabla S| &\le  C_1\left\langle \frac{y}{R_0} \right\rangle^{-\Lambda}, \label{nableS2}\\
			|\nabla \widetilde{V}| + |\nabla V| &\le  C_1\left\langle \frac{y}{R_0} \right\rangle^{-\Lambda},\label{nableU2}\\
			\max \left\{\|\nabla^4 \widetilde{S}\|_{L^2(B^c(0,R_1))}, \|\nabla^4 \widetilde{V}\|_{L^2(B^c(0,R_1))}\right\} &\le \frac{\sigma_0}{2} e^{-\varpi(\tau-\tau_0)},  \label{nable42}\\
			E_K = \int (|\nabla^K S|^2 + |\nabla^K V|^2) \phi^K \, \mathrm{d}y &\le \frac{E}{2}\label{E2}. 
		\end{align}
		Here, the constant $C_1$ is given by \eqref{C1define}. We write
		$X_{\rm s}$ and $X_{\rm u}$ for the
		stable and unstable subspaces in Lemma~\ref{lem:new-stable-unstable-summary};
		$P_{\mathrm{sta}}$ and $P_{\mathrm{uns}}$ denote the corresponding spectral
		projections.
		
		Specifically, \eqref{Puns} is utilized to derive the temporal decay of the unstable components, while the nonlinear stability estimates in \eqref{S,U,1} constitute the core of the proof for the singularity formation. The lower bound estimate for the density is established in \eqref{S,lower,1}. Furthermore, \eqref{nableS1}, \eqref{nableU1}, \eqref{nable41}, and \eqref{E1} are primarily designed to obtain lower-order decay and higher-order estimates. This bootstrap argument serves as the central mechanism for establishing both global existence and nonlinear stability in the self-similar regime.
	\end{proposition}
	\subsection{Lower-Order estimates}
	\begin{lemma}
		Under the bootstrap hypotheses \eqref{Puns}-\eqref{E1}, we have
		\begin{equation}
			\label{123decay}
			|\nabla S| + |\nabla V| \le \sigma_0^{\frac{9}{10}} \phi^{-\frac{1}{2}}, \quad |\nabla^2 S| + |\nabla^2 V| \le \sigma_0^{\frac{9}{10}} \phi^{-1}, \quad |\nabla^3 S| + |\nabla^3 V| \le \sigma_0^{\frac{9}{10}} \phi^{-\frac{3}{2}},
		\end{equation}
		in the far-field region $B^c(0, R_0)$.
	\end{lemma}
	\begin{proof}
		Regarding the behavior in the far-field region $B^c(0, R_0)$, the bootstrap bounds \eqref{S,U,1} and \eqref{eq:profile-decay} ensure that $(S,V)$ remain bounded, specifically 
		\begin{equation}\label{S,V,maxCsigma0}
			\max\{|S|, |V|\} \le C\sigma_0. 
		\end{equation}
		To obtain higher-order control, we interpolate these estimates by the weighted Gagliardo--Nirenberg inequality in Lemma \ref{lem:weighted-GN-new}, with the parameter configuration
		\begin{equation*}
			\varphi = \phi^{\frac{1}{2}}, \quad l = K, \quad i = j, \quad \theta = \frac{j}{K - 3/2},
		\end{equation*}
		it follows that for each index $j \in \{1, 2, 3\}$, we have:
		\begin{equation*}
			|\nabla^j S| + |\nabla^j V| \le C(K) \left( \sigma_0^{1-\frac{j}{K-3/2}} E^{\frac{j}{2K-3}} \phi^{-\frac{Kj}{2(K-3/2)}} + \sigma_0 \langle y \rangle^{-j} \right).
		\end{equation*}
		In view of the weight properties in \eqref{eq:weight-definition-notation}--\eqref{eq:weight-power-properties}, we arrive at the following localized decay estimates:
		\begin{equation}
			|\nabla S| + |\nabla V| \le \sigma_0^{\frac{9}{10}} \phi^{-\frac{1}{2}}, \quad |\nabla^2 S| + |\nabla^2 V| \le \sigma_0^{\frac{9}{10}} \phi^{-1}, \quad |\nabla^3 S| + |\nabla^3 V| \le \sigma_0^{\frac{9}{10}} \phi^{-\frac{3}{2}}. 
		\end{equation}
	\end{proof}
	\begin{proposition}
		Under the bootstrap hypotheses \eqref{Puns}-\eqref{E1}, we have that \eqref{S,lower,2} holds; that is,
		\begin{equation}
			\frac{\sigma_1}{5} \left\langle \frac{y}{R_0} \right\rangle^{-(\Lambda - 1)} \le S(\tau, y), \quad \text{for } \tau \in [\tau_0, \tau_1] \quad \text{and} \quad y \in \mathbb{R}^3,
		\end{equation}
		when $|y| \ge R_0$, there holds
		\begin{equation}\label{upper bound}
			S \le C \sigma_0 \max \left\{ \left\langle \frac{y}{R_0} \right\rangle^{-(\Lambda-1)}, e^{-(\tau-\tau_0)(\Lambda-1)} \right\} \quad \text{for } \tau \in [\tau_0, \tau_1].
		\end{equation}
	\end{proposition}
	\begin{proof}
		\textbf{1.} For the inner region $|y| < R_0$, the lower bound for the cutoff profile and \eqref{S,U,1} yield the following strict lower bound for the density: 
		\begin{equation}
			S \ge 2\sigma_0 - \|\widetilde S\|_{L^\infty} \ge \sigma_0\ge \frac{\sigma_1}{5}\ge\frac{\sigma_1}{5} \left\langle \frac{y}{R_0} \right\rangle^{-(\Lambda - 1)}.
		\end{equation}
		
		\textbf{2.} For the far-field regime where $|y| \ge R_0$, we introduce the characteristic-type $$\omega_{S,y_0}(\tau) := e^{(\Lambda-1)(\tau-\bar{\tau})} S(\tau, y_0 e^{(\tau-\bar{\tau})}),$$
		here $y=y_0e^{\tau-\bar{\tau}}$ ($\bar{\tau}\ge \tau_0$) and $|y_0| \ge R_0$ will be defined later.
		
		Taking the derivative with respect to $\tau$ and invoking the transport structure $\eqref{Eq_self_similar_S}_1$, we obtain:
		\begin{align}
			\partial_\tau \omega_{S,y_0}(\tau) &= (\Lambda - 1)e^{(\Lambda-1)(\tau-\bar{\tau})} S(\tau, y_0e^{(\tau-\bar{\tau})}) + e^{(\Lambda-1)(\tau-\bar{\tau})} \partial_\tau S(\tau, y_0e^{(\tau-\bar{\tau})}) \nonumber \\
			&\quad + e^{(\Lambda-1)(\tau-\bar{\tau})} y_0e^{(\tau-\bar{\tau})} \cdot \nabla S(\tau, y_0e^{(\tau-\bar{\tau})})  \nonumber\\
			&= -e^{(\Lambda-1)(\tau-\bar{\tau})} (V \cdot \nabla S)(\tau, y_0e^{(\tau-\bar{\tau})}) - \delta e^{(\Lambda-1)(\tau-\bar{\tau})} (S \operatorname{div} V)(\tau, y_0e^{(\tau-\bar{\tau})}) \nonumber\\
			&\quad +e^{-f_1 \tau}e^{(\Lambda-1)(\tau-\bar{\tau})} d \alpha (\delta S)^{\frac{\alpha-1}{\delta}} \Delta S(\tau, y_0e^{(\tau-\bar{\tau})}) \nonumber \\
			&\quad + e^{-f_1 \tau}e^{(\Lambda-1)(\tau-\bar{\tau})} d \alpha (\alpha-\delta) (\delta S)^{\frac{\alpha-\delta-1}{\delta}} |\nabla S|^2(\tau, y_0e^{(\tau-\bar{\tau})}). \label{wS1}
		\end{align}

		We obtain from \eqref{lambdafanwei2}, \eqref{Eq_parameter_hierarchy}, \eqref{123decay} and \eqref{S,V,maxCsigma0} that
		\begin{align}
			&\left| e^{(\Lambda-1)(\tau-\bar{\tau})} V(\tau, y_0 e^{(\tau-\bar{\tau})})\cdot \nabla S(\tau, y_0 e^{(\tau-\bar{\tau})}) \right| \nonumber \\
			&\le C \left| \frac{y_0 e^{(\tau-\bar{\tau})}}{R_0} \right|^{-1+\eta} e^{(\Lambda-1)(\tau-\bar{\tau})} \sigma_0^{\frac{19}{10}} \nonumber \\
			&\le C \sigma_0^{\frac{19}{10}} \left\langle \frac{y_0}{R_0} \right\rangle^{-1+\eta} e^{(\tau-\bar{\tau})(\eta+\Lambda-2)} \le \frac{1}{2} \sigma_0^{\frac{9}{5}} \left\langle \frac{y_0}{R_0} \right\rangle^{1-\Lambda} e^{-\frac{\tau-\bar{\tau}}{10}}, \label{VS1}
		\end{align}
		and
		\begin{equation}
			\left| \delta e^{(\Lambda-1)(\tau-\bar{\tau})} S(\tau, y_0 e^{(\tau-\bar{\tau})}) \operatorname{div} V(\tau, y_0 e^{(\tau-\bar{\tau})}) \right| \le \frac{1}{2} \sigma_0^{\frac{9}{5}} \left\langle \frac{y_0}{R_0} \right\rangle^{1-\Lambda} e^{-\frac{\tau-\bar{\tau}}{10}}.\label{SV1}
		\end{equation}
		Furthermore, by invoking \eqref{Eq_parameter_hierarchy}, \eqref{S,lower,1}, and \eqref{123decay} once more, we obtain
		\begin{equation}
			\begin{split}\label{SSS1}
				&\left| e^{-f_1 \tau}e^{(\Lambda-1)(\tau-\bar{\tau})} d \alpha (\delta S)^{\frac{\alpha-1}{\delta}} \Delta S(\tau, y_0e^{(\tau-\bar{\tau})}) \right| \\
				&\le C e^{-f_1 \tau} \left| \frac{y_0 e^{(\tau-\bar{\tau})}}{R_0} \right|^{-2+2\eta} e^{(\Lambda-1)(\tau-\bar{\tau})} \sigma_0^{\frac{9}{10}}{\sigma_1}^{\frac{\alpha-1}{\delta}} \left\langle \frac{y_0 e^{(\tau-\bar{\tau})}}{R_0} \right\rangle^{(\Lambda - 1)\frac{1-\alpha}{\delta}}  \\
				&\le C \sigma_0^{\frac{9}{10}}{\sigma_1}^{\frac{\alpha-1}{\delta}}e^{(-f_1-2+2\eta+\Lambda-1+\frac{(\Lambda-1)}{\delta}(1-\alpha))(\tau-\bar{\tau})} \left\langle \frac{y_0}{R_0} \right\rangle^{-2+2\eta+(\Lambda - 1)\frac{1-\alpha}{\delta}} e^{-f_1\bar{\tau}}\\
				&\le Ce^{-f_1\tau_0}\sigma_0^{\frac{9}{10}}{\sigma_1}^{\frac{\alpha-1}{\delta}}e^{-\frac{\tau-\bar{\tau}}{10}}\left\langle \frac{y_0}{R_0} \right\rangle^{-2+2\eta+(\Lambda - 1)\frac{1-\alpha}{\delta}}\\
				&\le \frac{1}{2}\sigma_0^{\frac{9}{5}}\left\langle \frac{y_0}{R_0} \right\rangle^{1-\Lambda}e^{-\frac{\tau-\bar{\tau}}{10}}.
			\end{split}
		\end{equation}
		Here, we have utilized the following facts in view of \eqref{Eq_parameter_hierarchy}:
		\begin{align}
			-f_1 - 2 + 2\eta + \Lambda - 1 + \frac{(\Lambda-1)}{\delta}(1-\alpha) = -1 + 2\eta \le -\frac{1}{10}&, \\
			f_1 - \Lambda + 2 - 2 + 2\eta = f_1 - \Lambda + 2\eta \le 1 - \Lambda&, \label{1-lamda}\\
			C e^{-f_1\tau_0} \sigma_0^{\frac{9}{10}} \sigma_1^{\frac{\alpha-1}{\delta}} \le \frac{1}{2} \sigma_0^{\frac{9}{5}}&.
		\end{align}
		In addition, we provide an explicit proof for $\eqref{1-lamda}$ as follows:
		\begin{itemize}
			\item In the adiabatic index regime $1 < \gamma < \frac{5}{3}$, the estimate for the parameter $f_1$ is modified as
			\begin{align}
				-1 + f_1 + 2\eta &= \left( \frac{1 - \alpha}{\delta} + 1 \right)(\Lambda - 1) - 2 + 2\eta \nonumber \\
				&\le \frac{2}{\left(1 + \sqrt{\frac{2}{\gamma-1}}\right)^2} \frac{2(1 - \alpha) + \gamma - 1}{\gamma - 1} - 2 + 2\eta \nonumber \\
				&= \frac{2\gamma + 2 - 4\alpha}{\gamma + 1 + 2\sqrt{2(\gamma - 1)}} - 2 + 2\eta < 0,
			\end{align}
			where the final inequality holds provided that $\eta > 0$ is chosen to be sufficiently small. 
			\item For the regime $\frac{5}{3} \le \gamma < 1 + \frac{2}{\sqrt{3}}$, we have:
			\begin{align}
				-1 + f_1 + 2\eta &= \left( \frac{1 - \alpha}{\delta} + 1 \right)(\Lambda - 1) - 2 + 2\eta \nonumber \\
				&\le \frac{(3 - \sqrt{3})(\gamma - 1) + (6 - 2\sqrt{3})(1 - \alpha)}{2 + \sqrt{3}(\gamma - 1)} - 2 + 2\eta \nonumber \\
				&= \frac{(3 - 3\sqrt{3})(\gamma - 1) + (6 - 2\sqrt{3})(1 - \alpha) - 4}{2 + \sqrt{3}(\gamma - 1)} + 2\eta < 0.
			\end{align}
		\end{itemize}
		Regarding the term
		\begin{equation*}
			e^{-f_1 \tau}e^{(\Lambda-1)(\tau-\bar{\tau})} d \alpha (\alpha-\delta) (\delta S)^{\frac{\alpha-\delta-1}{\delta}} |\nabla S|^2(\tau, y_0e^{(\tau-\bar{\tau})}),
		\end{equation*}
		we can obtain the following estimate by invoking \eqref{Eq_parameter_hierarchy}, \eqref{S,lower,1}, and \eqref{123decay}:
		\begin{equation}
			\begin{split}\label{SSS2}
				&\left| e^{-f_1 \tau}e^{(\Lambda-1)(\tau-\bar{\tau})} d \alpha (\alpha-\delta) (\delta S)^{\frac{\alpha-\delta-1}{\delta}} |\nabla S|^2(\tau, y_0e^{(\tau-\bar{\tau})}) \right| \\
				&\le C e^{-f_1 \tau} \left| \frac{y_0 e^{(\tau-\bar{\tau})}}{R_0} \right|^{-2+2\eta} e^{(\Lambda-1)(\tau-\bar{\tau})} \sigma_0^{\frac{9}{5}}{\sigma_1}^{\frac{\alpha-\delta-1}{\delta}} \left\langle \frac{y_0 e^{(\tau-\bar{\tau})}}{R_0} \right\rangle^{(\Lambda - 1)\frac{1-\alpha+\delta}{\delta}}  \\
				&\le C \sigma_0^{\frac{9}{5}}{\sigma_1}^{\frac{\alpha-\delta-1}{\delta}}e^{(-f_1-2+2\eta+\Lambda-1+\frac{(\Lambda-1)}{\delta}(1-\alpha+\delta))(\tau-\bar{\tau})} \left\langle \frac{y_0}{R_0} \right\rangle^{-2+2\eta+(\Lambda - 1)\frac{1-\alpha+\delta}{\delta}} e^{-f_1\bar{\tau}}\\
				&\le Ce^{-f_1\tau_0}\sigma_0^{\frac{9}{5}}{\sigma_1}^{\frac{\alpha-\delta-1}{\delta}}e^{-\frac{\tau-\bar{\tau}}{10}}\left\langle \frac{y_0}{R_0} \right\rangle^{-2+2\eta+(\Lambda - 1)\frac{1-\alpha+\delta}{\delta}}\\
				&\le \frac{1}{2}\sigma_0^{\frac{9}{5}}\left\langle \frac{y_0}{R_0} \right\rangle^{1-\Lambda}e^{-\frac{\tau-\bar{\tau}}{10}}.
			\end{split}
		\end{equation}
		Here we use the following set of relations, which are established by virtue of the parameter hierarchy \eqref{Eq_parameter_hierarchy}:
		\begin{align}
			-f_1 - 2 + 2\eta + \Lambda - 1 + \frac{(\Lambda-1)}{\delta}(1-\alpha+\delta) =\Lambda -2 + 2\eta \le -\frac{1}{10}&,\label{Lambda-19/10} \\
			f_1 - \Lambda + 2+\Lambda-1 - 2 + 2\eta = f_1 - 1 + 2\eta \le 1 - \Lambda&, \label{2-lamda}\\
			C e^{-f_1\tau_0} \sigma_0^{\frac{9}{5}} \sigma_1^{\frac{\alpha-\delta-1}{\delta}} \le \frac{1}{2} \sigma_0^{\frac{9}{5}}&.
		\end{align}
		We provide the proof for the estimate \eqref{2-lamda}:
		\begin{itemize}
			\item For the regime $1 < \gamma < \frac{5}{3}$,
			$$\begin{aligned}
				f_1 - 2 + \Lambda + 2\eta &= \left( \frac{1-\alpha}{\delta} \right) (\Lambda - 1) + 2\Lambda - 4 + 2\eta \\
				&\le \frac{2}{\gamma-1} \cdot \frac{2}{\left( 1 + \sqrt{\frac{2}{\gamma-1}} \right)^2} + \frac{4}{\left( 1 + \sqrt{\frac{2}{\gamma-1}} \right)^2} - 2 + 2\eta \\
				&= \frac{4 + 4\gamma - 4 - 2(\gamma-1) \left( 1 + \frac{2}{\gamma-1} + 2\sqrt{\frac{2}{\gamma-1}} \right)}{(\gamma-1) \left( 1 + \sqrt{\frac{2}{\gamma-1}} \right)^2} + 2\eta \\
				&= \frac{\sqrt{2} \left( \sqrt{2(\gamma-1)} - 4 \right)}{\sqrt{\gamma-1} \left( 1 + \sqrt{\frac{2}{\gamma-1}} \right)^2} + 2\eta \le 0.
			\end{aligned}$$
			\item For the regime $\frac{5}{3} \le \gamma < 1 + \frac{2}{\sqrt{3}}$,
			$$\begin{aligned}
				f_1 - 2 + \Lambda + 2\eta &= \left( \frac{1-\alpha}{\delta} \right) (\Lambda - 1) + 2\Lambda - 4 + 2\eta \\
				&\le \frac{2}{\gamma-1} \left( \frac{3\gamma-1}{2+\sqrt{3}(\gamma-1)} - 1 \right) + \frac{6\gamma-2}{2+\sqrt{3}(\gamma-1)} - 4 + 2\eta \\
				&= \frac{1}{(\gamma-1)(2+\sqrt{3}(\gamma-1))} \Big[ 6\gamma-2 - 2(2+\sqrt{3}(\gamma-1)) \\
				&\quad + (6\gamma-2)(\gamma-1) - 4(\gamma-1)(2+\sqrt{3}(\gamma-1)) \Big] + 2\eta \\
				&= \frac{(6-4\sqrt{3})\gamma^2 + (6\sqrt{3}-10)\gamma + 4-2\sqrt{3}}{(\gamma-1)(\sqrt{3}\gamma + 2-\sqrt{3})} + 2\eta \le 0.
			\end{aligned}$$
		\end{itemize}
		Substituting \eqref{VS1}, \eqref{SV1}, \eqref{SSS1}, and \eqref{SSS2} into \eqref{wS1}, we obtain
		\begin{equation} \label{WS1t}
			|\partial_\tau \omega_{S,y_0}(\tau)| \le 2\sigma_0^{\frac{9}{5}} \left\langle \frac{y_0}{R_0} \right\rangle^{1-\Lambda} e^{-\frac{\tau-\bar{\tau}}{10}}.
		\end{equation}
		Integrating \eqref{wS1} over $[\bar{\tau}, \tau]$, we obtain
		\begin{equation}\label{differencetau}
			|\omega_{S,y_0}(\tau) - \omega_{S,y_0}(\bar{\tau})| \le 20 \sigma_0^{\frac{9}{5}} \left\langle \frac{y_0}{R_0} \right\rangle^{1-\Lambda},
		\end{equation}
		which holds uniformly for all $\tau \ge \bar{\tau}$.
		
		By the normalization \eqref{eq:profile-R0-normalization},
		\[
		2\sigma_0 \le \bar{S} \le C\sigma_0 \quad \text{for } |y| = R_0.
		\]
		Consequently, we have
		\[
		\sigma_0 \le \omega_{S,y_0}(\bar{\tau}) \le C\sigma_0 \quad \text{for } |y_0| = R_0.
		\]
		
		\vspace{1em}
		
		\noindent Invoking \eqref{differencetau}, we deduce that for all $\tau \ge \bar{\tau}$,
		\[
		\frac{1}{2}\sigma_0 \le \omega_{S,y_0}(\tau) \le C\sigma_0 \quad \text{for } |y_0| = R_0.
		\]
		Equivalently, this can be rewritten as
		\[
		\frac{1}{2}\sigma_0 e^{-(\Lambda-1)(\tau-\bar{\tau})} \le S(\tau, y_0 e^{(\tau-\bar{\tau})}) \le C\sigma_0 e^{-(\Lambda-1)(\tau-\bar{\tau})}.
		\]
		Noting the identity $y = y_0e^{\tau-\bar{\tau}}$, we arrive at the estimate
		\begin{equation}\label{tR0}
			\frac{\sigma_0}{C} \left\langle \frac{y}{R_0} \right\rangle^{-(\Lambda-1)} \le S(\tau, y) \le C\sigma_0 \left\langle \frac{y}{R_0} \right\rangle^{-(\Lambda-1)},\quad \text{for} \,\,|y_0|=R_0 \,\,\text{and} \,\,\tau \ge \bar{\tau}.
		\end{equation}
		Next, by setting $\bar{\tau} = \tau_0$ and considering $|y_0| \ge R_0$, \eqref{initialdatacondition1} implies
		\[
		\frac{\sigma_1}{2} \left\langle \frac{y_0}{R_0} \right\rangle^{1-\Lambda} \le S(\tau_0, y_0) \le C\sigma_0.
		\]
		Therefore, we obtain from \eqref{differencetau}
		\begin{equation}\label{t0bigR0}
			e^{-(\Lambda-1)(\tau-\tau_0)} \frac{2\sigma_1}{5} \left\langle \frac{y_0}{R_0} \right\rangle^{1-\Lambda} \le S(\tau, y) = e^{-(\Lambda-1)(\tau-\tau_0)} \omega_{S,y_0}(\tau) \le C\sigma_0 e^{-(\Lambda-1)(\tau-\tau_0)}.
		\end{equation}
		For any $\tau \ge \tau_0$ and $|y| > R_0$, we shall first prove \eqref{S,lower,2} by distinguishing between two cases:
		\begin{itemize}
			\item \textbf{Case I}. When $R_0 \le |y|<R_0e^{\tau-\tau_0}$, we can choose $\bar{\tau}\ge\tau_0$ and $|y_0|=R_0$ s.t.  
			$$ y=y_0 e^{\tau-\bar{\tau}}.$$
			Using \eqref{tR0}, we get 
			\begin{equation}
				\frac{\sigma_1}{5} \left\langle \frac{y}{R_0} \right\rangle^{-(\Lambda - 1)}\le \frac{\sigma_0}{C} \left\langle \frac{y}{R_0} \right\rangle^{-(\Lambda-1)}  \le S(\tau, y).
			\end{equation}
			\item \textbf{Case II}. When $|y|\ge R_0e^{\tau-\tau_0}$, we can choose $|y_0|\ge R_0$ s.t.
			$$y=y_0 e^{\tau-\tau_0}.$$
			Using \eqref{t0bigR0}, we arrive at 
			\begin{equation}
				\frac{\sigma_1}{5} \left\langle \frac{y}{R_0} \right\rangle^{-(\Lambda - 1)}\le  e^{-(\Lambda-1)(\tau-\tau_0)} \frac{2\sigma_1}{5} \left\langle \frac{y_0}{R_0} \right\rangle^{1-\Lambda} \le S(\tau, y).
			\end{equation}
			Here we have used the fact that $e^{(1-\Lambda)(\tau-\tau_0)} =\frac{|y|^{1-\Lambda}}{|y_0|^{1-\Lambda}}$ and 
			\begin{equation}\label{Eq_bracket_estimate}
				\left\langle \frac{y}{R_0} \right\rangle^2 = \left( \frac{|y|^2}{R_0^2} + 1 \right) \ge \frac{1}{2} \left( \frac{|y_0|^2}{R_0^2} + 1 \right) \frac{|y|^2}{|y_0|^2} \ge \left( \frac{1}{2} \right)^{\frac{2}{\Lambda - 1}} \left\langle \frac{y_0}{R_0} \right\rangle^2 \frac{|y|^2}{|y_0|^2}.
			\end{equation}
		\end{itemize}
		The proof of the upper bound estimate \eqref{upper bound} is similar to that of the lower bound estimate by using \eqref{differencetau}, and is thus omitted here. This completes the proof.
	\end{proof} 
	\begin{lemma} \label{lem:energy-decay}
		Under the bootstrap hypotheses \eqref{Puns}-\eqref{E1}, for any $\tau \in [\tau_0, \tau_1]$, the perturbed energy functional $\widetilde{E}_K(\tau)$ satisfies the uniform bound
		\begin{equation}\label{energy decayEE}
			\widetilde{E}_K(\tau) := \int_{\mathbb{R}^3} \left( |\nabla^K \widetilde{S}|^2 + |\nabla^K \widetilde{V}|^2 \right) \phi^K dy \le 2E.
		\end{equation}
	\end{lemma}
	
	\begin{proof}
		Recall the decomposition $(\widetilde{S}, \widetilde{V}) = (S, V) - (\widehat{\chi} \bar{S}, \widehat{\chi} \bar{V})$. In view of \eqref{E1} and the assumption $C(K) \ll E$, it suffices to validate the following upper bound for the cutoff profiles:
		\[
		\int_{\mathbb{R}^3} \left( |\nabla^K (\widehat{\chi}\bar{S})|^2 + |\nabla^K (\widehat{\chi}\bar{V})|^2 \right) \phi^K dy \le C(K).
		\]
		By virtue of \eqref{eq:weight-definition-notation} and the profile decay \eqref{eq:profile-decay}, we obtain the pointwise estimate
		\[
		\left( |\nabla^K \bar{S}|^2 + |\nabla^K \bar{V}|^2 \right) \phi^K \le C(K) \langle y \rangle^{-2(\Lambda-1)-2K\eta}.
		\]
		Consequently, under the parameter choice $\frac{1}{K} \ll \eta$, we arrive at the verification bound
		\begin{equation}\label{barSUE}
			\int_{\mathbb{R}^3} \left( |\nabla^K \bar{S}|^2 + |\nabla^K \bar{V}|^2 \right) \phi^K dy \le C(K).
		\end{equation}
		Next, using \eqref{xdecayestim} and the profile decay \eqref{eq:profile-decay}, the cross-terms can be evaluated through the following summation:
		\begin{equation}\label{K-1SUE}
			\begin{aligned}
				\sum_{0 \le j \le K-1} & \int_{\mathbb{R}^3} |\nabla^{K-j} \widehat{\chi}|^2 \left( |\nabla^j \bar{S}|^2 + |\nabla^j \bar{V}|^2 \right) \phi^K dy \\
				&\le C(K) \sum_{0 \le j \le K-1} e^{-2(K-j)\tau} \int_{\frac{1}{2}e^\tau}^{e^\tau} r^{-2(\Lambda-1)-2j} r^{2(1-\eta)K} r^2 dr \\
				&\le C(K) \sum_{0 \le j \le K-1} e^{-2(K-j)\tau + \left(3 + 2(1-\eta)K - 2(\Lambda-1) - 2j\right)\tau} \\
				&\le C(K) e^{-(2K\eta - 3 + 2(\Lambda-1))\tau} \le C(K).
			\end{aligned}
		\end{equation}
		As a consequence, the desired estimate \eqref{energy decayEE} flows directly from combining \eqref{barSUE} with \eqref{K-1SUE}.
	\end{proof}
	\begin{proposition}
		Under the bootstrap hypotheses \eqref{Puns}-\eqref{E1}, we have that \eqref{nableS2} and \eqref{nableU2}  hold:
		\begin{align}
			|\nabla \widetilde{S}| + |\nabla S| &\le  C_1\left\langle \frac{y}{R_0} \right\rangle^{-\Lambda},  \quad \text{for } \tau \in [\tau_0, \tau_1] \quad \text{and} \quad y \in \mathbb{R}^3,\label{SSC1yR0}\\
			|\nabla \widetilde{V}| + |\nabla V| &\le  C_1\left\langle \frac{y}{R_0} \right\rangle^{-\Lambda}, \quad \text{for } \tau \in [\tau_0, \tau_1] \quad \text{and} \quad y \in \mathbb{R}^3.\label{UUC1yR0}
		\end{align}
	\end{proposition}
	\begin{proof}
		\noindent \textbf{1. Interior bounds for $S$ and $V$ within the regional domain $|y| \le R_0$.} \\
		To establish the localized estimates in the core region where $|y| \le R_0$, we invoke the fundamental relation \eqref{Eq_parameter_hierarchy} in tandem with the weighted Gagliardo--Nirenberg inequality in Lemma~\ref{lem:weighted-GN-new}, bridging the gaps between the energy states \eqref{S,U,1} and \eqref{energy decayEE}. Specifically, selecting the test-weight parameters as
		\[
		\varphi = \phi^{\frac{1}{2}}, \quad l = K, \quad i = 1, \quad \theta = \frac{1}{K - 3/2},
		\]
		gives the following uniform pointwise control for the gradients of the perturbed fluid variables:
		\[
		|\nabla \widetilde{S}| + |\nabla \widetilde{V}| \le C(K) \left( \sigma_0^{1 - \frac{1}{K-3/2}} E^{\frac{1}{2K-3}} \phi^{-\frac{K}{2(K-3/2)}} + \sigma_0 \langle y \rangle^{-1} \right) \le C(K) \sigma_0^{\frac{9}{10}}\le C(K) \sigma_0^{\frac{4}{5}} \left\langle \frac{y}{R_0} \right\rangle^{-\Lambda}.
		\]
		Subsequently, using \eqref{xdecayestim} and \eqref{eq:profile-decay} for the truncated reference profiles $(\widehat{\chi} \bar{S}, \widehat{\chi} \bar{V})$, we have
		\begin{equation}\label{XSUC2yR}
			\begin{aligned}
				|\nabla (\widehat{\chi} \bar{S})| +|\nabla (\widehat{\chi} \bar{V})|
				&\le |\nabla \widehat{\chi} \bar{S}| + |\widehat{\chi} \nabla \bar{S}|+|\nabla \widehat{\chi} \bar{V}|+ |\widehat{\chi} \nabla \bar{V}|\\
				&\le C\mathbf{1}_{\left\{\frac{e^\tau}{2}\leq |y|\leq e^\tau\right\}}(y) \cdot e^{-\tau} \langle r \rangle^{-(\Lambda - 1)} + C \langle r \rangle^{-(\Lambda - 1) - 1} \\
				&\le C_2 \left\langle \frac{y}{R_0} \right\rangle^{-\Lambda}, \quad \text{for } \tau \in [\tau_0, \tau_1] \quad \text{and} \quad y \in \mathbb {R}^3,
			\end{aligned}
		\end{equation}
		where $\mathbf{1}_{\left\{\frac{e^\tau}{2}\leq |y|\leq e^\tau\right\}}(y)$ denotes the the characteristic function of the annulus.
		Set
		\[
		M:=\sup_{y\in \mathbb {R}^3}
		\left\langle\frac{y}{R_0}\right\rangle^{\Lambda}
		\left(|\nabla S(\tau_0,y)|+|\nabla V(\tau_0,y)|\right)<\infty,
		\]
		which is finite by the weighted initial-data assumptions.
		Define $C_1$ by
		\begin{equation}\label{C1define}
			C_1 := 10\max \left\{ 2C(K) \sigma_0^{\frac{4}{5}} + C_2, \, 20,\, 2M+C_2 \right\}.
		\end{equation}
		Then we obtain
		\begin{align}
			|\nabla \tilde{S}| + |\nabla S| &\le \frac{C_1}{10}\left\langle \frac{y}{R_0} \right\rangle^{-\Lambda}, \label{SSinter10}\\
			|\nabla \tilde{V}| + |\nabla V| &\le \frac{C_1}{10}\left\langle \frac{y}{R_0} \right\rangle^{-\Lambda}, 
		\end{align}
		when $|y| \le R_0$.\\
		\noindent \textbf{2. Exterior bounds for $S$ and $\widetilde{S}$ within the regional domain $|y| > R_0$.} \\
		Applying the differential operator $\partial_{y_i}$ to $\eqref{Eq_self_similar_S}_1$, we obtain
		\begin{equation}
			\begin{aligned}
				\partial_\tau \partial_{y_i} S =& -\Lambda \partial_{y_i} S - y \cdot \nabla \partial_{y_i} S - \partial_{y_i} (V \cdot \nabla S) - \delta \partial_{y_i} (S \, \mathrm{div} \, V)\\
				&+e^{-f_1 \tau} d \alpha \, \delta^{\frac{\alpha-1}{\delta}} \frac{\alpha-1}{\delta}S^{\frac{\alpha-1}{\delta}-1}\partial_{y_i}S\, \Delta S 
				+ e^{-f_1 \tau} d \alpha \, \delta^{\frac{\alpha-1}{\delta}} S^{\frac{\alpha-1}{\delta}} \Delta \partial_{y_i} S \\
				&+e^{-f_1 \tau} d \alpha (\alpha-\delta) \, \delta^{\frac{\alpha-\delta-1}{\delta}} \frac{\alpha-\delta-1}{\delta}S^{\frac{\alpha-2\delta-1}{\delta}} \partial_{y_i} S \, |\nabla S|^2\\ 
				&+ 2e^{-f_1 \tau} d \alpha (\alpha-\delta) (\delta S)^{\frac{\alpha-\delta-1}{\delta}} \nabla S \cdot \nabla \partial_{y_i} S.
			\end{aligned}
		\end{equation}
		
		In the exterior region $|y| \ge R_0$, we choose the characteristic curves $y = y_0 e^{\tau-\bar{\tau}}$, where either $\bar{\tau} = \tau_0$ or $|y_0| = R_0$. For any fixed index $i \in \{1, 2, 3\}$, we define:
		\[
		\omega_{S,y_0,i}(\tau) := e^{\Lambda(\tau-\bar{\tau})} \partial_{y_i} S(\tau, y_0 e^{\tau-\bar{\tau}}).
		\]
		Then, the localized time derivative $\partial_\tau \omega_{S,y_0,i}$ can be formulated as follows:
		\begin{equation}
			\begin{aligned}
				\partial_\tau \omega_{S,y_0,i} =& -e^{\Lambda(\tau-\bar{\tau})} \left( \partial_{y_i} V \cdot \nabla S + V \cdot \nabla \partial_{y_i} S + \delta \partial_{y_i} S \, \mathrm{div} \, V + \delta S \, \mathrm{div}(\partial_{y_i} V) \right)\\
				&+e^{\Lambda(\tau-\bar{\tau})}e^{-f_1 \tau} d \alpha \, \delta^{\frac{\alpha-1}{\delta}} \frac{\alpha-1}{\delta}\partial _{y_i}S S^{\frac{\alpha-1}{\delta}-1} \Delta S + e^{\Lambda(\tau-\bar{\tau})}e^{-f_1 \tau} d \alpha \, \delta^{\frac{\alpha-1}{\delta}} S^{\frac{\alpha-1}{\delta}} \Delta \partial_{y_i} S \\
				&+e^{\Lambda(\tau-\bar{\tau})}e^{-f_1 \tau} d \alpha (\alpha-\delta) \, \delta^{\frac{\alpha-\delta-1}{\delta}} \frac{\alpha-\delta-1}{\delta}S^{\frac{\alpha-2\delta-1}{\delta}} \partial_{y_i} S \, |\nabla S|^2 \\
				&+ 2e^{\Lambda(\tau-\bar{\tau})}e^{-f_1 \tau} d \alpha (\alpha-\delta) (\delta S)^{\frac{\alpha-\delta-1}{\delta}} \nabla S \cdot \nabla \partial_{y_i} S,
			\end{aligned}
		\end{equation}
		where all constitutive terms on the right-hand side are evaluated along the path $(\tau, y_0 e^{\tau-\bar{\tau}})$. Utilizing the uniform bounds $|S|, |V| \le C\sigma_0$ together with the relation \eqref{123decay}, we restrict the growth rate via
		\begin{equation}\label{vsvssvsv}
			\begin{aligned}
				&|e^{\Lambda(\tau-\bar{\tau})} \left( \partial_{y_i} V \cdot \nabla S + V \cdot \nabla \partial_{y_i} S + \delta \partial_{y_i} S \, \mathrm{div} \, V + \delta S \, \mathrm{div}(\partial_{y_i} V) \right)|\\
				&\le C \sigma_0^{\frac{9}{5}} \phi(y_0 e^{\tau-\bar{\tau}})^{-1} e^{\Lambda(\tau-\bar{\tau})} \le C \sigma_0^{\frac{9}{5}} e^{(\Lambda-2+2\eta)(\tau-\bar{\tau})} \left\langle \frac{y_0}{R_0} \right\rangle^{-2+2\eta}.
			\end{aligned}
		\end{equation}
		Applying \eqref{S,lower,1}, \eqref{nableS1}, and \eqref{123decay}, we obtain
		\begin{equation}\label{S1SS2}
			\begin{aligned}
				&|e^{\Lambda(\tau-\bar{\tau})}e^{-f_1 \tau} d \alpha \, \delta^{\frac{\alpha-1}{\delta}} \frac{\alpha-1}{\delta}\partial _{y_i}S S^{\frac{\alpha-1}{\delta}-1} \Delta S|\\
				&\le C e^{(\Lambda - f_1)(\tau - \bar{\tau})} e^{-f_1 \bar{\tau}} \sigma_1^{\frac{\alpha - \delta - 1}{\delta}} \left\langle \frac{y}{R_0} \right\rangle^{(\Lambda - 1)\frac{\delta + 1 - \alpha}{\delta}} \cdot \left( 2C_1 \left\langle \frac{y}{R_0} \right\rangle^{-\Lambda} \right) \cdot \sigma_0^{\frac{9}{10}} \phi^{-1} \\
				&\le C C_1 e^{-f_1 \bar{\tau}} \sigma_0^{\frac{9}{10}} \sigma_1^{\frac{\alpha - \delta - 1}{\delta}} \left\langle \frac{y_0}{R_0} \right\rangle^{(\Lambda - 1)\frac{1 - \alpha}{\delta} - 3 + 2\eta} e^{\left( -f_1 + (\Lambda - 1)\frac{1 - \alpha}{\delta} + \Lambda - 3 + 2\eta \right)(\tau - \bar{\tau})} \\
				&\le C \sigma_0^{\frac{9}{5}} \left\langle \frac{y_0}{R_0} \right\rangle^{-\Lambda} e^{(-1 + 2\eta)(\tau - \bar{\tau})}.
			\end{aligned}
		\end{equation}
		Here we have used the following facts: 
		\begin{align}
			C_1 e^{-f_1 \bar{\tau}} \sigma_0^{\frac{9}{10}}\sigma_1^{\frac{\alpha-\delta-1}{\delta}} \le\sigma_0^{\frac{9}{5}}&,\\
			(\Lambda - 1)\frac{1 - \alpha}{\delta} - 3 + 2\eta=f_1-\Lambda-1+2\eta\le-\Lambda&,\\
			-f_1 + (\Lambda - 1)\frac{1 - \alpha}{\delta} + \Lambda - 3 + 2\eta=-1+2\eta&.
		\end{align}
		By virtue of \eqref{S,lower,1} and \eqref{123decay}, we arrive at
		\begin{equation}\label{SS333}
			\begin{aligned}
				&|e^{\Lambda(\tau-\bar{\tau})}e^{-f_1 \tau} d \alpha \, \delta^{\frac{\alpha-1}{\delta}} S^{\frac{\alpha-1}{\delta}} \Delta \partial_{y_i} S|\\
				&\le C e^{(\Lambda-f_1)(\tau-\bar{\tau})} e^{-f_1 \bar{\tau}} \sigma_1^{\frac{\alpha-1}{\delta}} \left\langle \frac{y}{R_0} \right\rangle^{(\Lambda-1)\frac{1-\alpha}{\delta}} \sigma_0^{\frac{9}{10}} \phi^{-\frac{3}{2}} \\
				&\le C e^{-f_1 \bar{\tau}} \sigma_1^{\frac{\alpha-1}{\delta}} \sigma_0^{\frac{9}{10}} \left\langle \frac{y_0}{R_0} \right\rangle^{(\Lambda-1)\frac{1-\alpha}{\delta}-3+3\eta} e^{\left( \Lambda-f_1 + (\Lambda-1)\frac{1-\alpha}{\delta} - 3 + 3\eta \right)(\tau-\bar{\tau})} \\
				&\le C \sigma_0^{\frac{9}{5}} \left\langle \frac{y_0}{R_0} \right\rangle^{-\Lambda} e^{(-1+3\eta)(\tau-\bar{\tau})}.
			\end{aligned}
		\end{equation}
		Here we have used:
		\begin{align}
			e^{-f_1 \bar{\tau}} \sigma_1^{\frac{\alpha-1}{\delta}} \sigma_0^{\frac{9}{10}}\le\sigma_0^{\frac{9}{5}}&,\\
			(\Lambda - 1)\frac{1 - \alpha}{\delta} - 3 + 3\eta=f_1-\Lambda-1+3\eta\le-\Lambda&,\\
			-f_1 + (\Lambda - 1)\frac{1 - \alpha}{\delta} + \Lambda - 3 + 3\eta=-1+3\eta&.
		\end{align}
		Using \eqref{S,lower,1} and \eqref{nableS1}, we obtain
		\begin{equation}\label{SS1S1S1}
			\begin{aligned}
				& |e^{\Lambda(\tau-\bar{\tau})}e^{-f_1 \tau} d \alpha (\alpha-\delta) \, \delta^{\frac{\alpha-\delta-1}{\delta}} \frac{\alpha-\delta-1}{\delta}S^{\frac{\alpha-2\delta-1}{\delta}} \partial_{y_i} S \, |\nabla S|^2|\\
				&\le C e^{(\Lambda-f_1)(\tau-\bar{\tau})} e^{-f_1 \bar{\tau}} \sigma_1^{\frac{\alpha-2\delta-1}{\delta}} \left\langle \frac{y}{R_0} \right\rangle^{(\Lambda-1)\frac{1-\alpha+2\delta}{\delta}} \left( 2C_1 \left\langle \frac{y}{R_0} \right\rangle^{-\Lambda} \right)^3 \\
				&\le C e^{-f_1 \bar{\tau}} \sigma_1^{\frac{\alpha-2\delta-1}{\delta}} C_1^3 e^{\left( \Lambda-f_1 + (\Lambda-1)\frac{1-\alpha+2\delta}{\delta} - 3\Lambda \right)(\tau-\bar{\tau})} \left\langle \frac{y_0}{R_0} \right\rangle^{(\Lambda-1)\frac{1-\alpha+2\delta}{\delta} - 3\Lambda} \\
				&\le C \sigma_0^{\frac{9}{5}} \left\langle \frac{y_0}{R_0} \right\rangle^{-\Lambda} e^{-\Lambda(\tau-\bar{\tau})}.
			\end{aligned}
		\end{equation}
		Here, we have utilized the following facts:
		\begin{align}
			e^{-f_1 \bar{\tau}} \sigma_1^{\frac{\alpha-2\delta-1}{\delta}} C_1^3\le \sigma_0^{\frac{9}{5}}&,\\
			(\Lambda-1)\frac{1-\alpha+2\delta}{\delta} - 3\Lambda=f_1-2\Lambda\le-\Lambda&,\\
			\Lambda-f_1 + (\Lambda-1)\frac{1-\alpha+2\delta}{\delta} - 3\Lambda=-\Lambda&.
		\end{align}
		By an estimate similar to \eqref{S1SS2}, we have
		\begin{equation}\label{SS1S222}
			|2e^{\Lambda(\tau-\bar{\tau})}e^{-f_1 \tau} d \alpha (\alpha-\delta) (\delta S)^{\frac{\alpha-\delta-1}{\delta}} \nabla S \cdot \nabla \partial_{y_i} S |\le C \sigma_0^{\frac{9}{5}} \left\langle \frac{y_0}{R_0} \right\rangle^{-\Lambda} e^{(-1 + 2\eta)(\tau - \bar{\tau})}.
		\end{equation}
		In view of \eqref{vsvssvsv}, \eqref{S1SS2}, \eqref{SS333}, \eqref{SS1S1S1}, and \eqref{SS1S222}, it follows that
		\[
		|\partial_\tau \omega_{S,y_0,i}| \le C \sigma_0^{\frac{9}{5}} e^{(\Lambda-2+2\eta)(\tau-\bar{\tau})} \left\langle \frac{y_0}{R_0} \right\rangle^{-\Lambda}.
		\]
		Integrating this inequality directly over the interval $[\bar{\tau}, \tau]$, we arrive at the oscillation bound:
		\begin{equation}\label{partialyiS}
			\begin{aligned}
				&\left| \partial_{y_i} S(\tau, y_0 e^{\tau-\bar{\tau}}) - e^{-\Lambda(\tau-\bar{\tau})} \partial_{y_i} S(\bar{\tau}, y_0) \right|\\
				&= e^{-\Lambda(\tau-\bar{\tau})} |\omega_{S,y_0,i}(\tau) - \omega_{S,y_0,i}(\bar{\tau})| \\
				&\le C \sigma_0^{\frac{9}{5}} e^{-\Lambda(\tau-\bar{\tau})} \left\langle \frac{y_0}{R_0} \right\rangle^{-\Lambda}.
			\end{aligned}
		\end{equation}
		On the other hand, the definition of $M$ and the prescribed initial weighted bound give, for all $|y_0| \ge R_0$,
		\begin{equation}\label{Mtau0y0-S}
			|\partial_{y_i} S(\tau_0, y_0)| \le M \left\langle \frac{y_0}{R_0} \right\rangle^{-\Lambda}.
		\end{equation}
		Furthermore, the boundary values are controlled by:
		\begin{equation}\label{C1bartauy0-S}
			|\partial_{y_i} S(\bar{\tau}, y_0)| \le \frac{C_1}{10}\left\langle \frac{y_0}{R_0} \right\rangle^{-\Lambda} \quad \text{under the constraints } |y_0| = R_0 \text{ and } \bar{\tau} \ge \tau_0.
		\end{equation}
		Similarly, for any $\tau \ge \tau_0$ and $|y| > R_0$, we divide the analysis into two cases:
		\begin{itemize}
			\item \textbf{Case I}. When $R_0 \le |y|<R_0e^{\tau-\tau_0}$, we can choose $\bar{\tau}\ge\tau_0$ and $|y_0|=R_0$ s.t.  
			$$ y=y_0 e^{\tau-\bar{\tau}}.$$
			Using \eqref{partialyiS} and \eqref{C1bartauy0-S}, we obtain 
			\begin{equation}\label{bartaugetau0y0R0}
				\begin{aligned}
					|\partial_{y_i} S(\tau, y_0 e^{\tau-\bar{\tau}})|&\le \left| \partial_{y_i} S(\tau, y_0 e^{\tau-\bar{\tau}}) - e^{-\Lambda(\tau-\bar{\tau})} \partial_{y_i} S(\bar{\tau}, y_0) \right|+|e^{-\Lambda(\tau-\bar{\tau})} \partial_{y_i} S(\bar{\tau}, y_0)|\\
					& \le C \sigma_0^{\frac{9}{5}} e^{-\Lambda(\tau-\bar{\tau})} \left\langle \frac{y_0}{R_0} \right\rangle^{-\Lambda}+\frac{C_1}{10}\left\langle \frac{y_0}{R_0} \right\rangle^{-\Lambda} e^{-\Lambda(\tau-\bar{\tau})}\\
					&\le \frac{C_1}{8}(\frac{|y|}{R_0})^{-\Lambda}\le \frac{C_1}{4}\left\langle \frac{y}{R_0} \right\rangle^{-\Lambda}.       
				\end{aligned}
			\end{equation}
			\item \textbf{Case II}. When $|y|\ge R_0e^{\tau-\tau_0}$, we can choose $|y_0|\ge R_0$ s.t.
			$$y=y_0 e^{\tau-\tau_0}.$$
			Using  \eqref{partialyiS} and \eqref{Mtau0y0-S}, we arrive at 
			\begin{equation}\label{y0R0etautau04}
				\begin{aligned}
					|\partial_{y_i} S(\tau, y_0 e^{\tau-{\tau_0}})|&\le \left| \partial_{y_i} S(\tau, y_0 e^{\tau-{\tau_0}}) - e^{-\Lambda(\tau-{\tau_0})} \partial_{y_i} S({\tau_0}, y_0) \right|+|e^{-\Lambda(\tau-{\tau_0})} \partial_{y_i} S({\tau_0}, y_0)|\\
					& \le C \sigma_0^{\frac{9}{5}} e^{-\Lambda(\tau-{\tau_0})} \left\langle \frac{y_0}{R_0} \right\rangle^{-\Lambda}+\frac{C_1}{10}\left\langle \frac{y_0}{R_0} \right\rangle^{-\Lambda} e^{-\Lambda(\tau-{\tau_0})}\\
					&\le \frac{C_1}{8}(\frac{|y|}{|y_0|})^{-\Lambda}\left\langle \frac{y_0}{R_0} \right\rangle^{-\Lambda}\le \frac{C_1}{4}\left\langle \frac{y}{R_0} \right\rangle^{-\Lambda}.       
				\end{aligned}
			\end{equation}
			Here we have used the fact that $|y_0|^{-\Lambda}e^{-\Lambda(\tau-\tau_0)} =|y|^{-\Lambda}.$
		\end{itemize}
		Hence, we deduce that
		\begin{equation}
			|\nabla S| \le \frac{\sqrt{3}C_1}{4}\left\langle \frac{y}{R_0} \right\rangle^{-\Lambda}.
		\end{equation}
		From \eqref{XSUC2yR}, we have
		\begin{equation}
			|\nabla (\widehat{\chi} \bar{S})| \le \frac{C_1}{10} \left\langle \frac{y}{R_0} \right\rangle^{-\Lambda}.
		\end{equation}
		This concludes the estimate
		\[
		|\nabla \tilde{S}| + |\nabla S| \le  (\frac{\sqrt{3}}{4}+\frac{\sqrt{3}}{4}+\frac{1}{10})C_1\left\langle \frac{y}{R_0} \right\rangle^{-\Lambda}\le C_1\left\langle \frac{y}{R_0} \right\rangle^{-\Lambda}.
		\]
		Consequently, the bound \eqref{SSC1yR0} follows from a combination with \eqref{SSinter10}.\\ 
		\noindent \textbf{3. Exterior bounds for $V$ and $\widetilde{V}$ within the regional domain $|y| > R_0$.} \\
		Acting on $\eqref{Eq_self_similar_S}_2$ with the differential operator $\partial_{y_i}$ yields
		\begin{equation}
			\begin{aligned}
				\partial_\tau \partial_{y_i} V = &- \Lambda \partial_{y_i} V - y \cdot \nabla \partial_{y_i} V - \partial_{y_i} V \cdot \nabla V - V \cdot \nabla(\partial_{y_i} V) - \delta \partial_{y_i} S \nabla S - \delta S \nabla(\partial_{y_i} S) \\
				&+e^{-f_1 \tau}(\delta S)^{\frac{\alpha-1}{\delta}}\mathbb{L}(\partial_{y_i}V)+e^{-f_1 \tau}\delta ^{\frac{\alpha-1}{\delta}} \frac{\alpha-1}{\delta}S^{\frac{\alpha-\delta-1}{\delta}}\partial_{y_i} S\mathbb{L}V\\
				&+e^{-f_1\tau} \alpha (\delta S)^{\frac{\alpha-\delta-1}{\delta}}\nabla(\partial_{y_i} S) \cdot \mathbb{D}(V) + e^{-f_1\tau}\alpha{\delta}^{\frac{\alpha-\delta-1}{\delta}}\frac{\alpha-\delta-1}{\delta} S^{\frac{\alpha-1}{\delta}-2} \partial_{y_i} S \nabla S \cdot \mathbb{D}(V)\\
				&+e^{-f_1\tau} \alpha (\delta S)^{\frac{\alpha-\delta-1}{\delta}}\nabla S \cdot \mathbb{D}(\partial_{y_i}V).
			\end{aligned}
		\end{equation}
		For any fixed index $i \in \{1, 2, 3\}$, we define :
		\[
		\omega_{V,y_0,i}(\tau) := e^{\Lambda(\tau-\bar{\tau})} \partial_{y_i} V(\tau, y_0 e^{\tau-\bar{\tau}}),
		\]
		where $y=y_0e^{\tau-\bar{\tau}}$ for either $\bar{\tau}=\tau_0$ or $|y_0|=R_0$. Thus, we find that $\omega_{V,y_0,i}(\tau)$ satisfies
		\begin{equation}\label{wVy0itau}
			\begin{aligned}
				\partial_\tau \omega_{V,y_0,i}(\tau) = 
				& -e^{\Lambda(\tau-\bar{\tau})} \Big( \partial_{y_i} V \cdot \nabla V + V \cdot \nabla(\partial_{y_i} V) + \delta \partial_{y_i} S \nabla S + \delta S \nabla(\partial_{y_i} S) \Big) \\
				&+e^{\Lambda(\tau-\bar{\tau})}e^{-f_1 \tau}(\delta S)^{\frac{\alpha-1}{\delta}}\mathbb{L}(\partial_{y_i}V)+e^{\Lambda(\tau-\bar{\tau})}e^{-f_1 \tau}\delta ^{\frac{\alpha-1}{\delta}} \frac{\alpha-1}{\delta}S^{\frac{\alpha-\delta-1}{\delta}}\partial_{y_i} S\mathbb{L}V\\
				&+e^{\Lambda(\tau-\bar{\tau})}e^{-f_1\tau} \alpha (\delta S)^{\frac{\alpha-\delta-1}{\delta}}\nabla(\partial_{y_i} S) \cdot \mathbb{D}(V) \\
				&+e^{\Lambda(\tau-\bar{\tau})} e^{-f_1\tau}\alpha{\delta}^{\frac{\alpha-\delta-1}{\delta}}\frac{\alpha-\delta-1}{\delta} S^{\frac{\alpha-1}{\delta}-2} \partial_{y_i} S \nabla S \cdot \mathbb{D}(V)\\
				&+e^{\Lambda(\tau-\bar{\tau})}e^{-f_1\tau} \alpha (\delta S)^{\frac{\alpha-\delta-1}{\delta}}\nabla S \cdot \mathbb{D}(\partial_{y_i}V).
			\end{aligned}
		\end{equation}
		By virtue of  \eqref{S,lower,1}, \eqref{nableU1}, and \eqref{123decay}, and proceeding similarly to the estimates in \eqref{vsvssvsv}, \eqref{S1SS2}, \eqref{SS333}, \eqref{SS1S1S1}, and \eqref{SS1S222}, we have
		\begin{align}
			&|e^{\Lambda(\tau-\bar{\tau})} \Big( \partial_{y_i} V \cdot \nabla V + V \cdot \nabla(\partial_{y_i} V) + \delta \partial_{y_i} S \nabla S + \delta S \nabla(\partial_{y_i} S) \Big)|\nonumber \\
			&\le C \sigma_0^{\frac{9}{5}} \phi(y_0 e^{\tau-\bar{\tau}})^{-1} e^{\Lambda(\tau-\bar{\tau})} \le C \sigma_0^{\frac{9}{5}} e^{(\Lambda-2+2\eta)(\tau-\bar{\tau})} \left\langle \frac{y_0}{R_0} \right\rangle^{-2+2\eta}, \\
			&|e^{\Lambda(\tau-\bar{\tau})}e^{-f_1 \tau}(\delta S)^{\frac{\alpha-1}{\delta}}\mathbb{L}(\partial_{y_i}V)|\le C \sigma_0^{\frac{9}{5}} \left\langle \frac{y_0}{R_0} \right\rangle^{-\Lambda} e^{(-1+3\eta)(\tau-\bar{\tau})},\\
			&|e^{\Lambda(\tau-\bar{\tau})}e^{-f_1 \tau}\delta ^{\frac{\alpha-1}{\delta}} \frac{\alpha-1}{\delta}S^{\frac{\alpha-\delta-1}{\delta}}\partial_{y_i} S\mathbb{L}V|\le C \sigma_0^{\frac{9}{5}} \left\langle \frac{y_0}{R_0} \right\rangle^{-\Lambda} e^{(-1 + 2\eta)(\tau - \bar{\tau})},\\
			&|e^{\Lambda(\tau-\bar{\tau})}e^{-f_1\tau} \alpha (\delta S)^{\frac{\alpha-\delta-1}{\delta}}\nabla(\partial_{y_i} S) \cdot \mathbb{D}(V)|\le C \sigma_0^{\frac{9}{5}} \left\langle \frac{y_0}{R_0} \right\rangle^{-\Lambda} e^{(-1 + 2\eta)(\tau - \bar{\tau})},\\
			&|e^{\Lambda(\tau-\bar{\tau})} e^{-f_1\tau}\alpha{\delta}^{\frac{\alpha-\delta-1}{\delta}}\frac{\alpha-\delta-1}{\delta} S^{\frac{\alpha-1}{\delta}-2} \partial_{y_i} S \nabla S \cdot \mathbb{D}(V)|\le C \sigma_0^{\frac{9}{5}} \left\langle \frac{y_0}{R_0} \right\rangle^{-\Lambda} e^{-\Lambda(\tau-\bar{\tau})},\\
			&|e^{\Lambda(\tau-\bar{\tau})}e^{-f_1\tau} \alpha (\delta S)^{\frac{\alpha-\delta-1}{\delta}}\nabla S \cdot \mathbb{D}(\partial_{y_i}V)|\le C \sigma_0^{\frac{9}{5}} \left\langle \frac{y_0}{R_0} \right\rangle^{-\Lambda} e^{(-1 + 2\eta)(\tau - \bar{\tau})}.\label{SnablaSDyiV}
		\end{align}
		In the fourth estimate above, the factor
		\(S^{(\alpha-1)/\delta-2}\) is controlled using the same smallness condition
		\[
		e^{-f_1\tau_0}\sigma_1^{\frac{\alpha-2\delta-1}{\delta}}C_1^3
		\le\sigma_0^{9/5}
		\]
		as in \eqref{SS1S1S1}.
		Combining \eqref{wVy0itau}-\eqref{SnablaSDyiV}, we arrive at 
		\begin{equation}
			|\partial_\tau \omega_{V,y_0,i}(\tau)| \le C\sigma_0^{\frac{9}{5}}\left\langle \frac{y_0}{R_0} \right\rangle^{-\Lambda}e^{(\Lambda-2+2\eta)(\tau - \bar{\tau})}.
		\end{equation}
		Integrating this inequality over the interval $[\bar{\tau}, \tau]$, we arrive at the oscillation bound:
		\begin{equation}\label{partialyiV}
			\begin{aligned}
				&\left| \partial_{y_i} V(\tau, y_0 e^{\tau-\bar{\tau}}) - e^{-\Lambda(\tau-\bar{\tau})} \partial_{y_i} V(\bar{\tau}, y_0) \right|\\
				&= e^{-\Lambda(\tau-\bar{\tau})} |\omega_{V,y_0,i}(\tau) - \omega_{V,y_0,i}(\bar{\tau})| \\
				&\le C \sigma_0^{\frac{9}{5}} e^{-\Lambda(\tau-\bar{\tau})} \left\langle \frac{y_0}{R_0} \right\rangle^{-\Lambda}.
			\end{aligned}
		\end{equation}
		In addition, we impose the spatial decay condition on the initial data that for all $|y_0| \ge R_0$:
		\begin{equation}\label{Mtau0y0-V}
			|\partial_{y_i} V(\tau_0, y_0)|\le M \left\langle \frac{y_0}{R_0} \right\rangle^{-\Lambda},
		\end{equation}
		the boundary values are controlled by:
		\begin{equation}\label{C1bartauy0-V}
			|\partial_{y_i} V(\bar{\tau}, y_0)| \le \frac{C_1}{10}\left\langle \frac{y_0}{R_0} \right\rangle^{-\Lambda} \quad \text{under the constraints } |y_0| = R_0 \text{ and } \bar{\tau} \ge \tau_0.
		\end{equation}
		By arguments completely parallel to those for \eqref{bartaugetau0y0R0} and \eqref{y0R0etautau04}, we can deduce that
		\[
		|\nabla \tilde{V}| + |\nabla V| \le  (\frac{\sqrt{3}}{4}+\frac{\sqrt{3}}{4}+\frac{1}{10})C_1\left\langle \frac{y}{R_0} \right\rangle^{-\Lambda}\le C_1\left\langle \frac{y}{R_0} \right\rangle^{-\Lambda}.
		\]
		Therefore, we obtain \eqref{UUC1yR0}. This completes the proof.
	\end{proof}
	
	We next derive estimates for the higher-order derivatives by interpolating between the weighted energy bounds and the lower-order estimates. The proof follows the arguments in Lemma 3.7 of \cite{Cao-Labora-Gomez-Serrano-Shi-Staffilani-2023} and Lemma 5.4 of \cite{Chen-Zhang-Zhu}, and is therefore omitted.
	\begin{lemma}\label{lem:higher-order-perturbation}
		For every integer $1\leq j\leq K-2$, one has
		\begin{equation}\label{eq:lessKKKK-2-perturbation}
			|\nabla^j V(\tau,y)|+|\nabla^j S(\tau,y)|
			\leq
			C(\sigma_0)\phi^{-\frac{j}{2}}
			\langle y\rangle^{-(\Lambda-1)}
			\langle y\rangle^{
				\frac{
					\Lambda j-\Lambda-K+\frac{5}{2}
					+(1-\eta)\left(K-\frac{5j}{2}\right)
				}{
					K-\frac{5}{2}
			}},
		\end{equation}
		and
		\begin{equation}\label{eq:lessK-2-perturbation}
			|\nabla^jS(\tau,y)|
			+|\nabla^j V(\tau,y)|
			\leq
			C(\sigma_0)\phi(y)^{-\frac j2}
			\langle y\rangle^{-(\Lambda-1)}
		\end{equation}for
		$(\tau,y)\in[\tau_0,\tau_1]\times\mathbb{R}^3$.
		Furthermore, for every $\bar{\varepsilon}>0$ and
		$\tau\in[\tau_0,\tau_1]$,
		\begin{equation}\label{eq:K-1-perturbation}
			\begin{aligned}
				&\left\|
				\langle y\rangle^{
					K(1-\eta)\frac{K-2}{K-1}-\bar{\varepsilon}}
				\nabla^{K-1} S(\tau,\cdot)
				\right\|_{L^{2+\frac{2}{K-2}}(\mathbb{R}^3)}
				\\
				&\qquad+
				\left\|
				\langle y\rangle^{
					K(1-\eta)\frac{K-2}{K-1}-\bar{\varepsilon}}
				\nabla^{K-1} V(\tau,\cdot)
				\right\|_{L^{2+\frac{2}{K-2}}(\mathbb{R}^3)}
				\leq C(\sigma_0,\bar{\varepsilon}).
			\end{aligned}
		\end{equation}
	\end{lemma}
	Next, we establish the estimate \eqref{S,U,2}.
	\begin{proposition}
		Under the bootstrap hypotheses \eqref{Puns}--\eqref{E1}, the bound in \eqref{S,U,2} holds:
		\begin{equation}
			\max \{|\widetilde{S}|, |\widetilde{V}|\} \le \frac{\sigma_0}{100} e^{-\varpi(\tau-\tau_0)}.
		\end{equation}
	\end{proposition}
	\begin{proof}
		\noindent \textbf{1. Interior bounds for $\widetilde{S}$ and $\widetilde{V}$ within the regional domain $|y| \le R_1$.}\\
		Recalling the structural framework of the template space and Lemma \ref{localidentity}, we have the following localized embedding control for the higher-order gradients of the velocity field:
		\begin{equation}\label{Eq_L_infinity_subspace_norm}
			\|(\nabla^j \widetilde{S},\nabla^j \widetilde{V})\|_{L^\infty(B(0,R_1))} = \|(\nabla^j\widetilde{\widetilde{S}},\nabla^j \widetilde{\widetilde{V}})\|_{L^\infty(B(0,R_1))} \le C \|(\widetilde{\widetilde{S}},\widetilde{\widetilde{V}})\|_X, \quad \text{for} \quad j \le4.
		\end{equation}
		From \eqref{truncatedsystem}, the perturbation vector $(\widetilde{\widetilde{S}}, \widetilde{\widetilde{V}})$ obeys the linearized dynamics
		\begin{equation}\label{Eq_perturbed_matrix_system}
			\partial_\tau (\widetilde{\widetilde{S}}, \widetilde{\widetilde{V}}) = \mathcal{L}(\widetilde{\widetilde{S}}, \widetilde{\widetilde{V}}) + \widehat{\chi}_2 \mathcal{F}(\widetilde{S}, \widetilde{V}),
		\end{equation}
		Here, the perturbation term $\mathcal{F}(\widetilde{S}, \widetilde{V})$ obeys the following formulation:
		\begin{equation}\label{Eq_F_vector_block}
			\mathcal{F}(\widetilde{S}, \widetilde{V}) = 
			\left( 
			\begin{aligned}
				& \mathcal{N}_s(\widetilde{S},\widetilde{V}) + \mathcal{E}_s(\bar{S},\bar{V}) + e^{-f_1 \tau} d \alpha (\delta S)^{\frac{\alpha-1}{\delta}} \Delta S + e^{-f_1 \tau} d \alpha (\alpha-\delta) (\delta S)^{\frac{\alpha-\delta-1}{\delta}} |\nabla S|^2 \\[1mm]
				& \mathcal{N}_v(\widetilde{S},\widetilde{V}) + \mathcal{E}_v(\bar{S},\bar{V}) + e^{-f_1 \tau} (\delta S)^{\frac{\alpha-1}{\delta}}\mathbb{L}(V) + e^{-f_1 \tau} \alpha (\delta S)^{\frac{\alpha-\delta-1}{\delta}} \nabla S \cdot \mathbb{D}V
			\end{aligned} 
			\right)^{\top}.
		\end{equation} 
		Then, the system is projected onto the stable subspace to extract the contractive features
		\begin{equation}\label{Eq_projected_stable_eq}
			\partial_\tau P_{\mathrm{sta}}(\widetilde{\widetilde{S}}, \widetilde{\widetilde{V}}) = \mathcal{L} P_{\mathrm{sta}}(\widetilde{\widetilde{S}}, \widetilde{\widetilde{V}}) + P_{\mathrm{sta}}(\widehat{\chi}_2 \mathcal{F}(\widetilde{S}, \widetilde{V})).
		\end{equation}
		Using Duhamel's principle we derive
		\begin{equation}\label{Eq_Duhamel_representation}
			P_{\mathrm{sta}}(\widetilde{\widetilde{S}}, \widetilde{\widetilde{V}})(\tau) = e^{(\tau - \tau_0)\mathcal{L}} P_{\mathrm{sta}}(\widetilde{\widetilde{S}}, \widetilde{\widetilde{V}})(\tau_0) + \int_{\tau_0}^\tau e^{(\tau - \bar{\tau})\mathcal{L}} \left( P_{\mathrm{sta}}(\widehat{\chi}_2 \mathcal{F}(\widetilde{S}, \widetilde{V})) \right)(\bar{\tau}) \, \mathrm{d}\bar{\tau}.
		\end{equation}
		Here, $e^{t\mathcal{L}}$ represents the linear operator semigroup generated by the differential operator $\mathcal{L}$. Now, we proceed to estimate the bound of $\|\widehat{\chi}_2 \mathcal{F}(\widetilde{S}, \widetilde{V})\|_X$:
		\begin{itemize}
			\item \noindent\textbf{Estimates for $\mathcal{N}_s$ and $\mathcal{N}_v$.}\\
			Using the weighted Gagliardo--Nirenberg inequality \eqref{eq:weighted-GN-Linfty-special-new} in Lemma~\ref{lem:weighted-GN-new}, together with \eqref{S,U,1} and \eqref{energy decayEE}, we arrive at
			\begin{equation}\label{Eq_nonlinear_X_norm_merged}
				\begin{split}
					& \left\|\widehat\chi_2 \left( \mathcal{N}_s(\widetilde{S},\widetilde{V}), \, \mathcal{N}_v(\widetilde{S},\widetilde{V}) \right) \right\|_X \\
					& \le C(m)\sum_{j\le m}(\|\nabla^j(\widetilde{V} \cdot \nabla \widetilde{V})\|_{L^\infty} + \|\nabla^j(\widetilde{V} \cdot \nabla \widetilde{S})\|_{L^\infty} + \|\nabla^j(\widetilde{S}\nabla \widetilde{S})\|_{L^\infty} + \|\nabla^j(\widetilde{S} \operatorname{div}\widetilde{V})\|_{L^\infty}) \\
					& \le C(m) \sum_{0 \le \ell \le j \le m} \left( \|\nabla^\ell \widetilde{V}\|_{L^\infty} \|\nabla^{j-\ell+1} \widetilde{V}\|_{L^\infty} + \|\nabla^\ell \widetilde{V}\|_{L^\infty} \|\nabla^{j-\ell+1} \widetilde{S}\|_{L^\infty} \right. \\
					& \quad \left. + \, \|\nabla^\ell \widetilde{S}\|_{L^\infty} \|\nabla^{j-\ell+1} \widetilde{S}\|_{L^\infty} + \|\nabla^\ell \widetilde{S}\|_{L^\infty} \|\nabla^{j-\ell+1} \widetilde{V}\|_{L^\infty} \right) \\
					& \le C(m) E^{\frac{1}{20}} \sigma_0^{\frac{39}{20}} e^{-\frac{39}{20}\varpi(\tau - \tau_0)} \\
					& \le C(m) \sigma_0^{\frac{19}{10}} e^{-\frac{39}{20}\varpi(\tau - \tau_0)}.
				\end{split}
			\end{equation}
			\item \noindent\textbf{Estimates for $\mathcal{E}_s$ and $\mathcal{E}_v$.}\\ 
			It is readily verified that $\widehat{\chi}_2 \mathcal{E}_s(\bar{S},\bar{V}) = 0$ and $\widehat{\chi}_2 \mathcal{E}_v(\bar{S},\bar{V}) = 0$ hold identically.
			\item \noindent\textbf{Estimates for Remaining Nonlinear Terms.}\\
			Using the bootstrap bounds and the estimates \eqref{eq:lessK-2-perturbation}, we have 
			\begin{equation}\label{Eq_F_dis_Sobolev_bound}
				\begin{aligned}
					&\|\widehat\chi_2 (e^{-f_1 \tau} (\delta S)^{\frac{\alpha-1}{\delta}}\mathbb{L}(V) + e^{-f_1 \tau} \alpha (\delta S)^{\frac{\alpha-\delta-1}{\delta}} \nabla S \cdot \mathbb{D}V)\|_{X}\\
					&= \|\widehat\chi_2 (e^{-f_1 \tau} (\delta S)^{\frac{\alpha-1}{\delta}}\mathbb{L}(V) + e^{-f_1 \tau} \alpha (\delta S)^{\frac{\alpha-\delta-1}{\delta}} \nabla S \cdot \mathbb{D}V)\|_{H^m}\\
					&\le C(m) e^{-f_1 \tau} \left\| S^{\frac{\alpha-1}{\delta}} \mathbb{L}(V) + \frac{\alpha}{\delta}S^{\frac{\alpha-1}{\delta}} \frac{\nabla S}{S} \mathbb{D}(V) \right\|_{L^\infty(B(0,3R_1))} \\
					&\quad + C(m) e^{-f_1 \tau} \left\| \nabla^m \left( S^{\frac{\alpha-1}{\delta}} \mathbb{L}(V) \right) + \frac{\alpha}{\delta} \nabla^m \left( S^{\frac{\alpha-1}{\delta}} \frac{\nabla S}{S} \mathbb{D}(V) \right) \right\|_{L^\infty(B(0,3R_1))} \\
					&\le C(\sigma_0,\sigma_1) e^{-f_1 \tau} \le C(\sigma_0,\sigma_1) \sigma_0^{-\frac{9}{5}} e^{-f_1 \frac{\tau}{2}} \sigma_0^{\frac{9}{5}} e^{-f_1 \frac{\tau}{2}} \\
					&\le C(\sigma_0,\sigma_1) e^{-f_1 \frac{\tau_0}{2}} \sigma_1^{\frac{6}{5}} e^{-f_1 \frac{\tau}{2}} \le \sigma_1^{\frac{6}{5}} e^{-f_1 \frac{\tau}{2}}.
				\end{aligned}
			\end{equation}
			Analogously, the same estimates give
			\begin{equation}\label{Eq_F_disSSS_Sobolev_bound}
				\begin{aligned}
					&\|\widehat\chi_2 (e^{-f_1 \tau} d \alpha (\delta S)^{\frac{\alpha-1}{\delta}} \Delta S + e^{-f_1 \tau} d \alpha (\alpha-\delta) (\delta S)^{\frac{\alpha-\delta-1}{\delta}} |\nabla S|^2)\|_{X}\\
					&\le C(\sigma_0,\sigma_1) e^{-f_1 \frac{\tau_0}{2}} \sigma_1^{\frac{6}{5}} e^{-f_1 \frac{\tau}{2}} \le \sigma_1^{\frac{6}{5}} e^{-f_1 \frac{\tau}{2}}.
				\end{aligned}
			\end{equation}
		\end{itemize}
		By virtue of \eqref{Eq_nonlinear_X_norm_merged}-\eqref{Eq_F_disSSS_Sobolev_bound}, we arrive at
		\begin{equation}\label{eq:full-forcing-X-bound}
			\|\widehat{\chi}_2 \mathcal{F}(\widetilde{S}, \widetilde{V})\|_X
			\le C\sigma_1^{\frac{6}{5}}e^{-\frac{3}{2}\varpi(\tau-\tau_0)}
			\le \sigma_1e^{-\frac{3}{2}\varpi(\tau-\tau_0)}.
		\end{equation}
		Taking the parameter in Lemma~\ref{lem:new-stable-unstable-summary} to be
		$\mathfrak c_g=\sigma_g=\frac{25}{12}\varpi$ and using its
		stable-semigroup estimate and \eqref{initialdatacondition1}, we obtain
		\begin{align}\label{Eq_stable_subspace_decay_estimate}
			\|P_{\mathrm{sta}}(\widetilde{\widetilde{S}}, \widetilde{\widetilde{V}})\|_X 
			&\le \sigma_1 e^{-\frac{\sigma_g}{2}(\tau - \tau_0)} + \int_{\tau_0}^\tau \sigma_1 e^{-\frac{\sigma_g}{2}(\tau - \bar{\tau})} e^{-\frac{3}{2}\varpi(\bar{\tau} - \tau_0)} \, \mathrm{d}\bar{\tau} \nonumber \\
			&\le C \frac{\sigma_1}{\sigma_g} e^{-\frac{\sigma_g}{2}(\tau - \tau_0)} \le C \frac{\sigma_1}{\sigma_g} e^{-\varpi(\tau - \tau_0)}.
		\end{align}
		Finally, using \eqref{eq:sigma-g-hierarchy}, \eqref{Puns} and \eqref{Eq_L_infinity_subspace_norm}, we obtain 
		\begin{equation}\label{4R1SVbound}
			\begin{aligned}
				\|( \nabla^j\widetilde{S},\nabla^j\widetilde{V})\|_{L^\infty(B(0,R_1))} &\le C \|\widetilde{\widetilde{S}},\widetilde{\widetilde{V}}\|_X \\
				&\le C(\|P_{\mathrm{sta}}(\widetilde{\widetilde{S}}, \widetilde{\widetilde{V}})\|_X +\|P_{\mathrm{uns}}(\widetilde{\widetilde{S}}, \widetilde{\widetilde{V}})\|_X) \\
				&\le C(\frac{\sigma_1}{\sigma_g} e^{-\varpi(\tau - \tau_0)}+\sigma_1 e^{-\varpi(\tau-\tau_0)}) \le\frac{\sigma_0}{100}e^{-\varpi(\tau-
					\tau_0)},
			\end{aligned}
		\end{equation}
		for $0\le j \le 4$.\\
		\noindent \textbf{2. Exterior bounds for $\widetilde{S}$ and $\widetilde{V}$ within the regional domain $|y| > R_1$.} \\
		From \eqref{perturSEq} and \eqref{perturVEq}, we see that the perturbation variables $(\widetilde{S}, \widetilde{V})$ satisfy the system
		\begin{equation}\label{Eq_perturbation_system}
			\begin{cases}
				\begin{aligned}
					\partial_\tau \widetilde{S} = &-(\Lambda - 1)\widetilde{S} - (y + \widehat{\chi}\bar{V}) \cdot \nabla \widetilde{S} - \delta(\widehat{\chi}\bar{S}) \operatorname{div} \widetilde{V} - \widetilde{V} \cdot \nabla(\widehat{\chi}\bar{S}) - \delta \widetilde{S} \operatorname{div}(\widehat{\chi}\bar{V}) \\
					&+{\mathcal{N}_s(\widetilde{S},\widetilde{V})}+{\mathcal{E}_s(\bar{S},\bar{V})} \\
					&+e^{-f_1 \tau} d \alpha (\delta S)^{\frac{\alpha-1}{\delta}} \Delta S + e^{-f_1 \tau} d \alpha (\alpha-\delta) (\delta S)^{\frac{\alpha-\delta-1}{\delta}} |\nabla S|^2,
				\end{aligned} \\
				\begin{aligned}
					\partial_\tau \widetilde{V} = &-(\Lambda - 1)\widetilde{V} - (y + \widehat{\chi}\bar{V}) \cdot \nabla \widetilde{V} - \delta(\widehat{\chi}\bar{S})\nabla \widetilde{S} - \widetilde{V} \cdot \nabla(\widehat{\chi}\bar{V}) - \delta \widetilde{S} \nabla(\widehat{\chi}\bar{S}) \\
					&+{\mathcal{N}_v(\widetilde{S},\widetilde{V})}+{\mathcal{E}_v(\bar{S},\bar{V})} \\
					&+e^{-f_1 \tau} (\delta S)^{\frac{\alpha-1}{\delta}}\mathbb{L}(V) + e^{-f_1 \tau} \alpha (\delta S)^{\frac{\alpha-\delta-1}{\delta}} \nabla S \cdot \mathbb{D}V.
				\end{aligned}
			\end{cases}
		\end{equation}
		Invoking the Gagliardo-Nirenberg inequality for the hypotheses \eqref{S,U,1} and \eqref{nable41} under the specific parameter choices
		\begin{equation*}
			p = \infty, \quad q = 2, \quad \theta = \frac{2}{5}, \quad \bar{r} = \infty, \quad i = 1, \quad l = 4,
		\end{equation*}
		we deduce the following $L^\infty$-bound:
		\begin{equation}\label{Eq_GN_interpolation_result}
			\max \left\{ \|\nabla \widetilde{S}\|_{L^\infty(B^c(0,R_1))}, \, \|\nabla \widetilde{V}\|_{L^\infty(B^c(0,R_1))} \right\} \le C_3\sigma_0 \left(\frac{1}{50}\right)^{\frac{1}{20}} e^{-\varpi(\tau-\tau_0)}.
		\end{equation}
		Consequently, by virtue of the estimates \eqref{est_decay_1}, \eqref{S,U,1}  and \eqref{Eq_GN_interpolation_result}, the linear part can be bounded by
		\begin{align}
			& \left|\widehat{\chi}\bar{V} \cdot \nabla \widetilde{S}\right| + \left|\delta(\widehat{\chi}\bar{S}) \operatorname{div}\widetilde{V}\right| + \left|\widetilde{V} \cdot \nabla(\widehat{\chi}\bar{S})\right| + \left|\delta \widetilde{S} \operatorname{div}(\widehat{\chi}\bar{V})\right| \le \frac{\sigma_0e^{-\varpi(\tau-\tau_0)}}{50M_1}[C_3\left(\frac{1}{50}\right)^{\frac{1}{20}}+\frac{1}{50}], \label{Eq_cross_term_bound_1} \\
			& \left|\widehat{\chi}\bar{V} \cdot \nabla \widetilde{V}\right| + \left|\delta(\widehat{\chi}\bar{S})\nabla \widetilde{S}\right| + \left|\widetilde{V} \cdot \nabla(\widehat{\chi}\bar{V})\right| + \left|\delta \widetilde{S} \nabla(\widehat{\chi}\bar{S})\right| \le\frac{\sigma_0e^{-\varpi(\tau-\tau_0)}}{50M_1}[C_3\left(\frac{1}{50}\right)^{\frac{1}{20}}+\frac{1}{50}]. \label{Eq_cross_term_bound_2}
		\end{align}
		Using \eqref{S,U,1}  and \eqref{Eq_GN_interpolation_result}, we arrive at 
		\begin{align}
			|\mathcal{N}_s(\widetilde{S},\widetilde{V})|\le|\widetilde{V} \cdot \nabla \widetilde{S}|+ |\delta \widetilde{S} \operatorname{div} \widetilde{V}|\le2C_3\left(\frac{1}{50}\right)^{\frac{21}{20}}\sigma_0^2e^{-2\varpi(\tau-\tau_0)},\\
			|\mathcal{N}_v(\widetilde{S},\widetilde{V})|\le |\widetilde{V} \cdot \nabla \widetilde{V}| + |\delta \widetilde{S} \nabla \widetilde{S}|\le2C_3\left(\frac{1}{50}\right)^{\frac{21}{20}}\sigma_0^2e^{-2\varpi(\tau-\tau_0)}.
		\end{align}
		In view of the decay estimates for the profile, we deduce that for $j \ge 0$ 
		$$|\nabla^j \bar{S}| + |\nabla^j \bar{V}| \le C \langle y \rangle^{1 - \Lambda - j}.$$
		It follows from \eqref{xdecayestim} that
		\begin{equation}
			\begin{aligned}
				\|{\mathcal{E}_s(\bar{S},\bar{V})}\|_{L^{\infty}(B^c(0,R_1))}&\le \|(\widehat{\chi}^2 - \widehat{\chi})\bar{V} \cdot \nabla \bar{S}\|_{L^{\infty}(B^c(0,R_1))} +\|(1 + \delta)\widehat{\chi}\bar{V} \cdot \nabla \widehat{\chi} \, \bar{S}\|_{L^{\infty}(B^c(0,R_1))}\\
				&\quad+ \|\delta(\widehat{\chi}^2 - \widehat{\chi})\bar{S} \operatorname{div} \bar{V}\|_{L^{\infty}(B^c(0,R_1))}\\ 
				&\le C\|(\widehat{\chi}^2 - \widehat{\chi}) \left\langle y  \right\rangle^{1-2\Lambda}\|_{L^{\infty}(B^c(0,R_1))}+C\||\widehat{\chi}||\nabla\widehat{\chi}| \left\langle y  \right\rangle^{2-2\Lambda}\|_{L^{\infty}(B^c(0,R_1))}\\
				&\le Ce^{(1-2\Lambda)\tau}\le Ce^{-\tau}\le Ce^{-\tau_0}e^{-(\tau-\tau_0)}\le \sigma_1e^{-(\tau-\tau_0)},
			\end{aligned}
		\end{equation}
		and 
		\begin{equation}
			\begin{aligned}
				\|{\mathcal{E}_v(\bar{S},\bar{V})}\|_{L^{\infty}(B^c(0,R_1))}&\le \|(\widehat{\chi}^2 - \widehat{\chi})\bar{V} \cdot \nabla \bar{V}\|_{L^{\infty}(B^c(0,R_1))} +\|\widehat{\chi}\bar{V} \cdot \nabla \widehat{\chi} \, \bar{V}\|_{L^{\infty}(B^c(0,R_1))}\\
				&\quad+ \|\delta(\widehat{\chi}^2 - \widehat{\chi})\bar{S} \nabla \bar{S}\|_{L^{\infty}(B^c(0,R_1))}+\|\delta\widehat{\chi}\bar{S}\nabla\widehat{\chi}\bar{S}\|_{L^{\infty}(B^c(0,R_1))}\\ 
				&\le C\|(\widehat{\chi}^2 - \widehat{\chi}) \left\langle y  \right\rangle^{1-2\Lambda}\|_{L^{\infty}(B^c(0,R_1))}+C\||\widehat{\chi}||\nabla\widehat{\chi}| \left\langle y  \right\rangle^{2-2\Lambda}\|_{L^{\infty}(B^c(0,R_1))}\\
				&\le Ce^{(1-2\Lambda)\tau}\le Ce^{-\tau}\le Ce^{-\tau_0}e^{-(\tau-\tau_0)}\le \sigma_1e^{-(\tau-\tau_0)}.
			\end{aligned}
		\end{equation}
		By virtue of \eqref{SSS1} and \eqref{SSS2}, together with similarly performed estimates, we can conclude that
		\begin{align}
			|e^{-f_1 \tau} d \alpha (\delta S)^{\frac{\alpha-1}{\delta}} \Delta S| &\le C \sigma_0^{\frac{9}{10}}{\sigma_1}^{\frac{\alpha-1}{\delta}} \left\langle \frac{y}{R_0} \right\rangle^{-2+2\eta+(\Lambda - 1)\frac{1-\alpha}{\delta}} e^{-f_1\tau}\nonumber \\
			&\le C {\sigma_1}^{\frac{6}{5}} e^{-\frac{f_1}{2}\tau},\\
			|e^{-f_1 \tau} d \alpha (\alpha-\delta) (\delta S)^{\frac{\alpha-\delta-1}{\delta}} |\nabla S|^2|&\le C\sigma_0^{\frac{9}{5}}{\sigma_1}^{\frac{\alpha-\delta-1}{\delta}}
			\left\langle \frac{y}{R_0} \right\rangle^{-2+2\eta+(\Lambda - 1)\frac{1-\alpha+\delta}
				{\delta}}e^{-f_1\tau}\nonumber \\
			&\le C{\sigma_1}^{\frac{6}{5}}e^{-\frac{f_1}{2}\tau},\\
			|e^{-f_1 \tau} (\delta S)^{\frac{\alpha-1}{\delta}}\mathbb{L}(V)|&\le C\sigma_0^{\frac{9}{10}}{\sigma_1}^{\frac{\alpha-1}{\delta}} \left\langle \frac{y}{R_0} \right\rangle^{-2+2\eta+(\Lambda - 1)\frac{1-\alpha}{\delta}} e^{-f_1\tau}\nonumber \\
			&\le C {\sigma_1}^{\frac{6}{5}} e^{-\frac{f_1}{2}\tau},\\
			|e^{-f_1 \tau} \alpha (\delta S)^{\frac{\alpha-\delta-1}{\delta}} \nabla S \cdot \mathbb{D}V|&\le C\sigma_0^{\frac{9}{5}}{\sigma_1}^{\frac{\alpha-\delta-1}{\delta}}
			\left\langle \frac{y}{R_0} \right\rangle^{-2+2\eta+(\Lambda - 1)\frac{1-\alpha+\delta}
				{\delta}}e^{-f_1\tau}\nonumber \\
			&\le C{\sigma_1}^{\frac{6}{5}}e^{-\frac{f_1}{2}\tau}. \label{ef1ssdvestcom}
		\end{align}
		Combining the estimates \eqref{Eq_perturbation_system}-\eqref{ef1ssdvestcom}, we arrive at
		\begin{equation}\label{Eq_differential_ineq}
			\begin{split}
				\left|\partial_\tau \left( e^{(\Lambda - 1)(\tau - \tau_0)} \widetilde{S}(\tau, e^{\tau- \bar{\tau}} y_0) \right)\right| 
				&\le e^{(\Lambda - 1)(\tau - \tau_0)} \left| \frac{\sigma_0 e^{-\varpi(\tau - \tau_0)}}{50 M_1} \left( C_3 \left( \frac{1}{50} \right)^{\frac{1}{20}} + \frac{1}{50} \right) \right. \\
				&\quad \left. + \, 2 C_3 \left( \frac{1}{50} \right)^{\frac{21}{20}} \sigma_0^2 e^{-2\varpi(\tau - \tau_0)} + \sigma_1 e^{-\varpi(\tau - \tau_0)}+ C{\sigma_1}^{\frac{6}{5}}e^{-\frac{f_1}{2}\tau}\right|\\ 
				&\le e^{(\Lambda - 1)(\tau - \tau_0)} \frac{\sigma_0(\Lambda-1-\varpi)}{200} e^{-\varpi(\tau - \tau_0)} .
			\end{split}
		\end{equation}
		To ensure that the last inequality holds, we utilize \eqref{Eq_parameter_hierarchy} and choose $M_1$ such that
		\begin{equation}\label{Eq_constant_choice}
			\frac{\left( C_3 \left(\frac{1}{50}\right)^{\frac{1}{20}} + \frac{1}{50} \right)}{50 M_1} = \frac{\Lambda - 1 - \varpi}{400}.
		\end{equation}
		Next, integrating the above inequality over the temporal interval $[\bar{\tau}, \tau]$ yields
		\begin{equation}\label{Eq_integrated_diff}
			\left| e^{(\Lambda - 1)(\tau - \tau_0)} \widetilde{S}(\tau, e^{\tau-\bar{\tau}}y_0) - e^{(\Lambda - 1)(\bar{\tau} - \tau_0)} \widetilde{S}(\bar{\tau}, y_0) \right| \le \frac{\sigma_0}{200} e^{(\Lambda - 1 - \varpi)(\tau - \tau_0)}.
		\end{equation}
		Applying the triangle inequality to \eqref{Eq_integrated_diff} immediately leads to the following pointwise bound for the perturbation variable:
		\begin{equation}\label{Eq_pointwise_S_bound}
			\left| \widetilde{S}(\tau, e^{\tau-\bar{\tau}} y_0) \right| \le \frac{\sigma_0}{200} e^{-\varpi(\tau - \tau_0)} + e^{(\Lambda - 1)(\bar{\tau} - \tau)} \left| \widetilde{S}(\bar{\tau}, y_0) \right|.
		\end{equation}
		Now, we split our analysis into two geometric cases depending on the localization of the initial point $y_0$:
		
		\begin{itemize}
			\item \textbf{Case 1: The localized domain on the boundary $|y_0| = R_1$ for $\bar{\tau} \ge \tau_0$.} \\
			In this regime, taking the supremum over all possible trajectories, the combination of \eqref{4R1SVbound} and the bootstrap assumptions yields
			\begin{align}
				e^{(\Lambda - 1)(\bar{\tau} - \tau)} \max \left| \widetilde{S}(\bar{\tau}, y_0) \right| &\le C e^{-\varpi(\bar{\tau} - \tau_0)} \frac{\sigma_1}{\sigma_g} e^{(\Lambda - 1)(\bar{\tau} - \tau)} \nonumber \\
				&\le \frac{\sigma_0}{200} e^{-\varpi(\tau - \tau_0)}. \label{Eq_case1_final}
			\end{align}
			\item \textbf{Case 2: The far-field domain $|y_0| > R_1$ at the initial layer $\bar{\tau} = \tau_0$.} \\
			For the initial data outside the ball of radius $R_1$, the localized smallness assumption implies that
			\begin{equation*}
				e^{(\Lambda - 1)(\tau_0 - \tau)} \left| \widetilde{S}(\tau_0, y_0) \right| \le \sigma_1 e^{(\Lambda - 1)(\tau_0 - \tau)},
			\end{equation*}
			which can be bounded uniformly by the target decay:
			\begin{equation}\label{Eq_case2_final}
				e^{(\Lambda - 1)(\tau_0 - \tau)} \left| \widetilde{S}(\tau_0, y_0) \right| \le \frac{\sigma_0}{200} e^{-\varpi(\tau - \tau_0)}.
			\end{equation}
		\end{itemize}
		
		Consequently, gathering \eqref{Eq_case1_final} and \eqref{Eq_case2_final}, the desired uniform decay estimates for $\widetilde{S}$ are successfully established for all $\tau \ge \tau_0$ and $|y|\ge R_1$, i.e.
		$$|\widetilde{S}|\le \frac{\sigma_0}{100} e^{-\varpi(\tau-\tau_0)}.$$
		The proof for the boundedness of $\widetilde{V}$ is entirely analogous. This completes the proof.
	\end{proof}
	\subsection{Higher-Order estimates}
	\begin{proposition}
		Under the bootstrap hypotheses \eqref{Puns}-\eqref{E1}, we can prove \eqref{nable42}:
		\begin{equation}
			\max \left\{\|\nabla^4 \widetilde{S}\|_{L^2(B^c(0,R_1))}, \|\nabla^4 
			\widetilde{V}\|_{L^2(B^c(0,R_1))}\right\} \le \frac{\sigma_0}{2} e^{-\varpi(\tau-\tau_0)}.
		\end{equation}
	\end{proposition}
	\begin{proof}
		To perform the high-order energy analysis, we first differentiate \eqref{perturSEq}--\eqref{perturVEq} by $\partial_\beta$ with $|\beta| = 4$, which leads to the following linearized auxiliary system:
		\begin{align}
			\partial_\tau \partial_\beta \widetilde{S} &= \mathcal{G}_{s,0}(\widetilde{V}, \widetilde{S}) + \mathcal{G}_{s,1}(\widetilde{V}, \widetilde{S}) + \mathcal{G}_{s,2}(\widetilde{V}, \widetilde{S}) + \mathcal{G}_{s,3}(\widetilde{V}, \widetilde{S}) + \partial_\beta \mathcal{N}_s + \partial_\beta \mathcal{E}_s+\partial_\beta \mathcal{F}_s, \label{Eq_high_order_S_new} \\
			\partial_\tau \partial_\beta \widetilde{V} &= \mathcal{G}_{v,0}(\widetilde{V}, \widetilde{S}) + \mathcal{G}_{v,1}(\widetilde{V}, \widetilde{S}) + \mathcal{G}_{v,2}(\widetilde{V}, \widetilde{S}) + \mathcal{G}_{v,3}(\widetilde{V}, \widetilde{S}) + \partial_\beta \mathcal{N}_v + \partial_\beta \mathcal{E}_v + \partial_\beta \mathcal{F}_{v}. \label{Eq_high_order_V_new}
		\end{align}
		In this representation, the differential operators $\mathcal{G}_{s,j}$ and $\mathcal{G}_{v,j}$ ($j=0,1,2,3$) are classified by the distribution of derivative orders. Specifically, the fifth-order derivatives act on $\widetilde{S}$ and $\widetilde{V}$ exclusively within $\mathcal{G}_{s,0}$ and $\mathcal{G}_{v,0}$. The linear expressions containing exactly fourth-order and third-order derivatives are grouped into $(\mathcal{G}_{s,1}, \mathcal{G}_{v,1})$ and $(\mathcal{G}_{s,2}, \mathcal{G}_{v,2})$, respectively. Finally, the remaining linear combinations with at most two derivatives are collected in $(\mathcal{G}_{s,3}, \mathcal{G}_{v,3})$.
		
		Now, testing \eqref{Eq_high_order_S_new} and \eqref{Eq_high_order_V_new} by $\partial_\beta \widetilde{S}$ and $\partial_\beta \widetilde{V}$ respectively, followed by summing over $|\beta| = 4$ and integrating by parts outside the ball $B(0, R_1)$, we arrive at
		\begin{equation}\label{Eq_energy_identity_H_chain}
			\begin{aligned}
				H &= \frac{\mathrm{d}}{\mathrm{d}\tau} \left( \|\nabla^4 \widetilde{S}\|_{L^2(B^c(0,R_1))}^2 + \|\nabla^4 \widetilde{V}\|_{L^2(B^c(0,R_1))}^2 \right) \\
				&= \frac{\mathrm{d}}{\mathrm{d}\tau} \sum_{|\beta|=4} \int_{B^c(0,R_1)} \left( |\partial_\beta \widetilde{S}|^2 + |\partial_\beta \widetilde{V}|^2 \right) \mathrm{d}y \\
				&= 2 \sum_{|\beta|=4} \sum_{j=0}^{3}  \int_{B^c(0,R_1)} \partial_\beta \widetilde{V} \cdot \mathcal{G}_{v,j} +\partial_\beta \widetilde{S} \, \mathcal{G}_{s,j} \, \mathrm{d}y + 2 \sum_{|\beta|=4}  \int_{B^c(0,R_1)} \partial_\beta \widetilde{V} \cdot \partial_\beta \mathcal{N}_v + \partial_\beta \widetilde{S} \, \partial_\beta \mathcal{N}_s \, \mathrm{d}y \\
				&\quad + 2 \sum_{|\beta|=4} \int_{B^c(0,R_1)} \partial_\beta \widetilde{V} \cdot \partial_\beta \mathcal{E}_v+ \partial_\beta \widetilde{S} \, \partial_\beta \mathcal{E}_s \, \mathrm{d}y + 2 \sum_{|\beta|=4}\int_{B^c(0,R_1)} \partial_\beta \widetilde{V} \cdot \partial_\beta \mathcal{F}_{v}+\partial_\beta \widetilde{S} \cdot \partial_\beta \mathcal{F}_{s} \, \mathrm{d}y \\
				&:= \sum_{|\beta|=4} \left( \sum_{j=0}^{3} H_{j,\beta} + H_{N,\beta} + H_{E,\beta} + H_{F,\beta} \right).
			\end{aligned}
		\end{equation}
		\noindent\textbf{1. Estimates for $H_{0,\beta}$.}\par
		To handle $H_{0,\beta}$ in \eqref{Eq_energy_identity_H_chain}, we integrate by parts over the exterior region $B^c(0,R_1)$ and use \eqref{R1sigma1sigmag1}, \eqref{4R1SVbound}, which yields the following estimate:
		\begin{equation}\label{Eq_H_0_beta_identity_bar}
			\begin{aligned}
				H_{0,\beta} &= -2 \int_{B^c(0,R_1)} \left( \partial_\beta \widetilde{V} \cdot \left( (y + \widehat{\chi}\bar{V}) \cdot \nabla \right) \partial_\beta \widetilde{V} + \left( \partial_\beta \widetilde{V} \cdot \nabla \right) \partial_\beta \widetilde{S} (\delta \widehat{\chi}\bar{S}) \right) \mathrm{d}y \\
				&\quad - 2 \int_{B^c(0,R_1)} \left( \partial_\beta \widetilde{S} (y + \widehat{\chi}\bar{V}) \cdot \nabla \partial_\beta \widetilde{S} + \partial_\beta \widetilde{S} \operatorname{div}(\partial_\beta \widetilde{V}) \delta \widehat{\chi}\bar{S} \right) \mathrm{d}y \\
				&\le \int_{B^c(0,R_1)} \left( 3 + |\operatorname{div}(\widehat{\chi}\bar{V})| \right) \left( |\partial_\beta \widetilde{S}|^2 + |\partial_\beta \widetilde{V}|^2 \right) \mathrm{d}y \\
				&\quad + 2 \int_{B^c(0,R_1)} |\nabla(\delta \widehat{\chi}\bar{S})| |\partial_\beta \widetilde{S}| |\partial_\beta \widetilde{V}| \, \mathrm{d}y \\
				&\quad + \int_{\partial B(0,R_1)} \left( R_1 + |\widehat{\chi}\bar{V}| \right) \left( |\partial_\beta \widetilde{S}|^2 + |\partial_\beta \widetilde{V}|^2 \right) \mathrm{d}S + 2 \int_{\partial B(0,R_1)} |\delta \widehat{\chi}\bar{S}| |\partial_\beta \widetilde{S}| |\partial_\beta \widetilde{V}| \, \mathrm{d}S \\
				&\le \int_{B^c(0,R_1)} \left( 3 + 3|\nabla(\widehat{\chi}\bar{V})|+|\nabla(\delta \widehat{\chi}\bar{S})| \right) \left( |\partial_\beta \widetilde{S}|^2 + |\partial_\beta \widetilde{V}|^2 \right) \mathrm{d}y\\
				&\quad +C(R_1+ |\bar{V}|(R_1)+|\delta\bar{S}|(R_1))(\frac{\sigma_1}{\sigma_g})^2 e^{-2\varpi(\tau - \tau_0)}\\
				&\le \int_{B^c(0,R_1)} \left( 3 + 3|\nabla(\widehat{\chi}\bar{V})|+|\nabla(\delta \widehat{\chi}\bar{S})| \right) \left( |\partial_\beta \widetilde{S}|^2 + |\partial_\beta \widetilde{V}|^2 \right) \mathrm{d}y+(\frac{\sigma_1}{\sigma_g})^{\frac{17}{10}}e^{-2\varpi(\tau - \tau_0)}.
			\end{aligned}
		\end{equation}
		\noindent\textbf{2. Estimates for $H_{1,\beta}$.}\par
		In fact, the explicit expressions for $\mathcal{G}_{s,1}$ and $\mathcal{G}_{v,1}$ can be verified as follows:
		$$\begin{aligned}
			\mathcal{G}_{s,1} &= -(\Lambda - 1)\partial_\beta \widetilde{S} - (\partial_\beta \widetilde{V} \cdot \nabla)(\widehat{\chi}\bar{S}) - \delta\partial_\beta \widetilde{S} \operatorname{div}(\widehat{\chi}\bar{V}) \\
			&\quad - \sum_{j=1}^{4}\sum_{k=1}^{3} \partial_{y_{\beta_j}} (\widehat{\chi}\bar{V}_k + y_k) \partial_{y_k}\partial_{\beta^{(j)}} \widetilde{S} - \delta \sum_{j=1}^{4}\partial_{y_{\beta_j}} (\widehat{\chi}\bar{S})\partial_{\beta^{(j)}} \operatorname{div}\widetilde{V} \\
			&= -(\Lambda + 3)\partial_\beta \widetilde{S} - (\partial_\beta \widetilde{V} \cdot \nabla)(\widehat{\chi}\bar{S}) - \delta\partial_\beta \widetilde{S} \operatorname{div}(\widehat{\chi}\bar{V}) \\
			&\quad - \sum_{j=1}^{4}\sum_{k=1}^{3} \partial_{y_{\beta_j}} (\widehat{\chi}\bar{V}_k)\partial_{y_k}\partial_{\beta^{(j)}} \widetilde{S} - \delta \sum_{j=1}^{4}\partial_{y_{\beta_j}} (\widehat{\chi}\bar{S})\partial_{\beta^{(j)}} \operatorname{div}\widetilde{V},
		\end{aligned}$$
		and
		$$\begin{aligned}
			\mathcal{G}_{v,1} &= -(\Lambda - 1)\partial_\beta \widetilde{V} - (\partial_\beta \widetilde{V} \cdot \nabla)(\widehat{\chi}\bar{V}) - \delta\partial_\beta \widetilde{S}\nabla(\widehat{\chi}\bar{S}) \\
			&\quad - \sum_{j=1}^{4}\sum_{k=1}^{3} \partial_{y_{\beta_j}} (\widehat{\chi}\bar{V}_k + y_k) \partial_{y_k}\partial_{\beta^{(j)}} \widetilde{V} - \delta \sum_{j=1}^{4}\partial_{y_{\beta_j}} (\widehat{\chi}\bar{S})\nabla(\partial_{\beta^{(j)}} \widetilde{S}) \\
			&= -(\Lambda + 3)\partial_\beta \widetilde{V} - (\partial_\beta \widetilde{V} \cdot \nabla)(\widehat{\chi}\bar{V}) - \delta\partial_\beta \widetilde{S}\nabla(\widehat{\chi}\bar{S}) \\
			&\quad - \sum_{j=1}^{4}\sum_{k=1}^{3} \partial_{y_{\beta_j}} (\widehat{\chi}\bar{V}_k)\partial_{y_k}\partial_{\beta^{(j)}} \widetilde{V} - \delta \sum_{j=1}^{4}\partial_{y_{\beta_j}} (\widehat{\chi}\bar{S})\nabla(\partial_{\beta^{(j)}} \widetilde{S}).
		\end{aligned}$$
		Consequently, we obtain (using $\delta<\frac{1}{\sqrt{3}}$)
		\begin{equation}\label{Eq_H_1_beta_estimate}
			\begin{aligned}
				H_{1,\beta} &= 2 \int_{B^c(0,R_1)} \partial_\beta \widetilde{V} \cdot \mathcal{G}_{v,1} \, \mathrm{d}y+2 \int_{B^c(0,R_1)} \partial_\beta \widetilde{S} \, \mathcal{G}_{s,1} \, \mathrm{d}y \\
				&\le -2(\Lambda + 3) \int_{B^c(0,R_1)} \left( |\partial_\beta \widetilde{S}|^2 + |\partial_\beta \widetilde{V}|^2 \right) \mathrm{d}y \\
				&\quad + 20 \left( \|\nabla(\widehat{\chi}\bar{S})\|_{L^\infty(B^c(0,R_1))} + \|\nabla(\widehat{\chi}\bar{V})\|_{L^\infty(B^c(0,R_1))} \right) \\
				&\quad \quad \times \left( \|\nabla^4 \widetilde{S}\|_{L^2(B^c(0,R_1))}^2 + \|\nabla^4 \widetilde{V}\|_{L^2(B^c(0,R_1))}^2 \right).
			\end{aligned}
		\end{equation}
		\noindent\textbf{3. Estimates for $H_{2,\beta}$.}\par
		Employing Hölder and the Gagliardo--Nirenberg inequalities, together with \eqref{est_decay_3}, \eqref{S,U,1}, and \eqref{nable41}, we deduce that
		\begin{equation*}
			\begin{aligned}
				&\|\mathcal{G}_{s,2}\|_{L^2(B^c(0,R_1))} + \|\mathcal{G}_{v,2}\|_{L^2(B^c(0,R_1))} \\
				&\le 80 \left( \|\nabla^2(\widehat{\chi}\bar{V})\|_{L^8(B^c(0,R_1))} + \|\nabla^2(\widehat{\chi}\bar{S})\|_{L^8(B^c(0,R_1))} \right) \\
				&\quad \times \left( \|\nabla^3 \widetilde{V}\|_{L^{\frac{8}{3}}(B^c(0,R_1))} + \|\nabla^3 \widetilde{S}\|_{L^{\frac{8}{3}}(B^c(0,R_1))} \right) \\
				&\le C \left( \|\nabla^2(\widehat{\chi}\bar{V})\|_{L^8(B^c(0,R_1))} + \|\nabla^2(\widehat{\chi}\bar{S})\|_{L^8(B^c(0,R_1))} \right) \sigma_0 50^{-\frac{1}{4}} e^{-\varpi(\tau - \tau_0)}.
			\end{aligned}
		\end{equation*}
		Thus, we arrive at
		\begin{equation}\label{Eq_H_2_beta_final_estimate}
			\begin{aligned}
				H_{2,\beta} &= 2 \int_{B^c(0,R_1)} \partial_\beta \widetilde{V} \cdot \mathcal{G}_{v,2} \, \mathrm{d}y + 2 \int_{B^c(0,R_1)} \partial_\beta \widetilde{S} \, \mathcal{G}_{s,2} \, \mathrm{d}y \\
				&\le C \left( \|\nabla^2(\widehat{\chi}\bar{S})\|_{L^8(B^c(0,R_1))} + \|\nabla^2(\widehat{\chi}\bar{V})\|_{L^8(B^c(0,R_1))} \right) \\
				&\quad \times \left( \int_{B^c(0,R_1)} \left( |\partial_\beta \widetilde{S}|^2 + |\partial_\beta \widetilde{V}|^2 \right) \mathrm{d}y \right)^{\frac{1}{2}} \sigma_0 50^{-\frac{1}{4}} e^{-\varpi(\tau - \tau_0)}.
			\end{aligned}
		\end{equation}
		\noindent\textbf{4. Estimates for $H_{3,\beta}$.}\par
		A combined application of Hölder inequality, the Gagliardo--Nirenberg inequality, and the estimates \eqref{est_decay_4}, \eqref{S,U,1}, \eqref{nable41} enables us to establish that
		\begin{equation*}
			\begin{aligned}
				&\|\mathcal{G}_{s,3}\|_{L^2(B^c(0,R_1))} + \|\mathcal{G}_{v,3}\|_{L^2(B^c(0,R_1))}\\ &\le 80 \left( \|\widetilde{V}\|_{W^{2,\infty}(B^c(0,R_1))} + \|\widetilde{S}\|_{W^{2,\infty}(B^c(0,R_1))} \right) \\
				&\quad \times \sum_{j=3}^{5} \left( \|\nabla^j(\widehat{\chi}\bar{S})\|_{L^2(B^c(0,R_1))} + \|\nabla^j(\widehat{\chi}\bar{V})\|_{L^2(B^c(0,R_1))} \right) \\
				&\le C \sum_{j=3}^{5} \left( \|\nabla^j(\widehat{\chi}\bar{S})\|_{L^2(B^c(0,R_1))} + \|\nabla^j(\widehat{\chi}\bar{V})\|_{L^2(B^c(0,R_1))} \right) \sigma_0 50^{-\frac{2}{5}} e^{-\varpi(\tau - \tau_0)}.
			\end{aligned}
		\end{equation*}
		Therefore, we obtain
		\begin{equation}\label{Eq_H_3_beta_final_estimate}
			\begin{aligned}
				H_{3,\beta} &= 2 \int_{B^c(0,R_1)} \partial_\beta \widetilde{V} \cdot \mathcal{G}_{v,3} \, \mathrm{d}y + 2 \int_{B^c(0,R_1)} \partial_\beta \widetilde{S} \, \mathcal{G}_{s,3} \, \mathrm{d}y \\
				&\le C \sum_{j=3}^{5} \left( \|\nabla^j(\widehat{\chi}\bar{S})\|_{L^2(B^c(0,R_1))} + \|\nabla^j(\widehat{\chi}\bar{V})\|_{L^2(B^c(0,R_1))} \right) \\
				&\quad \times \left( \int_{B^c(0,R_1)} \left( |\partial_\beta \widetilde{S}|^2 + |\partial_\beta \widetilde{V}|^2 \right) \mathrm{d}y \right)^{\frac{1}{2}} \sigma_0 50^{-\frac{2}{5}} e^{-\varpi(\tau - \tau_0)}.
			\end{aligned}
		\end{equation}
		\noindent\textbf{5. Estimates for $H_{N,\beta}$.}\par
		Using the Gagliardo-Nirenberg inequality for \eqref{nable41} and \eqref{energy decayEE}, we have
		$$
		\begin{aligned}
			\|\nabla^5 \widetilde{V}\|_{L^2(B^c(0,R_1))} &\le C(K) \left( \|\nabla^4 \widetilde{V}\|_{L^2(B^c(0,R_1))}^{\frac{K-5}{K-4}} \|\nabla^K \widetilde{V}\|_{L^2(B^c(0,R_1))}^{\frac{1}{K-4}} + \|\nabla^4 \widetilde{V}\|_{L^2(B^c(0,R_1))} \right) \nonumber \\
			&\le C(K) \left( E^{\frac{1}{10}} \sigma_0^{\frac{9}{10}} e^{-\frac{9}{10}\varpi(\tau - \tau_0)} + \sigma_0 e^{-\varpi(\tau - \tau_0)} \right) \le C(K) \sigma_0^{\frac{4}{5}} e^{-\frac{9}{10}\varpi(\tau - \tau_0)}.
		\end{aligned}
		$$
		By parallel arguments, one can establish the corresponding estimates for $\widetilde{S}$. Consequently, under the condition $C(K)\sigma_0 \ll 1$ and using the Gagliardo-Nirenberg inequality, we deduce that
		$$\begin{aligned}
			&\|\partial_\beta \mathcal{N}_s\|_{L^2(B^c(0,R_1))} + \|\partial_\beta \mathcal{N}_v\|_{L^2(B^c(0,R_1))} \nonumber \\
			&\le C \left( \|\nabla^5 \widetilde{V}\|_{L^2(B^c(0,R_1))} + \|\nabla^5 \widetilde{S}\|_{L^2(B^c(0,R_1))} \right) \left( \|\widetilde{V}\|_{L^\infty(B^c(0,R_1))} + \|\widetilde{S}\|_{L^\infty(B^c(0,R_1))} \right) \nonumber \\
			&\quad + C \left( \|\nabla^4 \widetilde{V}\|_{L^2(B^c(0,R_1))} + \|\nabla^4 \widetilde{S}\|_{L^2(B^c(0,R_1))} \right) \left( \|\nabla \widetilde{V}\|_{L^\infty(B^c(0,R_1))} + \|\nabla \widetilde{S}\|_{L^\infty(B^c(0,R_1))} \right) \nonumber \\
			&\quad + C \left( \|\nabla^3 \widetilde{V}\|_{L^{\frac{8}{3}}(B^c(0,R_1))} + \|\nabla^3 \widetilde{S}\|_{L^{\frac{8}{3}}(B^c(0,R_1))} \right) \left( \|\nabla^2 \widetilde{V}\|_{L^8(B^c(0,R_1))} + \|\nabla^2 \widetilde{S}\|_{L^8(B^c(0,R_1))} \right) \nonumber \\
			&\le C(K) \sigma_0^{\frac{9}{5}} e^{-\frac{19}{10}\varpi(\tau - \tau_0)} \le \sigma_0^{\frac{3}{2}} e^{-\frac{3}{2}\varpi(\tau - \tau_0)}.
		\end{aligned}$$
		Consequently,
		\begin{equation}\label{Eq_H_N_beta_final_estimate}
			\begin{aligned}
				H_{N,\beta} &= 2 \int_{B^c(0,R_1)} \partial_\beta \widetilde{V} \cdot \partial_\beta \mathcal{N}_v \, \mathrm{d}y + 2 \int_{B^c(0,R_1)} \partial_\beta \widetilde{S} \, \partial_\beta \mathcal{N}_s \, \mathrm{d}y \\
				&\le 2 \sigma_0^{\frac{3}{2}} e^{-\frac{3}{2}\varpi(\tau - \tau_0)} \left( \int_{B^c(0,R_1)} \left( |\partial_\beta \widetilde{V}|^2 + |\partial_\beta \widetilde{S}|^2 \right) \mathrm{d}y \right)^{\frac{1}{2}}.
			\end{aligned}
		\end{equation}
		\noindent\textbf{6. Estimates for $H_{E,\beta}$.}\par
		Taking into account \eqref{xdecayestim}, the cutoff term satisfies the pointwise estimate
		\begin{equation}\label{Eq_cutoff_error_pointwise_new}
			\left| \nabla^4 \left( (\widehat{\chi}^2 - \widehat{\chi})\bar{V} \cdot \nabla \bar{V} \right) \right| \le C \sum_{0 \le j \le 4} \left| \nabla^{4-j}(\widehat{\chi}^2 - \widehat{\chi}) \right| \left| \nabla^j (\bar{V} \cdot \nabla \bar{V}) \right| \le C \sum_{0 \le j \le 4} e^{-(4-j)\tau} \langle y \rangle^{1 - 2\Lambda - j}.
		\end{equation}
		It then follows from a direct integration over the exterior domain that
		\begin{equation}\label{Eq_cutoff_error_L2_integration_new}
			\begin{aligned}
				\left\| \nabla^4 \left( (\widehat{\chi}^2 - \widehat{\chi})\bar{V} \cdot \nabla \bar{V} \right) \right\|_{L^2(B^c(0,R_1))}^2 &\le C \sum_{0 \le j \le 4} e^{-2(4-j)\tau} \int_{\frac{1}{2}e^\tau}^{e^\tau} r^{2(1 - 2\Lambda - j)} r^2 \, \mathrm{d}r \\
				&\le C \sum_{0 \le j \le 4} e^{\left(-2(4-j) + 5 - 4\Lambda - 2j\right)\tau} \\
				&\le C e^{-(4\Lambda + 3)\tau} \le C e^{-2\tau}.
			\end{aligned}
		\end{equation}
		By a similar method, the remaining components $\nabla^4 \mathcal{E}_s$ and $\nabla^4 \mathcal{E}_v$ can be bounded identically. Assuming $\tau_0 \gg 1$ and utilizing \eqref{S,U,1}, we obtain
		\begin{equation}\label{Eq_E_zeta_E_u_final_bound_new}
			\|\nabla^4 \mathcal{E}_s\|_{L^2(B^c(0,R_1))} + \|\nabla^4 \mathcal{E}_v\|_{L^2(B^c(0,R_1))} \le C e^{-\tau} \le C e^{-\tau_0} e^{-(\tau - \tau_0)}  \le \sigma_1 e^{-(\tau - \tau_0)}.
		\end{equation}
		Therefore, we obtain 
		\begin{equation}\label{Eq_H_E_beta_estimate_revised}
			\begin{aligned}
				H_{E,\beta} &= 2 \int_{B^c(0,R_1)} \partial_\beta \widetilde{V} \cdot \partial_\beta \mathcal{E}_v \, \mathrm{d}y + 2 \int_{B^c(0,R_1)} \partial_\beta \widetilde{S} \, \partial_\beta \mathcal{E}_s \, \mathrm{d}y \\
				&\le 4 \sigma_1 e^{-(\tau - \tau_0)} \left( \int_{B^c(0,R_1)} \left( |\partial_\beta \widetilde{V}|^2 + |\partial_\beta \widetilde{S}|^2 \right) \mathrm{d}y \right)^{\frac{1}{2}}.
			\end{aligned}
		\end{equation}
		\noindent\textbf{7. Estimates for $H_{F,\beta}$.}\par
		
		Prior to evaluating the dissipation term $\mathcal{F}_{s}$ and $\mathcal{F}_{v}$, the pointwise control on $\left|\nabla^{4-j}\Delta S\right|$ is established via (using \eqref{eq:lessKKKK-2-perturbation})
		\begin{equation}\label{Eq_diffusion_L_V_pointwise}
			\left|\nabla^{4-j}\Delta S\right| \le C(\sigma_0, j)\langle y \rangle^{-\frac{(K-(6-j)-3/2)\Lambda+K(1-\eta)(6-j-1)}{K-5/2}} \le C(\sigma_0, j)\langle y \rangle^{-\frac{K+j-15/2+K(1-\eta)(5-j)}{K-5/2}}.
		\end{equation}
		Moreover, considering any arbitrary integer $\ell \in [1, j]$ and positive indices $k_i \ge 1$ ($1 \le i \le \ell$) that comply with the partition restriction $k_1 + \dots + k_\ell = j$, the asymptotic smallness criterion $\sigma_0^{\frac{3}{2}} \ll \sigma_1 \ll \sigma_0$ and \eqref{eq:lessK-2-perturbation} guarantee that
		\begin{equation}\label{Eq_diffusion_S_fractional_power_bound}
			\begin{aligned}
				\left|\nabla^j\left(S^{\frac{\alpha-1}{\delta}}\right)\right| &\le C(j) \sum_{k_1+\dots+k_\ell=j} \frac{\left|\nabla^{k_1}S\right|\left|\nabla^{k_2}S\right| \dots \left|\nabla^{k_\ell}S\right|}{S^l} S^{\frac{\alpha-1}{\delta}} \\
				&\le C(\sigma_0, j) \sum_{k_1+\dots+k_\ell=j} \sigma_1^{-l+\frac{\alpha-1}{\delta}} \phi^{-\frac{k_1+k_2+\dots+k_\ell}{2}} \left\langle \frac{y}{R_0} \right\rangle^{-\frac{\alpha-1}{\delta}(\Lambda-1)} \\
				&\le C(\sigma_0, j) \phi^{-\frac{j}{2}} \langle y \rangle^{\frac{1-\alpha}{\delta}(\Lambda-1)}.
			\end{aligned}
		\end{equation}
		Combining \eqref{Eq_diffusion_L_V_pointwise} with \eqref{Eq_diffusion_S_fractional_power_bound}, we arrive at
		\begin{equation}\label{Eq_diffusion_total_pointwise_final}
			\begin{aligned}
				\left|\partial^\beta\left(S^{\frac{\alpha-1}{\delta}}\Delta S\right)\right| &\le C \sum_{0\le j\le 4} \left|\nabla^{4-j}\Delta S\right| \left|\nabla^j\left(S^{\frac{\alpha-1}{\delta}}\right)\right| \\
				&\le C(\sigma_0) \sum_{0\le j\le 4} \langle y \rangle^{\frac{1-\alpha}{\delta}(\Lambda-1)-j(1-\eta)-\frac{K+j-15/2+K(1-\eta)(5-j)}{K-5/2}} \\
				&\le C(\sigma_0) \langle y \rangle^{-5(1-\eta)+2-\Lambda+f_1} \le C(\sigma_0) \langle y \rangle^{-3+5\eta},
			\end{aligned}
		\end{equation}
		which holds for any multi-index $\beta$ with $|\beta|=4$.
		Similarly, we have 
		\begin{equation}
			\begin{aligned}
				\left| \partial^\beta \left( (S)^{\frac{\alpha-\delta-1}{\delta}} |\nabla S|^2 \right) \right| &\le C \sum_{0 \le j \le 4} \left| \nabla^{4-j} (\nabla S) \right| \left| \nabla^j \left( (S)^{\frac{\alpha-\delta-1}{\delta}}\nabla S \right) \right| \\
				&\le C(\sigma_0) \sum_{0 \le j \le 4} \left| \nabla^{4-j} (\nabla S) \right| \phi^{-\frac{j+1}{2}} \langle y \rangle^{\frac{1-\alpha}{\delta}(\Lambda-1)} \\
				&\le C(\sigma_0) \sum_{0 \le j \le 4} \langle y \rangle^{\frac{1-\alpha}{\delta}(\Lambda-1) - (j+1)(1-\eta) - \frac{K+j-13/2+K(1-\eta)(4-j)}{K-5/2}} \\
				&\le C(\sigma_0) \langle y \rangle^{-5(1-\eta) + 2 - \Lambda + f_1} \le C(\sigma_0) \langle y \rangle^{-3 + 5\eta},
			\end{aligned}
		\end{equation}
		together with
		\begin{equation}\label{Eq_diffusion_total_pointwise_final}
			\left|\partial^\beta\left(S^{\frac{\alpha-1}{\delta}}\mathbb{L}(V)\right)\right|\le C(\sigma_0) \langle y \rangle^{-3+5\eta},
		\end{equation}
		and
		\begin{equation}\label{Eq_diffusion_strong_nonlinear_swapped_final_111}
			\begin{aligned}
				\left| \partial^\beta \left( S^{\frac{\alpha-1}{\delta}-1} \nabla S \cdot \mathbb{D}(V) \right) \right| &\le C \sum_{0 \le j \le 4} \left| \nabla^{4-j} (\mathbb{D}(V)) \right| \left| \nabla^j \left( S^{\frac{\alpha-1}{\delta}-1} \nabla S \right) \right| \\
				&\le C(\sigma_0) \sum_{0 \le j \le 4} \left| \nabla^{4-j} (\mathbb{D}(V)) \right| \phi^{-\frac{j+1}{2}} \langle y \rangle^{\frac{1-\alpha}{\delta}(\Lambda-1)} \\
				&\le C(\sigma_0) \sum_{0 \le j \le 4} \langle y \rangle^{\frac{1-\alpha}{\delta}(\Lambda-1) - (j+1)(1-\eta) - \frac{K+j-13/2+K(1-\eta)(4-j)}{K-5/2}} \\
				&\le C(\sigma_0) \langle y \rangle^{-5(1-\eta) + 2 - \Lambda + f_1} \le C(\sigma_0) \langle y \rangle^{-3 + 5\eta}.
			\end{aligned}
		\end{equation}
		By gathering the above-derived estimates together, we conclude that 
		\begin{equation}\label{HFbeta22VS}
			\begin{aligned}
				H_{F,\beta} &= 2 \int_{B^c(0,R_1)} \partial_\beta \widetilde{V} \cdot \partial_\beta \mathcal{F}_{v} \, \mathrm{d}y +2 \int_{B^c(0,R_1)} \partial_\beta \widetilde{S} \cdot \partial_\beta \mathcal{F}_{s} \, \mathrm{d}y\\
				&\le 2 \left( \int_{B^c(0,R_1)} \left( |\partial_\beta \widetilde{S}|^2 + |\partial_\beta \widetilde{V}|^2 \right) \mathrm{d}y \right)^{\frac{1}{2}}\cdot C(\sigma_0) e^{-f_1\tau} \left( \int_{R_1}^\infty \langle r \rangle^{-6+10\eta} r^2 \, \mathrm{d}r \right)^{\frac{1}{2}} \\
				&\le C(\sigma_0) e^{-f_1\tau} \left( \int_{B^c(0,R_1)} \left( |\partial_\beta \widetilde{S}|^2 + |\partial_\beta \widetilde{V}|^2 \right) \mathrm{d}y \right)^{\frac{1}{2}} \\
				&\le C(\sigma_0)\sigma_0^{-\frac{3}{2}} e^{-f_1\tau_0} \sigma_0^{\frac{3}{2}} e^{-f_1(\tau-\tau_0)} \left( \int_{B^c(0,R_1)} \left( |\partial_\beta \widetilde{S}|^2 + |\partial_\beta \widetilde{V}|^2 \right) \mathrm{d}y \right)^{\frac{1}{2}} \\
				&\le C(\sigma_0) e^{-f_1\tau_0} \sigma_1 e^{-\sigma_g(\tau-\tau_0)} \left( \int_{B^c(0,R_1)} \left( |\partial_\beta \widetilde{S}|^2 + |\partial_\beta \widetilde{V}|^2 \right) \mathrm{d}y \right)^{\frac{1}{2}} \\
				&\le 2\sigma_1 e^{-\sigma_g(\tau-\tau_0)} \left( \int_{B^c(0,R_1)} \left( |\partial_\beta \widetilde{S}|^2 + |\partial_\beta \widetilde{V}|^2 \right) \mathrm{d}y \right)^{\frac{1}{2}}.
			\end{aligned}
		\end{equation}
		Substituting \eqref{Eq_H_0_beta_identity_bar}, \eqref{Eq_H_1_beta_estimate}, \eqref{Eq_H_2_beta_final_estimate}, \eqref{Eq_H_3_beta_final_estimate}, \eqref{Eq_H_N_beta_final_estimate}, \eqref{HFbeta22VS} into \eqref{Eq_energy_identity_H_chain} and using \eqref{est_decay_3}, \eqref{est_decay_4}, we arrive at
		\begin{equation}
			\begin{aligned}
				H &\le \left( -2\Lambda - 3 + 3|\nabla(\widehat{\chi}\bar{V})| + |\nabla(\delta\widehat{\chi}\bar{S})| \right) \left( \sum_{|\beta|=4} \int_{B^c(0,R_1)} \left( |\partial_\beta \widetilde{V}|^2 + |\partial_\beta \widetilde{S}|^2 \right) \mathrm{d}y \right) \\
				&\quad + 81 \left( \frac{\sigma_1}{\sigma_g} \right)^{\frac{17}{10}} e^{-2\varpi(\tau-\tau_0)} + 1620 \left( \|\nabla(\widehat{\chi}\bar{V})\|_{L^\infty(B^c(0,R_1))} + \|\nabla(\widehat{\chi}\bar{S})\|_{L^\infty(B^c(0,R_1))} \right) \\
				&\quad \times \left( \|\nabla^4 \widetilde{V}\|_{L^2(B^c(0,R_1))}^2 + \|\nabla^4 \widetilde{S}\|_{L^2(B^c(0,R_1))}^2 \right) \\
				&\quad + \sum_{|\beta|=4} \left( \int_{B^c(0,R_1)} \left( |\partial_\beta \widetilde{V}|^2 + |\partial_\beta \widetilde{S}|^2 \right) \mathrm{d}y \right)^{\frac{1}{2}} \\
				&\quad \times \Big( C \left( \|\nabla^2(\widehat{\chi}\bar{V})\|_{L^8(B^c(0,R_1))} + \|\nabla^2(\widehat{\chi}\bar{S})\|_{L^8(B^c(0,R_1))} \right) \sigma_0 50^{-\frac{1}{4}} e^{-\varpi(\tau - \tau_0)} \\
				&\qquad + C \sum_{j=3}^{5} \left( \|\nabla^j(\widehat{\chi}\bar{V})\|_{L^2(B^c(0,R_1))} + \|\nabla^j(\widehat{\chi}\bar{S})\|_{L^2(B^c(0,R_1))} \right) \sigma_0 50^{-\frac{2}{5}} e^{-\varpi(\tau - \tau_0)} \\
				&\qquad + 2\sigma_0^{\frac{3}{2}} e^{-\frac{3}{2}\varpi(\tau-\tau_0)} + 4\sigma_1 e^{-(\tau-\tau_0)} + 2\sigma_1 e^{-\sigma_g(\tau-\tau_0)} \Big),
			\end{aligned}
		\end{equation}
		Therefore, we have
		\begin{equation}\label{Eq_energy_inequality_substitution_final}
			\begin{aligned}
				H &\le \left( -3 - 2\Lambda + 2000 \|(\nabla(\widehat{\chi}\bar{V}), \nabla(\widehat{\chi}\bar{S}))\|_{L^\infty(B^c(0,R_1))} \right) \\
				&\quad \times \left( \|\nabla^4 \widetilde{V}\|_{L^2(B^c(0,R_1))}^2 + \|\nabla^4 \widetilde{S}\|_{L^2(B^c(0,R_1))}^2 \right) + 81 \left( \frac{\sigma_1}{\sigma_g} \right)^{\frac{17}{10}} e^{-2\varpi(\tau-\tau_0)} \\
				&\quad + C_4 \left( \|\nabla^4 \widetilde{V}\|_{L^2(B^c(0,R_1))}^2 + \|\nabla^4 \widetilde{S}\|_{L^2(B^c(0,R_1))}^2 \right)^{\frac{1}{2}} \\
				&\quad \times \Big( \left( \|\nabla^2(\widehat{\chi}\bar{V})\|_{L^8(B^c(0,R_1))} + \|\nabla^2(\widehat{\chi}\bar{S})\|_{L^8(B^c(0,R_1))} \right) \sigma_0 50^{-\frac{1}{4}} e^{-\varpi(\tau - \tau_0)} \\
				&\qquad +\sum_{j=3}^{5} \left( \|\nabla^j(\widehat{\chi}\bar{V})\|_{L^2(B^c(0,R_1))} + \|\nabla^j(\widehat{\chi}\bar{S})\|_{L^2(B^c(0,R_1))} \right) \sigma_0 50^{-\frac{2}{5}} e^{-\varpi(\tau - \tau_0)} \\
				&\qquad + 2\sigma_0^{\frac{3}{2}} e^{-\frac{3}{2}\varpi(\tau-\tau_0)} + 4\sigma_1 e^{-(\tau-\tau_0)} + 2\sigma_1 e^{-\sigma_g(\tau-\tau_0)} \Big).
			\end{aligned}
		\end{equation}
		From \eqref{est_decay_2}-\eqref{est_decay_4}, we obtain
		\begin{equation}
			\begin{aligned}
				&\frac{\mathrm{d}}{\mathrm{d}\tau} \left( \|\nabla^4 \widetilde{S}\|_{L^2(B^c(0,R_1))}^2 + \|\nabla^4 \widetilde{V}\|_{L^2(B^c(0,R_1))}^2 \right) \\
				&\le -3\left( \|\nabla^4 \widetilde{S}\|_{L^2(B^c(0,R_1))}^2 + \|\nabla^4 \widetilde{V}\|_{L^2(B^c(0,R_1))}^2 \right)+81 \left( \frac{\sigma_1}{\sigma_g} \right)^{\frac{17}{10}} e^{-2\varpi(\tau-\tau_0)}\\
				&\quad+\frac{6}{5}\sigma_0e^{-\varpi(\tau-\tau_0)}50^{-\frac{1}{4}}\left( \|\nabla^4 \widetilde{S}\|_{L^2(B^c(0,R_1))}^2 + \|\nabla^4 \widetilde{V}\|_{L^2(B^c(0,R_1))}^2 \right)^{\frac{1}{2}}
			\end{aligned}
		\end{equation}
		Let
		\begin{equation}\label{Eq_X_tau_definition}
			\mathcal{Q}(\tau) = \left( \|\nabla^4 \widetilde V\|_{L^2(B^c(0,R_1))}^2 + \|\nabla^4 \widetilde S\|_{L^2(B^c(0,R_1))}^2 \right) e^{2\varpi(\tau - \tau_0)},
		\end{equation}
		From the closed energy differential inequality, we have
		\begin{equation}\label{Eq_dQ_dtau_chain_final}
			\begin{aligned}
				\frac{\mathrm{d}}{\mathrm{d}\tau} \mathcal{Q}(\tau) &\le (-3 + 2\varpi) \mathcal{Q}(\tau) + 81 \left( \frac{\sigma_1}{\sigma_g} \right)^{\frac{17}{10}} + \frac{6}{5}\sigma_0 50^{-\frac{1}{4}} \mathcal{Q}(\tau)^{\frac{1}{2}} \\
				&\le -\mathcal{Q}(\tau) + 81 \left( \frac{\sigma_1}{\sigma_g} \right)^{\frac{17}{10}} + \left( \frac{3}{5}\sigma_0 50^{-\frac{1}{4}} \right)^2.
			\end{aligned}
		\end{equation}
		By applying Grönwall's inequality to the above relation over the interval $[\tau_0, \tau]$, we deduce that
		\begin{equation}\label{Eq_Gronwall_Q_integration}
			\begin{aligned}
				\mathcal{Q}(\tau) &\le \mathcal{Q}(\tau_0) e^{-(\tau - \tau_0)} + \left[ 81 \left( \frac{\sigma_1}{\sigma_g} \right)^{\frac{17}{10}} + \left( \frac{3}{5}\sigma_0 50^{-\frac{1}{4}} \right)^2 \right] (1 - e^{-(\tau - \tau_0)}) \\
				&= e^{-(\tau - \tau_0)} \left( \mathcal{Q}(\tau_0) - 81 \left( \frac{\sigma_1}{\sigma_g} \right)^{\frac{17}{10}} - \left( \frac{3}{5}\sigma_0 50^{-\frac{1}{4}} \right)^2 \right) + 81 \left( \frac{\sigma_1}{\sigma_g} \right)^{\frac{17}{10}} + \left( \frac{3}{5}\sigma_0 50^{-\frac{1}{4}} \right)^2.
			\end{aligned}
		\end{equation}
		Notice that the initial localized energy at $\tau = \tau_0$ is defined and bounded by
		\begin{equation}\label{Eq_initial_Q_bound}
			\mathcal{Q}(\tau_0) = \|\nabla^4 \widetilde{V}_0\|_{L^2(B^c(0,R_1))}^2 + \|\nabla^4 \widetilde{S}_0\|_{L^2(B^c(0,R_1))}^2 \le C(K) \left( \sigma_1^{\frac{2K-8}{K-3/2}} E^{\frac{5}{2K-3}}+ \sigma_1^2 \right).
		\end{equation}
		Consequently, choosing $\sigma_0 \ll 1$ yields the smallness of the initial data component, satisfying
		\begin{equation}\label{Eq_final_uniform_smallness_Q}
			\mathcal{Q}(\tau_0) \le \sigma_1^{\frac{9}{5}} \le 81 \left( \frac{\sigma_1}{\sigma_g} \right)^{\frac{17}{10}} + \left( \frac{3}{5}\sigma_0 50^{-\frac{1}{4}} \right)^2 \le \frac{1}{4} \sigma_0^2.
		\end{equation}
		Substituting \eqref{Eq_final_uniform_smallness_Q} into \eqref{Eq_Gronwall_Q_integration}, we establish the desired property.
	\end{proof}
	Before deriving the weighted energy estimates, we first state the following commutator lemma, whose proof can be found in \cite[Lemma~3.13]{Cao-Labora-Gomez-Serrano-Shi-Staffilani-2023} and \cite[Lemma~6.6]{Chen-Zhang-Zhu}.
	\begin{lemma}\label{lemmacommunator}
		Let
		\[
		\beta=(\beta_1,\beta_2,\ldots,\beta_K),
		\qquad
		\beta_i\in\{1,2,3\},
		\]
		be an ordered multi-index of length \(K\). For \(1\le j\le K\), let
		\(\beta^{(j)}\) be the ordered multi-index obtained from \(\beta\) by deleting
		its \(j\)-th entry. Then, for every
		\(\tau\in[\tau_0,\tau_1]\), the following estimates hold:
		\begin{align}
			&\left\|
			\sum_{i=1}^3
			\left(
			\partial_\beta(V\cdot\nabla V_i)
			-
			V\cdot\nabla\partial_\beta V_i
			-
			\sum_{j=1}^K\sum_{k=1}^3
			L_{k\beta_j}\,
			\partial_{y_k}\partial_{\beta^{(j)}}V_i
			\right)
			\phi^{K/2}
			\right\|_{L^2}
			\nonumber\\
			&\qquad\le
			E^{\frac12-\frac{1}{4K}}
			+
			C\left\|
			\partial_\beta V\,\phi^{K/2}
			\right\|_{L^2},
			\label{Est:I21-new}\\[1mm]
			&\left\|
			\left(
			\partial_\beta(V\cdot\nabla S)
			-
			V\cdot\nabla\partial_\beta S
			-
			\sum_{j=1}^K\sum_{k=1}^3
			L_{k\beta_j}\,
			\partial_{y_k}\partial_{\beta^{(j)}}S
			\right)
			\phi^{K/2}
			\right\|_{L^2}
			\nonumber\\
			&\qquad\le
			E^{\frac12-\frac{1}{4K}}
			+
			C\left\|
			\partial_\beta S\,\phi^{K/2}
			\right\|_{L^2},
			\label{Est:I22-new}\\[1mm]
			&\left\|
			\sum_{i=1}^3
			\left(
			\partial_\beta(S\partial_{y_i}S)
			-
			S\partial_{y_i}\partial_\beta S
			-
			\sum_{j=1}^K
			\left[
			\partial_r(\widehat{\chi}\,\overline S)
			\frac{y_{\beta_j}}{|y|}
			\right]
			\partial_{y_i}\partial_{\beta^{(j)}}S
			\right)
			\phi^{K/2}
			\right\|_{L^2}
			\nonumber\\
			&\qquad\le
			E^{\frac12-\frac{1}{4K}}
			+
			C\left\|
			\partial_\beta S\,\phi^{K/2}
			\right\|_{L^2},
			\label{Est:I32-new}\\[1mm]
			&\left\|
			\left(
			\partial_\beta(S\operatorname{div}V)
			-
			S\partial_\beta\operatorname{div}V
			-
			\sum_{j=1}^K
			\left[
			\partial_r(\widehat{\chi}\,\overline S)
			\frac{y_{\beta_j}}{|y|}
			\right]
			\partial_{\beta^{(j)}}\operatorname{div}V
			\right)
			\phi^{K/2}
			\right\|_{L^2}
			\nonumber\\
			&\qquad\le
			E^{\frac12-\frac{1}{4K}}
			+
			C\left\|
			\partial_\beta V\,\phi^{K/2}
			\right\|_{L^2}.
			\label{Est:I31-new}
		\end{align}
		Here
		\[
		L_{k\beta_j}
		:=
		\left(
		\partial_r(\widehat{\chi}\,\overline{\mathcal V})
		-
		\frac{\widehat{\chi}\,\overline{\mathcal V}}{|y|}
		\right)
		\frac{y_k y_{\beta_j}}{|y|^2}
		+
		\delta_{k, \beta_j}
		\frac{\widehat{\chi}\,\overline{\mathcal V}}{|y|},
		\]
		where
		\[
		\overline V(y)
		=
		\overline{\mathcal V}(|y|)\frac{y}{|y|}
		\]
		and \(\delta_{k,\beta_j}\) denotes the Kronecker delta.
	\end{lemma}
	\begin{proposition}
		Under the bootstrap hypotheses \eqref{Puns}-\eqref{E1}, we can prove \eqref{E2}:
		\begin{equation}
			E_K = \int (|\nabla^K S|^2 + |\nabla^K V|^2) \phi^K \, \mathrm{d}y \le \frac{E}{2}.
		\end{equation}
	\end{proposition}
	\begin{proof}
		Applying the differential operator $\partial_\beta$ with $|\beta|=K$ to \eqref{Eq_self_similar_S}, and utilizing the identity $\partial_\beta (y \cdot \nabla f) = K \partial_\beta f + y \cdot \nabla \partial_\beta f$, we obtain
		\begin{equation}\label{Eq_coupled_higher_derivative_system}
			\left\{
			\begin{aligned}
				&(\partial_\tau + \Lambda - 1 + K)\partial_\beta S + y \cdot \nabla \partial_\beta S + \partial_\beta(V \cdot \nabla S) + \delta \partial_\beta(S \operatorname{div}V) = \partial_\beta \mathcal{F}_{s}, \\
				&(\partial_\tau + \Lambda - 1 + K)\partial_\beta V + y \cdot \nabla \partial_\beta V + \partial_\beta(V \cdot \nabla V) + \delta \partial_\beta(S \nabla S) = \partial_\beta \mathcal{F}_{v}.
			\end{aligned}
			\right.
		\end{equation}
		To achieve the high-order energy identity, we multiply the system \eqref{Eq_coupled_higher_derivative_system} by $\phi^K \partial_\beta S$ and $\phi^K \partial_\beta V$ respectively, perform a summation over all entries $|\beta|=K$, and integrate by parts
		\begin{equation}\label{Eq_rewrite_full_energy_integral_system}
			\begin{aligned}
				&\left(\frac{1}{2} \frac{\mathrm{d}}{\mathrm{d}\tau} + \Lambda - 1 + K\right) E_K \\
				&= -\sum_{|\beta|=K} \int_{\mathbb{R}^3} \phi^K y \cdot \left( \nabla \partial_\beta S \partial_\beta S+ \sum_{i=1}^3 \nabla\partial_\beta V_i \partial_\beta V_i  \right) \mathrm{d}y \\
				&\quad - \sum_{|\beta|=K} \int_{\mathbb{R}^3} \phi^K \left( \partial_\beta S \partial_\beta(V \cdot \nabla S) + \sum_{i=1}^3 \partial_\beta V_i \partial_\beta(V \cdot \nabla V_i) \right) \mathrm{d}y \\
				&\quad - \sum_{|\beta|=K} \int_{\mathbb{R}^3} \delta \phi^K \left( \partial_\beta S \partial_\beta(S \operatorname{div}V) + \sum_{i=1}^3 \partial_\beta V_i \partial_\beta(S \partial_{y_i} S) \right) \mathrm{d}y \\
				&\quad + \sum_{|\beta|=K} \int_{\mathbb{R}^3} \phi^K \left( \partial_\beta S \partial_\beta\mathcal{F}_{s}+ \partial_\beta V \partial_\beta\mathcal{F}_{v} \right)\mathrm{d}y\\
				&=: -K_1 - K_2 - K_3 + K_4.
			\end{aligned}
		\end{equation}
		We now estimate each term separately and then combine the resulting bounds. The estimates for $K_1$--$K_3$ follow the argument in
		\cite[Section~3.4]{Cao-Labora-Gomez-Serrano-Shi-Staffilani-2023},
		and we include the corresponding calculations for completeness.
		The term $K_4$ is the principal term that requires special attention
		in the present setting, and its estimate is derived in detail below.\\
		\noindent\textbf{1. Estimates for $K_1$.}\\
		Integrating by parts and using \eqref{E1}, we arrive at
		\begin{equation}\label{Eq_K1_integration_by_parts}
			\begin{aligned}
				-K_1
				&=
				\frac{1}{2}
				\sum_{|\beta|=K}
				\int_{\mathbb{R}^3}
				\operatorname{div}\!\left(\phi^K y\right)
				\left(
				|\partial_\beta S|^2+|\partial_\beta V|^2
				\right)\,\mathrm{d}y\\
				&=
				\frac{3}{2}
				\sum_{|\beta|=K}
				\int_{\mathbb{R}^3}
				\left(
				|\partial_\beta S|^2+|\partial_\beta V|^2
				\right)\phi^K\,\mathrm{d}y\\
				&\quad+
				\frac{K}{2}
				\sum_{|\beta|=K}
				\int_{\mathbb{R}^3}
				\frac{y\cdot\nabla\phi}{\phi}
				\left(
				|\partial_\beta S|^2+|\partial_\beta V|^2
				\right)\phi^K\,\mathrm{d}y\\
				&\leq
				\frac{3}{2}E
				+
				\frac{K}{2}
				\sum_{|\beta|=K}
				\int_{\mathbb{R}^3}
				\frac{y\cdot\nabla\phi}{\phi}
				\left(
				|\partial_\beta S|^2+|\partial_\beta V|^2
				\right)\phi^K\,\mathrm{d}y.
			\end{aligned}
		\end{equation}
		\noindent\textbf{2. Estimates for $K_2$.}\\
		By using the commutator estimate \eqref{Est:I21-new}, the remaining non-linear terms can be bounded by
		\begin{equation}
			\begin{aligned}
				&-\sum_{|\beta|=K} \int_{\mathbb{R}^3} \phi^K \sum_{i=1}^3 \partial_\beta(V \cdot \nabla V_i) \partial_\beta V_i \, \mathrm{d}y \\
				&= -\sum_{|\beta|=K} \int_{\mathbb{R}^3} \phi^K \sum_{i=1}^3 \left( \partial_\beta(V \cdot \nabla V_i) - V \cdot \nabla \partial_\beta V_i - \sum_{j=1}^K \sum_{k=1}^3 \partial_{y_{\beta_j}}(\widehat{\chi}\bar{\mathcal{V}}\frac{y_k}{|y|})\partial_{y_k} \partial_{\beta^{(j)}} V_i \right) \partial_\beta V_i \, \mathrm{d}y \\
				&\quad - \sum_{|\beta|=K} \int_{\mathbb{R}^3} \phi^K \sum_{i=1}^3 \left( V \cdot \nabla \partial_\beta V_i + \sum_{j=1}^K \sum_{k=1}^3 \partial_{y_{\beta_j}}(\widehat{\chi}\bar{\mathcal{V}}\frac{y_k}{|y|}) \partial_{y_k} \partial_{\beta^{(j)}} V_i \right) \partial_\beta V_i \, \mathrm{d}y \\
				&\le \sum_{|\beta|=K} \left( E^{\frac{1}{2}-\frac{1}{4K}} + C\|\partial_\beta V \phi^{\frac{K}{2}}\|_{L^2} \right) \|\partial_\beta V \phi^{\frac{K}{2}}\|_{L^2} - \sum_{|\beta|=K} \int_{\mathbb{R}^3} \phi^K \sum_{i=1}^3 V \cdot \nabla \partial_\beta V_i \partial_\beta V_i \, \mathrm{d}y \\
				&\quad - \sum_{j=1}^K \sum_{|\beta^{(j)}|=K-1} \sum_{\beta_j=1}^3 \sum_{k=1}^3 \int_{\mathbb{R}^3} \phi^K \sum_{i=1}^3 \left( \partial_{y_{\beta_j}}(\widehat{\chi}\bar{\mathcal{V}}\frac{y_k}{|y|}) \partial_{y_k} \partial_{\beta^{(j)}} V_i \partial_{y_{\beta_j}} \partial_{\beta^{(j)}} V_i \right) \mathrm{d}y. \end{aligned}
		\end{equation}
		Consequently, we have
		\begin{equation}\label{Eq_nonlinear_transport_commutator_estimate}
			\begin{aligned}
				&-\sum_{|\beta|=K}
				\int_{\mathbb{R}^3}
				\phi^K\,
				\partial_\beta(V\cdot\nabla V)
				\cdot\partial_\beta V\,\mathrm{d}y
				\\
				&\le
				\frac{3^K}{2}E^{1-\frac{1}{2K}}
				+
				C\sum_{|\beta|=K}
				\left\|
				\phi^{\frac K2}\partial_\beta V
				\right\|_{L^2}^2+
				\frac12\sum_{|\beta|=K}
				\int_{\mathbb{R}^3}
				\left(
				\operatorname{div}V
				+
				K\frac{V\cdot\nabla\phi}{\phi}
				\right)
				|\partial_\beta V|^2\phi^K\,\mathrm{d}y
				\\
				&\quad-
				\sum_{j=1}^K
				\sum_{|\beta^{(j)}|=K-1}
				\sum_{i=1}^3
				\int_{\mathbb{R}^3}
				\phi^K
				\left(
				\partial_r(\widehat{\chi}\bar{\mathcal V})
				-
				\frac{\widehat{\chi}\bar{\mathcal V}}{|y|}
				\right)
				\left|
				\frac{y}{|y|}\cdot
				\nabla\partial_{\beta^{(j)}}V_i
				\right|^2
				\,\mathrm{d}y
				\\
				&\quad-
				\sum_{j=1}^K
				\sum_{|\beta^{(j)}|=K-1}
				\sum_{i=1}^3
				\int_{\mathbb{R}^3}
				\phi^K
				\frac{\widehat{\chi}\bar{\mathcal V}}{|y|}
				\left|
				\nabla\partial_{\beta^{(j)}}V_i
				\right|^2
				\,\mathrm{d}y .
			\end{aligned}
		\end{equation}
		Here we use the identity 
		\begin{equation}
			\sum_{k=1}^3 \sum_{\beta_j=1}^3 \frac{y_k y_{\beta_j}}{|y|^2} \partial_{y_k} \partial_{\beta^{(j)}} V_i \partial_{y_{\beta_j}} \partial_{\beta^{(j)}} V_i=\left|\frac{y}{|y|}\cdot
			\nabla\partial_{\beta^{(j)}}V_i\right|^2,
		\end{equation}
		For $y\in\mathbb{R}^3\setminus\{0\}$, let
		\[
		r:=|y|,
		\qquad
		\omega:=\frac{y}{|y|}\in\mathbb S^2.
		\]
		We define the radial and angular derivatives, respectively, by
		\begin{equation}\label{eq:radial-angular-derivative-definition}
			\partial_r
			:=
			\omega\cdot\nabla_y
			=
			\sum_{k=1}^3\frac{y_k}{|y|}\partial_{y_k},
			\qquad
			\frac{1}{r}\nabla_\theta
			:=
			\nabla_y-\omega\partial_r.
		\end{equation}
		Therefore, we arrive at
		\begin{equation}\label{Eq_nonlinear_transport_commutator_estimate_finally}
			\begin{aligned}
				&-\sum_{|\beta|=K} \int_{\mathbb{R}^3} \phi^K \sum_{i=1}^3 \partial_\beta(V \cdot \nabla V_i) \partial_\beta V_i \, \mathrm{d}y \\
				&\le C E + \sum_{|\beta|=K} \int_{\mathbb{R}^3} \frac{K \nabla \phi \cdot V + \phi \operatorname{div}V}{2\phi} |\partial_\beta V|^2 \phi^K \, \mathrm{d}y \\
				&\quad - K \sum_{|\bar{\beta}|=K-1} \int_{\mathbb{R}^3} \phi^K \left( \partial_r (\widehat{\chi}\bar{\mathcal{V}}) - \frac{\widehat{\chi}\bar{\mathcal{V}}}{r} \right) \sum_{i=1}^3 |\partial_r \partial_{\bar{\beta}} V_i|^2 \, \mathrm{d}y \\
				&\quad - K \sum_{|\bar{\beta}|=K-1}\sum_{i=1}^3 \sum_{k=1}^3 \int_{\mathbb{R}^3} \phi^K \frac{\widehat{\chi}\bar{\mathcal{V}}}{r} |\partial_{y_k} \partial_{\bar{\beta}} V_i|^2 \, \mathrm{d}y \\
				&\le C E + \sum_{|\beta|=K} \int_{\mathbb{R}^3} \left( \frac{K \nabla \phi \cdot \widetilde{V}}{2\phi} + \frac{1}{2}\operatorname{div}V \right) |\partial_\beta V|^2 \phi^K \, \mathrm{d}y \\
				&\quad + K \sum_{|\bar{\beta}|=K-1} \int_{\mathbb{R}^3} \left( \frac{\nabla \phi \cdot \widehat{\chi}\bar{V}}{2\phi} |\nabla \partial_{\bar{\beta}} V|^2 - \partial_r (\widehat{\chi}\bar{\mathcal{V}}) |\partial_r \partial_{\bar{\beta}} V|^2 - \frac{\widehat{\chi}\bar{\mathcal V}}{r^3} |\nabla_\theta \partial_{\bar{\beta}} V|^2 \right) \phi^K \, \mathrm{d}y \\
				&\le C E + K \sum_{|\bar{\beta}|=K-1} \int_{\mathbb{R}^3} \left( \frac{\nabla \phi \cdot \widehat{\chi}\bar{V}}{2\phi} |\nabla \partial_{\bar{\beta}} V|^2 - \partial_r (\widehat{\chi}\bar{\mathcal{V}}) |\partial_r \partial_{\bar{\beta}} V|^2 - \frac{\widehat{\chi}\bar{\mathcal V}}{r^3} |\nabla_\theta \partial_{\bar{\beta}} V|^2 \right) \phi^K \, \mathrm{d}y.
			\end{aligned}
		\end{equation}
		For the final inequality, we have used following facts
		$$ \|\operatorname{div} V\|_{L^\infty} \le C, \quad \left\| \frac{\nabla \phi}{2\phi} \widetilde{V} \right\|_{L^\infty} \le C \sigma_0, \quad \sigma_0\ll K^{-1}.$$
		Similarly, we have 
		\begin{equation}\label{SimilarlyestimateSs}
			\begin{aligned}
				&-\sum_{|{\beta}|=K} \int_{\mathbb{R}^3} \phi^K \partial_{{\beta}} S \partial_{\beta}(V \cdot \nabla S) \, \mathrm{d}y \\
				&\le C E +  K \sum_{|\bar{\beta}|=K-1} \int_{\mathbb{R}^3} \left( \frac{\nabla \phi \cdot \widehat{\chi}\bar{V}}{2\phi} |\nabla\partial_{\bar{\beta}} S|^2 - \partial_r (\widehat{\chi}\bar{\mathcal{V}}) |\partial_r \partial_{\bar{\beta}} S|^2 - \frac{\widehat{\chi}\bar{\mathcal V}}{r^3} |\nabla_\theta \partial_{\bar{\beta}} S|^2 \right) \phi^K \, \mathrm{d}y.
			\end{aligned}
		\end{equation}
		Combining \eqref{Eq_nonlinear_transport_commutator_estimate_finally} with \eqref{SimilarlyestimateSs}, we obtain
		\begin{equation}\label{Eq_image_0bc0e4}
			\begin{aligned}
				-K_2 &\le C E + K \int_{\mathbb{R}^3} \left( \frac{\nabla \phi \cdot \widehat{\chi}\bar{V}}{2\phi} - \partial_r (\widehat{\chi}\bar{\mathcal{V}}) \right) \left( |\partial_r \nabla^{K-1} V|^2 + |\partial_r \nabla^{K-1} S|^2 \right) \phi^K \, \mathrm{d}y \\
				&\quad + K \int_{\mathbb{R}^3} \left( \frac{\nabla \phi \cdot \widehat{\chi}\bar{V}}{2\phi r^2} - \frac{\widehat{\chi}\bar{\mathcal{V}}}{r^3} \right) \left( |\nabla_\theta \nabla^{K-1} V|^2 + |\nabla_\theta \nabla^{K-1} S|^2 \right) \phi^K \, \mathrm{d}y.
			\end{aligned}
		\end{equation}
		\noindent\textbf{3. Estimates for $K_3$.}\\
		Using the commutator estimates \eqref{Est:I32-new}, one can get
		\begin{equation}\label{estimate_K3_1VSS}
			\begin{aligned}
				&-\sum_{|\beta|=K}
				\int_{\mathbb R^3}
				\phi^K\,
				\partial_\beta V\cdot
				\partial_\beta(S\nabla S)\,\mathrm dy
				\\
				&=
				-\sum_{|\beta|=K}
				\int_{\mathbb R^3}
				\phi^K\,\partial_\beta V\cdot
				\Bigg[
				\partial_\beta(S\nabla S)
				-S\nabla\partial_\beta S
				-\sum_{j=1}^K
				\partial_r(\widehat\chi\bar S)
				\frac{y_{\beta_j}}{r}
				\nabla\partial_{\beta^{(j)}}S
				\Bigg]\,\mathrm dy
				\\
				&\quad
				-\sum_{|\beta|=K}
				\int_{\mathbb R^3}
				\phi^K S\nabla\partial_\beta S
				\cdot\partial_\beta V\,\mathrm dy
				\\
				&\quad
				-\sum_{j=1}^K
				\sum_{|\beta^{(j)}|=K-1}
				\sum_{\beta_j=1}^3\sum_{i=1}^3
				\int_{\mathbb R^3}
				\phi^K\,
				\partial_r(\widehat\chi\bar S)
				\frac{y_{\beta_j}}{r}
				\partial_{y_i}\partial_{\beta^{(j)}}S\,
				\partial_{y_{\beta_j}}\partial_{\beta^{(j)}}V_i
				\,\mathrm dy
				\\
				&\leq
				\sum_{|\beta|=K}
				\left(
				E^{\frac12-\frac{1}{4K}}
				+C\|\phi^{\frac K2}\partial_\beta S\|_{L^2}
				\right)
				\|\phi^{\frac K2}\partial_\beta V\|_{L^2}
				-\sum_{|\beta|=K}
				\int_{\mathbb R^3}
				\phi^K S\nabla\partial_\beta S
				\cdot\partial_\beta V\,\mathrm dy
				\\
				&\quad
				-\sum_{j=1}^K
				\sum_{|\beta^{(j)}|=K-1}
				\sum_{\beta_j=1}^3\sum_{i=1}^3
				\int_{\mathbb R^3}
				\phi^K\,
				\partial_r(\widehat\chi\bar S)
				\frac{y_{\beta_j}}{r}
				\partial_{y_i}\partial_{\beta^{(j)}}S\,
				\partial_{y_{\beta_j}}\partial_{\beta^{(j)}}V_i
				\,\mathrm dy
				\\
				&\leq
				\frac{3^K}{2}E^{1-\frac{1}{2K}}+CE
				-\sum_{|\beta|=K}
				\int_{\mathbb R^3}
				\phi^K S\nabla\partial_\beta S
				\cdot\partial_\beta V\,\mathrm dy
				\\
				&\quad
				+\frac{K}{2}
				\sum_{|\bar\beta|=K-1}\sum_{i=1}^3
				\int_{\mathbb R^3}
				\phi^K
				\left|\partial_r(\widehat\chi\bar S)\right|
				\left(
				|\nabla\partial_{\bar\beta}V_i|^2
				+
				|\partial_{y_i}\partial_{\bar\beta}S|^2
				\right)\,\mathrm dy
				\\
				&\leq
				CE
				-\sum_{|\beta|=K}
				\int_{\mathbb R^3}
				\phi^K S\nabla\partial_\beta S
				\cdot\partial_\beta V\,\mathrm dy
				+\frac{K}{2}
				\sum_{|\beta|=K}
				\int_{\mathbb R^3}
				\phi^K
				\left|\partial_r(\widehat\chi\bar S)\right|
				\left(
				|\partial_\beta S|^2+|\partial_\beta V|^2
				\right)\,\mathrm dy .
			\end{aligned}
		\end{equation}
		
		Similarly, we can use the commutator estimates \eqref{Est:I31-new} to obtain
		\begin{equation}\label{estimate_K3_1SSV123}
			\begin{aligned}
				&-\sum_{|\beta|=K} \int_{\mathbb{R}^3} \phi^K \sum_{i=1}^3 \partial_{{\beta}} S\partial_{\beta}(S \partial_{y_i} V_i) \, \mathrm{d}y \\
				&\le C E -\sum_{|{\beta}|=K} \sum_{i=1}^3\int_{\mathbb{R}^3} \phi^K S \partial_{{\beta}} S \cdot \partial_{y_i} \partial_{{\beta}} V_i \, \mathrm{d}y + \frac{K}{2} \sum_{|\beta|=K} \int_{\mathbb{R}^3} \phi^K |\partial_r(\widehat{\chi}\bar{S})| \left( | \partial_{{\beta}} S|^2 + | \partial_{{\beta}} V|^2 \right) \mathrm{d}y.
			\end{aligned}
		\end{equation}
		Combining \eqref{estimate_K3_1VSS} and \eqref{estimate_K3_1SSV123}, integrating by parts and using Cauchy-Schwarz inequality, we have
		\begin{equation}\label{Eq_K3_nonlinear_coupling_bound_chain}
			\begin{aligned}
				-K_3 &\le C E + \delta \sum_{|\beta|=K} \int_{\mathbb{R}^3} \frac{\nabla(\phi^K S)}{\phi^K} \cdot \left( \phi^K \partial_{\beta} S \partial_{\beta} V \right) \mathrm{d}y \\
				&\quad + \delta K \sum_{|\beta|=K} \int_{\mathbb{R}^3} \phi^K |\partial_r (\widehat{\chi}\bar{S})| \left( |\partial_{\beta} S|^2 + |\partial_{\beta} V|^2 \right) \mathrm{d}y \\
				&\le C E + \delta \sum_{|\beta|=K} \int_{\mathbb{R}^3} \left( \frac{K \nabla \phi \widetilde{S}}{\phi} + \frac{K \nabla \phi \widehat{\chi}\bar{S}}{\phi} + \nabla S \right) \cdot \left( \phi^K \partial_{\beta} S \partial_{\beta} V \right) \mathrm{d}y \\
				&\quad + \delta K \sum_{|\beta|=K} \int_{\mathbb{R}^3} \phi^K |\partial_r (\widehat{\chi}\bar{S})| \left( |\partial_{\beta} S|^2 + |\partial_{\beta} V|^2 \right) \mathrm{d}y \\
				&\le C E + K \delta \sum_{|\beta|=K} \int_{\mathbb{R}^3} \phi^K \left( \frac{|\nabla \phi| \widehat{\chi}\bar{S}}{2\phi} + |\partial_r (\widehat{\chi}\bar{S})| \right) \left( |\partial_{\beta} S|^2 + |\partial_{\beta} V|^2 \right) \mathrm{d}y.
			\end{aligned}
		\end{equation}
		Here we have used the following estimates
		$$ \|\nabla V\|_{L^\infty} \le C, \quad \left\| \frac{\nabla \phi}{2\phi} \widetilde{S} \right\|_{L^\infty} \le C \sigma_0, \quad \sigma_0\ll K^{-1}.$$
		\noindent\textbf{4. Estimates for $K_4$.}\\
		We denote
		\begin{equation}\label{K4_initial_boundestimate}
			e^{f_1 \tau}K_4=e^{f_1 \tau}\sum_{|\beta|=K} \int_{\mathbb{R}^3} \phi^K \partial_\beta V \partial_\beta\mathcal{F}_{v} \mathrm{d}y+e^{f_1 \tau}\sum_{|\beta|=K} \int_{\mathbb{R}^3} \phi^K  \partial_\beta S \partial_\beta\mathcal{F}_{s} \mathrm{d}y:=e^{f_1 \tau}(K_{4,v}+K_{4,s}).
		\end{equation}
		By integrating by parts and rewriting it in commutator form, we obtain
		\begin{equation}
			\begin{aligned}
				&\delta^{\frac{1-\alpha}{\delta}}K_{4,v}\\
				&=\sum_{|\beta|=K}\big( \int_{\mathbb{R}^3} \phi^K \left[ \partial_\beta, S^{\frac{\alpha-1}{\delta}} \right] \mathbb{L}(V) \cdot \partial_\beta V \, \mathrm{d}y \\
				&\quad- \int_{\mathbb{R}^3} \left( \partial_i\left(\phi^K S^{\frac{\alpha-1}{\delta}}\right) \cdot \partial_\beta (\nu \partial_i V_j + ( \sqrt{\nu^2 - \varepsilon^2} + (\alpha - 1) ( \nu + \sqrt{\nu^2 - \varepsilon^2} ) ) \partial_j
				V_i) \right) \cdot \partial_\beta V_j \, \mathrm{d}y \\
				&\quad - \int_{\mathbb{R}^3} \phi^K S^{\frac{\alpha-1}{\delta}} \left( \nu | \partial_\beta\partial_i V_j|^2 +  ( \sqrt{\nu^2 - \varepsilon^2} + (\alpha - 1) ( \nu + \sqrt{\nu^2 - \varepsilon^2} ) )\partial_\beta\partial_i V_j\partial_\beta\partial_j V_i \right) \mathrm{d}y \\
				&\quad + \frac{\alpha}{\delta} \int_{\mathbb{R}^3} \phi^K \left[ \partial_\beta, S^{\frac{\alpha-1}{\delta}-1} \right] (\nabla S \cdot \mathbb{D}(V)) \cdot \partial_\beta V \, \mathrm{d}y \\
				&\quad + \frac{\alpha}{\delta} \int_{\mathbb{R}^3} \phi^K S^{\frac{\alpha-1}{\delta}-1} (\partial_\beta(\nabla S) \cdot \mathbb{D}(V)) \cdot \partial_\beta V \, \mathrm{d}y \\
				&\quad + \frac{\alpha}{\delta} \int_{\mathbb{R}^3} \phi^K S^{\frac{\alpha-1}{\delta}-1} (\partial_\beta(\nabla S \cdot \mathbb{D}(V)) - \partial_\beta(\nabla S) \cdot \mathbb{D}(V)) \cdot \partial_\beta V \, \mathrm{d}y\big) \\
				&:= I_{v,1} - I_{v,2} - I_{v,3} + I_{v,4} + I_{v,5} + I_{v,6}.
			\end{aligned}
		\end{equation}
		and
		\begin{equation}
			\begin{aligned}
				\frac{\delta^{\frac{1-\alpha}{\delta}}}{d\alpha} K_{4,s} &=\sum_{|\beta|=K}\big( \int_{\mathbb{R}^3} \phi^K \left[ \partial_\beta, S^{\frac{\alpha-1}{\delta}} \right] \Delta S \cdot \partial_\beta S \, \mathrm{d}y \\
				&\quad- \int_{\mathbb{R}^3} \left( \nabla\left(\phi^K S^{\frac{\alpha-1}{\delta}}\right) \cdot \nabla \partial_\beta S \right) \cdot \partial_\beta S \, \mathrm{d}y - \int_{\mathbb{R}^3} \phi^K S^{\frac{\alpha-1}{\delta}} |\nabla \partial_{\beta}S|^2 \mathrm{d}y \\
				&\quad + \frac{\alpha-\delta}{\delta} \int_{\mathbb{R}^3} \phi^K \left[ \partial_\beta, S^{\frac{\alpha-1}{\delta}-1} \right] (|\nabla S|^2) \cdot \partial_\beta S \, \mathrm{d}y \\
				&\quad + \frac{\alpha-\delta}{\delta} \int_{\mathbb{R}^3} \phi^K S^{\frac{\alpha-1}{\delta}-1} (\partial_\beta(\nabla S) \cdot\nabla S) \cdot \partial_\beta S \, \mathrm{d}y \\
				&\quad + \frac{\alpha-\delta}{\delta} \int_{\mathbb{R}^3} \phi^K S^{\frac{\alpha-1}{\delta}-1} (\partial_\beta(|\nabla S|^2) - \partial_\beta(\nabla S) \cdot \nabla S) \cdot \partial_\beta S \, \mathrm{d}y \big)\\
				&:= I_{s,1} - I_{s,2} - I_{s,3} + I_{s,4} + I_{s,5} + I_{s,6}.
			\end{aligned}
		\end{equation}
		
		We first estimate $K_{4,s}$. \textbf{For $I_{s,3}$, we treat it as a dissipative term.}
		\textbf{As for $I_{s,1}$, owing to the following fact}
		\begin{equation}\label{Eq_commutator_estimate_S}
			\begin{aligned}
				&\left| \left[ \partial_\beta, S^{\frac{\alpha-1}{\delta}} \right] \Delta S - \sum_{j=1}^K \partial_{y_{\beta_j}} S^{\frac{\alpha-1}{\delta}} \partial_{\beta^{(j)}} \Delta S \right| \\
				&\le C(K)3^K \max_{\substack{\sum_{j=0}^\ell |\bar{\beta}^j| = K \\ 0 \le |\bar{\beta}^0| \le K-2}} \left| S^{\frac{\alpha-1}{\delta}} \partial_{\bar{\beta}^0} \Delta S \frac{\partial_{\bar{\beta}^1} S}{S} \frac{\partial_{\bar{\beta}^2} S}{S} \cdots \frac{\partial_{\bar{\beta}^\ell} S}{S} \right|,
			\end{aligned}
		\end{equation}
		where $|\bar{\beta}^j| \ge 1$ for all $1 \le j \le \ell$. Therefore, we can decompose $I_{s,1}$ as
		\begin{equation}\label{Eq_Iv1_decomposition}
			\begin{aligned}
				I_{s,1} &\le C \sum_{|\beta|=K}\sum_{j=1}^K \int_{\mathbb{R}^3} \left| \phi^K \partial_\beta S \cdot \partial_{\beta^{(j)}} \Delta S  S^{\frac{\alpha-1}{\delta}-1} \partial_{y_{\beta_j}} S \right| \mathrm{d}y \\
				&\quad + \sum_{|\beta|=K}C(K)3^K E^{\frac{1}{2}} \max_{\substack{\sum_{j=0}^\ell |\bar{\beta}^j| = K \\ 0 \le |\bar{\beta}^0| \le K-2}} \left\| S^{\frac{\alpha-1}{\delta}} \partial_{\bar{\beta}^0} \Delta S \frac{\partial_{\bar{\beta}^1} S}{S} \frac{\partial_{\bar{\beta}^2} S}{S} \cdots \frac{\partial_{\bar{\beta}^\ell} S}{S} \phi^{\frac{K}{2}} \right\|_{L^2} \\
				&=: I_{s,1,1} + I_{s,1,2}.
			\end{aligned}
		\end{equation}
		For $I_{s,1,1}$, using the fact $$\frac{|\nabla S|^2}{S^2} S^{\frac{\alpha-1}{\delta}} \le C(\sigma_0) \left\langle \frac{y}{R_0} \right\rangle^{-2+\frac{1-\alpha}{\delta}(\Lambda-1)} \le C(\sigma_0),$$ we have
		\begin{equation}\label{Eq_Iv11_Holder_estimate}
			\begin{aligned}
				I_{s,1,1} &\le C(K)\sum_{|\beta|=K} \int_{\mathbb{R}^3} \phi^K |\partial_\beta S| |\nabla^{K+1} S| \frac{|\nabla S|}{S} S^{\frac{\alpha-1}{\delta}} \, \mathrm{d}y \\
				&\le C(K)\sum_{|\beta|=K} I_{s,3}^{\frac{1}{2}} \left( \int_{\mathbb{R}^3} \phi^K |\partial_\beta S|^2 \frac{|\nabla S|^2}{S^2} S^{\frac{\alpha-1}{\delta}} \, \mathrm{d}y \right)^{\frac{1}{2}} \\
				&\le C(K) I_{s,3}^{\frac{1}{2}} E^{\frac{1}{2}} \left\| \frac{|\nabla S|^2}{S^2} S^{\frac{\alpha-1}{\delta}} \right\|_{L^\infty}^{\frac{1}{2}} \le C(\sigma_0) I_{s,3}^{\frac{1}{2}}.
			\end{aligned}
		\end{equation}
		For $I_{s,1,2}$, we divide our discussion into the following three cases:
		\begin{itemize}
			\item \textbf{Case 1: $0 \le |\bar{\beta}^0| \le K-4$ and $1 \le |\bar{\beta}^i|\le K-2$ for all $i \ge1$.} \\
			In this case, using the weighted interpolation estimates \eqref{eq:lessKKKK-2-perturbation}, we have
			\begin{equation}\label{Eq_Iv12_weight_decay_chain}
				\begin{aligned}
					&\left\| S^{\frac{\alpha-1}{\delta}} \partial_{\bar{\beta}^0} \Delta S \frac{\partial_{\bar{\beta}^1} S}{S} \frac{\partial_{\bar{\beta}^2} S}{S} \cdots \frac{\partial_{\bar{\beta}^\ell} S}{S} \phi^{\frac{K}{2}} \right\|_{L^2} \\
					&\le C(\sigma_0) \left\| \phi^{-1} \langle y \rangle^{\left(\frac{1-\alpha}{\delta}-1\right)(\Lambda-1)} \langle y \rangle^{\frac{(K+2)\Lambda - (\ell+1)\Lambda - (\ell+1)K + \frac{5}{2}(\ell+1) + (1-\eta)(K(\ell+1)-\frac{5}{2}K-5)}{K-5/2}} \right\|_{L^2} \\
					&= C(\sigma_0) \left\| \phi^{-1} \langle y \rangle^{\left(\frac{1-\alpha}{\delta}-1\right)(\Lambda-1)} \langle y \rangle^{\frac{\Lambda(K+1-\ell) - \frac{5}{2}(1-\eta)(K+2) - K\eta(\ell+1) + \frac{5}{2}(\ell+1)}{K-5/2}} \right\|_{L^2}.
				\end{aligned}
			\end{equation}
			For $\ell \ge 1$, since we have taken $\eta$ sufficiently small and $K$ sufficiently large with respect to $\eta$ yields
			\begin{equation}\label{Eq_index_inequality_case2}
				\Lambda(K+1-\ell) - \frac{5}{2}(1-\eta)(K+2) - K\eta(\ell+1) + \frac{5}{2}(\ell+1) \le \left(\Lambda - \frac{5}{2} + \eta\right)K.
			\end{equation}
			Using  $-1 + f_1 < 0$, and $\Lambda > 1$, we deduce
			\begin{equation}\label{Eq_L2_radial_integration_chain}
				\begin{aligned}
					&\left\| S^{\frac{\alpha-1}{\delta}} \partial_{\bar{\beta}^0} \Delta S \frac{\partial_{\bar{\beta}^1} S}{S} \frac{\partial_{\bar{\beta}^2} S}{S} \cdots \frac{\partial_{\bar{\beta}^\ell} S}{S} \phi^{\frac{K}{2}} \right\|_{L^2} \\
					&\le C(\sigma_0) \left\| \langle y \rangle^{-2(1-\eta) + \left(\frac{1-\alpha}{\delta}-1\right)(\Lambda-1) + \Lambda - \frac{5}{2} + \eta} \right\|_{L^2} \\
					&= C(\sigma_0) \left\| \langle y \rangle^{-\Lambda - \frac{3}{2} + 3\eta + f_1} \right\|_{L^2} \\
					&= C(\sigma_0) \left( \int_0^\infty \langle r \rangle^{-2\Lambda - 3 + 6\eta + 2f_1} r^2 \, \mathrm{d}r \right)^{\frac{1}{2}} \le C(\sigma_0).
				\end{aligned}
			\end{equation}
			\item \textbf{Case 2: For a certain $i \ge 1$, there holds $ |\bar{\beta}^i|\in\{K-1, K\}$. } \\
			Without loss of generality, we assume $i = 1$. If $|\bar{\beta}^1| = K$, we have $f_1<1-2\eta$, then
			\begin{equation}\label{Eq_case_beta1_K}
				\begin{aligned}
					\left\| \phi^{\frac{K}{2}} S^{\frac{\alpha-1}{\delta}} \partial_{\bar{\beta}^1} S \frac{\Delta S}{S} \right\|_{L^2} &\le C E^{\frac{1}{2}} \left\| S^{\frac{\alpha-1}{\delta}} \frac{\nabla^2 S}{S} \right\|_{L^\infty} \\
					&\le C(\sigma_0) \left\langle \frac{y}{R_0} \right\rangle^{-2+2\eta+{(\frac{1-\alpha}{\delta}+1)(\Lambda-1)}} \le C(\sigma_0).
				\end{aligned}
			\end{equation}
			If $|\bar{\beta}^1| = K - 1$, there are two possible scenarios, one of which is $|\bar{\beta}^0|=1$. For fixed $j \in \{1,2,3\}$, using \eqref{eq:K-1-perturbation} and selecting $\bar{\varepsilon}$ small enough, we have
			\begin{equation}\label{Eq_case_beta1_Kminus1_sub2}
				\begin{aligned}
					&\left\| \phi^{\frac{K}{2}} \partial_{\bar{\beta}^1} S \frac{\partial_{y_j} \Delta S}{S} S^{\frac{\alpha-1}{\delta}} \right\|_{L^2} \le C(\sigma_0, \bar{\varepsilon}) \left\| \langle y \rangle^{K(1-\eta)\frac{K-2}{K-1}-\bar{\varepsilon}} \partial_{\bar{\beta}^1} S \right\|_{L^{2+\frac{2}{K-2}}} \\
					&\qquad\qquad\qquad\qquad\qquad\quad \times \left\| \frac{\langle y \rangle^{K(1-\eta)\frac{1}{K-1}+\bar{\varepsilon}} |\nabla^3 S|}{S} S^{\frac{\alpha-1}{\delta}} \right\|_{L^{2(K-1)}} \\
					&\qquad\qquad\qquad\qquad\qquad \le C(\sigma_0, \bar{\varepsilon}) \left\| \langle y \rangle^{-3(1-\eta) + \frac{1-\alpha}{\delta}(\Lambda-1) + \frac{K}{K-1}(1-\eta) + \bar{\varepsilon}} \right\|_{L^{2(K-1)}} \\
					&\qquad\qquad\qquad\qquad\qquad \le C(\sigma_0, \bar{\varepsilon}) \left\| \langle y \rangle^{-\Lambda + 3\eta + \bar{\varepsilon} + \frac{K}{K-1}(1-\eta)} \right\|_{L^{2(K-1)}} \le C(\sigma_0, \bar{\varepsilon}).
				\end{aligned}
			\end{equation}
			In the other case, assume without loss of generality that $|\bar{\beta}^2|=1$. For fixed $j \in \{1,2,3\}$, selecting $\bar{\varepsilon}$ small enough, we have
			\begin{equation}\label{Eq_case_beta1_Kminus1_sub1}
				\begin{aligned}
					\left\| \phi^{\frac{K}{2}} \partial_{\bar{\beta}^1} S \frac{\partial_{y_j} S \Delta S}{S^2} S^{\frac{\alpha-1}{\delta}} \right\|_{L^2} &\le C(\sigma_0, \bar{\varepsilon}) \left\| \langle y \rangle^{K(1-\eta)\frac{K-2}{K-1}-\bar{\varepsilon}} \partial_{\bar{\beta}^1} S \right\|_{L^{2+\frac{2}{K-2}}} \\
					&\quad \times \left\| \frac{|\nabla S| \langle y \rangle}{S} \right\|_{L^\infty} \left\| \frac{\langle y \rangle^{K(1-\eta)\frac{1}{K-1}+\bar{\varepsilon}-1} |\nabla^2 S|}{S} S^{\frac{\alpha-1}{\delta}} \right\|_{L^{2(K-1)}} \\
					&\le C(\sigma_0, \bar{\varepsilon}) \left\| \langle y \rangle^{-2(1-\eta) + \frac{1-\alpha}{\delta}(\Lambda-1) - 1 + \frac{K}{K-1}(1-\eta) + \bar{\varepsilon}} \right\|_{L^{2(K-1)}} \\
					&\le C(\sigma_0, \bar{\varepsilon}) \left\| \langle y \rangle^{-\Lambda + 2\eta + \bar{\varepsilon} + \frac{K}{K-1}(1-\eta)} \right\|_{L^{2(K-1)}} \le C(\sigma_0, \bar{\varepsilon}).
				\end{aligned}
			\end{equation}
			\item \textbf{Case 3: $|\bar{\beta}^0| \in \{K-3,K-2\}$.}\\
			If $|\bar{\beta}^0|=K-2$, using \eqref{eq:lessK-2-perturbation}, we have
			\begin{equation}\label{Eq_case_beta0_low_order_sub1}
				\begin{aligned}
					&\left\| \phi^{\frac{K}{2}} \partial_{\bar{\beta}^0} \Delta S \nabla^2 \left( S^{\frac{\alpha-1}{\delta}} \right) \right\|_{L^2} \\
					&\le C\left\| \phi^{\frac{K}{2}} \nabla^K S \right\|_{L^2} \left( \left\| \frac{|\nabla^2 S|}{S} S^{\frac{\alpha-1}{\delta}} \right\|_{L^\infty} + \left\| \frac{|\nabla S|^2}{S^2} S^{\frac{\alpha-1}{\delta}} \right\|_{L^\infty} \right) \\
					&\le CE^{\frac{1}{2}} \left( \left\| \langle y \rangle^{-2(1-\eta) + \frac{1-\alpha}{\delta}(\Lambda-1)} \right\|_{L^\infty} + \left\| \langle y \rangle^{-2 + \frac{1-\alpha}{\delta}(\Lambda-1)} \right\|_{L^\infty} \right) \le C(\sigma_0).
				\end{aligned}
			\end{equation}
			If $|\bar{\beta}^0|=K-3$, using \eqref{eq:lessK-2-perturbation}, we have
			\begin{equation}\label{Eq_case_beta0_low_order_sub2}
				\begin{aligned}
					&\left\| \phi^{\frac{K}{2}} \partial_{\bar{\beta}^0} \Delta S \nabla^3 \left( S^{\frac{\alpha-1}{\delta}} \right) \right\|_{L^2} \\
					&\le C(\sigma_0, \bar{\varepsilon}) \left\| \langle y \rangle^{K(1-\eta)\frac{K-2}{K-1}-\bar{\varepsilon}} \nabla^{K-1} S \right\|_{L^{2+\frac{2}{K-2}}} \left\| \langle y \rangle^{(1-\eta)\frac{K}{K-1} + \bar{\varepsilon} + \frac{1-\alpha}{\delta}(\Lambda-1) - 3(1-\eta)} \right\|_{L^{2(K-1)}} \\
					&\le C(\sigma_0, \bar{\varepsilon}) \left\| \langle y \rangle^{-\Lambda + 3\eta + \bar{\varepsilon} + (1-\eta)\frac{K}{K-1}} \right\|_{L^{2(K-1)}} \le C(\sigma_0, \bar{\varepsilon}).
				\end{aligned}
			\end{equation}
		\end{itemize}
		Combining the estimates from all three cases and fixing $\bar{\varepsilon}$ to be sufficiently small (relative to $\eta$, $\Lambda$, and $K$), we eventually arrive at the uniform bound for $I_{s,1,2}$
		\begin{equation}\label{Eq_Iv12_final_bound_v1}
			I_{s,1,2} \le C(\sigma_0, \bar{\varepsilon}) \le C(\sigma_0).
		\end{equation}
		\textbf{For $I_{s,2}$, we have the following estimates}
		\begin{equation}\label{Eq_Iv2_holder_bounds_updated}
			\begin{aligned}
				|I_{s,2}| &\le \sum_{|\beta|=K} \left| \int_{\mathbb{R}^3} \left( K\phi^{K-1}\nabla\phi S^{\frac{\alpha-1}{\delta}} + \frac{\alpha-1}{\delta} S^{\frac{\alpha-1}{\delta}-1}\phi^K\nabla S \right) \cdot\nabla \partial_{\beta}S \cdot \partial_{\beta}S \mathrm{d}y \right| \\
				&\le \sum_{|\beta|=K}C I_{s,3}^{\frac{1}{2}} K \left( \int_{\mathbb{R}^3} \phi^{K-2} |\nabla\phi|^2 S^{\frac{\alpha-1}{\delta}} |\partial_\beta S|^2 \, \mathrm{d}y \right)^{\frac{1}{2}} \\
				&\quad + \sum_{|\beta|=K}C I_{s,3}^{\frac{1}{2}} \frac{1-\alpha}{\delta} \left( \int_{\mathbb{R}^3} \phi^K S^{\frac{\alpha-1}{\delta}} \frac{|\nabla S|^2}{S^2} |\partial_\beta S|^2 \, \mathrm{d}y \right)^{\frac{1}{2}} \\
				&\le C(K) I_{s,3}^{\frac{1}{2}} E^{\frac{1}{2}} \left( K \left\| \frac{|\nabla\phi|^2}{\phi^2} S^{\frac{\alpha-1}{\delta}} \right\|_{L^\infty}^{\frac{1}{2}} + \frac{1-\alpha}{\delta} \left\| S^{\frac{\alpha-1}{\delta}} \frac{|\nabla S|^2}{S^2} \right\|_{L^\infty}^{\frac{1}{2}} \right)\\
				&\le C(\sigma_0)I_{s,3}^{\frac{1}{2}}.
			\end{aligned}
		\end{equation}
		\textbf{As for $I_{s,4}$, employing an analogous estimation procedure to that of $I_{s,1}$ yields}
		\begin{equation}\label{Iv4csigma}
			|I_{s,4}| \le C(\sigma_0).
		\end{equation}
		\textbf{For $I_{s,5}$, we arrive at the following estimates}
		\begin{equation}\label{Eq_Iv555_Holder_estimate}
			\begin{aligned}
				I_{s,5} &\le C(K)\sum_{|\beta|=K} \int_{\mathbb{R}^3} \phi^K |\partial_\beta S| |\nabla^{K+1} S| \frac{|\nabla S|}{S} S^{\frac{\alpha-1}{\delta}} \, \mathrm{d}y \\
				&\le C(K)\sum_{|\beta|=K} I_{s,3}^{\frac{1}{2}} \left( \int_{\mathbb{R}^3} \phi^K |\partial_\beta S|^2 \frac{|\nabla S|^2}{S^2} S^{\frac{\alpha-1}{\delta}} \, \mathrm{d}y \right)^{\frac{1}{2}} \\
				&\le C(K) I_{s,3}^{\frac{1}{2}} E^{\frac{1}{2}} \left\| \frac{|\nabla S|^2}{S^2} S^{\frac{\alpha-1}{\delta}} \right\|_{L^\infty}^{\frac{1}{2}} \le C(\sigma_0) I_{s,3}^{\frac{1}{2}}.
			\end{aligned}
		\end{equation}
		\textbf{For $I_{s,6}$, we can control it to obtain}
		\begin{equation}\label{Eq_Iv6_commutator_bounds_main}
			\begin{aligned}
				I_{s,6} &= \sum_{|\beta|=K}\frac{\alpha-\delta}{\delta} \int_{\mathbb{R}^3} \phi^K S^{\frac{\alpha-1}{\delta}-1} \left( \partial_\beta (|\nabla S|^2)  - \partial_\beta (\nabla S) \cdot \nabla S \right) \cdot \partial_\beta S \, \mathrm{d}y \\
				&\le C(K) E^{\frac{1}{2}} I_{s,3}^{\frac{1}{2}} \left\| \frac{|\nabla S|^2}{S^2} S^{\frac{\alpha-1}{\delta}} \right\|_{L^\infty}^{\frac{1}{2}} \\
				&\quad + C(K) E^{\frac{1}{2}} \max_{\substack{|\bar{\beta}^0| + |\bar{\beta}^1| = K \\ 1 \le |\bar{\beta}^0| \le K-1}} \left\| \phi^{\frac{K}{2}} S^{\frac{\alpha-1}{\delta}} \partial_{\bar{\beta}^0} \nabla S \frac{\partial_{\bar{\beta}^1} \nabla S}{S} \right\|_{L^2} \\
				&\le C(\sigma_0) I_{s,3}^{\frac{1}{2}} + I_{s,6,2} \le C(\sigma_0) \left( 1 + I_{s,3}^{\frac{1}{2}} \right).
			\end{aligned}
		\end{equation}
		The estimate for $I_{s,6,2}$ is completely analogous to that for $I_{s,1,2}$.
		
		We now turn to the estimate of $K_{4,v}$. \textbf{For $I_{v,3}$, since $0 < \alpha < \frac{1}{2}$, we find that $I_{v,3}$ is positive definite.} Specifically, decomposing the gradient into its symmetric and skew-symmetric parts gives the uniform positivity estimate
		\begin{equation}\label{Lvellpiticoo}
			\begin{aligned}
				&\nu |\nabla \partial_\beta V|^2 + \left( \sqrt{\nu^2 - \varepsilon^2} + (\alpha - 1) \left( \nu + \sqrt{\nu^2 - \varepsilon^2} \right) \right)\partial_\beta\partial_i V_j\partial_\beta\partial_j V_i \\
				&\quad\ge \min\left\{\alpha\left(\nu+\sqrt{\nu^2-\varepsilon^2}\right),\,
				2\nu-\alpha\left(\nu+\sqrt{\nu^2-\varepsilon^2}\right)\right\}|\nabla \partial_\beta V|^2 \\
				&\quad=\alpha\left(\nu+\sqrt{\nu^2-\varepsilon^2}\right)|\nabla \partial_\beta V|^2.
			\end{aligned}
		\end{equation}
		Here we used $0<\alpha<\frac12$ and $\sqrt{\nu^2-\varepsilon^2}\le\nu$. Therefore, $I_{v,3}$ is strictly positive definite. \textbf{We have already established the positive definiteness of
			$I_{v,3}$. Then the estimates for the remaining terms in $K_{4,v}$ are
			entirely analogous to those for the corresponding terms in $K_{4,s}$;
			we therefore state only the resulting bounds below.}
		\begin{align}
			&|I_{v,1}|\le C(\sigma_0)(1+I_{v,3}^{\frac{1}{2}}),\label{Is11s3sigma0}\\
			&|I_{v,2}| \le C(\sigma_0)I_{v,3}^{\frac{1}{2}},\\
			&|I_{v,4}|\le C(\sigma_0),\\
			&|I_{v,5}|\le C(\sigma_0)I_{v,3}^{\frac{1}{2}},\\
			&|I_{v,6}|\le C(\sigma_0)(1+I_{v,3}^{\frac{1}{2}}).\label{Is61s3sigma0}
		\end{align}
		Substituting \eqref{Eq_Iv11_Holder_estimate}, \eqref{Eq_Iv12_final_bound_v1}, \eqref{Iv4csigma}, \eqref{Eq_Iv555_Holder_estimate}, \eqref{Eq_Iv6_commutator_bounds_main}, and \eqref{Is11s3sigma0}-\eqref{Is61s3sigma0} into \eqref{K4_initial_boundestimate} and using Young's inequality, we arrive at
		\begin{equation}\label{K4_final_Is3,Iv,3}
			e^{f_1 \tau}K_4+\frac{1}{2}\delta^{\frac{\alpha-1}{\delta}}d\alpha I_{s,3}+\frac{1}{2}\delta^{\frac{\alpha-1}{\delta}}I_{v,3} \le C(\sigma_0).    
		\end{equation}
		\textbf{5. Final energy estimates}\\
		Now substituting \eqref{Eq_K1_integration_by_parts}, \eqref{Eq_image_0bc0e4}, \eqref{Eq_K3_nonlinear_coupling_bound_chain}, \eqref{K4_final_Is3,Iv,3} into \eqref{Eq_rewrite_full_energy_integral_system} and using Young's inequality, the identity \eqref{eq:radial-angular-derivative-definition}, we obtain
		\begin{equation}\label{Eq_rewrite_full_energy_integral_system_finalss}
			\begin{aligned}
				&\left(\frac{1}{2} \frac{\mathrm{d}}{\mathrm{d}\tau} + \Lambda - 1 + K\right) E_K+\frac{1}{2}e^{-f_1 \tau}\delta^{\frac{\alpha-1}{\delta}}d\alpha I_{s,3}+\frac{1}{2}e^{-f_1 \tau}\delta^{\frac{\alpha-1}{\delta}}I_{v,3} \\
				&\le CE + \frac{K}{2} \sum_{|\beta|=K} \int_{\mathbb{R}^3} \frac{y \cdot \nabla \phi}{\phi} \left( |\partial_\beta S|^2 + |\partial_\beta V|^2 \right) \phi^K \mathrm{d}y\\
				&\quad +K \int_{\mathbb{R}^3} \left( \frac{\nabla \phi \cdot \widehat{\chi}\bar{V}}{2\phi} - \partial_r (\widehat{\chi}\bar{\mathcal{V}}) \right) \left( |\partial_r \nabla^{K-1} V|^2 + |\partial_r \nabla^{K-1} S|^2 \right) \phi^K \, \mathrm{d}y \\
				&\quad + K \int_{\mathbb{R}^3} \left( \frac{\nabla \phi \cdot \widehat{\chi}\bar{V}}{2\phi |y|^2} - \frac{ (\widehat{\chi}\bar{\mathcal{V}})}{|y|^3} \right) \left( |\nabla_\theta \nabla^{K-1} V|^2 + |\nabla_\theta \nabla^{K-1} S|^2 \right) \phi^K \, \mathrm{d}y\\
				&\quad+K \delta \sum_{|\beta|=K} \int_{\mathbb{R}^3} \phi^K \left( \frac{|\nabla \phi| \widehat{\chi}\bar{S}}{2\phi} + |\partial_r (\widehat{\chi}\bar{S})| \right) \left( |\partial_{\beta} S|^2 + |\partial_{\beta} V|^2 \right) \mathrm{d}y+ e^{-f_1\tau}C(\sigma_0)\\
				&\le CE +e^{-f_1(\tau-\tau_0)}+K \int_{\mathbb{R}^3} \left( \delta|\partial_r(  \widehat{\chi}\bar{S})|-\partial_r (\widehat{\chi}\bar{\mathcal{V}}) \right) \left( |\partial_r \nabla^{K-1} V|^2 + |\partial_r \nabla^{K-1} S|^2 \right) \phi^K \, \mathrm{d}y \\
				&\quad + K \int_{\mathbb{R}^3} \left( \delta\frac{|\partial_r(\widehat{\chi}\bar{S})|}{|y|^2} - \frac{(\widehat{\chi}\bar{\mathcal{V}})}{|y|^3} \right) \left( |\nabla_\theta \nabla^{K-1} V|^2 + |\nabla_\theta \nabla^{K-1} S|^2 \right) \phi^K \, \mathrm{d}y\\
				&\quad+K \sum_{|\beta|=K} \int_{\mathbb{R}^3} \phi^K \left( \frac{|\nabla \phi| \widehat{\chi}\bar{S}+\nabla \phi \cdot (y+\widehat{\chi}\bar{V})}{2\phi}\right) \left( |\partial_{\beta} S|^2 + |\partial_{\beta} V|^2 \right) \mathrm{d}y.
			\end{aligned}
		\end{equation}
		Since $\phi=1$ in $B(0,R_0)$, the estimate
		$|\partial_r\widehat{\chi}|\leq C e^{-\tau_0}$, together with the strict
		core-profile inequalities
		\eqref{eq:profile-radial-repulsivity}--\eqref{eq:profile-angular-repulsivity}
		and the parameter hierarchy \eqref{Eq_parameter_hierarchy}, allows us to
		choose a constant $\bar{\eta}>0$ such that the following estimates hold
		uniformly throughout $B(0,R_0)$:
		$$
		-\partial_r (\widehat{\chi}\bar{\mathcal{V}})+\delta|\partial_r(  \widehat{\chi}\bar{S})|+\frac{|\nabla \phi| \widehat{\chi}\bar{S}+\nabla \phi \cdot (y+\widehat{\chi}\bar{V})}{2\phi} \le 1-\frac{\bar{\eta}}{2},
		$$
		$$
		\delta|\partial_r(\widehat{\chi}\bar{S})| - \frac{(\widehat{\chi}\bar{\mathcal{V}})}{r}+\frac{|\nabla \phi| \widehat{\chi}\bar{S}+\nabla \phi \cdot (y+\widehat{\chi}\bar{V})}{2\phi} \le 1-\frac{\bar{\eta}}{2}.
		$$
		In $B^c(0,R_0)$, using \eqref{eq:weight-radial-ratio-bound} in
		Remark~\ref{remark1-etar}, \eqref{xdecayestim}, \eqref{Eq_parameter_hierarchy}, \eqref{eq:R0-choice}, we have
		$$
		-\partial_r (\widehat{\chi}\bar{\mathcal{V}})+\delta|\partial_r(  \widehat{\chi}\bar{S})|+\frac{|\nabla \phi| \widehat{\chi}\bar{S}+\nabla \phi \cdot (y+\widehat{\chi}\bar{V})}{2\phi} \le  \frac{\eta}{4} + \frac{r\partial_r \phi}{2\phi}  \le \frac{\eta}{4} + 1-\eta\le 1-\frac{{\eta}}{2},
		$$
		$$
		\delta|\partial_r(\widehat{\chi}\bar{S})| - \frac{(\widehat{\chi}\bar{\mathcal{V}})}{r}+\frac{|\nabla \phi| \widehat{\chi}\bar{S}+\nabla \phi \cdot (y+\widehat{\chi}\bar{V})}{2\phi} \le  \frac{\eta}{4} + \frac{r\partial_r \phi}{2\phi}\le \frac{\eta}{4} + 1-\eta\le 1-\frac{{\eta}}{2}.
		$$
		Therefore, we arrive at 
		\begin{equation}\label{Eq_energy_inequality_skeleton_final}
			\begin{aligned}
				&\frac{\mathrm{d}}{\mathrm{d}\tau} E_K\le -K \min\{\bar{\eta}, \eta\} E_K + C E.
			\end{aligned}
		\end{equation}
		Using Gr\"onwall's inequality and \eqref{Eq_parameter_hierarchy}, \eqref{eq:K-repulsivity-relation}, we obtain
		\begin{equation}
			\begin{aligned}
				E_K(\tau)&\le e^{-K \min\{\bar{\eta},\eta\}(\tau-\tau_0)} E_K(\tau_0) + \frac{CE}{K \min\{\bar{\eta},\eta\}} \left( 1 - e^{-K \min\{\bar{\eta},\eta\}(\tau-\tau_0)} \right) \\
				&\le e^{-K \min\{\bar{\eta},\eta\}(\tau-\tau_0)} E_K(\tau_0) + \frac{E}{2} \left( 1 - e^{-K \min\{\bar{\eta},\eta\}(\tau-\tau_0)} \right).
			\end{aligned}
		\end{equation}
		If $E_K(\tau_0)\le \frac{E}{2}$, then $E_K(\tau)\le \frac{E}{2}$.
	\end{proof}
	\subsection{Control of the unstable component under the bootstrap assumptions}
	\label{subsec:bootstrap-unstable-control}
	
	We apply the stable--unstable decomposition of the truncated linearized
	operator to the bootstrap argument. Let
	\[
	X=X_{\rm s}\oplus X_{\rm u},
	\qquad
	\dim X_{\rm u}=N,
	\]
	and let \(P_{\rm uns}\) denote the spectral projection onto \(X_{\rm u}\).
	Choose a smooth normalized basis
	\[
	\Phi_j=\{(\Phi_{j,S},\Phi_{j,V})\}_{j=1}^{N}
	\]
	of \(X_{\rm u}\). For the truncated perturbation
	\((\widetilde{\widetilde S},\widetilde{\widetilde V})\), define
	\[
	k(\tau)
	:=
	P_{\rm uns}
	(\widetilde{\widetilde S},\widetilde{\widetilde V})(\tau)
	=
	\sum_{j=1}^{N}
	k_j(\tau)(\Phi_{j,S},\Phi_{j,V}).
	\]
	By the initial condition \eqref{initialdatatruncated}, we have $k_j(\tau_0)=\widehat{k_j}$.
	Let
	\(\langle\cdot,\cdot\rangle_{\Upsilon}\) be the inner product on
	\(X_{\rm u}\) furnished by the spectral decomposition (see Lemma \ref{lem:new-stable-unstable-summary}), so that
	\begin{equation}\label{eq:unstable-spectral-lower-bound}
		\Re\langle\mathcal L h,h\rangle_{\Upsilon}
		\ge
		-\frac{3\mathfrak c_g}{5}\|h\|_{\Upsilon}^2,
		\qquad h\in X_{\rm u}.
	\end{equation}
	Since \(X_{\rm u}\) is finite-dimensional, there exists \(C_m\ge1\) such
	that
	\begin{equation}\label{eq:unstable-equivalent-norms}
		C_m^{-1}\|h\|_X
		\le
		\|h\|_{\Upsilon}
		\le
		C_m\|h\|_X,
		\qquad h\in X_{\rm u}.
	\end{equation}
	Since each $\Phi_j$ has compact support and is normalized by
	$\|\Phi_j\|_X=1$, taking $\sigma_1$ sufficiently small, with its
	smallness depending on $m$ and $E$, ensures that
	\eqref{initialdatacondition1}--\eqref{initialdatacondition2} hold for
	every admissible choice of the coefficients $\hat{k}_i$, provided that
	the reference perturbation $(\widetilde S_0^*,\widetilde V_0^*)$
	satisfies
	\begin{equation}\label{eq:conditions_for_tilde}
		\begin{aligned}
			&\max\left\{
			\|\widetilde S_0^*\|_{L^\infty},
			\|\widetilde V_0^*\|_{L^\infty}
			\right\}
			\leq \frac{3\sigma_1}{4}, \quad\left\|
			\phi^{\frac K2}\nabla^K\widetilde S_0^*
			\right\|_{L^2}
			+
			\left\|
			\phi^{\frac K2}\nabla^K\widetilde V_0^*
			\right\|_{L^2}
			\leq \frac{E}{4},
			\\
			&\widetilde S_0^*+\widehat{\chi}_0\bar S
			\geq \frac{\sigma_1}{2}\left\langle\frac{y}{R_0}\right\rangle^{1-\Lambda}, \quad \left|
			\nabla\bigl(\widetilde S_0^*
			+\widehat{\chi}_0\bar S\bigr)
			\right|
			+
			\left|
			\nabla\bigl(\widetilde V_0^*
			+\widehat{\chi}_0\bar V\bigr)
			\right|
			\leq C\langle y\rangle^{-\Lambda}.
		\end{aligned}
	\end{equation}
	
	Finally, since each $\Phi_j$ is compactly supported and satisfies
	$\|\Phi_j\|_X=1$, the full initial perturbation
	\[
	(\widetilde S_0,\widetilde V_0)
	=
	(\widetilde S_0^*,\widetilde V_0^*)
	+
	\sum_{j=1}^N\hat{k}_j\Phi_j
	\]
	continues to satisfy
	\eqref{initialdatacondition1}--\eqref{initialdatacondition2}, provided
	that $\sigma_1$ and the coefficients $\hat{k}_j$ are sufficiently
	small, with their smallness depending only on $m$ and $E$.
	
	The next proposition connects the spectral splitting with the improvement of
	the unstable-mode bootstrap estimate \eqref{Puns}. 
	Throughout this subsection, we identify the spectral-gap notations by setting
	\[
	\mathfrak c_g:=\sigma_g=\frac{25}{12}\varpi.
	\]
	
	\begin{proposition}[Selection and outgoing property of the unstable modes]
		\label{prop:bootstrap-unstable-outgoing}
		Assume that the bootstrap bounds
		\eqref{Puns}--\eqref{E1} hold on \([\tau_0,\tau_1]\). The remaining bootstrap estimates
		\eqref{S,U,1}--\eqref{E1} imply (using \eqref{eq:full-forcing-X-bound} and \eqref{eq:unstable-equivalent-norms})
		\begin{equation}\label{eq:bootstrap-forcing-uns}
			\left\|
			P_{\rm uns}\bigl(\widehat\chi_2\mathcal F(\tau)\bigr)
			\right\|_{\Upsilon}
			\le
			C(m)\sigma_1^{6/5}
			e^{-\frac32\varpi(\tau-\tau_0)}.
		\end{equation}
		Then the initial unstable coefficients can be chosen so that
		\begin{equation}\label{eq:improved-unstable-conclusion}
			\left\|
			P_{\rm uns}
			(\widetilde{\widetilde S},\widetilde{\widetilde V})(\tau)
			\right\|_X
			\le
			\sigma_1^{21/20}
			e^{-\frac43\varpi(\tau-\tau_0)},
			\qquad
			\tau\in[\tau_0,\tau_1].
		\end{equation}
		In particular, this proves the improved bootstrap estimate \eqref{Puns2}.
	\end{proposition}
	
	\begin{proof}
		Introduce the time-dependent neighborhoods
		\[
		\begin{aligned}
			\mathcal U_X(\tau)
			&:=
			\left\{
			h\in X_{\rm u}:
			\|h\|_X
			\le
			\sigma_1^{21/20}
			e^{-\frac43\varpi(\tau-\tau_0)}
			\right\},\\
			\mathcal U_{\Upsilon}(\tau)
			&:=
			\left\{
			h\in X_{\rm u}:
			\|h\|_{\Upsilon}
			\le
			\sigma_1^{11/10}
			e^{-\frac43\varpi(\tau-\tau_0)}
			\right\}.
		\end{aligned}
		\]
		Since \(11/10>21/20\), the norm equivalence
		\eqref{eq:unstable-equivalent-norms} implies, after taking \(\sigma_1\)
		sufficiently small, that
		\begin{equation}\label{eq:unstable-ball-inclusion}
			\mathcal U_{\Upsilon}(\tau)
			\Subset
			\mathcal U_X(\tau),
			\qquad \tau\ge\tau_0.
		\end{equation}
		
		The invariance of \(X_{\rm u}\) under \(\mathcal L\) gives the
		finite-dimensional evolution equation
		\begin{equation}\label{eq:k-finite-dimensional}
			\partial_\tau k
			=
			\mathcal L k
			+
			P_{\rm uns}(\widehat\chi_2\mathcal F).
		\end{equation}
		Taking the \(\Upsilon\)-inner product of \eqref{eq:k-finite-dimensional} with
		\(k\), and using \eqref{eq:unstable-spectral-lower-bound} and
		\eqref{eq:bootstrap-forcing-uns}, we find
		\begin{equation}\label{eq:k-differential-lower}
			\begin{aligned}
				\frac12\frac{d}{d\tau}\|k(\tau)\|_{\Upsilon}^2
				&=
				\Re\langle\mathcal Lk,k\rangle_{\Upsilon}
				+
				\Re\left\langle
				P_{\rm uns}(\widehat\chi_2\mathcal F),k
				\right\rangle_{\Upsilon}\\
				&\ge
				-\frac{3\mathfrak c_g}{5}\|k(\tau)\|_{\Upsilon}^2
				-
				C\sigma_1^{6/5}
				e^{-\frac32\varpi(\tau-\tau_0)}
				\|k(\tau)\|_{\Upsilon}.
			\end{aligned}
		\end{equation}
		
		Suppose that \(k(\tau)\) reaches
		\(\partial\mathcal U_{\Upsilon}(\tau)\) for the first time at
		\(\tau=\tau_e\). Then
		\[
		\|k(\tau_e)\|_{\Upsilon}
		=
		\sigma_1^{11/10}
		e^{-\frac43\varpi(\tau_e-\tau_0)}.
		\]
		Consequently, the forcing term in \eqref{eq:k-differential-lower} satisfies
		\[
		\begin{aligned}
			&C\sigma_1^{6/5}
			e^{-\frac32\varpi(\tau_e-\tau_0)}
			\|k(\tau_e)\|_{\Upsilon}\\
			&\qquad=
			C\sigma_1^{1/10}
			e^{-\frac16\varpi(\tau_e-\tau_0)}
			\|k(\tau_e)\|_{\Upsilon}^2
			\le
			\frac{\mathfrak c_g}{50}
			\|k(\tau_e)\|_{\Upsilon}^2,
		\end{aligned}
		\]
		provided that \(\sigma_1\) is sufficiently small. Thus
		\begin{equation}\label{eq:k-boundary-derivative}
			\frac12\frac{d}{d\tau}\|k(\tau)\|_{\Upsilon}^2
			\bigg|_{\tau=\tau_e}
			\ge
			-\frac{31\mathfrak c_g}{50}
			\|k(\tau_e)\|_{\Upsilon}^2.
		\end{equation}
		It follows that
		\[
		\begin{aligned}
			&\frac{d}{d\tau}
			\left[
			e^{\frac83\varpi(\tau-\tau_0)}
			\|k(\tau)\|_{\Upsilon}^2
			\right]_{\tau=\tau_e}\\
			&\quad\ge
			\left(
			\frac83\varpi-\frac{31}{25}\mathfrak c_g
			\right)
			e^{\frac83\varpi(\tau_e-\tau_0)}
			\|k(\tau_e)\|_{\Upsilon}^2
			>0,
		\end{aligned}
		\]
		where the last inequality follows from the choice of \(\varpi\).
		Therefore, whenever \(k(\tau)\) touches the moving boundary
		\(\partial\mathcal U_{\Upsilon}(\tau)\), it crosses that boundary strictly
		outward.
		
		We finally select the initial unstable coefficients. Suppose that no choice
		\[
		k(\tau_0)\in\mathcal U_{\Upsilon}(\tau_0)
		\]
		produces a trajectory remaining in
		\(\mathcal U_{\Upsilon}(\tau)\) throughout \([\tau_0,\tau_1]\). For each
		initial coefficient vector, let \(\tau_e\) be its first exit time. The strict
		outgoing property proved above shows that the normalized exit map
		\[
		k(\tau_0)
		\longmapsto
		\sigma_1^{-11/10}
		e^{\frac43\varpi(\tau_e-\tau_0)}
		k(\tau_e)
		\]
		is continuous and restricts to the identity (Here, the identity map is understood after identifying
		$\mathcal U_{\Upsilon}(\tau_0)$ with the unit ball via the normalization
		$k(\tau_0)\mapsto \sigma_1^{-11/10}k(\tau_0)$) on
		\(\partial\mathcal U_{\Upsilon}(\tau_0)\). It would therefore define a
		continuous retraction of the closed \(N\)-dimensional ball onto its boundary (see \cite[Proposition 8.15]{Buckmaster-Cao-Labora-Gomez-Serrano} for more details),
		which is impossible. Hence there exists a choice of the initial unstable
		coefficients for which
		\[
		k(\tau)\in\mathcal U_{\Upsilon}(\tau),
		\qquad
		\tau\in[\tau_0,\tau_1].
		\]
		By \eqref{eq:unstable-ball-inclusion},
		\[
		k(\tau)\in\mathcal U_X(\tau),
		\qquad
		\tau\in[\tau_0,\tau_1],
		\]
		which is precisely
		\[
		\left\|
		P_{\rm uns}
		(\widetilde{\widetilde S},\widetilde{\widetilde V})(\tau)
		\right\|_X
		\le
		\sigma_1^{21/20}
		e^{-\frac43\varpi(\tau-\tau_0)}.
		\]
		This completes the improvement of \eqref{Puns}.
	\end{proof}
	Furthermore, the condition
	\[
	(\widehat\chi_2\widetilde S_0^*,\widehat\chi_2\widetilde V_0^*)
	\in X_{\mathrm{s}}
	\]
	can equivalently be expressed as
	\begin{equation}\label{eq:stable-initial-data-condition}
		P_{\mathrm{uns}}
		\bigl(\widehat\chi_2\widetilde S_0^*,\widehat\chi_2\widetilde V_0^*\bigr)=0.
	\end{equation}
	This imposes only finitely many closed constraints on the initial data,
	since the range of $P_{\mathrm{uns}}$ is the finite-dimensional unstable
	subspace $X_{\mathrm{u}}$. Consequently, the collection of initial data
	leading to finite-time implosion contains a finite-codimensional manifold.
	\subsection{Global Existence and Nonlinear Stability}
	We now provide the regularity estimates for the equation.
	\begin{proposition}[Regularity consequence of the bootstrap bounds]
		\label{prop:global-regularity-SS}
		Assume that the solution \((S,V)\) of \eqref{Eq_self_similar_S} exists on
		\([\tau_0,\tau_1]\) and satisfies the bootstrap estimates. Then, for every
		\(\tau\in[\tau_0,\tau_1]\),
		\begin{equation}\label{eq:global-HK-SV-1}
			\|S(\tau)-S^*(\tau)\|_{H^K}^2
			+
			\|V(\tau)\|_{H^K}^2
			+
			\int_{\tau_0}^{\tau}
			\left(
			\|S(\tilde\tau)-S^*(\tilde\tau)\|_{H^{K+1}}^2
			+\|V(\tilde\tau)\|_{H^{K+1}}^2
			\right)\,d\tilde\tau
			\le C(\tau),
		\end{equation}
		and
		\begin{equation}\label{eq:global-HK-SV-2}
			\|\partial_\tau(S-S^*(\tau))\|_{H^{K-2}}
			+
			\|V_\tau(\tau)\|_{H^{K-2}}
			+
			\int_{\tau_0}^{\tau}
			\left(
			\|\partial_{\tilde\tau}(S-S^*(\tilde\tau))\|_{H^{K-1}}^2
			+\|V_\tau(\tilde\tau)\|_{H^{K-1}}^2
			\right)\,d\tilde\tau
			\le C(\tau).
		\end{equation}
	\end{proposition}
	
	\begin{proof}
		We begin with the zeroth-order energy. Set
		\[
		\mathcal E_0(\tau)
		:=
		\frac12\int_{\R^3}
		\bigl(
		|S-S^*|^2+|V|^2
		\bigr)\,dy .
		\]
		Since \(S^*_{\tau}+(\Lambda-1)S^*=0\), subtracting the equation for the
		far-field state from the first equation in \eqref{Eq_self_similar_S}, and
		multiplying the two equations by \(S-S^*\) and
		\(V\), respectively, and integrating over \(\R^3\), we obtain
		\begin{equation}\label{eq:L2-energy-SV-new}
			\begin{aligned}
				\frac{d}{d\tau}\mathcal E_0(\tau)
				&=
				-(\Lambda-1)\int_{\R^3}
				\bigl(
				|S-S^*|^2+|V|^2
				\bigr)\,dy  \\
				&\quad
				-\int_{\R^3}
				\left[
				(y+V)\cdot\nabla S\,(S-S^*)
				+
				(y+V)\cdot\nabla V\cdot V
				\right]\,dy  \\
				&\quad
				-\delta\int_{\R^3}
				\left[
				S(S-S^*)\diver V
				+
				S\nabla S\cdot V
				\right]\,dy  \\
				&\quad
				+e^{-f_1\tau}d\alpha
				\int_{\R^3}
				(\delta S)^{\frac{\alpha-1}{\delta}}
				\Delta S\,(S-S^*)\,dy  \\
				&\quad
				+e^{-f_1\tau}d\alpha(\alpha-\delta)
				\int_{\R^3}
				(\delta S)^{\frac{\alpha-\delta-1}{\delta}}
				|\nabla S|^2(S-S^*)\,dy  \\
				&\quad
				+e^{-f_1\tau}
				\int_{\R^3}
				(\delta S)^{\frac{\alpha-1}{\delta}}
				\mathbb L(V)\cdot V\,dy  \\
				&\quad
				+e^{-f_1\tau}\alpha
				\int_{\R^3}
				(\delta S)^{\frac{\alpha-\delta-1}{\delta}}
				(\nabla S\cdot\mathbb D V)\cdot V\,dy  \\
				&=: \mathcal H_1+\mathcal H_2+\mathcal H_3+\mathcal H_4+\mathcal H_5+\mathcal H_6+\mathcal H_7 .
			\end{aligned}
		\end{equation}
		
		We now estimate the terms on the right-hand side. Since \(\Lambda>1\),
		\[
		\mathcal H_1
		=
		-2(\Lambda-1)\mathcal E_0(\tau)
		\le 0.
		\]
		For the transport term, integration by parts gives
		\[
		\begin{aligned}
			\mathcal H_2
			&=
			-\int_{\R^3}
			(y+V)\cdot\nabla(S-S^*)\,(S-S^*)\,dy
			-\int_{\R^3}
			(y+V)\cdot\nabla V\cdot V\,dy\\
			&=
			\frac12\int_{\R^3}
			\diver(y+V)
			\bigl(
			|S-S^*|^2+|V|^2
			\bigr)\,dy.
		\end{aligned}
		\]
		Using the bootstrap bounds on \(V\) and the explicit far-field profile, we get
		\begin{equation}\label{eq:H2-bound-new}
			|\mathcal H_2|
			\le
			C\mathcal E_0(\tau).
		\end{equation}
		For the term $\mathcal H_3$, we have
		\begin{align*}
			\mathcal H_3
			&=
			-\delta\int_{\mathbb R^3}
			\left[
			S(S-S^*)\operatorname{div}V
			+
			S\nabla(S-S^*)\cdot V
			\right]\,\mathrm dy.
		\end{align*}
		Integrating by parts in the second term, we obtain
		\begin{align*}
			-\delta\int_{\mathbb R^3}
			S\nabla(S-S^*)\cdot V\,\mathrm dy
			&=
			\delta\int_{\mathbb R^3}
			(S-S^*)\operatorname{div}(SV)\,\mathrm dy\\
			&=
			\delta\int_{\mathbb R^3}
			S(S-S^*)\operatorname{div}V\,\mathrm dy\\
			&\quad+
			\delta\int_{\mathbb R^3}
			(S-S^*)V\cdot\nabla S\,\mathrm dy.
		\end{align*}
		Thus, the two terms containing
		$S(S-S^*)\operatorname{div}V$ cancel, and consequently
		\[
		\mathcal H_3
		=
		\delta\int_{\mathbb R^3}
		(S-S^*)V\cdot\nabla S\,\mathrm dy.
		\]
		Using the bootstrap estimates \eqref{nableS2} and Young's inequality, we conclude that
		\begin{equation}\label{eq:H3-bound-new}
			|\mathcal H_3|
			\leq C\mathcal E_0(\tau).
		\end{equation}

		For the \(S\)-diffusion contribution, since \(S\) stays in a fixed positive
		range (Here, one can derive a time-dependent lower bound by following the same
		argument as in \eqref{CSrlowerboundssss}--\eqref{c*lowerboundedsss}),  integration by parts yields
		\[
		\begin{aligned}
			\mathcal H_4
			&=
			e^{-f_1\tau}d\alpha
			\int_{\R^3}
			(\delta S)^{\frac{\alpha-1}{\delta}}
			\Delta S\,(S-S^*)\,dy  \\
			&=
			-e^{-f_1\tau}d\alpha
			\int_{\R^3}
			(\delta S)^{\frac{\alpha-1}{\delta}}
			|\nabla S|^2\,dy  \\
			&\quad
			-e^{-f_1\tau}d\alpha
			\int_{\R^3}
			(S-S^*)\nabla\left(
			(\delta S)^{\frac{\alpha-1}{\delta}}
			\right)\cdot\nabla S\,dy .
		\end{aligned}
		\]
		Therefore, using the lower and upper bounds of $S$, \eqref{nableS2} and Young's inequality
		\begin{equation}\label{eq:H4-bound-new}
			\mathcal H_4
			\le
			-C(\sigma_0)e^{-f_1\tau}\|\nabla S\|_{L^2}^2
			+
			C e^{-f_1\tau}\mathcal E_0(\tau)
			+
			C(\tau).
		\end{equation}
		Using the lower and upper bound of $S$ and \eqref{nableS2}, we have
		\[
		\begin{aligned}
			|\mathcal H_5|
			&\le
			Ce^{-f_1\tau}
			\int_{\R^3}
			|\nabla S|^2|S-S^*|\,dy  \\
			&\le
			Ce^{-f_1\tau}
			\int_{\R^3}
			\langle y\rangle^{-2\Lambda}|S-S^*|\,dy  \\
			&\le
			C\mathcal E_0(\tau)+C(\tau),
		\end{aligned}
		\]
		where we used the bootstrap decay of \(\nabla S\) and \(1<\Lambda<2\).
		
		For the \(V\)-diffusion term, ellipticity of \(\mathbb L\), upper and lower bound of $S$ and Young's inequality give
		\[
		\begin{aligned}
			\mathcal H_6
			&=
			e^{-f_1\tau}
			\int_{\R^3}
			(\delta S)^{\frac{\alpha-1}{\delta}}
			\mathbb L(V)\cdot V\,dy  \\
			&\le
			-C(\sigma_0)e^{-f_1\tau}\|\nabla V\|_{L^2}^2
			+
			Ce^{-f_1\tau}
			\int_{\R^3}
			|\nabla S||\nabla V||V|\,dy  \\
			&\le
			-C(\sigma_0)e^{-f_1\tau}\|\nabla V\|_{L^2}^2
			+
			Ce^{-f_1\tau}
			\int_{\R^3}
			|\nabla S|^2|V|^2\,dy  \\
			&\le
			-C(\sigma_0)e^{-f_1\tau}\|\nabla V\|_{L^2}^2
			+
			C\mathcal E_0(\tau).
		\end{aligned}
		\]
		Similarly, the first-order viscous coupling satisfies
		\[
		\begin{aligned}
			|\mathcal H_7|
			&\le
			Ce^{-f_1\tau}
			\int_{\R^3}
			|\nabla S||\nabla V||V|\,dy  \\
			&\le
			\vartheta e^{-f_1\tau}\|\nabla V\|_{L^2}^2
			+
			C_\vartheta e^{-f_1\tau}
			\int_{\R^3}
			|\nabla S|^2|V|^2\,dy  \\
			&\le
			\vartheta e^{-f_1\tau}\|\nabla V\|_{L^2}^2
			+
			C\mathcal E_0(\tau).
		\end{aligned}
		\]
		Choosing \(\vartheta>0\) sufficiently small and combining
		\eqref{eq:L2-energy-SV-new}--\eqref{eq:H4-bound-new}, we arrive at
		\[
		\frac{d}{d\tau}\mathcal E_0(\tau)
		+
		C(\sigma_0)e^{-f_1\tau}
		\left(\|\nabla S\|_{L^2}^2+\|\nabla V\|_{L^2}^2\right)
		\le
		C\mathcal E_0(\tau)+C(\tau).
		\]
		By Gr\"onwall's inequality,
		\[
		\mathcal E_0(\tau)
		+
		\int_{\tau_0}^{\tau}
		e^{-f_1\tilde\tau}
		\left(
		\|\nabla S(\tilde\tau)\|_{L^2}^2
		+\|\nabla V(\tilde\tau)\|_{L^2}^2
		\right)\,d\tilde\tau
		\le
		C(\tau).
		\]
		
		Next, the bootstrap weighted high-order estimate gives
		\[
		\int_{\R^3}
		\bigl(
		|\nabla^K S|^2+|\nabla^K V|^2
		\bigr)\,dy
		\le
		CE_K(\tau)
		\le C(E),
		\]
		and from the dissipative terms on the left-hand side of
		\eqref{Eq_rewrite_full_energy_integral_system_finalss}, we infer that
		\[
		\int_{\tau_0}^{\tau}
		\left(
		\|S(\tilde\tau)-S^*(\tilde\tau)\|_{H^{K+1}}^2
		+\|V(\tilde\tau)\|_{H^{K+1}}^2
		\right)\,d\tilde\tau
		\le C(\tau).
		\]
		Together with the \(L^2\)-bound and standard interpolation, this proves
		\[
		\|S(\tau)-S^*(\tau)\|_{H^K}^2
		+
		\|V(\tau)\|_{H^K}^2
		+
		\int_{\tau_0}^{\tau}
		\left(
		\|S(\tilde\tau)-S^*(\tilde\tau)\|_{H^{K+1}}^2
		+\|V(\tilde\tau)\|_{H^{K+1}}^2
		\right)\,d\tilde\tau
		\le C(\tau).
		\]
		
		Finally, the time derivative estimates follow from the equations. Indeed,
		using \eqref{Eq_self_similar_S}, the Sobolev product estimates, the pointwise decay,
		and the bounds already obtained (may be necessary to establish an additional global weighted
		$L^2$ estimate for the gradient by testing the equation against a
		suitable weighted function, as in the appendix), we can prove \eqref{eq:global-HK-SV-2}.
	\end{proof}
	\begin{proposition}[Continuation under the bootstrap bounds]
		\label{prop:continuation-self-similar}
		Let \((S,V)\) be the local solution of \eqref{Eq_self_similar_S} constructed in
		Theorem~\ref{thm:local-self-similar}. Assume that the improved estimates
		obtained in Proposition~\ref{prop_English}, together with the regularity
		bounds established above, remain valid throughout the lifespan of the
		solution. Then \((S,V)\) can be continued for all
		\(\tau\ge \tau_0\).
	\end{proposition}
	
	\begin{proof}
		Let \(\mathscr T\) denote the set of times up to which the solution exists and
		remains in the bootstrap regime:
		\[
		\mathscr T
		:=
		\left\{
		\sigma>\tau_0:
		\begin{array}{l}
			\text{\eqref{Eq_self_similar_S} admits a solution on }[\tau_0,\sigma],\\
			\text{and the estimates \eqref{Puns}--\eqref{E1} hold on }[\tau_0,\sigma]
		\end{array}
		\right\}.
		\]
		The local existence theorem ensures that \(\mathscr T\neq\varnothing\). Define
		\[
		\tau_{\max}:=\sup\mathscr T.
		\]
		We prove that \(\tau_{\max}=\infty\).
		
		Suppose, on the contrary, that \(\tau_{\max}<\infty\). Set
		\(s_*:=S^*(\tau_{\max})>0\). For each \(\tau<\tau_{\max}\), the fact that
		\(S(\tau)-S^*(\tau)\in H^K(\R^3)\), with \(K>3\) and the estimate \eqref{nableS2} implies that
		\(S(\tau,r\omega)\to S^*(\tau)\) as \(r\to\infty\). Indeed,
		\eqref{nableS2} gives, uniformly in \(\tau\) and \(\omega\in\mathbb S^2\),
		\begin{equation}\label{CSrlowerboundssss}
			|S(\tau,r\omega)-S^*(\tau)|
			\le C_1\int_r^\infty
			\left\langle\frac{s}{R_0}\right\rangle^{-\Lambda}\,ds
			\le C r^{1-\Lambda}.
		\end{equation}
		Choose \(R\ge R_0\) so large that the last expression is at most
		\(s_*/2\). Hence \(S(\tau,y)\ge s_*/2\) for \(|y|\ge R\), whereas
		\eqref{S,lower,2} gives
		\(S(\tau,y)\ge(\sigma_1/5)\langle R/R_0\rangle^{-(\Lambda-1)}\)
		for \(|y|\le R\). Thus the uniform constant
		\begin{equation}\label{c*lowerboundedsss}
			c_*:=\min\left\{\frac{s_*}{2},
			\frac{\sigma_1}{5}\left\langle\frac{R}{R_0}\right\rangle^{-(\Lambda-1)}
			\right\}>0
		\end{equation}
		is available. The improved bootstrap bounds and
		Proposition~\ref{prop:global-regularity-SS} now yield
		\begin{equation}\label{eq:continuation-uniform-bounds}
			\begin{aligned}
				&\sup_{\tau_0\le \tau<\tau_{\max}}
				\left(
				\|S(\tau)-S^*(\tau)\|_{H^K}^2
				+
				\|V(\tau)\|_{H^K}^2
				\right)\\
				&\quad+
				\int_{\tau_0}^{\tau_{\max}}
				\left(
				\|S(\tau)-S^*(\tau)\|_{H^{K+1}}^2
				+
				\|V(\tau)\|_{H^{K+1}}^2
				\right)\,d\tau
				\le C(\tau_{\max}),\\
				&\sup_{\tau_0\le \tau<\tau_{\max}}
				\left(
				\|\partial_\tau(S-S^*(\tau))\|_{H^{K-2}}^2
				+
				\|V_\tau(\tau)\|_{H^{K-2}}^2
				\right)
				\le C(\tau_{\max}),\\
				&
				\int_{\tau_0}^{\tau_{\max}}
				\left(
				\|\partial_\tau(S-S^*(\tau))\|_{H^{K-1}}^2
				+\|V_\tau(\tau)\|_{H^{K-1}}^2
				\right)\,d\tau
				\le C(\tau_{\max}),\\
				&\sup_{\tau_0\le \tau<\tau_{\max}}E_K(\tau)
				\le \frac{E}{2},\\
				&\inf_{\substack{\tau_0\le\tau<\tau_{\max}\\y\in\R^3}}S(\tau,y)
				\ge c_*>0.
			\end{aligned}
		\end{equation}
		In addition, the remaining improved bootstrap estimates give
		\begin{equation}\label{eq:continuation-improved-bounds}
			\begin{aligned}
				\|P_{\rm uns}(\widetilde{\widetilde S},\widetilde{\widetilde V})(\tau)\|_X
				&\le \sigma_1^{21/20}
				e^{-\frac43\varpi(\tau-\tau_0)},\\
				\max\left\{
				\|\widetilde S(\tau)\|_{L^\infty},
				\|\widetilde V(\tau)\|_{L^\infty}
				\right\}
				&\le \frac{\sigma_0}{100}e^{-\varpi(\tau-\tau_0)},\\
				\max\left\{
				\|\nabla^4\widetilde S(\tau)\|_{L^2(B^c(0,R_1))},
				\|\nabla^4\widetilde V(\tau)\|_{L^2(B^c(0,R_1))}
				\right\}
				&\le \frac{\sigma_0}{2}e^{-\varpi(\tau-\tau_0)},
				\qquad \tau_0\le\tau<\tau_{\max}.
			\end{aligned}
		\end{equation}
		
		Choose a sequence \(\tau_j\rightarrow\tau_{\max}\). By
		\eqref{eq:continuation-uniform-bounds}, weak compactness, and the time
		regularity of \((S,V)\), there exist terminal data
		\[
		S(\tau_{\max},y)-S^*(\tau_{\max})\in H^K(\R^3),
		\qquad
		V(\tau_{\max},y)\in H^K(\R^3),
		\]
		such that, after extracting a subsequence,
		\begin{equation}\label{eq:terminal-convergence}
			\begin{aligned}
				S(\tau_j,y)-S^*(\tau_j)
				&\rightharpoonup
				S(\tau_{\max},y)-S^*(\tau_{\max})
				&&\text{weakly in }H^K(\R^3),\\
				V(\tau_j,y)
				&\rightharpoonup V(\tau_{\max},y)
				&&\text{weakly in }H^K(\R^3),\\
			\end{aligned}
		\end{equation}
		The convergence is strong in lower-order local Sobolev spaces. In particular,
		the pointwise lower bound passes to the terminal state:
		\[
		\inf_{y\in\R^3}S(\tau_{\max},y)\ge c_*>0.
		\]
		The truncated perturbation also has a terminal value in \(X\). Indeed, its
		Duhamel representation \eqref{Eq_Duhamel_representation}, the strong
		continuity of the semigroup, and the time-integrability of the forcing in
		\(X\) show that
		\[
		(\widetilde{\widetilde S},\widetilde{\widetilde V})(\tau)
		\rightarrow
		(\widetilde{\widetilde S}(\tau_{\max}),\widetilde{\widetilde V}(\tau_{\max}))
		\quad\text{strongly in }X
		\quad\text{as }\tau\rightarrow\tau_{\max}.
		\]
		In particular, the finite-dimensional unstable projection converges strongly.
		By weak lower semicontinuity and
		\eqref{eq:continuation-improved-bounds}, the terminal data also satisfy
		\begin{equation}\label{eq:terminal-bootstrap-bounds}
			\begin{aligned}
				&E_K(\tau_{\max})
				\le \frac{E}{2},\qquad
				\|S(\tau_{\max},y)-S^*(\tau_{\max})\|_{H^K}^2
				+
				\|V(\tau_{\max},y)\|_{H^K}^2
				\le C(\tau_{\max}),\\
				&\|P_{\rm uns}(\widetilde{\widetilde S}_{\max},
				\widetilde{\widetilde V}_{\max})\|_X
				\le \sigma_1^{21/20}
				e^{-\frac43\varpi(\tau_{\max}-\tau_0)},\\
				&\max\left\{
				\|\widetilde S(\tau_{\max},y)\|_{L^\infty},
				\|\widetilde V(\tau_{\max},y)\|_{L^\infty}
				\right\}
				\le \frac{\sigma_0}{100}e^{-\varpi(\tau_{\max}-\tau_0)},\\
				&\max\left\{
				\|\nabla^4\widetilde S(\tau_{\max},y)\|_{L^2(B^c(0,R_1))},
				\|\nabla^4\widetilde V(\tau_{\max},y)\|_{L^2(B^c(0,R_1))}
				\right\}
				\le \frac{\sigma_0}{2}e^{-\varpi(\tau_{\max}-\tau_0)} .
			\end{aligned}
		\end{equation}
		
		We now restart \eqref{Eq_self_similar_S} at \(\tau=\tau_{\max}\), with $S(\tau_{\max},y),$
		and $V(\tau_{\max},y)$.
		Since \(S_{\max}\) is separated from zero and the terminal Sobolev norms are
		finite, Theorem~\ref{thm:local-self-similar} provides a number
		\(\iota_*>0\) and a unique solution on
		\[
		[\tau_{\max},\tau_{\max}+\iota_*].
		\]
		The Duhamel formulation of \eqref{truncatedsystem}, initialized by
		\((\widetilde{\widetilde S}_{\max},\widetilde{\widetilde V}_{\max})\),
		continues the auxiliary truncated perturbation on the same interval.
		Uniqueness shows that this solution agrees with the original one on their
		common interval and hence gives a genuine continuation beyond
		\(\tau_{\max}\).
		
		It remains to verify that the bootstrap inequalities persist for a short
		time. The local weighted energy estimate gives, after taking \(0<\iota_1\le\iota_*\) sufficiently small,
		\[
		E_K(\tau)<E,
		\qquad
		\tau_{\max}\le\tau\le\tau_{\max}+\iota_1.
		\]
		Likewise, since
		\[
		S-S^*(\tau)\in C([\tau_{\max},\tau_{\max}+\iota_*];H^K),
		\qquad
		V\in C([\tau_{\max},\tau_{\max}+\iota_*];H^K),
		\]
		Sobolev embedding, the finite dimensionality of the unstable space, and the
		strict improvements in \eqref{eq:terminal-bootstrap-bounds} imply that there
		exists \(0<\iota_2\le\iota_*\) such that
		\begin{equation}\label{eq:extended-lower-bounds}
			\begin{aligned}
				&\inf_{\substack{\tau_{\max}\le\tau\le\tau_{\max}+\iota_2\\y\in\R^3}}
				S(\tau,y)>0,\qquad
				\|P_{\rm uns}(\widetilde{\widetilde S},\widetilde{\widetilde V})(\tau)\|_X
				<\sigma_1e^{-\varpi(\tau-\tau_0)},\\
				&\max\left\{
				\|\widetilde S(\tau)\|_{L^\infty},
				\|\widetilde V(\tau)\|_{L^\infty}
				\right\}
				<\frac{\sigma_0}{50}e^{-\varpi(\tau-\tau_0)},\\
				&\max\left\{
				\|\nabla^4\widetilde S(\tau)\|_{L^2(B^c(0,R_1))},
				\|\nabla^4\widetilde V(\tau)\|_{L^2(B^c(0,R_1))}
				\right\}
				<\sigma_0e^{-\varpi(\tau-\tau_0)}.
			\end{aligned}
		\end{equation}
		and all the remaining lower-order bootstrap estimates continue to hold on the
		same interval.
		
		Taking
		\[
		\iota:=\min\{\iota_1,\iota_2\}>0,
		\]
		we conclude that the solution exists on
		\[
		[\tau_0,\tau_{\max}+\iota]
		\]
		and still satisfies \eqref{Puns}--\eqref{E1}. This
		contradicts the definition of \(\tau_{\max}\). Hence the assumption
		\(\tau_{\max}<\infty\) is impossible, and therefore
		\[
		\tau_{\max}=\infty.
		\]
		Thus the solution of \eqref{Eq_self_similar_S} extends globally in the
		self-similar time variable.
	\end{proof}
	
	\section{Proof of the Implosion}
	\label{subsec:return-eulerian}
	\subsection{Proof of Theorem \ref{Thm_Main_Blowup_Profiles}}
	\begin{proof}
		We now transfer the global solution of the self-similar system
		\eqref{Eq_self_similar_S} back to the Eulerian coordinates and verify the
		formation of an implosion singularity for the original system. Recall that
		\[
		\tau=-\frac{\log(T-t)}{\Lambda},
		\qquad
		y=\frac{x}{(T-t)^{1/\Lambda}},
		\qquad
		\frac{\rho^\delta(t,x)}{\delta}
		=
		\frac{S(\tau,y)}
		{\Lambda(T-t)^{1-\frac1\Lambda}},
		\qquad
		v(t,x)
		=
		\frac{V(\tau,y)}
		{\Lambda(T-t)^{1-\frac1\Lambda}}.
		\]
		We discuss separately the two parameter regimes in the statement of the main
		theorem.
		
		\smallskip
		\noindent\textbf{Case 1: condition \((P_1)\).}
		Fix
		\[
		1<\gamma<1+\frac{2}{\sqrt3}.
		\]
		By the existence theorem for the self-similar profiles, one may choose
		\[
		1<\Lambda<\Lambda^*(\gamma)
		\]
		such that the stationary profile system admits a smooth spherically symmetric
		solution \((\bar S,\bar V)\). The restrictions on the viscosity
		exponent in \((P_1)\) also ensure that
		\[
		f_1
		=
		\frac{(1-\alpha)(\Lambda-1)}{\delta}+\Lambda-2
		=
		\frac{2(1-\alpha)(\Lambda-1)}{\gamma-1}+\Lambda-2
		>0.
		\]
		Hence all the hypotheses of the global stability result for
		\eqref{Eq_self_similar_S} are fulfilled.
		
		Let $(\widetilde S_0^*,\widetilde V_0^*)$ be chosen as in
		\eqref{eq:conditions_for_tilde}. By selecting the unstable coefficients as in the finite
		dimensional argument above, we take the initial data in the form
		\[
		\begin{aligned}
			S_0
			&=
			\widehat\chi\,\bar S+\widetilde S_0^*
			+\sum_{j=1}^{N}\widehat k_j\Phi_{j,S},\\
			V_0
			&=
			\widehat\chi\,\bar V+\widetilde V_0^*
			+\sum_{j=1}^{N}\widehat k_j\Phi_{j,V},
		\end{aligned}
		\]
		where \(\{(\Phi_{j,S},\Phi_{j,V})\}_{j=1}^N\) is a basis of the unstable
		subspace. It is also chosen with the weighted regularity and smallness used in the
		bootstrap argument. After adding the selected (sufficiently small) unstable
		component, the full datum \(S_0\) has the lower bound; hence it also meets the local
		compatibility conditions \eqref{eq:local-initial-data}. The resulting solution
		exists for every \(\tau\ge\tau_0\) and can be
		decomposed as
		\[
		S(\tau,y)=\widehat\chi(\tau,y)\bar S(y)+\widetilde S(\tau,y),
		\qquad
		V(\tau,y)=\widehat\chi(\tau,y)\bar V(y)+\widetilde V(\tau,y).
		\]
		Moreover, the global stability estimate gives
		\[
		\begin{aligned}
			\|\widetilde S(\tau)\|_{L^\infty}
			+\|\widetilde V(\tau)\|_{L^\infty}
			&\le
			C\sigma_0e^{-\varpi(\tau-\tau_0)}\\
			&=
			C\sigma_0
			\left(\frac{T-t}{T}\right)^{\varpi/\Lambda},
			\qquad \tau\ge\tau_0.
		\end{aligned}
		\]
		
		Since
		\[
		\widehat\chi(\tau,y)=\chi(e^{-\tau}y)
		\qquad\text{and}\qquad
		e^{-\tau}y=x,
		\]
		the solution in the physical coordinates is given by
		\[
		\begin{aligned}
			\frac{\rho^\delta(t,x)}{\delta}
			&=
			\frac{1}{\Lambda(T-t)^{1-\frac1\Lambda}}
			\left[
			\widehat\chi(x)\,
			\bar S\left(
			\frac{x}{(T-t)^{1/\Lambda}}
			\right)
			+
			\widetilde S\left(
			-\frac{\log(T-t)}{\Lambda},
			\frac{x}{(T-t)^{1/\Lambda}}
			\right)
			\right],\\
			v(t,x)
			&=
			\frac{1}{\Lambda(T-t)^{1-\frac1\Lambda}}
			\left[
			\widehat\chi(x)\,
			\bar V\left(
			\frac{x}{(T-t)^{1/\Lambda}}
			\right)
			+
			\widetilde V\left(
			-\frac{\log(T-t)}{\Lambda},
			\frac{x}{(T-t)^{1/\Lambda}}
			\right)
			\right].
		\end{aligned}
		\]
		Equivalently,
		\[
		\begin{aligned}
			\left(
			\frac{\Lambda(T-t)^{1-\frac1\Lambda}}{\delta}
			\right)^{1/\delta}
			\rho(t,x)
			&=
			\left[
			\widehat\chi(x)\,
			\bar S\left(
			\frac{x}{(T-t)^{1/\Lambda}}
			\right)
			+
			\widetilde S\left(
			-\frac{\log(T-t)}{\Lambda},
			\frac{x}{(T-t)^{1/\Lambda}}
			\right)
			\right]^{1/\delta},\\
			\Lambda(T-t)^{1-\frac1\Lambda}v(t,x)
			&=
			\widehat\chi(x)\,
			\bar V\left(
			\frac{x}{(T-t)^{1/\Lambda}}
			\right)
			+
			\widetilde V\left(
			-\frac{\log(T-t)}{\Lambda},
			\frac{x}{(T-t)^{1/\Lambda}}
			\right).
		\end{aligned}
		\]
		Because the perturbation tends to zero in \(L^\infty\) as \(t\to T^-\), the
		leading-order behavior is determined by
		\((\bar S,\bar V)\). In particular,
		\[
		\begin{aligned}
			\lim_{t\to T^-}
			\left(
			\frac{\Lambda(T-t)^{1-\frac1\Lambda}}{\delta}
			\right)^{1/\delta}
			\rho(t,x)
			&=
			\chi(x)^{1/\delta}
			\lim_{t\to T^-}
			\bar S\left(
			\frac{x}{(T-t)^{1/\Lambda}}
			\right)^{1/\delta},\\
			\lim_{t\to T^-}
			\Lambda(T-t)^{1-\frac1\Lambda}v(t,x)
			&=
			\chi(x)
			\lim_{t\to T^-}
			\bar V\left(
			\frac{x}{(T-t)^{1/\Lambda}}
			\right),
		\end{aligned}
		\]
		whenever the profile limits on the right-hand sides are understood in the
		corresponding asymptotic sense.
		
		Since \(\Lambda>1\),
		\[
		(T-t)^{-(1-\frac1\Lambda)}\rightarrow\infty
		\qquad\text{as }t\to T^-.
		\]
		The nontrivial behavior of the profile near the self-similar core therefore
		implies
		\[
		\rho(t,0)\rightarrow\infty
		\qquad\text{as }t\to T^-.
		\]
		Furthermore, for any \(r>0\), the shrinking self-similar region
		\[
		|x|\lesssim (T-t)^{1/\Lambda}
		\]
		is eventually contained in \(B_r(0)\). Since \(\overline V\) is nontrivial,
		\[
		\sup_{|x|<r}|v(t,x)|\rightarrow\infty
		\qquad\text{as }t\to T^-.
		\]
		Thus the Eulerian solution develops an implosion singularity at the prescribed
		time \(T\).
		
		\smallskip
		\noindent\textbf{Case 2: condition \((P_2)\).}
		Let
		\[
		\gamma\in
		\left(1,1+\frac{2}{\sqrt3}\right)\setminus\mathcal J,
		\]
		where \(\mathcal J\) is the exceptional set arising in the construction of the
		stationary profiles. In this regime there exists a sequence
		\(\{\Lambda_n\}_{n\ge1}\) such that
		\[
		1<\Lambda_n<\Lambda^*(\gamma),
		\qquad
		\Lambda_n\rightarrow\Lambda^*(\gamma),
		\]
		and, for each \(n\), the stationary system associated with \(\Lambda_n\)
		admits a smooth spherically symmetric profile.
		
		For a fixed \(\Lambda_n\), the decay coefficient in
		\eqref{Eq_self_similar_S} is positive precisely when
		\[
		f_1(\gamma,\alpha,\Lambda_n)
		=
		\frac{2(1-\alpha)(\Lambda_n-1)}{\gamma-1}
		+\Lambda_n-2
		>0.
		\]
		Equivalently, it is sufficient to require
		\[
		0<\alpha<
		\frac{\Lambda_n(\gamma+1)-2\gamma}
		{2(\Lambda_n-1)}.
		\]
		Define
		\[
		\alpha^*(\gamma)
		:=
		\frac{\Lambda^*(\gamma)(\gamma+1)-2\gamma}
		{2(\Lambda^*(\gamma)-1)}<\frac{1}{2}.
		\]
		Since \(\Lambda_n\to\Lambda^*(\gamma)\), for every
		\[
		0<\alpha<\alpha^*(\gamma),
		\]
		one can choose \(n\) sufficiently large so that
		\[
		0<\alpha<
		\frac{\Lambda_n(\gamma+1)-2\gamma}
		{2(\Lambda_n-1)}.
		\]
		Taking \(\Lambda=\Lambda_n\), we then have \(f_1>0\), and all the hypotheses
		of the global stability theorem are satisfied. The construction of the
		initial data, the control of the unstable modes, and the return to the
		Eulerian coordinates proceed exactly as in Case~1. Consequently, the
		corresponding solution \((\rho,v)\) also forms an implosion singularity at
		time \(T\).
	\end{proof}
	\subsection{Proof of Corollary \ref{Thm_density_blowup_original_system}}
	\begin{proof}
		Let $(\rho,v)$ be the solution constructed in
		Theorem~\ref{Thm_Main_Blowup_Profiles}. Define the physical velocity by
		\[
		u=v-d\alpha\rho^{\alpha-2}\nabla\rho.
		\]
		Since $\rho$ remains strictly positive and $(\rho,v)$ is smooth for every
		$t<T$, the pair $(\rho,u)$ is a smooth solution of the original system on
		$[0,T)$. The lower bound
		\eqref{Eq_original_initial_density_positive} follows from the choice of
		the initial data. Finally, Theorem~\ref{Thm_Main_Blowup_Profiles} gives
		\[
		\lim_{t\to T^-}\rho(t,0)=+\infty.
		\]
		Hence the solution of the original system develops a finite-time
		singularity at $t=T$.
	\end{proof}
	\section{Appendix}
	\subsection{Basic Lemmas}
	\begin{lemma}[Weighted Gagliardo--Nirenberg inequality]
		\label{lem:weighted-GN-new}
		Let \(f:\R^3\to\R^3\), and let \(1\le i\le l\) be integers. For \(j\ge0\), set
		\[
		I_j=[2^j,2^{j+1}],
		\qquad
		I_{-1}=[0,1].
		\]
		Let \(\varphi,\psi:\R^3\to(0,\infty)\) be two weights which are essentially
		constant on each dyadic annulus, in the following sense: there exist positive
		numbers \(\varphi_j,\psi_j\) and a constant \(C_*>0\) such that, whenever
		\(|y|\in I_j\),
		\[
		C_*^{-1}\varphi_j\le \varphi(y)\le C_*\varphi_j,
		\qquad
		C_*^{-1}\psi_j\le \psi(y)\le C_*\psi_j.
		\]
		Assume also that the weights are normalized near the origin:
		\[
		\varphi_{-1}=\psi_{-1}=\varphi_0=\psi_0=1.
		\]
		Let \(p,q,\bar r,\theta\) satisfy
		\begin{equation}\label{eq:weighted-GN-balance-new}
			\frac1{\bar r}
			=
			\frac{i}{3}
			+
			\theta\left(\frac1q-\frac{l}{3}\right)
			+
			\frac{1-\theta}{p},
			\qquad
			\frac{i}{l}\le \theta<1.
		\end{equation}
		
		If \(\bar r=\infty\), then
		\begin{equation}\label{eq:weighted-GN-Linfty-new}
			\begin{aligned}
				|\nabla^i f(y)|
				\le C
				\Big(
				&\|\psi^l f\|_{L^p}^{1-\theta}
				\|\varphi^l\nabla^l f\|_{L^q}^{\theta}
				\psi(y)^{-l(1-\theta)}
				\varphi(y)^{-l\theta}  \\
				&+
				\|\psi^l f\|_{L^p}
				\langle y\rangle^{3\theta(\frac1q-\frac1p)-l\theta}
				\psi(y)^{-l}
				\Big).
			\end{aligned}
		\end{equation}
		In particular, when \(p=\infty\), \(q=2\), and \(\psi\equiv1\), this reduces to
		\begin{equation}\label{eq:weighted-GN-Linfty-special-new}
			|\nabla^i f(y)|
			\le
			C
			\left(
			\|f\|_{L^\infty}^{1-\frac{i}{l-3/2}}
			\|\varphi^l\nabla^l f\|_{L^2}^{\frac{i}{l-3/2}}
			\varphi(y)^{-l\frac{i}{l-3/2}}
			+
			\|f\|_{L^\infty}\langle y\rangle^{-i}
			\right).
		\end{equation}
		
		If \(1\le \bar r<\infty\), assume in addition that
		\begin{equation}\label{eq:weighted-GN-extra-new}
			\left(
			\frac{\varphi(y)}{\langle y\rangle\psi(y)}
			\right)^{l\theta}
			\langle y\rangle^{3\theta(\frac1q-\frac1p)}
			\le C
		\end{equation}
		uniformly in \(y\). Then, for every \(\varepsilon_0>0\),
		\begin{equation}\label{eq:weighted-GN-finite-r-new}
			\left\|
			\langle y\rangle^{-\varepsilon_0}
			\psi^{l(1-\theta)}
			\varphi^{l\theta}
			\nabla^i f
			\right\|_{L^{\bar r}}
			\le
			C
			\left(
			\|\psi^l f\|_{L^p}^{1-\theta}
			\|\varphi^l\nabla^l f\|_{L^q}^{\theta}
			+
			\|\psi^l f\|_{L^p}
			\right).
		\end{equation}
		Here the constant \(C\) depends only on the parameters appearing in the
		assumptions, namely \(p,q,i,l,\theta,\bar r,\varepsilon_0\) and \(C_*\), but is
		independent of \(f,\varphi,\psi\).
	\end{lemma}
	\begin{lemma}[Exterior Gagliardo--Nirenberg inequality]
		\label{lem:exterior-GN-new}
		Let \(C_0\ge1\), and let \(f:\R^3\to\R^3\) satisfy \(f\in H^l(\R^3)\). Suppose
		that \(1\le i\le l\) and that \(p,q,\bar r,\theta\) obey
		\begin{equation}\label{eq:exterior-GN-balance-new}
			\frac1{\bar r}
			=
			\frac{i}{3}
			+
			\theta\left(\frac1q-\frac{l}{3}\right)
			+
			\frac{1-\theta}{p},
			\qquad
			\frac{i}{l}\le \theta<1,
			\qquad
			\frac{l}{3}\ge \frac1q-\frac1p .
		\end{equation}
		Then
		\begin{equation}\label{eq:exterior-GN-new}
			\|\nabla^i f\|_{L^{\bar r}(B^c(0,C_0))}
			\le
			C
			\left(
			\|f\|_{L^p(B^c(0,C_0))}^{1-\theta}
			\|\nabla^l f\|_{L^q(B^c(0,C_0))}^{\theta}
			+
			\|f\|_{L^p(B^c(0,C_0))}
			\right).
		\end{equation}
		The constant \(C\) depends only on
		\[
		p,q,i,l,\theta,\bar r,
		\]
		and is independent of both \(f\) and \(C_0\).
	\end{lemma}
	\subsection{Smooth Self-Similar Euler Profiles}
	\label{sec:self-similar-euler-profiles}
	
	In this section, we recall the existence and the estimates of the smooth
	self-similar Euler profiles used in the stability analysis. Let
	\[
	\delta=\frac{\gamma-1}{2},
	\qquad \gamma>1.
	\]
	We consider the stationary self-similar Euler system
	\begin{equation}\label{Eq_self_similar_profile_system_1}
		\left\{
		\begin{aligned}
			(\Lambda-1)\overline S
			+(y+\overline V)\cdot\nabla\overline S
			+\delta\overline S\,\operatorname{div}\overline V
			&=0,\\
			(\Lambda-1)\overline V
			+(y+\overline V)\cdot\nabla\overline V
			+\delta\overline S\,\nabla\overline S
			&=0.
		\end{aligned}
		\right.
	\end{equation}
	The profile is assumed to be spherically symmetric and therefore has the form
	\[
	\overline S(y)=\overline S(r),
	\qquad
	\overline V(y)
	=
	\overline{\mathcal V}(r)\frac{y}{r},
	\qquad
	r=|y|.
	\]
	
	The existence theory developed in
	\cite{Buckmaster-Cao-Labora-Gomez-Serrano,Merle-Raphael-Rodnianski-Szeftel-2022-I,Cao-Labora-Gomez-Serrano-Shi-Staffilani-2023,Shao-Wang-Wei-Zhang} gives the following result.
	
	\begin{lemma}[Existence of smooth self-similar profiles]
		\label{thm:existence_profiles}
		There exists an at most countable exceptional set
		\[
		\mathcal J\subset(1,\infty),
		\]
		whose possible accumulation points belong to
		\[
		\left\{1,\frac53,\infty\right\},
		\qquad
		\left\{\frac75,\frac53\right\}\cap\mathcal J=\varnothing,
		\]
		such that the following assertions hold.
		
		\begin{enumerate}
			\item[\textnormal{(i)}]
			For every
			\[
			\gamma\in(1,\infty)\setminus\mathcal J,
			\]
			there exists a sequence of admissible scaling parameters
			\(\{\Lambda_n\}_{n\ge1}\) satisfying
			\[
			1<\Lambda_n<\Lambda^*(\gamma),
			\qquad
			\Lambda_n\rightarrow\Lambda^*(\gamma),
			\]
			for which system \eqref{Eq_self_similar_profile_system_1} admits a smooth,
			spherically symmetric solution.
			
			\item[\textnormal{(ii)}]
			For every \(\gamma>1\), there exists a nonempty set
			\[
			I_\gamma\subset(1,\Lambda^*(\gamma))
			\]
			of admissible scaling parameters such that, for at least one
			\(\Lambda\in I_\gamma\), system
			\eqref{Eq_self_similar_profile_system_1} possesses a smooth, spherically
			symmetric solution.
		\end{enumerate}
		
		In particular, the above conclusions apply in the range relevant to the
		present paper,
		\[
		1<\gamma<1+\frac{2}{\sqrt3},
		\]
		with the exceptional set \(\mathcal J\) excluded whenever the discrete
		sequence in \textnormal{(i)} is used.
	\end{lemma}
	
	\begin{proof}
		The existence of the admissible scaling parameters and the corresponding
		smooth profiles follows from
		\cite[Theorem~1.2]{Buckmaster-Cao-Labora-Gomez-Serrano},
		\cite[Theorem~1.3]{Merle-Raphael-Rodnianski-Szeftel-2022-I}, and
		\cite[Theorem~1.1]{Shao-Wang-Wei-Zhang}. We omit the phase-plane construction here.
	\end{proof}
	
	The profile can be chosen to be positive and regular at the origin. Its
	properties needed below are summarized as follows.
	
	\begin{lemma}[Positivity and decay of the profile]
		\label{prop:profiles-R}
		Let \((\bar S,\bar V)\) be a smooth profile obtained in
		Lemma~\ref{thm:existence_profiles}. Then
		\begin{equation}\label{eq:profile-origin}
			\bar S(0)=S_*>0,
			\qquad
			\bar V(0)=0,
		\end{equation}
		and
		\begin{equation}\label{eq:profile-infinity}
			(\bar S,\bar V)(y)\longrightarrow(0,0)
			\qquad\text{as }|y|\to\infty.
		\end{equation}
		Moreover, there exist constants \(C_*>1\), \(C_j>0\), and
		\(\bar{\eta}>0\), depending only on the selected profile and on
		\((\gamma,\Lambda)\), such that
		\begin{align}
			\bar S(r)
			&\ge
			C_*^{-1}\langle r\rangle^{-(\Lambda-1)},
			\label{eq:profile-lower-bound}\\
			|\nabla^j\bar S(y)|
			+
			|\nabla^j\bar V(y)|
			&\le
			C_j\langle y\rangle^{-(\Lambda-1)-j},
			\qquad j\ge0,
			\label{eq:profile-decay}\\
			1+\partial_r\bar{\mathcal V}(r)
			-\delta|\partial_r\bar S(r)|
			&\ge\bar{\eta},
			\label{eq:profile-radial-repulsivity}\\
			1+\frac{\overline{\mathcal V}(r)}{r}
			-\delta|\partial_r\bar S(r)|
			&\ge\bar{\eta}.
			\label{eq:profile-angular-repulsivity}
		\end{align}
		In addition, for \(r>1\),
		\begin{equation}\label{eq:profile-outgoing-property}
			r+\bar{\mathcal V}(r)-\delta\bar S(r)
			\ge
			(r-1)\bar{\eta}.
		\end{equation}
	\end{lemma}
	
	\begin{proof}
		The positivity and decay estimates follow from the asymptotic behavior of the
		smooth profile at the origin and at infinity. The repulsivity estimates are
		consequences of the smooth crossing of the sonic point; see
		\cite{Cao-Labora-Gomez-Serrano-Shi-Staffilani-2023,Shao-Wang-Wei-Zhang}.
	\end{proof}
	
	\subsection{Spectral Analysis of Truncated Linear Operators}
	\label{sec:truncated-linear-operator}
	
	This section records the spectral facts concerning the truncated linearized
	operator
	\[
	\mathcal L=(\mathcal L_s,\mathcal L_v)
	\]
	introduced in \eqref{Ls,LvSSVV}. These results provide an invariant splitting
	of the phase space into a finite-dimensional unstable component and an
	exponentially decaying stable component.
	Let \(\widetilde\eta\in(0,1)\) denote the auxiliary cutoff tolerance used in
	the truncation parameters below.
	
	\begin{lemma}[\cite{Cao-Labora-Gomez-Serrano-Shi-Staffilani-2023}, Lemma 2.9][Spectral decomposition of \(\mathcal L\)]
		\label{lem:truncated-spectral-decomposition}
		Let \(\mathfrak c_g\in(0,1)\). Assume that the integers \(m\) and \(J\) satisfy
		\[
		m\ge C\widetilde\eta^{-1},
		\qquad
		J\ge 2m,
		\]
		where \(C>1\) is sufficiently large and independent of
		\(m,J,\widetilde\eta\). Then the following statements hold.
		
		\begin{enumerate}
			
			\item Define
			\[
			\Sigma_{\rm u}
			:=
			\sigma(\mathcal L)\cap
			\left\{
			z\in\mathbb C:
			\Re z>-\frac{\mathfrak c_g}{2}
			\right\},
			\]
			where $\Re(\lambda)$ denotes the real part of $\lambda$. The set \(\Sigma_{\rm u}\) contains only finitely many points, each of which is
			an eigenvalue of \(\mathcal L\) with finite algebraic multiplicity. If
			\(q_\lambda\) denotes the least positive integer for which
			\[
			\ker(\mathcal L-\lambda)^{q_\lambda}
			=
			\ker(\mathcal L-\lambda)^{q_\lambda+1},
			\]
			then
			\begin{equation}\label{eq:new-unstable-space}
				X_{\rm u}
				:=
				\bigoplus_{\lambda\in\Sigma_{\rm u}}
				\ker(\mathcal L-\lambda)^{q_\lambda}
			\end{equation}
			is finite-dimensional.
			
			\item Let \(\mathcal L^*\) be the adjoint of \(\mathcal L\), and set
			\[
			\Sigma_{\rm u}^*
			:=
			\sigma(\mathcal L^*)\cap
			\left\{
			z\in\mathbb C:
			\Re z>-\frac{\mathfrak c_g}{2}
			\right\}.
			\]
			Writing \(q_\lambda^*\) for the least positive integer such that
			\[
			\ker(\mathcal L^*-\lambda)^{q_\lambda^*}
			=
			\ker(\mathcal L^*-\lambda)^{q_{\lambda}^*+1},
			\]
			define
			\begin{equation}\label{eq:new-stable-space}
				X_{\rm s}
				:=
				\left(
				\bigoplus_{\lambda\in\Sigma_{\rm u}^*}
				\ker(\mathcal L^*-\lambda)^{q_\lambda^*}
				\right)^\perp .
			\end{equation}
			Then \(X_{\rm s}\) and \(X_{\rm u}\) are invariant under \(\mathcal L\), and
			\[
			X=X_{\rm s}\oplus X_{\rm u}.
			\]
			Moreover,
			\[
			\Sigma_{\rm u}^*=\overline{\Sigma_{\rm u}},
			\qquad
			q_{\bar\lambda}^*=q_\lambda .
			\]
			
			\item Every eigenvalue of
			\(\mathcal L|_{X_{\rm u}}\) has real part greater than
			\(-\mathfrak c_g/2\). If
			\[
			N:=\dim X_{\rm u},
			\]
			then one may choose a basis of \(X_{\rm u}\) in which the restricted operator
			has block-diagonal form
			\[
			\mathcal L|_{X_{\rm u}}
			=
			\operatorname{diag}(\mathbb J_1,\ldots,\mathbb J_L),
			\qquad
			\sum_{j=1}^{L}\dim\mathbb J_j=N,
			\]
			where
			\[
			\mathbb J_j
			=
			\begin{pmatrix}
				\lambda_j & \frac{\mathfrak c_g}{10} & & 0\\
				&\lambda_j&\ddots&\\
				&&\ddots&\frac{\mathfrak c_g}{10}\\
				0&&&\lambda_j
			\end{pmatrix}.
			\]
			Consequently,
			\begin{equation}\label{eq:new-unstable-lower-bound}
				\Re\bigl(\overline z^{\,\top}
				(\mathcal L|_{X_{\rm u}})z\bigr)
				\ge
				-\frac{3\mathfrak c_g}{5}|z|^2,
				\qquad z\in\mathbb C^N.
			\end{equation}
			On the stable subspace, the semigroup generated by \(\mathcal L\) satisfies
			\begin{equation}\label{eq:new-stable-semigroup}
				\|e^{\tau\mathcal L}h\|_X
				\le
				e^{-\frac{\mathfrak c_g}{2}\tau}\|h\|_X,
				\qquad
				h\in X_{\rm s},\quad \tau\ge0.
			\end{equation}
		\end{enumerate}
	\end{lemma}
	
	The unstable subspace enjoys additional regularity once the truncation
	parameters are chosen sufficiently large.
	
	\begin{lemma}[\cite{Cao-Labora-Gomez-Serrano-Shi-Staffilani-2023}][Smoothness of the unstable directions]
		\label{lem:new-smooth-unstable-space}
		Fix \(\mathfrak c_g\in(0,1)\). There exist thresholds
		\[
		m_*=m_*(\widetilde\eta),
		\qquad
		J_*=J_*(R_1,\mathfrak c_g,\widetilde\eta,
		\overline S,\overline V)
		\]
		such that, whenever
		\[
		m\ge m_*,
		\qquad
		J\ge J_*,
		\]
		every element of the unstable space \(X_{\rm u}\) defined in
		\eqref{eq:new-unstable-space} is smooth. In particular, \(X_{\rm u}\) admits a
		basis
		\[
		\left\{
		(\Phi_{j,S},\Phi_{j,V})
		\right\}_{j=1}^{N}
		\subset C^\infty(\mathbb R^3)\times C^\infty(\mathbb R^3; \mathbb R^3).
		\]
	\end{lemma}

	Combining the preceding two lemmas gives the form used in the nonlinear
	analysis.
	
	\begin{lemma}[\cite{Cao-Labora-Gomez-Serrano-Shi-Staffilani-2023}, Proposition 1.8][Stable--unstable splitting]
		\label{lem:new-stable-unstable-summary}
		Choose \(m\) and \(J\) according to
		Lemmas~\ref{lem:truncated-spectral-decomposition} and
		\ref{lem:new-smooth-unstable-space}. Then
		\[
		X=X_{\rm s}\oplus X_{\rm u},
		\]
		where both subspaces are invariant under \(\mathcal L\),
		\[
		\dim X_{\rm u}=N<\infty,
		\]
		and \(X_{\rm u}\) is generated by smooth functions
		\[
		X_{\rm u}
		=
		\operatorname{span}
		\left\{
		(\Phi_{j,S},\Phi_{j,V})
		\right\}_{j=1}^{N}.
		\]
		There exists an inner product
		\(\langle\cdot,\cdot\rangle_{\Upsilon}\) on \(X_{\rm u}\), with associated
		norm \(\|\cdot\|_{\Upsilon}\), such that
		\begin{equation}\label{eq:new-splitting-estimates}
			\begin{aligned}
				\Re\langle\mathcal L h,h\rangle_{\Upsilon}
				&\ge
				-\frac{3\mathfrak c_g}{5}\|h\|_{\Upsilon}^2,
				&&h\in X_{\rm u},\\
				\|e^{\tau\mathcal L}g\|_X
				&\le
				e^{-\frac{\mathfrak c_g}{2}\tau}\|g\|_X,
				&&g\in X_{\rm s},\quad \tau\ge0.
			\end{aligned}
		\end{equation}
		Furthermore, because \(X_{\rm u}\) is finite-dimensional, its \(X\)-norm and
		\(\Upsilon\)-norm are equivalent: there is a constant \(C_m\ge1\)
		such that
		\begin{equation}\label{eq:new-unstable-norm-equivalence}
			C_{m}^{-1}\|h\|_X
			\le
			\|h\|_{\Upsilon}
			\le
			C_{m}\|h\|_X,
			\qquad h\in X_{\rm u}.
		\end{equation}
	\end{lemma}
	
	\subsection{Existence of Local-in-time Solutions and Uniqueness}
	In this section, we study the local existence theory for the original
	system, together with weighted gradient estimates.
	
	We consider a Navier--Stokes--Korteweg-type system on $\R^3$:
	\begin{equation}\label{eq:nsk}
		\begin{cases}
			\rho_t + \diver(\rho v) - d\Delta(\rho^\alpha)=0,\\[2mm]
			\rho v_t + \rho v\cdot\nabla v + \nabla\left(\frac{\rho^\gamma}{\gamma}\right)
			= \nu\rho^\alpha\Delta v
			+ \bigl[(\nu-d)+(\alpha-1)(2\nu-d)\bigr]\rho^\alpha\nabla\diver v\\
			\qquad\qquad
			+ \alpha\rho^{\alpha-1}\nabla\rho\cdot
			\bigl[(\nu+d)\nabla v + (\nu-d)(\nabla v)^t
			+ (\alpha-1)(2\nu-d)(\diver v)\I\bigr].
		\end{cases}
	\end{equation}
	The parameters satisfy
	\begin{equation}\label{para3.2}
		0<\alpha\le 1,\qquad \gamma\ge 1,\qquad \nu\ge \varepsilon>0,
	\end{equation}
	and
	\begin{equation}
		d=\nu-\sqrt{\nu^2-\varepsilon^2},\qquad
		\beta:=(\nu-d)+(\alpha-1)(2\nu-d).
	\end{equation}
	We have the following coercivity estimate:
	\begin{equation}\label{ellpiticcondi}
		\nu |\nabla v|^2 + \beta \nabla v : (\nabla v)^{\top} \geqslant \min \Big\{ \alpha(2\nu - d), \, d + (1-\alpha)(2\nu - d) \Big\} |\nabla v|^2.
	\end{equation}
	We prescribe the far-field behavior to be a strictly positive constant state $\rstar > 0$, while the initial data satisfy
	\begin{equation}
		\rho_0-\rstar\in H^K(\R^3),\qquad v_0\in H^{K}(\R^3),\qquad K\ge 6,
	\end{equation}
	and
	\begin{equation}\label{initial3.5}
		0<\ubar\le \rho_0(x)\le \obar,\qquad x\in\R^3.
	\end{equation} 
	Now we state the local existence result, and then establish the spatial decay and local boundedness of the weighted energy for these solutions to satisfy the initial conditions for the bootstrap argument.
	\begin{proposition}\label{thm:main}
		Under the assumptions \eqref{para3.2}--\eqref{initial3.5} above, there exist $T^*>0$ and a unique local solution $(\rho,v)$ of \eqref{eq:nsk} on $[0,T^*]\times\R^3$ such that
		\begin{equation}\label{localboundsrho}
			\frac{\ubar}{2}\le \rho(t,x)\le 2\obar+1,
			\qquad (t,x)\in [0,T^*]\times\R^3,
		\end{equation}
		and
		\begin{equation}
			\rho-\rstar\in C([0,T^*];H^K(\R^3))\cap L^2(0,T^*;H^{K+1}(\R^3)),
		\end{equation}
		\begin{equation}
			\rho_t\in C([0,T^*];H^{K-2}(\R^3))\cap L^2(0,T^*;H^{K-1}(\R^3)),
		\end{equation}
		\begin{equation}
			v\in C([0,T^*];H^K(\R^3))\cap L^2(0,T^*;H^{K+1}(\R^3)),
		\end{equation}
		\begin{equation}\label{vtspaceHk-2}
			v_t\in L^\infty(0,T^*;H^{K-2}(\R^3))\cap L^2(0,T^*;H^{K-1}(\R^3)).
		\end{equation}
		Moreover, if the initial data satisfy
		\[
		\langle x\rangle^{\Lambda}\nabla\rho_0\in H^2,
		\qquad
		\langle x\rangle^{\Lambda}\nabla v_0\in H^2,
		\]
		then this solution satisfies the following decay estimates
		\begin{equation}\label{decayestx-2}
			|\nabla\rho(t,x)|+|\nabla v(t,x)|
			\le
			C\langle x\rangle^{-\Lambda},
			\qquad 0\le t\le T^*.
		\end{equation}
	\end{proposition}
	
	\subsubsection{Iteration scheme}
	
	Fix $R>0$ and $T>0$, to be chosen later, and define
	\[
	\mathcal{B}_{R,T}:=
	\Bigl\{
	w:
	w(0)=v_0,\ 
	w\in C([0,T];H^{K})\cap L^2(0,T;H^{K+1}),\
	\]
	\[w_t\in L^\infty(0,T;H^{K-2})\cap L^2(0,T;H^{K-1}),
	\]
	\[
	\hspace{1.9cm}
	\norm{w}_{L^\infty(0,T;H^{K})}
	+\norm{w}_{L^2(0,T;H^{K+1})}
	+\norm{w_t}_{L^\infty(0,T;H^{K-2})}
	+\norm{w_t}_{L^2(0,T;H^{K-1})}
	\le R
	\Bigr\}.
	\]
	Choose the initial seed
	\[
	v^0(t):=e^{t\Delta}v_0,\qquad \rho^0:=\rho_0.
	\]
	
	\begin{lemma}[The seed belongs to the velocity class]\label{lem:seed}
		There exists a constant $C_{\mathrm{heat}}>0$, depending only on $K$, such
		that, for every $0<T\le 1$,
		\begin{equation}\label{eq:seed-estimate}
			\norm{v^0}_{L^\infty(0,T;H^{K})}
			+
			\norm{v^0}_{L^2(0,T;H^{K+1})}
			+
			\norm{\partial_t v^0}_{L^\infty(0,T;H^{K-2})}
			+
			\norm{\partial_t v^0}_{L^2(0,T;H^{K-1})}
			\le
			C_{\mathrm{heat}}\norm{v_0}_{H^{K}}.
		\end{equation}
		Consequently, $v^0\in\mathcal B_{R,T}$ whenever
		\[
		R\ge C_{\mathrm{heat}}\norm{v_0}_{H^{K}}.
		\]
	\end{lemma}
	
	\begin{proof}
		Since $v^0=e^{t\Delta}v_0$ solves
		\begin{equation}\label{v0heat}
			\partial_t v^0-\Delta v^0=0,\qquad v^0(0)=v_0,
		\end{equation}
		the standard $L^2$ energy identity gives
		\[
		\frac12\frac{d}{dt}\norm{v^0(t)}_{H^{K}}^2+\norm{\nabla v^0(t)}_{H^{K}}^2=0.
		\]
		After integration in time we obtain
		\[
		\sup_{0\le t\le T}\norm{v^0(t)}_{H^{K}}^2
		+
		2\int_0^T \norm{\nabla v^0(t)}_{H^{K}}^2\,dt
		\le
		\norm{v_0}_{H^{K}}^2.
		\]
		Since
		\[
		\norm{v^0(t)}_{H^{K+1}}^2
		\le
		C\Bigl(
		\norm{v^0(t)}_{H^{K}}^2+\norm{\nabla v^0(t)}_{H^{K}}^2
		\Bigr),
		\]
		we deduce, for $0<T\le 1$,
		\[
		\int_0^T \norm{v^0(t)}_{H^{K+1}}^2\,dt
		\le
		CT\sup_{0\le t\le T}\norm{v^0(t)}_{H^{K}}^2
		+C\int_0^T \norm{\nabla v^0(t)}_{H^{K}}^2\,dt
		\le
		C\norm{v_0}_{H^{K}}^2.
		\]
		It follows from \eqref{v0heat} that
		\[
		\norm{\partial_t v^0(t)}_{H^{K-2}}
		\le
		C\norm{v^0(t)}_{H^{K}}
		\le
		C\norm{v_0}_{H^{K}},
		\]
		and similarly, we obtain
		\[
		\int_0^T \norm{\partial_t v^0(t)}_{H^{K-1}}^2\,dt
		\le
		C\int_0^T \norm{v^0(t)}_{H^{K+1}}^2\,dt
		\le
		C\norm{v_0}_{H^{K}}^2.
		\]
		Combining the above estimates gives \eqref{eq:seed-estimate}.
	\end{proof}
	
	We now linearize the system and proceed by induction to derive the $(n+1)$-th step estimates from the $n$-th step bounds. Given $v^n\in\mathcal{B}_{R,T}$, define $(\rho^{n+1},v^{n+1})$ by
	\begin{equation}\label{eq:rho-iter}
		\rho_t^{n+1}+\diver(\rho^{n+1}v^n)-d\Delta\bigl((\rho^{n+1})^\alpha\bigr)=0,
		\qquad
		\rho^{n+1}(0)=\rho_0,
	\end{equation}
	and
	\begin{equation}\label{eq:v-iter}
		\rho^{n+1}v_t^{n+1}
		+\rho^{n+1}v^n\cdot\nabla v^{n+1}
		+\nabla\left(\frac{(\rho^{n+1})^\gamma}{\gamma}\right)
		=\A_{\rho^{n+1}}v^{n+1},
		\qquad
		v^{n+1}(0)=v_0,
	\end{equation}
	where the operator $\mathcal{A}_{\rho}$, when applied to a function $v$, is given by
	\begin{equation}
		\mathcal{A}_{\rho}v=\text{div} \left( \nu \rho^{\alpha} \nabla v + (\nu - d) \rho^{\alpha} (\nabla v)^{t} + (\alpha - 1)(2\nu - d) \rho^{\alpha} (\text{div} \, v) \I \right)+\alpha d \rho^{\alpha - 1} \nabla \rho \cdot \nabla v,
	\end{equation}
	First, we establish the estimate for $\rho^{n+1}$ via \eqref{eq:rho-iter}.
	\begin{lemma}\label{lem:density}
		Assume $v^n\in\mathcal{B}_{R,T}$. Then, for $T$ sufficiently small depending only on
		\[
		\ubar,\ \obar,\ \rstar,\ \norm{\rho_0-\rstar}_{H^K},\ \norm{v_0}_{H^{K}},\ R,\ K,
		\]
		problem \eqref{eq:rho-iter} admits a unique solution such that
		\[
		\rho^{n+1}-\rstar\in C([0,T];H^K)\cap L^2(0,T;H^{K+1}),
		\qquad
		\rho_t^{n+1}\in C([0,T];H^{K-2})\cap L^2(0,T;H^{K-1}),
		\]
		and
		\begin{equation}\label{eq:rho-apriori}
			\begin{split}
				\sup_{0\le t\le T}\norm{\rho^{n+1}(t)-\rstar}_{H^K}^2
				&+\int_0^T \norm{\rho^{n+1}(t)-\rstar}_{H^{K+1}}^2\,dt\\
				&+\sup_{0\le t\le T}\norm{\rho_t^{n+1}(t)}_{H^{K-2}}^2
				+\int_0^T \norm{\rho_t^{n+1}(t)}_{H^{K-1}}^2\,dt
				\le C_R=:M.
			\end{split}
		\end{equation}
		Moreover,
		\begin{equation}\label{eq:rho-pointwise}
			\frac{\ubar}{2}\le \rho^{n+1}(t,x)\le 2\obar+1
			\qquad\text{on }[0,T]\times\R^3.
		\end{equation}
	\end{lemma}
	
	\begin{proof}
		We split the proof into four steps: construction of the iterate, the $H^K$ energy estimate, the time-derivative estimate, and the pointwise strip bound.
		
		\smallskip
		\noindent\textbf{Step 1: construction by a contraction map.}
		Fix $M>0$ and define
		\[
		\mathcal K_{M,T}:=
		\Bigl\{
		\eta:
		\eta-\rstar\in C([0,T];H^K)\cap L^2(0,T;H^{K+1}),
		\eta_t\in L^2(0,T;H^{K-1}),
		\]
		\[
		\eta(0)=\rho_0,\qquad
		\frac{\ubar}{2}\le \eta(t,x)\le 2\obar+1,
		\]
		\[
		\sup_{0\le t\le T}\norm{\eta(t)-\rstar}_{H^K}^2
		+\int_0^T\Bigl(
		\norm{\eta(t)-\rstar}_{H^{K+1}}^2+\norm{\eta_t(t)}_{H^{K-1}}^2
		\Bigr)\,dt
		\le M
		\Bigr\}.
		\]
		For $\eta\in\mathcal K_{M,T}$, set
		\[
		a_\eta:=d\alpha \eta^{\alpha-1}.
		\]
		Because $\eta$ stays in the strip $[\ubar/2,2\obar+1]$, there exist positive constants $\lambda_*,\Lambda_*$ depending only on $\ubar,\obar,\alpha,c$ such that
		\[
		0<\lambda_*\le a_\eta(t,x)\le \Lambda_*.
		\]
		Moreover, the Sobolev product estimates yield
		\[
		\norm{\nabla a_\eta}_{L^\infty(0,T;H^{K-1})}
		\le C_M.
		\]
		For fixed \(\eta\in\mathcal K_{M,T}\), consider the linear parabolic problem
		\begin{equation}\label{eq:rho-linear-map}
			\rho_t+\diver(\rho v^n)-\diver(a_\eta\nabla \rho)=0,
			\qquad
			\rho(0)=\rho_0 .
		\end{equation}
		Here \(a_\eta\) is uniformly positive and satisfies
		\begin{equation}\label{aetalambda}
			0<\lambda_*\le a_\eta(t,x)\le \Lambda_*,
		\end{equation}
		and
		\[
		\nabla a_\eta\in L^\infty(0,T;H^{K-1})\subset L^\infty(0,T;L^{\infty}),
		\qquad
		v^n\in L^\infty(0,T;H^{K-1}).
		\]
		Since \(K\) is sufficiently large, the coefficients of
		\eqref{eq:rho-linear-map} have enough spatial regularity and the operator
		\[
		-\diver(a_\eta\nabla\cdot)
		\]
		is uniformly elliptic. Therefore, by the standard theory for linear uniformly
		parabolic equations with Sobolev coefficients, there exists a unique solution
		\(\rho\) satisfying
		\[
		\rho-\rstar\in L^\infty(0,T;H^K)\cap L^2(0,T;H^{K+1}),
		\qquad
		\rho_t\in L^2(0,T;H^{K-1}).
		\]
		We therefore define
		\[
		S(\eta):=\rho .
		\]
		
		It remains to note that \(\mathcal K_{M,T}\) is closed in the contraction
		metric \(C([0,T];H^{K-1})\). Indeed, suppose
		\[
		\eta^m\to \eta
		\qquad\text{in } C([0,T];H^{K-1}),
		\]
		and \(\eta^m\in\mathcal K_{M,T}\) for every \(m\). The uniform bounds in the
		definition of \(\mathcal K_{M,T}\) imply weak-* compactness in
		\(L^\infty(0,T;H^K)\) and weak compactness in
		\(L^2(0,T;H^{K+1})\) for \(\eta^m-\rstar\), and in
		\(L^2(0,T;H^{K-1})\) for \(\partial_t\eta^m\). The strong convergence in
		\(C([0,T];H^{K-1})\) identifies the weak limit with \(\eta\). By lower
		semicontinuity, the same bounds hold for \(\eta\).
		
		Furthermore, since
		\[
		\eta-\rstar\in L^2(0,T;H^{K+1}),
		\qquad
		\eta_t\in L^2(0,T;H^{K-1}),
		\]
		we arrive at
		\[
		\eta-\rstar\in C([0,T];H^K).
		\]
		The convergence in \(C([0,T];H^{K-1})\) also implies
		\[
		\eta(0)=\rho_0.
		\]
		Finally, because \(K-1>3/2\), we have
		\[
		H^{K-1}(\mathbb R^3)\hookrightarrow L^\infty(\mathbb R^3).
		\]
		Thus
		\[
		\eta^m\to \eta
		\qquad\text{in } C([0,T];L^\infty(\mathbb R^3)).
		\]
		Since every \(\eta^m\) satisfies the pointwise strip condition
		\[
		\frac{\ubar}{2}\le \eta^m(t,x)\le 2\obar+1,
		\]
		passing to the limit gives
		\[
		\frac{\ubar}{2}\le \eta(t,x)\le 2\obar+1.
		\]
		Therefore \(\eta\in\mathcal K_{M,T}\), and
		\[
		(\mathcal K_{M,T},\|\cdot\|_{C([0,T];H^{K-1})})
		\]
		is a complete metric space.\\
		\smallskip
		\noindent\textbf{Step 2: the $H^K$ estimate for $S(\eta)$.}
		Applying $\partial^\beta$, $|\beta|\le K$, to \eqref{eq:rho-linear-map}, testing by $\partial^\beta(\rho-\rstar)$, and integrating over $\R^3$, we obtain
		\begin{align}
			\frac12\frac{d}{dt}\norm{\partial^\beta (\rho-\rstar)}_{L^2}^2
			+\int_{\R^3} a_\eta |\nabla \partial^\beta \rho|^2\,dx
			=I_\beta+J_\beta,\label{eq:rho-map-energy}
		\end{align}
		where
		\[
		I_\beta:=
		\int_{\R^3}\partial^\beta(\rho v^n)\cdot\nabla \partial^\beta \rho\,dx,
		\]
		and
		\[
		J_\beta:=
		-\int_{\R^3}
		\Bigl(
		\partial^\beta(a_\eta\nabla \rho)-a_\eta \nabla\partial^\beta \rho
		\Bigr)\cdot\nabla\partial^\beta \rho\,dx.
		\]
		
		For $I_\beta$, when $|\beta|\ge 1$, we define
		\[
		\mathcal C_\beta(\rho,v^n)
		:=
		\partial^\beta(\rho v^n)
		-
		v^n\partial^\beta\rho
		-
		\rho\,\partial^\beta v^n .
		\]
		Hence
		\[
		I_\beta=I_{\beta,1}+I_{\beta,2}+I_{\beta,3},
		\]
		with
		\[
		\begin{aligned}
			I_{\beta,1}
			&:=
			\int_{\R^3}
			v^n\partial^\beta\rho\cdot\nabla\partial^\beta\rho\,dx,\\
			I_{\beta,2}
			&:=
			\int_{\R^3}
			\rho\,\partial^\beta v^n\cdot\nabla\partial^\beta\rho\,dx,\\
			I_{\beta,3}
			&:=
			\int_{\R^3}
			\mathcal C_\beta(\rho,v^n)\cdot\nabla\partial^\beta\rho\,dx .
		\end{aligned}
		\]
		
		For the first term, integration by parts gives
		\[
		I_{\beta,1}
		=
		-\frac12
		\int_{\R^3}
		\diver v^n\,|\partial^\beta\rho|^2\,dx .
		\]
		Therefore, by Hölder's inequality and Sobolev embedding,
		\[
		\begin{aligned}
			|I_{\beta,1}|
			&\le
			C\|\nabla v^n\|_{L^3}
			\|\partial^\beta\rho\|_{L^6}
			\|\partial^\beta(\rho-\rstar)\|_{L^2}  \\
			&\le
			C\|v^n\|_{H^{K-1}}
			\|\rho-\rstar\|_{H^{K+1}}
			\|\rho-\rstar\|_{H^K}.
		\end{aligned}
		\]
		Since $v^n\in L^\infty(0,T;H^{K})$ and
		$\|v^n\|_{L^\infty(0,T;H^{K-1})}\le R$, we obtain
		\[
		|I_{\beta,1}|
		\le
		\vartheta\|\rho-\rstar\|_{H^{K+1}}^2
		+
		C_{\vartheta,R}\|\rho-\rstar\|_{H^K}^2 .
		\]
		
		For the second term, using the uniform upper bound of $\rho$, we have
		\[
		\begin{aligned}
			|I_{\beta,2}|
			&\le
			\|\rho\|_{L^\infty}
			\|\partial^\beta v^n\|_{L^2}
			\|\nabla\partial^\beta\rho\|_{L^2}  \\
			&\le
			(C\|\rho-\rstar\|_{H^{K}} +\rstar)
			\|v^n\|_{H^K}
			\|\rho-\rstar\|_{H^{K+1}} .
		\end{aligned}
		\]
		Hence
		\[
		|I_{\beta,2}|
		\le
		\vartheta\|\rho-\rstar\|_{H^{K+1}}^2
		+
		C_{\vartheta}\|\rho-\rstar\|_{H^{K}}^2\|v^n\|_{H^K}^2+C_{\vartheta}\|v^n\|_{H^K}^2.
		\]
		
		For the lower-order remainder, we have the following estimate:
		\[
		\|\mathcal C_\beta(\rho,v^n)\|_{L^2}
		\le
		C
		\|\rho-\rstar\|_{H^K}
		\|v^n\|_{H^{K-1}} .
		\]
		Thus
		\[
		\|\mathcal C_\beta(\rho,v^n)\|_{L^2}
		\le
		C_R\|\rho-\rstar\|_{H^K}.
		\]
		Consequently,
		\[
		\begin{aligned}
			|I_{\beta,3}|
			&\le
			\|\mathcal C_\beta(\rho,v^n)\|_{L^2}
			\|\nabla\partial^\beta\rho\|_{L^2}  \\
			&\le
			C_R
			\|\rho-\rstar\|_{H^K}
			\|\rho-\rstar\|_{H^{K+1}}  \\
			&\le
			\vartheta\|\rho-\rstar\|_{H^{K+1}}^2
			+
			C_{\vartheta,R}\|\rho-\rstar\|_{H^K}^2 .
		\end{aligned}
		\]
		
		Combining the above estimates, we get
		\[
		|I_\beta|
		\le
		3\vartheta\|\rho-\rstar\|_{H^{K+1}}^2
		+
		C_{\vartheta,R}\|\rho-\rstar\|_{H^K}^2
		+
		C_{\vartheta}\|v^n\|_{H^K}^2+C_{\vartheta}\|\rho-\rstar\|_{H^K}^2\|v^n\|_{H^K}^2.
		\]
		The case $\beta=0$ is simpler and satisfies the same bound.
		
		For $J_\beta$, the commutator estimate gives
		\[
		\norm{\partial^\beta(a_\eta\nabla\rho)-a_\eta\nabla\partial^\beta\rho}_{L^2}
		\le C \norm{\nabla a_\eta}_{H^{K-1}}\norm{\rho-\rstar}_{H^K},
		\]
		hence
		\[
		|J_\beta|
		\le \vartheta \norm{\rho-\rstar}_{H^{K+1}}^2
		+C_{\vartheta,M}\norm{\rho-\rstar}_{H^K}^2.
		\]
		Summing over $|\beta|\le K$ and taking $\vartheta>0$ small, we obtain
		\begin{equation}\label{eq:rho-map-HK}
			\frac{d}{dt}\norm{\rho-\rstar}_{H^K}^2
			+c_0 \norm{\rho-\rstar}_{H^{K+1}}^2
			\le C_{M,R}\norm{\rho-\rstar}_{H^K}^2+C\|v^n\|_{H^K}^2+C\|\rho-\rstar\|_{H^K}^2\|v^n\|_{H^K}^2,
		\end{equation}
		where $c_0>0$ depends only on $\lambda_*$. Gr\"onwall's inequality yields
		\[
		\sup_{0\le t\le T}\norm{\rho(t)-\rstar}_{H^K}^2
		+\int_0^T c_0\norm{\rho(t)-\rstar}_{H^{K+1}}^2\,dt
		\le C e^{C_{M,R}T+C_RT}(\norm{\rho_0-\rstar}_{H^K}^2+C_{R}T).
		\]
		
		\smallskip
		\noindent\textbf{Step 3: the $\rho_t$ estimate and self-mapping property.}
		From \eqref{eq:rho-linear-map},
		\[
		\rho_t=\diver(a_\eta\nabla \rho)-\diver(\rho v^n).
		\]
		Using
		\[
		\eta \in\mathcal K_{M,T},\quad \nabla a_\eta\in L^\infty(0,T;H^{K-1}),\quad
		\rho-\rstar\in L^\infty(0,T;H^K)\cap L^2(0,T;H^{K+1}),
		\]
		and
		\[
		v^n\in L^\infty(0,T;H^{K})\cap L^2(0,T;H^{K+1}),
		\]
		we infer
		\[
		\rho_t\in L^\infty(0,T;H^{K-2})\cap L^2(0,T;H^{K-1}),
		\]
		with
		\[
		\sup_{0\le t\le T}\norm{\rho_t(t)}_{H^{K-2}}^2
		+\int_0^T \norm{\rho_t(t)}_{H^{K-1}}^2\,dt
		\le C_{R}.
		\]
		Now choose $M$ large enough depending on $\norm{\rho_0-\rstar}_{H^K}^2$ and $C_R$, and then choose $T$ so small that
		\[
		C e^{C_{M,R}T+C_RT}\norm{\rho_0-\rstar}_{H^K}^2\le \frac{M}{4},
		\quad
		C e^{C_{M,R}T+C_RT}C_RT \le\frac{M}{4} ,\quad C_RT\le \frac{M}{4}.
		\]
		Then
		\[
		\sup_{0\le t\le T}\norm{\rho(t)-\rstar}_{H^K}^2
		+\int_0^T\Bigl(
		\norm{\rho(t)-\rstar}_{H^{K+1}}^2+\norm{\rho_t(t)}_{H^{K-1}}^2
		\Bigr)\,dt
		\le M,
		\]
		so the norm part of $S(\eta)\in \mathcal K_{M,T}$ is established. Furthermore,
		\[
		\rho-\rstar\in L^2(0,T;H^{K+1}),
		\qquad
		\rho_t\in L^2(0,T;H^{K-1}),
		\]
		and therefore yields
		\[
		\rho-\rstar\in C([0,T];H^K).
		\]
		Returning to
		\[
		\rho_t=\diver(a_\eta\nabla \rho)-\diver(\rho v^n),
		\]
		we use $\rho-\rstar\in C([0,T];H^K)$, $v^n\in C([0,T];H^{K-1})$, and the continuity of
		\[
		(\rho,w)\longmapsto -\diver(\rho w)+\diver(a_\eta\nabla \rho)
		\]
		from $H^K(\R^3)\times H^{K-1}(\R^3)$ to $H^{K-2}(\R^3)$ on the fixed strip to
		obtain
		\[
		\rho_t\in C([0,T];H^{K-2}).
		\]
		Moreover, since $K\ge 6$,
		\[
		\norm{\rho(t)-\rho_0}_{L^\infty}
		\le C \norm{\rho(t)-\rho_0}_{H^{K-1}}
		\le C \int_0^t \norm{\rho_s(s)}_{H^{K-1}}\,ds
		\le C T^{1/2} \norm{\rho_t}_{L^2(0,T;H^{K-1})}.
		\]
		Therefore
		\[
		\norm{\rho(t)-\rho_0}_{L^\infty}
		\le C T^{1/2} C_{R}^{1/2}.
		\]
		Shrinking $T$ further so that
		\[
		C T^{1/2} C_{R}^{1/2}
		\le
		\min\left\{\frac{\ubar}{2},\,\obar+1\right\},
		\]
		we get
		\[
		\frac{\ubar}{2}\le \rho(t,x)\le 2\obar+1
		\qquad\text{for }(t,x)\in [0,T]\times \R^3.
		\]
		Hence $S(\eta)\in \mathcal K_{M,T}$.
		
		\smallskip
		\noindent\textbf{Step 4: pointwise strip and contraction.}
		Our goal next is to show that $\mathcal{S}: \eta \mapsto \rho$ defines a contraction mapping. Let $\eta_1,\eta_2\in \mathcal K_{M,T}$ and set
		\[
		\rho_i:=S(\eta_i),\qquad a_i:=d\alpha \eta_i^{\alpha-1},\qquad r:=\rho_1-\rho_2.
		\]
		Then $r(0)=0$ and
		\[
		r_t+\diver(rv^n)-\diver(a_1\nabla r)=\diver((a_1-a_2)\nabla \rho_2).
		\]
		Applying $\partial^\beta$ with $|\beta|\le K-1$ to the equation, testing the resulting identity by $\partial^\beta r$, and integrating over $\mathbb{R}^3$, we obtain the following energy estimate by a similar argument as in Step 2:
		\[
		\frac{d}{dt}\norm{r}_{H^{K-1}}^2
		+c_1 \norm{r}_{H^{K}}^2
		\le C\norm{r}_{H^{K-1}}^2\|v^n\|_{H^K}^2+C_{M,R}\norm{r}_{H^{K-1}}^2
		+C_{M,R}\norm{a_1-a_2}_{H^{K-1}}^2.
		\]
		Here we use that the map $\eta\mapsto a_\eta$ is Lipschitz on the fixed strip:
		\[
		\norm{a_1-a_2}_{H^{K-1}}
		\le C_M \norm{\eta_1-\eta_2}_{H^{K-1}}.
		\]
		Gr\"onwall's inequality and $r(0)=0$ give
		\[
		\sup_{0\le t\le T}\norm{r(t)}_{H^{K-1}}^2
		\le C_{M,R} T e^{C_{M,R}T}
		\sup_{0\le t\le T}\norm{\eta_1(t)-\eta_2(t)}_{H^{K-1}}^2.
		\]
		Shrink $T$ one last time so that
		\[
		C_{M,R} Te^{C_{M,R}T}\le \frac12.
		\]
		Then $S$ is a strict contraction on $\mathcal K_{M,T}$ in the
		$C([0,T];H^{K-1})$ norm. Banach's fixed-point theorem yields a unique fixed
		point $\rho^{n+1}=S(\rho^{n+1})$, which solves \eqref{eq:rho-iter}. The
		definition of $\mathcal K_{M,T}$ gives
		\[
		\rho^{n+1}-\rstar\in C([0,T];H^K)\cap L^2(0,T;H^{K+1}),
		\qquad
		\rho_t^{n+1}\in L^2(0,T;H^{K-1}),
		\]
		while the continuity of $\rho_t^{n+1}$ in $H^{K-2}$ was established above in
		Step 3. Hence the fixed point has the full regularity asserted in the lemma.
	\end{proof}
	
	\begin{lemma}\label{lem:velocity}
		Assume that $\rho^{n+1}$ satisfies the conclusions of Lemma \ref{lem:density}. Then there exists a small time (still denoted by $T$ for convenience) such that problem \eqref{eq:v-iter} admits a unique solution satisfying
		\[
		v^{n+1}\in C([0,T];H^{K})\cap L^2(0,T;H^{K+1}),
		\qquad
		v_t^{n+1}\in L^\infty(0,T;H^{K-2})\cap L^2(0,T;H^{K-1}),
		\]
		and
		\begin{equation}\label{eq:v-apriori}
			\begin{split}
				&\sup_{0\le t\le T}\norm{v^{n+1}(t)}_{H^{K}}^2
				+\int_0^T \norm{v^{n+1}(t)}_{H^{K+1}}^2\,dt\\
				&+\sup_{0\le t\le T}\norm{v_t^{n+1}(t)}_{H^{K-2}}^2
				+\int_0^T \norm{v_t^{n+1}(t)}_{H^{K-1}}^2\,dt
				\le 2C(\norm{v^{n+1}(0)}_{H^{K}}^2+\norm{v_t^{n+1}(0)}_{H^{K-2}}^2).
			\end{split}
		\end{equation}
	\end{lemma}
	
	\begin{proof}
		We divide the argument into two steps.
		By classical theory for linear uniformly parabolic systems with Sobolev
		coefficients, there exists a unique strong solution \(v^{n+1}\) satisfying
		\[
		v^{n+1}\in L^\infty(0,T;H^{K})\cap L^2(0,T;H^{K+1}),
		\]
		and
		\[
		\partial_t v^{n+1}\in
		L^\infty(0,T;H^{K-2})\cap L^2(0,T;H^{K-1}).
		\]
		\smallskip
		\noindent\textbf{Step 1: the combined \(H^K\)-estimate.}
		For every multi-index \(\beta\) with \(|\beta|\le K\), applying
		\(\partial^\beta\) to \eqref{eq:v-iter}, testing the resulting equation by
		\(\partial^\beta v^{n+1}\), and integrating over \(\R^3\), we obtain
		\begin{equation}\label{eq:v-combined-energy}
			\begin{aligned}
				&\int_{\R^3}
				\partial^\beta(\rho^{n+1}v_t^{n+1})
				\cdot\partial^\beta v^{n+1}\,dx \\
				&\quad+
				\int_{\R^3}
				\partial^\beta(\rho^{n+1}v^n\cdot\nabla v^{n+1})
				\cdot\partial^\beta v^{n+1}\,dx \\
				&\quad+
				\int_{\R^3}
				\partial^\beta\nabla\left(\frac{(\rho^{n+1})^\gamma}{\gamma}\right)
				\cdot\partial^\beta v^{n+1}\,dx \\
				&=
				\int_{\R^3}
				\partial^\beta(\A_{\rho^{n+1}}v^{n+1})
				\cdot\partial^\beta v^{n+1}\,dx .
			\end{aligned}
		\end{equation}
		
		We first combine the principal time and transport terms. Writing
		\[
		\begin{aligned}
			\partial^\beta(\rho^{n+1}v_t^{n+1})
			&=
			\rho^{n+1}\partial^\beta v_t^{n+1}
			+
			[\partial^\beta,\rho^{n+1}]v_t^{n+1},\\
			\partial^\beta(\rho^{n+1}v^n\cdot\nabla v^{n+1})
			&=
			\rho^{n+1}v^n\cdot\nabla\partial^\beta v^{n+1}
			+
			[\partial^\beta,\rho^{n+1}v^n\cdot\nabla]v^{n+1},
		\end{aligned}
		\]
		and using
		\[
		\rho_t^{n+1}+\diver(\rho^{n+1}v^n)
		=
		d\Delta\bigl((\rho^{n+1})^\alpha\bigr),
		\]
		we find
		\begin{equation}\label{eq:v-principal-time-transport}
			\begin{aligned}
				&\int_{\R^3}
				\rho^{n+1}\partial^\beta v_t^{n+1}
				\cdot\partial^\beta v^{n+1}\,dx
				+
				\int_{\R^3}
				\rho^{n+1}v^n\cdot\nabla\partial^\beta v^{n+1}
				\cdot\partial^\beta v^{n+1}\,dx\\
				&\quad=
				\frac12\frac{d}{dt}
				\int_{\R^3}
				\rho^{n+1}|\partial^\beta v^{n+1}|^2\,dx
				-\frac d2
				\int_{\R^3}
				\Delta\bigl((\rho^{n+1})^\alpha\bigr)
				|\partial^\beta v^{n+1}|^2\,dx .
			\end{aligned}
		\end{equation}
		Since
		\[
		\|\nabla(\rho^{n+1})^\alpha\|_{L^\infty}
		\le C_M,
		\]
		integration by parts and Young's inequality give
		\begin{equation}\label{eq:density-diffusion-remainder}
			\begin{aligned}
				\left|
				\frac d2
				\int_{\R^3}
				\Delta\bigl((\rho^{n+1})^\alpha\bigr)
				|\partial^\beta v^{n+1}|^2\,dx
				\right|
				&=
				\left|
				d\int_{\R^3}
				\nabla\bigl((\rho^{n+1})^\alpha\bigr)\cdot
				\bigl(\partial^\beta v^{n+1}\cdot
				\nabla\partial^\beta v^{n+1}\bigr)\,dx
				\right|\\
				&\le
				\vartheta
				\|\nabla\partial^\beta v^{n+1}\|_{L^2}^2
				+
				C_{\vartheta,M}
				\|\partial^\beta v^{n+1}\|_{L^2}^2.
			\end{aligned}
		\end{equation}
		
		For the commutator generated by the time derivative, the standard estimate
		gives, for \(1\le|\beta|\le K\),
		\begin{equation}\label{eq:time-density-commutator}
			\begin{aligned}
				\|[\partial^\beta,\rho^{n+1}]v_t^{n+1}\|_{L^2}
				&\le
				C\|\nabla\rho^{n+1}\|_{L^\infty}
				\|v_t^{n+1}\|_{H^{K-1}}
				+
				C\|\rho^{n+1}-\rstar\|_{H^K}
				\|v_t^{n+1}\|_{L^\infty}\\
				&\le
				C_M\left(
				\|v_t^{n+1}\|_{H^{K-1}}
				+
				\|v_t^{n+1}\|_{H^{K-3}}
				\right).
			\end{aligned}
		\end{equation}
		Here we used \(K\ge 6\) and
		\(H^{K-3}(\R^3)\hookrightarrow L^\infty(\R^3)\). Consequently,
		\begin{equation}\label{eq:time-density-commutator-energy}
			\begin{aligned}
				&\sum_{1\le|\beta|\le K}
				\left|
				\int_{\R^3}
				[\partial^\beta,\rho^{n+1}]v_t^{n+1}
				\cdot\partial^\beta v^{n+1}\,dx
				\right|\\
				&\quad\le
				C_M
				\left(
				\|v_t^{n+1}\|_{H^{K-1}}
				+
				\|v_t^{n+1}\|_{H^{K-3}}
				\right)
				\|v^{n+1}\|_{H^K}\\
				&\quad\le
				\kappa\|v_t^{n+1}\|_{H^{K-1}}^2
				+
				C_{\kappa,M}\|v^{n+1}\|_{H^K}^2,
			\end{aligned}
		\end{equation}
		where the bounded \(H^{K-3}\)-norm of \(v_t^{n+1}\) has been absorbed into
		the constant.
		
		The transport commutator is estimated similarly. Using the regularity of
		\(\rho^{n+1}\) and the a priori bounds for \(v^n\), we have
		\begin{equation}\label{eq:transport-commutator-combined}
			\begin{aligned}
				&\sum_{1\le|\beta|\le K}
				\left|
				\int_{\R^3}
				[\partial^\beta,\rho^{n+1}v^n\cdot\nabla]v^{n+1}
				\cdot\partial^\beta v^{n+1}\,dx
				\right|\\
				&\quad\le
				C_{M,R}
				\|v^{n+1}\|_{H^K}
				\left(
				\|v^{n+1}\|_{H^K}
				+
				\|\nabla v^{n+1}\|_{H^K}
				\right)\\
				&\quad\le
				\vartheta\|v^{n+1}\|_{H^{K+1}}^2
				+
				C_{\vartheta,M,R}\|v^{n+1}\|_{H^K}^2.
			\end{aligned}
		\end{equation}
		For \(\beta=0\), the corresponding commutator vanishes.
		
		For the pressure contribution, the fixed-strip condition and the Moser
		composition estimate imply
		\[
		\left\|\nabla\left(\frac{(\rho^{n+1})^\gamma}{\gamma}\right)\right\|_{H^{K-1}}
		\le C_M.
		\]
		Thus, integrating by parts at the top order when necessary, we obtain
		\begin{equation}\label{eq:pressure-combined}
			\begin{aligned}
				&\sum_{|\beta|\le K}
				\left|
				\int_{\R^3}
				\partial^\beta\nabla\left(\frac{(\rho^{n+1})^\gamma}{\gamma}\right)
				\cdot\partial^\beta v^{n+1}\,dx
				\right|\\
				&\quad\le
				C
				\left\|\nabla\left(\frac{(\rho^{n+1})^\gamma}{\gamma}\right)\right\|_{H^{K-1}}
				\|v^{n+1}\|_{H^{K+1}}
				+
				C_M\|v^{n+1}\|_{L^2}\\
				&\quad\le
				\vartheta\|v^{n+1}\|_{H^{K+1}}^2
				+
				C_{\vartheta,M}
				\left(1+\|v^{n+1}\|_{H^K}^2\right).
			\end{aligned}
		\end{equation}
		
		We next consider the elliptic term. Recall that
		\[
		\A_{\rho^{n+1}}
		=
		\LL_{\rho^{n+1}}+\Kop_{\rho^{n+1}},
		\]
		where
		\[
		\begin{aligned}
			\LL_\rho v
			&=
			\diver\left(
			\nu\rho^\alpha\nabla v
			+
			(\nu-d)\rho^\alpha(\nabla v)^\top
			+
			(\alpha-1)(2\nu-d)\rho^\alpha\diver v\,\I
			\right),\\
			\Kop_\rho v
			&=
			\alpha d\rho^{\alpha-1}\nabla\rho\cdot\nabla v.
		\end{aligned}
		\]
		We decompose
		\[
		\partial^\beta(\A_{\rho^{n+1}}v^{n+1})
		=
		\A_{\rho^{n+1}}\partial^\beta v^{n+1}
		+
		[\partial^\beta,\A_{\rho^{n+1}}]v^{n+1}.
		\]
		The ellipticity condition and the lower-order structure of
		\(\Kop_{\rho^{n+1}}\) yield
		\begin{equation}\label{eq:elliptic-principal-combined}
			\begin{aligned}
				&-\sum_{|\beta|\le K}
				\int_{\R^3}
				\A_{\rho^{n+1}}\partial^\beta v^{n+1}
				\cdot\partial^\beta v^{n+1}\,dx\\
				&\quad\ge
				\lambda
				\|\nabla v^{n+1}\|_{H^K}^2
				-
				C_M\|v^{n+1}\|_{H^K}^2.
			\end{aligned}
		\end{equation}
		
		For the second-order commutator, the divergence structure gives
		\[
		\begin{aligned}
			[\partial^\beta,\LL_{\rho^{n+1}}]v^{n+1}
			&=
			\nu\,\partial_i\left(
			[\partial^\beta,(\rho^{n+1})^\alpha]\partial_i v^{n+1}
			\right)\\
			&\quad+
			(\nu-d)\,\partial_i\left(
			[\partial^\beta,(\rho^{n+1})^\alpha]\partial_jv_i^{n+1}
			\right)\\
			&\quad+
			(\alpha-1)(2\nu-d)\,
			\nabla\left(
			[\partial^\beta,(\rho^{n+1})^\alpha]\diver v^{n+1}
			\right).
		\end{aligned}
		\]
		After integration by parts, the commutator estimate gives
		\begin{equation}\label{eq:L-commutator-combined}
			\begin{aligned}
				&\sum_{1\le|\beta|\le K}
				\left|
				\int_{\R^3}
				[\partial^\beta,\LL_{\rho^{n+1}}]v^{n+1}
				\cdot\partial^\beta v^{n+1}\,dx
				\right|\\
				&\quad\le
				C_M
				\|v^{n+1}\|_{H^K}
				\|\nabla v^{n+1}\|_{H^K}\\
				&\quad\le
				\vartheta\|v^{n+1}\|_{H^{K+1}}^2
				+
				C_{\vartheta,M}\|v^{n+1}\|_{H^K}^2.
			\end{aligned}
		\end{equation}
		Similarly, since
		\[
		\Kop_{\rho^{n+1}}v^{n+1}
		=
		\alpha d(\rho^{n+1})^{\alpha-1}
		\nabla\rho^{n+1}\cdot\nabla v^{n+1},
		\]
		we have
		\begin{equation}\label{eq:K-commutator-combined}
			\begin{aligned}
				&\sum_{1\le|\beta|\le K}
				\left|
				\int_{\R^3}
				[\partial^\beta,\Kop_{\rho^{n+1}}]v^{n+1}
				\cdot\partial^\beta v^{n+1}\,dx
				\right|\\
				&\quad\le
				C_M\|v^{n+1}\|_{H^K}^2
				+
				\vartheta\|v^{n+1}\|_{H^{K+1}}^2,
			\end{aligned}
		\end{equation}
		Here for the highest-order term in $\rho^{n+1}$, we integrate by parts to transfer one spatial derivative from $\rho^{n+1}$ to $v^{n+1}$. 
		Combining \eqref{eq:elliptic-principal-combined}--\eqref{eq:K-commutator-combined},
		we conclude that
		\begin{equation}\label{eq:full-elliptic-combined}
			\begin{aligned}
				&-\sum_{|\beta|\le K}
				\int_{\R^3}
				\partial^\beta(\A_{\rho^{n+1}}v^{n+1})
				\cdot\partial^\beta v^{n+1}\,dx\\
				&\quad\ge
				c_0\|v^{n+1}\|_{H^{K+1}}^2
				-
				C_M\|v^{n+1}\|_{H^K}^2,
			\end{aligned}
		\end{equation}
		after choosing \(\vartheta>0\) sufficiently small. Here the \(L^2\)-part of
		the \(H^{K+1}\)-norm is absorbed into the lower-order term on the right.
		
		Finally, summing \eqref{eq:v-combined-energy} over all \(|\beta|\le K\), and
		using \eqref{eq:density-diffusion-remainder},
		\eqref{eq:time-density-commutator-energy},
		\eqref{eq:transport-commutator-combined},
		\eqref{eq:pressure-combined}, and
		\eqref{eq:full-elliptic-combined}, we obtain
		\begin{equation}\label{eq:v-HK-final-combined}
			\begin{aligned}
				&\frac{d}{dt}
				\sum_{|\beta|\le K}
				\int_{\R^3}
				\rho^{n+1}|\partial^\beta v^{n+1}|^2\,dx
				+
				c_1\|v^{n+1}\|_{H^{K+1}}^2\\
				&\quad\le
				C_{\kappa,M,R}
				\left(
				1+\|v^{n+1}\|_{H^K}^2
				\right)
				+
				\kappa_1
				\|v_t^{n+1}\|_{H^{K-1}}^2,
			\end{aligned}
		\end{equation}
		where \(\kappa_1>0\) is a constant multiple of \(\kappa\) and will be chosen
		later. Since \(\rho^{n+1}\) is uniformly bounded from below and above,
		\[
		\sum_{|\beta|\le K}
		\int_{\R^3}
		\rho^{n+1}|\partial^\beta v^{n+1}|^2\,dx
		\sim
		\|v^{n+1}\|_{H^K}^2.
		\]
		Thus \eqref{eq:v-HK-final-combined} provides simultaneously the \(L^2\)- and
		\(H^K\)-estimates for \(v^{n+1}\).
		
		\smallskip
		\noindent\textbf{Step 2: the $v_t$ estimate.}
		Differentiate \eqref{eq:v-iter} in time:
		\begin{align}
			\rho^{n+1}v_{tt}^{n+1}
			-\A_{\rho^{n+1}}v_t^{n+1}
			&=
			-\rho_t^{n+1}v_t^{n+1}
			-\rho_t^{n+1}v^n\cdot\nabla v^{n+1}
			-\rho^{n+1}v_t^n\cdot\nabla v^{n+1}\nonumber\\
			&\quad
			-\rho^{n+1}v^n\cdot\nabla v_t^{n+1}
			-\partial_t\nabla\left(\frac{(\rho^{n+1})^\gamma}{\gamma}\right)
			+(\partial_t\A_{\rho^{n+1}})v^{n+1},
			\label{eq:vt-eq}
		\end{align}
		where $(\partial_t\mathcal{A}_{\rho^{n+1}})v^{n+1}$ denotes the terms in which the time derivative acts only on $\rho^{n+1}$. Applying $\partial^\beta$, $|\beta|\le K-2$, testing by $\partial^\beta v_t^{n+1}$, integrating over $\R^3$, and using the following estimate
		\[
		-\int_{\R^3}
		\mathcal A_{\rho^{n+1}}\partial^\beta v_t^{n+1}\cdot \partial^\beta v_t^{n+1}\,dx
		\ge
		c_3\|\nabla \partial^\beta v_t^{n+1}\|_{L^2}^2-C_M \norm{\partial^\beta v_t^{n+1}}_{L^2}^2,
		\]
		we arrive at 
		\begin{equation}\label{vtJ6Ce}
			\frac12\frac{d}{dt}
			\int_{\R^3}\rho^{n+1}|\partial^\beta v_t^{n+1}|^2\,dx
			+c_3\|\nabla \partial^\beta v_t^{n+1}\|_{L^2}^2
			\le
			\sum_{m=1}^6\mathcal J_{\beta,m}
			+
			\mathcal C_\beta^{A}
			+
			\frac12\int_{\R^3}|\rho_t^{n+1}||\partial^\beta v_t^{n+1}|^2\,dx,
		\end{equation}
		for some $c_3>0$, where
		\[
		\mathcal C_\beta^{A}
		:=
		\left|
		\int_{\R^3}
		\Bigl(
		\partial^\beta(\mathcal A_{\rho^{n+1}}v_t^{n+1})
		-
		\mathcal A_{\rho^{n+1}}\partial^\beta v_t^{n+1}
		\Bigr)
		\cdot \partial^\beta v_t^{n+1}\,dx
		\right|,
		\]
		and the $\mathcal J_{\beta,m}$ correspond to the six right-hand side terms in \eqref{eq:vt-eq}. For $\mathcal J_{\beta,1}$, using $L^6-L^3-L^2$ or $L^2-L^{\infty}-L^2$ estimates, we have
		\[
		\mathcal J_{\beta,1}
		\lesssim
		\norm{\rho_t^{n+1}}_{H^{K-2}}
		\norm{v_t^{n+1}}_{H^{K-2}}
		\norm{v_t^{n+1}}_{H^{K-1}}
		\le
		\vartheta \norm{v_t^{n+1}}_{H^{K-1}}^2
		+C_{\vartheta,M} \norm{v_t^{n+1}}_{H^{K-2}}^2.
		\]
		Similarly, using $L^6-L^\infty-L^2-L^3$ or $L^2-L^6-L^6-L^6$ estimates, we arrive at
		\[
		\mathcal J_{\beta,2}
		\lesssim
		\norm{\rho_t^{n+1}}_{H^{K-2}}
		\norm{v^n}_{H^{K-1}}
		\norm{v^{n+1}}_{H^{K}}
		\norm{v_t^{n+1}}_{H^{K-1}}
		\le
		\vartheta \norm{v_t^{n+1}}_{H^{K-1}}^2
		+C_{\vartheta,R,M} \norm{v^{n+1}}_{H^{K}}^2,
		\]
		and using $L^{\infty}-L^2-L^3-L^6$ estimate, we have
		\[
		\mathcal J_{\beta,3}
		\lesssim
		\norm{\rho^{n+1}}_{H^{K}}\norm{v_t^n}_{H^{K-2}}
		\norm{v^{n+1}}_{H^{K}}
		\norm{v_t^{n+1}}_{H^{K-1}}
		\le
		\vartheta \norm{v_t^{n+1}}_{H^{K-1}}^2
		+C_{\vartheta,R,M} \norm{v^{n+1}}_{H^{K}}^2.
		\]
		For $\mathcal J_{\beta,4}$, using an $L^{\infty}$--$L^{\infty}$--$L^{2}$--$L^{2}$ estimate, we have
		\[
		\mathcal J_{\beta,4}
		\le
		\vartheta \norm{v_t^{n+1}}_{H^{K-1}}^2+C_{\vartheta,R,M} \norm{v_t^{n+1}}_{H^{K-2}}^2.
		\]
		For the pressure term, since
		\[
		\partial_t\nabla\left(\frac{(\rho^{n+1})^\gamma}{\gamma}\right)
		=
		\nabla\bigl((\rho^{n+1})^{\gamma-1}\rho_t^{n+1}\bigr),
		\]
		integrating by parts, we have
		\begin{equation}
			\begin{split}
				\left|\int \partial^{\beta} \partial_t\nabla\left(\frac{(\rho^{n+1})^\gamma}{\gamma}\right) \cdot \partial^{\beta}v_t^{n+1}dx\right|&=\left|\int \partial^{\beta} \partial_t\left(\frac{(\rho^{n+1})^\gamma}{\gamma}\right) \cdot \partial^{\beta}\diver v_t^{n+1}dx\right|\\
				&\le\left\|\partial_t\left(\frac{(\rho^{n+1})^\gamma}{\gamma}\right)\right\|_{H^{K-2}}\|v^{n+1}_t\|_{H^{K-1}}\\
				&\le \vartheta \norm{v_t^{n+1}}_{H^{K-1}}^2+C_{\vartheta,M}\norm{\rho_t^{n+1}}^2_{H^{K-2}}.
			\end{split}
		\end{equation}
		Hence
		\[
		\mathcal J_{\beta,5}
		\le
		\vartheta \norm{v_t^{n+1}}_{H^{K-1}}^2
		+C_M.
		\]
		For the term $\mathcal J_{\beta,6}$, we estimate
		\[
		\mathcal J_{\beta,6}
		:=
		\left|
		\int_{\R^3}
		\partial^\beta\bigl((\partial_t\A_{\rho^{n+1}})v^{n+1}\bigr)
		\cdot \partial^\beta v_t^{n+1}\,dx
		\right|.
		\]
		Recall that
		\[
		\begin{aligned}
			(\partial_t \mathcal{A}_{\rho^{n+1}})v^{n+1}
			&=
			\diver\left(
			\nu ((\rho^{n+1})^\alpha)_t \nabla v^{n+1}
			+
			(\nu-d)((\rho^{n+1})^\alpha)_t(\nabla v^{n+1})^\top
			\right) \\
			&\quad+
			\diver\left(
			(\alpha-1)(2\nu-d)((\rho^{n+1})^\alpha)_t
			\diver v^{n+1}\,\I
			\right) \\
			&\quad+
			\alpha c
			\bigl((\rho^{n+1})^{\alpha-1}\nabla\rho^{n+1}\bigr)_t
			\nabla v^{n+1}.
		\end{aligned}
		\]
		For the divergence terms, integrating by parts gives
		\[
		\begin{aligned}
			&\left|
			\int_{\R^3}
			\partial^\beta
			\diver\left(
			((\rho^{n+1})^\alpha)_t\nabla v^{n+1}
			\right)
			\cdot \partial^\beta v_t^{n+1}\,dx
			\right| \\
			&\le
			\left\|
			\partial^\beta\left(
			((\rho^{n+1})^\alpha)_t\nabla v^{n+1}
			\right)
			\right\|_{L^2}
			\|\nabla\partial^\beta v_t^{n+1}\|_{L^2}.
		\end{aligned}
		\]
		The same bound holds for the two other divergence terms. Since
		\[
		\left\|
		((\rho^{n+1})^\alpha)_t\nabla v^{n+1}
		\right\|_{H^{K-2}}
		+
		\left\|
		((\rho^{n+1})^\alpha)_t\diver v^{n+1}
		\right\|_{H^{K-2}}
		\le
		C_M
		\|\rho_t^{n+1}\|_{H^{K-2}}
		\|v^{n+1}\|_{H^{K-1}},
		\]
		we have
		\begin{equation}\label{Ibeta12diver}
			\begin{aligned}
				&\left|
				\int_{\R^3}
				\partial^\beta
				\diver\left(
				((\rho^{n+1})^\alpha)_t\nabla v^{n+1}
				\right)
				\cdot \partial^\beta v_t^{n+1}\,dx
				\right| \\
				&\le C_M
				\|\rho_t^{n+1}\|_{H^{K-2}}
				\|v^{n+1}\|_{H^{K-1}}\|\nabla\partial^\beta v_t^{n+1}\|_{L^2}\\
				&\le\vartheta
				\|v_t^{n+1}\|_{H^{K-1}}^2
				+
				C_{\vartheta,M}
				\|\rho_t^{n+1}\|_{H^{K-2}}^2
				\|v^{n+1}\|_{H^{K-1}}^2.
			\end{aligned}
		\end{equation}
		For the non-divergence part, we use
		\[
		\bigl((\rho^{n+1})^{\alpha-1}\nabla\rho^{n+1}\bigr)_t
		=
		(\rho^{n+1})^{\alpha-1}\nabla\rho_t^{n+1}
		+
		(\alpha-1)(\rho^{n+1})^{\alpha-2}
		\rho_t^{n+1}\nabla\rho^{n+1}.
		\]
		Thus
		\[
		\begin{aligned}
			&\left|
			\int_{\R^3}
			\partial^\beta
			\left[
			\bigl((\rho^{n+1})^{\alpha-1}\nabla\rho^{n+1}\bigr)_t
			\nabla v^{n+1}
			\right]\cdot
			\partial^\beta v_t^{n+1}\,dx
			\right| \\
			&\le I_{\beta}^{nd,1}+I_{\beta}^{nd,2},
		\end{aligned}
		\]
		where
		\[
		\begin{aligned}
			I_{\beta}^{nd,1}
			&:=
			\left|
			\int_{\R^3}
			\partial^\beta
			\left[
			(\rho^{n+1})^{\alpha-1}
			\nabla\rho_t^{n+1}
			\nabla v^{n+1}
			\right]\cdot
			\partial^\beta v_t^{n+1}\,dx
			\right|,\\
			I_{\beta}^{nd,2}
			&:=
			\left|
			\int_{\R^3}
			\partial^\beta
			\left[
			(\alpha-1)(\rho^{n+1})^{\alpha-2}
			\rho_t^{n+1}\nabla\rho^{n+1}
			\nabla v^{n+1}
			\right]\cdot
			\partial^\beta v_t^{n+1}\,dx
			\right|.
		\end{aligned}
		\]
		For \(I_{\beta}^{nd,1}\), by Leibniz' formula,
		\[
		\partial^\beta
		\left[
		(\rho^{n+1})^{\alpha-1}
		\nabla\rho_t^{n+1}
		\nabla v^{n+1}
		\right]
		=
		\sum_{\beta_1+\beta_2+\beta_3=\beta}
		C_{\beta_1,\beta_2,\beta_3}
		\partial^{\beta_1}(\rho^{n+1})^{\alpha-1}
		\nabla\partial^{\beta_2}\rho_t^{n+1}
		\nabla\partial^{\beta_3}v^{n+1}.
		\]
		For each term in the sum, integration by parts gives
		\[
		\begin{aligned}
			&\int_{\R^3}
			\partial^{\beta_1}(\rho^{n+1})^{\alpha-1}
			\nabla\partial^{\beta_2}\rho_t^{n+1}
			\nabla\partial^{\beta_3}v^{n+1}
			\cdot
			\partial^\beta v_t^{n+1}\,dx \\
			&\quad =
			-
			\int_{\R^3}
			\partial^{\beta_2}\rho_t^{n+1}
			\,
			\diver\left(
			\partial^{\beta_1}(\rho^{n+1})^{\alpha-1}
			\nabla\partial^{\beta_3}v^{n+1}
			\cdot
			\partial^\beta v_t^{n+1}
			\right)\,dx .
		\end{aligned}
		\]
		Therefore, using the Sobolev product estimates, the fixed strip bound of
		\(\rho^{n+1}\), and \(|\beta|\le K-2\), we obtain
		\[
		I_{\beta}^{nd,1}
		\le
		C_M
		\|\rho_t^{n+1}\|_{H^{K-2}}
		\left(
		\|v^{n+1}\|_{H^K}
		\|v_t^{n+1}\|_{H^{K-2}}
		+
		\|v^{n+1}\|_{H^{K-1}}
		\|v_t^{n+1}\|_{H^{K-1}}
		\right).
		\]
		For \(I_{\beta}^{nd,2}\), by the Sobolev product estimates,
		\[
		\begin{aligned}
			I_{\beta}^{nd,2}
			&\le
			\left\|
			(\alpha-1)(\rho^{n+1})^{\alpha-2}
			\rho_t^{n+1}\nabla\rho^{n+1}
			\nabla v^{n+1}
			\right\|_{H^{K-2}}
			\|v_t^{n+1}\|_{H^{K-2}} \\
			&\le
			C_M
			\|\rho_t^{n+1}\|_{H^{K-2}}
			\|v^{n+1}\|_{H^{K-1}}
			\|v_t^{n+1}\|_{H^{K-2}}.
		\end{aligned}
		\]
		Combining the estimates for \(I_{\beta}^{nd,1}\) and \(I_{\beta}^{nd,2}\), we get
		\[
		\begin{aligned}
			&\left|
			\int_{\R^3}
			\partial^\beta
			\left[
			\bigl((\rho^{n+1})^{\alpha-1}\nabla\rho^{n+1}\bigr)_t
			\nabla v^{n+1}
			\right]\cdot
			\partial^\beta v_t^{n+1}\,dx
			\right| \\
			&\quad\le
			C_M
			\|\rho_t^{n+1}\|_{H^{K-2}}
			\left(
			\|v^{n+1}\|_{H^K}
			\|v_t^{n+1}\|_{H^{K-2}}
			+
			\|v^{n+1}\|_{H^{K-1}}
			\|v_t^{n+1}\|_{H^{K-1}}
			\right).
		\end{aligned}
		\]
		Hence, by Young's inequality,
		\begin{equation}\label{vtk-1I12}
			\begin{aligned}
				&\left|
				\int_{\R^3}
				\partial^\beta
				\left[
				\bigl((\rho^{n+1})^{\alpha-1}\nabla\rho^{n+1}\bigr)_t
				\nabla v^{n+1}
				\right]\cdot
				\partial^\beta v_t^{n+1}\,dx
				\right| \\
				&\quad\le
				\vartheta
				\|v_t^{n+1}\|_{H^{K-1}}^2
				+
				C_{\vartheta,M}
				\|\rho_t^{n+1}\|_{H^{K-2}}^2
				\left(
				\|v^{n+1}\|_{H^K}^2
				+
				\|v^{n+1}\|_{H^{K-1}}^2
				\right).
			\end{aligned}
		\end{equation}
		Therefore, combining \eqref{Ibeta12diver} (two other divergence terms are similar) with \eqref{vtk-1I12}, we have
		\[
		\mathcal J_{\beta,6}
		\le
		4\vartheta
		\|v_t^{n+1}\|_{H^{K-1}}^2
		+
		C_{\vartheta,M}
		\|\rho_t^{n+1}\|_{H^{K-2}}^2
		\bigl(
		\|v^{n+1}\|_{H^{K}}^2
		+
		\|v^{n+1}\|_{H^{K-1}}^2
		\bigr).
		\]
		
		For the commutator term $\mathcal C_\beta^A$, write
		\[
		\mathcal A_{\rho^{n+1}} v_t^{n+1}
		=
		\diver\Bigl(
		b_1\nabla v_t^{n+1}
		+
		b_2(\nabla v_t^{n+1})^\top
		\Bigr)
		+
		\diver\Bigl(
		b_3\diver v_t^{n+1}\,\I
		\Bigr)
		+
		B\cdot\nabla v_t^{n+1},
		\]
		where
		\[
		b_1:=\nu(\rho^{n+1})^\alpha,\qquad
		b_2:=(\nu-d)(\rho^{n+1})^\alpha,
		\]
		\[
		b_3:=(\alpha-1)(2\nu-d)(\rho^{n+1})^\alpha,
		\qquad
		B:=\alpha d(\rho^{n+1})^{\alpha-1}\nabla\rho^{n+1}.
		\]
		Then
		\[
		\begin{aligned}
			&\partial^\beta(\mathcal A_{\rho^{n+1}}v_t^{n+1})
			-
			\mathcal A_{\rho^{n+1}}\partial^\beta v_t^{n+1} \\
			&=
			\diver\Bigl(
			[\partial^\beta,b_1]\nabla v_t^{n+1}
			+
			[\partial^\beta,b_2](\nabla v_t^{n+1})^\top
			\Bigr) \\
			&\quad+
			\diver\Bigl(
			[\partial^\beta,b_3]\diver v_t^{n+1}\,\I
			\Bigr)
			+
			[\partial^\beta,B\cdot\nabla]v_t^{n+1}.
		\end{aligned}
		\]
		Following an estimate similar to \eqref{eq:L-commutator-combined}, we have
		\[
		\begin{aligned}
			\mathcal C_\beta^A
			&\lesssim
			\sum_{i=1}^3
			\|[\partial^\beta,b_i]\nabla v_t^{n+1}\|_{L^2}
			\|\nabla \partial^\beta v_t^{n+1}\|_{L^2} \\
			&\quad+
			\|[\partial^\beta,B\cdot\nabla]v_t^{n+1}\|_{L^2}
			\|\partial^\beta v_t^{n+1}\|_{L^2}.\\
			&\le C_M
			\|v_t^{n+1}\|_{H^{K-2}}
			\|\partial^\beta v_t^{n+1}\|_{H^1}\\
			&\le
			\vartheta
			\|\partial^\beta v_t^{n+1}\|_{H^1}^2
			+
			C_{\vartheta,M}
			\|v_t^{n+1}\|_{H^{K-2}}^2.
		\end{aligned}
		\]
		
		For the term $\frac12\int_{\R^3}|\rho_t^{n+1}||\partial^\beta v_t^{n+1}|^2\,dx$, we have
		\[
		\frac12\int_{\R^3}|\rho_t^{n+1}||\partial^\beta v_t^{n+1}|^2\,dx \lesssim\|\rho_t^{n+1}\|_{L^{\infty}}\|v_t\|_{H^{K-2}}^2.
		\]
		Substituting the above estimates into \eqref{vtJ6Ce}, summing over $\beta$, and choosing $\vartheta$ sufficiently small, we obtain
		\begin{equation}\label{eq:vt-final}
			\sum_{|\beta|\le K-2} \frac{d}{dt}
			\int_{\R^3}\rho^{n+1}|\partial^\beta v_t^{n+1}|^2\,dx
			+c_3 \norm{v_t^{n+1}}_{H^{K-1}}^2
			\le
			C_{R,M} \norm{v_t^{n+1}}_{H^{K-2}}^2
			+C_{R,M} \norm{v^{n+1}}_{H^{K}}^2
			+C_{R,M}.
		\end{equation}
		
		We now combine \eqref{eq:v-HK-final-combined} and \eqref{eq:vt-final}. Choose
		\[
		\lambda_0:=\min\left\{1,\frac{c_1}{2C_{R,M}}\right\}.
		\]
		Multiplying \eqref{eq:vt-final} by $\lambda_0$ and adding the result to
		\eqref{eq:v-HK-final-combined}, we get
		\begin{equation}
			\begin{split}
				&\frac{d}{dt}\Bigl(\sum_{ |\beta| \le K}\int_{\R^3}\rho^{n+1}|\partial^\beta v^{n+1}|^2\,dx+\lambda_0\sum_{|\beta|\le K-2} \int_{\R^3}\rho^{n+1}|\partial^\beta v_t^{n+1}|^2\,dx
				\Bigr)
				\\
				&\quad+c_1 \norm{v^{n+1}}_{H^{K+1}}^2
				+\lambda_0 c_3 \norm{v_t^{n+1}}_{H^{K-1}}^2\\
				&\le
				(C_{\kappa,R,M}+\lambda_0 C_{R,M}) \norm{v^{n+1}}_{H^{K}}^2+\lambda_0 C_{R,M}\norm{v_t^{n+1}}_{H^{K-2}}^2
				+\lambda_0 C_{R,M} \norm{v^{n+1}}_{H^{K+1}}^2\\
				&\quad+C_{R,M}+\lambda_0 C_{R,M}+\kappa_1\norm{v_t^{n+1}}_{H^{K-1}}^2.
			\end{split}
		\end{equation}
		By the choice of $\lambda_0$, we have
		\[
		\lambda_0 C_{R,M}\le \frac{c_1}{2}.
		\]
		Then, choosing $\kappa_1$ sufficiently small, we obtain
		\begin{equation}
			\begin{split}
				&\frac{d}{dt}\Bigl(\sum_{ |\beta| \le K}\int_{\R^3}\rho^{n+1}|\partial^\beta v^{n+1}|^2\,dx+\lambda_0\sum_{|\beta|\le K-2} \int_{\R^3}\rho^{n+1}|\partial^\beta v_t^{n+1}|^2\,dx
				\Bigr)
				\\
				&\quad+\frac{c_1}{2} \norm{v^{n+1}}_{H^{K+1}}^2
				+\frac{\lambda_0 c_3}{2} \norm{v_t^{n+1}}_{H^{K-1}}^2 \\
				\le&
				C_{R,M}\Bigl(
				\norm{v^{n+1}}_{H^{K}}^2
				+\norm{v_t^{n+1}}_{H^{K-2}}^2
				\Bigr)
				+C_{R,M},
			\end{split}
		\end{equation}
		Applying Gr\"onwall's lemma and using the initial values
		\[
		v^{n+1}(0)=v_0,\qquad v_t^{n+1}(0)\in H^{K-2},
		\]
		with the latter controlled by the equation at $t=0$, namely
		\[
		v_t^{n+1}(0)
		=\frac{1}{\rho_0}\Bigl(
		\A_{\rho_0}v_0-\rho_0 v^n(0)\cdot\nabla v_0-\nabla\left(\frac{\rho_0^\gamma}{\gamma}\right)
		\Bigr),
		\]
		we conclude that
		\begin{equation}\label{vcombestima}
			\begin{split}
				&\sup_{0\le t\le T}\Bigl(
				\norm{v^{n+1}(t)}_{H^{K}}^2+\norm{v_t^{n+1}(t)}_{H^{K-2}}^2
				\Bigr)
				+\int_0^T \Bigl(
				\norm{v^{n+1}(t)}_{H^{K+1}}^2+\norm{v_t^{n+1}(t)}_{H^{K-1}}^2
				\Bigr)\,dt\\
				&\le e^{C_{R,M} T}(C_{R,M}T+C\norm{v^{n+1}(0)}_{H^{K}}^2+C\norm{v_t^{n+1}(0)}_{H^{K-2}}^2)\\
				&\le 2C(\norm{v^{n+1}(0)}_{H^{K}}^2+\norm{v_t^{n+1}(0)}_{H^{K-2}}^2),
			\end{split}
		\end{equation}
		provided $T$ is sufficiently small. Finally, since
		\[
		v^{n+1}\in L^2(0,T;H^{K+1}),
		\qquad
		v_t^{n+1}\in L^2(0,T;H^{K-1}),
		\]
		the Lions--Magenes theorem yields
		\[
		v^{n+1}\in C([0,T];H^{K}),
		\]
		which is the continuity asserted in the statement.
	\end{proof}
	
	\begin{lemma}[A genuine invariant set]\label{lem:invariant}
		There exist $R>0$ and $T>0$, depending only on
		\[
		\ubar,\ \obar,\ \rstar,\ \norm{\rho_0-\rstar}_{H^K},\ \norm{v_0}_{H^{K}},\ K,
		\]
		such that the following holds. If
		\[
		v^n\in\mathcal B_{R,T},
		\qquad
		\frac{\ubar}{2}\le \rho^n(t,x)\le 2\obar+1
		\quad\text{on }[0,T]\times\R^3,
		\]
		then
		\[
		v^{n+1}\in\mathcal B_{R,T},
		\qquad
		\frac{\ubar}{2}\le \rho^{n+1}(t,x)\le 2\obar+1
		\quad\text{on }[0,T]\times\R^3.
		\]
		Consequently,
		\[
		v^n\in\mathcal B_{R,T},
		\qquad
		\frac{\ubar}{2}\le \rho^n(t,x)\le 2\obar+1
		\quad\text{on }[0,T]\times\R^3
		\]
		for every $n\ge 0$.
	\end{lemma}
	
	\begin{proof}
		First, we choose $R$ sufficiently large such that
		\begin{equation}\label{eq:R-choice}
			R\ge \max\left\{C_{\mathrm{heat}}\norm{v_0}_{H^{K}},\left[8C\left(\norm{v^{n+1}(0)}_{H^{K}}^2+\norm{v_t^{n+1}(0)}_{H^{K-2}}^2\right)\right]^{1/2}\right\}.
		\end{equation}
		where $C$ is chosen as the constant appearing in \eqref{vcombestima}. Once $R$ is fixed, Lemma \ref{lem:density} and Lemma \ref{lem:velocity} produce
		constants still denoted by $C_R(T)$ such that
		\begin{equation}
			\begin{split}
				\sup_{0\le t\le T}\norm{\rho^{n+1}(t)-\rstar}_{H^K}^2
				&+\int_0^T \norm{\rho^{n+1}(t)-\rstar}_{H^{K+1}}^2\,dt
				\\
				&+\sup_{0\le t\le T}\norm{\rho_t^{n+1}(t)}_{H^{K-2}}^2
				+\int_0^T \norm{\rho_t^{n+1}(t)}_{H^{K-1}}^2\,dt
				\le C_R(T),
			\end{split}
		\end{equation}
		and
		\begin{equation}
			\begin{split}
				&\sup_{0\le t\le T}\norm{v^{n+1}(t)}_{H^{K}}^2
				+\int_0^T \norm{v^{n+1}(t)}_{H^{K+1}}^2\,dt\\
				&+\sup_{0\le t\le T}\norm{v_t^{n+1}(t)}_{H^{K-2}}^2
				+\int_0^T \norm{v_t^{n+1}(t)}_{H^{K-1}}^2\,dt
				\le 2C(\norm{v^{n+1}(0)}_{H^{K}}^2+\norm{v_t^{n+1}(0)}_{H^{K-2}}^2).
			\end{split}
		\end{equation}
		Then the velocity estimate gives
		\[
		\norm{v^{n+1}}_{L^\infty(0,T;H^{K})}
		+
		\norm{v^{n+1}}_{L^2(0,T;H^{K+1})}
		+
		\norm{v_t^{n+1}}_{L^\infty(0,T;H^{K-2})}
		+
		\norm{v_t^{n+1}}_{L^2(0,T;H^{K-1})}
		\le R,
		\]
		Hence $v^{n+1}\in\mathcal B_{R,T}$.
		
		The pointwise strip for $\rho^{n+1}$ is already contained in
		\eqref{eq:rho-pointwise}. Thus one iteration step preserves both the velocity
		ball and the density strip.
		
		Finally, Lemma \ref{lem:seed} and \eqref{eq:R-choice} show that
		$v^0\in\mathcal B_{R,T}$. Since $\rho^0=\rho_0$ and the initial density lies in
		$[\ubar,\obar]$, the strip condition also holds at level $n=0$. The previous
		one-step invariance then yields the conclusion by induction on $n$.
	\end{proof}
	
	By an argument similar to the proof of Proposition 8.3 in Gu--Huang--Meng--Zhou \cite{Gu-Huang-Meng-Zhou}, we can establish the strong convergence of this sequence. Set
	\[
	\delta\rho^{n+1}:=\rho^{n+1}-\rho^n,
	\qquad
	\delta v^{n+1}:=v^{n+1}-v^n.
	\]
	
	\begin{lemma}[Contraction Estimate and Cauchy Sequences]\label{thm:contraction}
		Let $n \ge 2$ and define the energy functional $\mathfrak{E}^{n}(t)$ and its associated dissipation rate $\mathfrak{D}^{n}(t)$ by
		\begin{equation}\label{mathfrakEn}
			\mathfrak{E}^{n}(t) :=\Theta\|\delta\rho^{n}(t)\|_{L^2}^2 + \frac{1}{2}\int_{\mathbb{R}^3} \rho^{n}(x,t)|\delta v^{n}(x,t)|^2 \, dx,
		\end{equation}
		and
		\begin{equation}\label{mathfrakDn1}
			\mathfrak{D}^{n}(t) := \|\delta\rho^{n}(t)\|_{H^1}^2 + \|\delta v^{n}(t)\|_{H^1}^2,
		\end{equation}
		where $\Theta > 0$ is a constant depending only on $\gamma, \alpha, \nu, \varepsilon$, $\underline{\rho_0}$, $\overline{\rho_0}$, $\|\rho_0-\rstar\|_{H^3}$, $\|v_0\|_{H^2}$, $R$, and $M$. 
		
		Then, there exists a continuous and monotonically increasing function $\kappa(\cdot)$ with $\kappa(s) \to 0$ as $s \to 0^+$, such that the uniform contraction bound
		\begin{equation}\label{contractionen}
			\sup_{0\le t\le T}\mathfrak{E}^{n+1}(t) + \int_0^T \mathfrak{D}^{n+1}(t) \, dt \le \kappa(T) \left( \sup_{0\le t\le T}\mathfrak{E}^n(t) + \int_0^T \mathfrak{D}^n(t) \, dt \right)
		\end{equation}
		holds for all $n \ge 2$.
		
		In particular, there exists a sufficiently small time span $T_* > 0$, such that $\kappa(T_*) \le \frac{1}{2}$. Consequently, for all $0 < T \le T_*$, we have
		\begin{equation}\label{contraction-half}
			\sup_{0\le t\le T}\mathfrak{E}^{n+1}(t) + \int_0^T \mathfrak{D}^{n+1}(t) \, dt \le \frac{1}{2} \left( \sup_{0\le t\le T}\mathfrak{E}^n(t) + \int_0^T \mathfrak{D}^n(t) \, dt \right), \quad \forall n \ge 2,
		\end{equation}
		which guarantees that $\{\rho^n\}_{n=1}^\infty$ and $\{v^n\}_{n=1}^\infty$ form Cauchy sequences in $$L^\infty(0,T; L^2(\mathbb{R}^3)) \quad\text{and} \quad L^2(0,T; H^1(\mathbb{R}^3)).$$
	\end{lemma}
	\subsubsection{The approximation process for local solutions}
	Now, we give the approximation procedure for the local solution. The uniform bounds give, after extraction,
	\[
	\rho^n-\rstar \rightharpoonup \rho-\rstar
	\quad\text{in }L^2(0,T;H^{K+1}),
	\qquad
	v^n\rightharpoonup v
	\quad\text{in }L^2(0,T;H^{K+1}),
	\]
	and, after extraction of the time derivatives,
	\[
	\rho_t^n\rightharpoonup \rho_t
	\quad\text{in }L^2(0,T;H^{K-1}),
	\qquad
	v_t^n\rightharpoonup v_t
	\quad\text{in }L^2(0,T;H^{K-1}).
	\]
	On the other hand, the contraction estimate shows that the full sequence is Cauchy in
	\[
	L^\infty(0,T;L^2(\R^3))\times L^\infty(0,T;L^2(\R^3))
	\]
	and in
	\[
	L^2(0,T;H^1(\R^3))\times L^2(0,T;H^1(\R^3)).
	\]
	Together with the uniform bounds
	\[
	\sup_n\sup_{0\le t\le T}\norm{\rho^n(t)-\rstar}_{H^K}
	+\sup_n\sup_{0\le t\le T}\norm{v^n(t)}_{H^{K}}
	\le C_R,
	\]
	\[
	\sup_n\int_0^T\norm{\rho^n(t)-\rstar}_{H^{K+1}}^2\,dt
	+\sup_n\int_0^T\norm{v^n(t)}_{H^{K+1}}^2\,dt
	\le C_R,
	\]
	this implies higher-order strong convergence by interpolation. Indeed, for the density difference,
	\[
	\sup_{0\le t\le T}\norm{\rho^n(t)-\rho^m(t)}_{H^{K-1}}
	\le
	C
	\sup_{0\le t\le T}\norm{\rho^n(t)-\rho^m(t)}_{L^2}^{1/K}
	\sup_{0\le t\le T}\norm{\rho^n(t)-\rho^m(t)}_{H^{K}}^{(K-1)/K},
	\]
	so $\{\rho^n\}$ is Cauchy in $C([0,T];H^{K-1})$. Likewise,
	\[
	\sup_{0\le t\le T}\norm{v^n(t)-v^m(t)}_{H^{K-1}}
	\le
	C
	\sup_{0\le t\le T}\norm{v^n(t)-v^m(t)}_{L^2}^{1/K}
	\sup_{0\le t\le T}\norm{v^n(t)-v^m(t)}_{H^{K}}^{(K-1)/K},
	\]
	so $\{v^n\}$ is Cauchy in $C([0,T];H^{K-1})$. Hence there exists a pair $(\rho,v)$ such that
	\[
	\rho^n\to \rho
	\quad\text{in }C([0,T];H^{K-1}),
	\]
	\[
	v^n\to v
	\quad\text{in }C([0,T];H^{K-1}).
	\]
	The approximation
	and compactness procedure is analogous to that in the proof of
	Gu--Huang--Meng--Zhou \cite[Lemma~3.1]{Gu-Huang-Meng-Zhou}; we therefore omit
	the details here. Consequently, \((\rho,v)\) belongs to the regularity class
	stated in Proposition~\ref{thm:main}. The uniqueness follows from the standard
	energy estimate for the difference of two solutions.

	\subsubsection{Weighted Energy Estimate}
	Fix $\Lambda>1$. We introduce the weighted energy
	\[
	E_\Lambda(t):=
	\sum_{1\leq |\beta|\leq 3}
	\left(
	\|\langle x\rangle^\Lambda\partial^\beta\rho(t)\|_{L^2}^2
	+
	\|\langle x\rangle^\Lambda\sqrt{\rho}\,
	\partial^\beta v(t)\|_{L^2}^2
	\right),
	\]
	and the corresponding dissipation
	\[
	D_\Lambda(t):=
	\sum_{1\leq |\beta|\leq 3}
	\left(
	\|\langle x\rangle^\Lambda\nabla\partial^\beta\rho(t)\|_{L^2}^2
	+
	\|\langle x\rangle^\Lambda\nabla\partial^\beta v(t)\|_{L^2}^2
	\right).
	\]
	Since $\rho$ is uniformly bounded from below and above, and
	\[
	|\nabla\langle x\rangle^{\Lambda}|
	+
	|\nabla^2\langle x\rangle^{\Lambda}|
	\leq
	C_\Lambda\langle x\rangle^{\Lambda},
	\]
	we have
	\[
	E_\Lambda(t)
	\sim
	\|\langle x\rangle^\Lambda\nabla\rho(t)\|_{H^2}^2
	+
	\|\langle x\rangle^\Lambda\nabla v(t)\|_{H^2}^2.
	\]
	
	\smallskip
	\noindent\textbf{Step 1: weighted estimate for the density.}
	We write the density equation as
	\[
	\rho_t-\diver(a(\rho)\nabla\rho)
	=
	-\diver(\rho v),
	\qquad
	a(\rho):=d\alpha\rho^{\alpha-1}.
	\]
	Since $\rho$ remains in a fixed positive strip, there exists $a_0>0$
	such that
	\[
	a(\rho)\geq a_0.
	\]
	Applying $\partial^\beta$, $1\leq|\beta|\leq3$, gives
	\[
	\partial_t\partial^\beta\rho
	-
	\diver(a(\rho)\nabla\partial^\beta\rho)
	=
	-\partial^\beta\diver(\rho v)
	+
	\mathcal C_\beta^\rho,
	\]
	where
	\[
	\mathcal C_\beta^\rho
	=
	\partial^\beta\diver(a(\rho)\nabla\rho)
	-
	\diver(a(\rho)\nabla\partial^\beta\rho)
	=
	\sum_{0<\eta\leq\beta}
	C_{\beta,\eta}
	\diver\left(
	\partial^\eta a(\rho)
	\nabla\partial^{\beta-\eta}\rho
	\right).
	\]
	
	Testing by $\langle x\rangle^{2\Lambda}\partial^\beta\rho$ (Strictly speaking, one should first test the equation against $\langle x\rangle_R^{2\Lambda}\partial^\beta\rho$,
	where $\langle x\rangle_R^{2\Lambda}$ is a smooth truncated version of
	$\langle x\rangle^{2\Lambda}$. After deriving estimates uniformly in
	$R$, one may pass to the limit $R\to\infty$. For notational simplicity,
	we carry out the computation directly with
	$\langle x\rangle^{2\Lambda}\partial^\beta\rho$) , we obtain
	\[
	\frac12\frac{d}{dt}
	\|\langle x\rangle^\Lambda\partial^\beta\rho\|_{L^2}^2
	-
	\int_{\R^3}
	\langle x\rangle^{2\Lambda}\partial^\beta\rho\,
	\diver(a(\rho)\nabla\partial^\beta\rho)\,dx
	=
	I_1+I_2,
	\]
	where
	\[
	I_1
	:=
	-\int_{\R^3}
	\langle x\rangle^{2\Lambda}\partial^\beta\rho\,
	\partial^\beta\diver(\rho v)\,dx,
	\]
	and
	\[
	I_2
	:=
	\int_{\R^3}
	\langle x\rangle^{2\Lambda}\partial^\beta\rho\,
	\mathcal C_\beta^\rho\,dx.
	\]
	
	For the diffusion term, integration by parts gives
	\[
	\begin{aligned}
		&-\int_{\R^3}
		\langle x\rangle^{2\Lambda}\partial^\beta\rho\,
		\diver(a(\rho)\nabla\partial^\beta\rho)\,dx\\
		&\quad=
		\int_{\R^3}
		a(\rho)\langle x\rangle^{2\Lambda}
		|\nabla\partial^\beta\rho|^2\,dx\\
		&\qquad+
		\int_{\R^3}
		a(\rho)\partial^\beta\rho\,
		\nabla\langle x\rangle^{2\Lambda}
		\cdot\nabla\partial^\beta\rho\,dx\\
		&\quad\geq
		\frac{a_0}{2}
		\|\langle x\rangle^\Lambda
		\nabla\partial^\beta\rho\|_{L^2}^2
		-
		C_\Lambda
		\|\langle x\rangle^\Lambda
		\partial^\beta\rho\|_{L^2}^2.
	\end{aligned}
	\]
	For $I_1$, after integration by parts,
	\[
	\begin{aligned}
		|I_1|
		&\leq
		C_\Lambda
		\left(
		\|\langle x\rangle^\Lambda
		\nabla\partial^\beta\rho\|_{L^2}
		+
		\|\langle x\rangle^\Lambda
		\partial^\beta\rho\|_{L^2}
		\right)
		\|\langle x\rangle^\Lambda
		\partial^\beta(\rho v)\|_{L^2}.
	\end{aligned}
	\]
	By Leibniz' formula and the unweighted high-order bounds,
	\[
	\begin{aligned}
		\|\langle x\rangle^\Lambda
		\partial^\beta(\rho v)\|_{L^2}
		&\leq
		C\|\langle x\rangle^\Lambda
		\partial^\beta v\|_{L^2}
		+
		C\|\langle x\rangle^\Lambda
		\partial^\beta\rho\|_{L^2}\\
		&\quad+
		C\sum_{0<\eta<\beta}
		\|\langle x\rangle^\Lambda
		\partial^\eta\rho\,
		\partial^{\beta-\eta}v\|_{L^2}\\
		&\leq
		C E_\Lambda(t)^{1/2}.
	\end{aligned}
	\]
	Therefore,
	\[
	|I_1|
	\leq
	\vartheta
	\|\langle x\rangle^\Lambda
	\nabla\partial^\beta\rho\|_{L^2}^2
	+
	C_{\vartheta,\Lambda}E_\Lambda(t).
	\]
	
	For $I_2$, using the divergence form of
	$\mathcal C_\beta^\rho$, we have
	\[
	\begin{aligned}
		|I_2|
		&\leq
		C_\Lambda\sum_{0<\eta\leq\beta}
		\left(
		\|\langle x\rangle^\Lambda
		\nabla\partial^\beta\rho\|_{L^2}
		+
		\|\langle x\rangle^\Lambda
		\partial^\beta\rho\|_{L^2}
		\right)\\
		&\qquad\qquad\qquad\times
		\|\langle x\rangle^\Lambda
		(\partial^\eta a(\rho))
		\nabla\partial^{\beta-\eta}\rho\|_{L^2}.
	\end{aligned}
	\]
	Since $a(\rho)$ is smooth on the fixed range of $\rho$, the composition
	estimate and the weighted Sobolev product estimate imply
	\[
	\sum_{0<\eta\leq\beta}
	\|\langle x\rangle^\Lambda
	(\partial^\eta a(\rho))
	\nabla\partial^{\beta-\eta}\rho\|_{L^2}
	\leq
	C E_\Lambda(t)^{1/2}.
	\]
	Thus,
	\[
	|I_2|
	\leq
	\vartheta
	\|\langle x\rangle^\Lambda
	\nabla\partial^\beta\rho\|_{L^2}^2
	+
	C_{\vartheta,\Lambda}E_\Lambda(t).
	\]
	Combining the above estimates, we obtain
	\[
	\frac{d}{dt}
	\|\langle x\rangle^\Lambda
	\partial^\beta\rho\|_{L^2}^2
	+
	c_1
	\|\langle x\rangle^\Lambda
	\nabla\partial^\beta\rho\|_{L^2}^2
	\leq
	C_\Lambda E_\Lambda(t).
	\]
	
	\smallskip
	\noindent\textbf{Step 2: weighted estimate for the velocity.}
	We use the momentum equation in the form
	\[
	\rho v_t
	+\rho v\cdot\nabla v
	+\nabla\left(\frac{\rho^\gamma}{\gamma}\right)
	=
	\A_\rho v,
	\qquad
	\A_\rho v=\LL_\rho v+\Kop_\rho v.
	\]
	Applying $\partial^\beta$, $1\leq|\beta|\leq3$, testing by
	$\langle x\rangle^{2\Lambda}\partial^\beta v$, and integrating over
	$\R^3$, we get
	\[
	\begin{aligned}
		&\int_{\R^3}
		\langle x\rangle^{2\Lambda}
		\partial^\beta(\rho v_t)\cdot\partial^\beta v\,dx\\
		&\quad+
		\int_{\R^3}
		\langle x\rangle^{2\Lambda}
		\partial^\beta(\rho v\cdot\nabla v)
		\cdot\partial^\beta v\,dx\\
		&\quad+
		\int_{\R^3}
		\langle x\rangle^{2\Lambda}
		\partial^\beta\nabla
		\left(\frac{\rho^\gamma}{\gamma}\right)
		\cdot\partial^\beta v\,dx\\
		&=
		\int_{\R^3}
		\langle x\rangle^{2\Lambda}
		\partial^\beta(\A_\rho v)
		\cdot\partial^\beta v\,dx.
	\end{aligned}
	\]
	
	For the time derivative term, we write
	\[
	\partial^\beta(\rho v_t)
	=
	\rho\,\partial^\beta v_t
	+
	[\partial^\beta,\rho]v_t.
	\]
	Then
	\[
	\begin{aligned}
		\int_{\R^3}
		\langle x\rangle^{2\Lambda}
		\rho\,\partial^\beta v_t\cdot\partial^\beta v\,dx
		&=
		\frac12\frac{d}{dt}
		\int_{\R^3}
		\langle x\rangle^{2\Lambda}
		\rho|\partial^\beta v|^2\,dx\\
		&\quad-
		\frac12
		\int_{\R^3}
		\langle x\rangle^{2\Lambda}
		\rho_t|\partial^\beta v|^2\,dx.
	\end{aligned}
	\]
	Using the bound on $\rho_t$ and the commutator estimate,
	\[
	\begin{aligned}
		&\left|
		\int_{\R^3}
		\langle x\rangle^{2\Lambda}
		\rho_t|\partial^\beta v|^2\,dx
		\right|\\
		&\quad+
		\left|
		\int_{\R^3}
		\langle x\rangle^{2\Lambda}
		[\partial^\beta,\rho]v_t
		\cdot\partial^\beta v\,dx
		\right|
		\leq
		C_\Lambda E_\Lambda(t).
	\end{aligned}
	\]
	
	For the transport term, we decompose
	\[
	\partial^\beta(\rho v\cdot\nabla v)
	=
	\rho v\cdot\nabla\partial^\beta v
	+
	[\partial^\beta,\rho v]\cdot\nabla v.
	\]
	The principal part satisfies
	\[
	\begin{aligned}
		&\left|
		\int_{\R^3}
		\langle x\rangle^{2\Lambda}
		\rho v\cdot\nabla\partial^\beta v
		\cdot\partial^\beta v\,dx
		\right|\\
		&\quad=
		\frac12
		\left|
		\int_{\R^3}
		\diver(\langle x\rangle^{2\Lambda}\rho v)
		|\partial^\beta v|^2\,dx
		\right|\\
		&\quad\leq
		C_\Lambda
		\|\langle x\rangle^\Lambda
		\partial^\beta v\|_{L^2}^2.
	\end{aligned}
	\]
	Moreover, the commutator part is estimated by
	\[
	\left|
	\int_{\R^3}
	\langle x\rangle^{2\Lambda}
	[\partial^\beta,\rho v]\cdot\nabla v
	\cdot\partial^\beta v\,dx
	\right|
	\leq
	C_\Lambda E_\Lambda(t).
	\]
	
	For the pressure term, since
	\[
	\nabla\left(\frac{\rho^\gamma}{\gamma}\right)
	=
	\rho^{\gamma-1}\nabla\rho,
	\]
	we have
	\[
	\partial^\beta\nabla
	\left(\frac{\rho^\gamma}{\gamma}\right)
	=
	\rho^{\gamma-1}\nabla\partial^\beta\rho
	+
	\sum_{0<\eta\leq\beta}
	C_{\beta,\eta}
	\partial^\eta(\rho^{\gamma-1})
	\nabla\partial^{\beta-\eta}\rho.
	\]
	Using the fixed strip condition, the composition estimate, and the
	weighted Sobolev product estimate,
	\[
	\begin{aligned}
		&\left|
		\int_{\R^3}
		\langle x\rangle^{2\Lambda}
		\partial^\beta\nabla
		\left(\frac{\rho^\gamma}{\gamma}\right)
		\cdot\partial^\beta v\,dx
		\right|\\
		&\quad\leq
		C
		\|\langle x\rangle^\Lambda
		\nabla\partial^\beta\rho\|_{L^2}
		\|\langle x\rangle^\Lambda
		\partial^\beta v\|_{L^2}\\
		&\qquad+
		C(1+E_\Lambda(t)^{1/2})
		\|\langle x\rangle^\Lambda
		\partial^\beta v\|_{L^2}\\
		&\quad\leq
		\vartheta
		\|\langle x\rangle^\Lambda
		\nabla\partial^\beta\rho\|_{L^2}^2
		+
		C_{\vartheta,\Lambda}(1+E_\Lambda(t)).
	\end{aligned}
	\]
	
	It remains to treat the elliptic term. We write
	\[
	\partial^\beta(\A_\rho v)
	=
	\A_\rho\partial^\beta v
	+
	[\partial^\beta,\A_\rho]v.
	\]
	By the weighted elliptic coercivity of $\A_\rho$, together with the
	lower-order estimate for $\Kop_\rho$,
	\[
	\begin{aligned}
		-&\int_{\R^3}
		\langle x\rangle^{2\Lambda}
		\A_\rho\partial^\beta v
		\cdot\partial^\beta v\,dx\\
		&\quad\geq
		c_2
		\|\langle x\rangle^\Lambda
		\nabla\partial^\beta v\|_{L^2}^2
		-
		C_\Lambda
		\|\langle x\rangle^\Lambda
		\partial^\beta v\|_{L^2}^2.
	\end{aligned}
	\]
	Indeed, all terms containing derivatives of
	$\langle x\rangle^{2\Lambda}$ are bounded by
	\[
	\vartheta
	\|\langle x\rangle^\Lambda
	\nabla\partial^\beta v\|_{L^2}^2
	+
	C_{\vartheta,\Lambda}
	\|\langle x\rangle^\Lambda
	\partial^\beta v\|_{L^2}^2.
	\]
	For the commutator, the explicit forms of $\LL_\rho$ and
	$\Kop_\rho$ imply
	\[
	\begin{aligned}
		&\left|
		\int_{\R^3}
		\langle x\rangle^{2\Lambda}
		[\partial^\beta,\A_\rho]v
		\cdot\partial^\beta v\,dx
		\right|\\
		&\quad\leq
		C
		\|\langle x\rangle^\Lambda
		\nabla\partial^\beta v\|_{L^2}
		\|\langle x\rangle^\Lambda
		\partial^\beta v\|_{L^2}\\
		&\qquad+
		C
		\|\langle x\rangle^\Lambda
		\nabla\partial^\beta v\|_{L^2}
		(1+E_\Lambda(t)^{1/2})
		+
		C(1+E_\Lambda(t))\\
		&\quad\leq
		\vartheta
		\|\langle x\rangle^\Lambda
		\nabla\partial^\beta v\|_{L^2}^2
		+
		C_{\vartheta,\Lambda}(1+E_\Lambda(t)).
	\end{aligned}
	\]
	Combining the above estimates gives
	\[
	\begin{aligned}
		&\frac{d}{dt}
		\int_{\R^3}
		\langle x\rangle^{2\Lambda}
		\rho|\partial^\beta v|^2\,dx
		+
		c_3
		\|\langle x\rangle^\Lambda
		\nabla\partial^\beta v\|_{L^2}^2\\
		&\quad\leq
		\vartheta
		\|\langle x\rangle^\Lambda
		\nabla\partial^\beta\rho\|_{L^2}^2
		+
		C_\Lambda(1+E_\Lambda(t)).
	\end{aligned}
	\]
	
	\smallskip
	\noindent\textbf{Step 3: summation and closure.}
	Summing the density and velocity estimates over
	$1\leq|\beta|\leq3$, and using the equivalence between
	\[
	\int_{\R^3}
	\langle x\rangle^{2\Lambda}
	\rho|\partial^\beta v|^2\,dx
	\]
	and
	\[
	\|\langle x\rangle^\Lambda
	\partial^\beta v\|_{L^2}^2,
	\]
	we obtain
	\[
	\begin{aligned}
		\frac{d}{dt}E_\Lambda(t)
		+
		c_4D_\Lambda(t)
		&\leq
		\vartheta
		\sum_{1\leq|\beta|\leq3}
		\|\langle x\rangle^\Lambda
		\nabla\partial^\beta\rho\|_{L^2}^2\\
		&\quad+
		C_\Lambda(1+E_\Lambda(t)).
	\end{aligned}
	\]
	Choosing $\vartheta>0$ sufficiently small and absorbing the first term
	on the right-hand side, we get
	\[
	\frac{d}{dt}E_\Lambda(t)
	+
	\frac{c_4}{2}D_\Lambda(t)
	\leq
	C_\Lambda(1+E_\Lambda(t)).
	\]
	Thus, by Gr\"onwall's inequality,
	\[
	E_\Lambda(t)
	+
	\int_0^tD_\Lambda(s)\,ds
	\leq
	C_{T,\Lambda}(1+E_\Lambda(0)),
	\qquad
	0\leq t\leq T.
	\]
	Consequently, if
	\[
	\langle x\rangle^\Lambda\nabla\rho_0\in H^2,
	\qquad
	\langle x\rangle^\Lambda\nabla v_0\in H^2,
	\]
	then
	\[
	\sup_{0\leq t\leq T}
	\left(
	\|\langle x\rangle^\Lambda\nabla\rho(t)\|_{H^2}
	+
	\|\langle x\rangle^\Lambda\nabla v(t)\|_{H^2}
	\right)
	\leq
	C_{T,\Lambda}.
	\]
	Since $H^2(\R^3)\hookrightarrow L^\infty(\R^3)$, this gives
	\[
	\|\langle x\rangle^\Lambda\nabla\rho(t)\|_{L^\infty}
	+
	\|\langle x\rangle^\Lambda\nabla v(t)\|_{L^\infty}
	\leq
	C_{T,\Lambda}.
	\]
	Therefore,
	\[
	|\nabla\rho(t,x)|+|\nabla v(t,x)|
	\leq
	C_{T,\Lambda}\langle x\rangle^{-\Lambda},
	\qquad
	0\leq t\leq T.
	\]

	\section*{Acknowledgments}
	X. Huang is partially supported by Chinese Academy of Sciences Project for Young Scientists in Basic Research (Grant No. YSBR-031), National Natural Science Foundation of China (Grant Nos. 12494542, 11688101) and National Key R\&D Program of China (Grant No. 2021YFA1000801). 
	
	\vspace{1cm}
	\noindent\textbf{Data availability statement.} Data sharing is not applicable to this article.
	
	\vspace{0.3cm}
	\noindent\textbf{Conflict of interest.} The authors declare that they have no conflict of interest.
	
\end{document}